\documentclass[english, 11pt]{article}

\usepackage{custom_tex}
\allowdisplaybreaks
\usepackage[page]{appendix}
\newcommand{\titlename}{Global and Individualized Community Detection in Inhomogeneous Multilayer Networks}

\newcommand{\floor}[1]{{\left\lfloor {#1} \right \rfloor}}

\newcommand{\globest}{global~estimation}
\newcommand{\Globest}{Global~estimation}
\newcommand{\GlobEst}{Global~Estimation}
\newcommand{\indest}{individualized~estimation}
\newcommand{\Indest}{Individualized~estimation}
\newcommand{\IndEst}{Individualized~Estimation}
\newcommand{\sakura}{\mathscr{E}_{S, x, \xi}}
\newcommand{\clover}{\mathscr{F}_{S, \xi}}

\newcommand{\R}{\mathbb{R}}

\newcommand{\pell}{{p_{\ell}}}
\newcommand{\qell}{{q_{\ell}}}
\newcommand{\tpell}{{\tilde p_{\ell}}}
\newcommand{\tqell}{{\tilde q_{\ell}}}
\newcommand{\tran}{^\intercal}
\newcommand{\modelname}{IMLSBM}
\newcommand{\sname}{IMLSBM}
\newcommand{\mlsbm}{MLSBM}

\newcommand{\kl}{\textnormal{KL}}

\newcommand{\bz}{\mathbf{z}}
\newcommand{\bhz}{\mathbf{\hat z}}
\newcommand{\btz}{\mathbf{\tilde z}}
\newcommand{\zzz}[2][]{{\bz}^{{#1}(#2)}}
\newcommand{\hzzz}[2][]{{\bhz}^{{#1}(#2)}}
\newcommand{\tzzz}[2][]{{\btz}^{{#1}(#2)}}
\newcommand{\barzzz}[2][]{{\bar{\bz}}^{{#1}(#2)}}
\newcommand{\ppp}[2][]{{p}^{{#1}(#2)}}

\newcommand{\nnn}[2][]{n^{{#1}(#2)}}

\newcommand{\fff}[2][]{f^{{#1}(#2)}}
\newcommand{\XXX}[2][]{X^{{#1}(#2)}}
\newcommand{\tXXX}[2][]{\tilde{X}^{{#1}(#2)}}
\newcommand{\YYY}[2][]{Y^{{#1}(#2)}}
\newcommand{\tYYY}[2][]{\tilde{Y}^{{#1}(#2)}}
\newcommand{\ZZZ}[2][]{Z^{{#1}(#2)}}

\newcommand{\AAA}[2][]{A^{{#1}(#2)}}
\newcommand{\BBB}[2][]{B^{{#1}(#2)}}

\newcommand{\III}[2][]{{I}^{{#1}(#2)}}
\newcommand{\MMM}[2][]{{M}^{{#1}(#2)}}
\newcommand{\mmm}[2][]{{m}^{{#1}(#2)}}
\newcommand{\tmmm}[2][]{{\tilde m}^{{#1}(#2)}}

\newcommand{\weightvec}{\boldsymbol{\omega}}
\newcommand{\weight}{\omega}
\newcommand{\ham}{\textnormal{H}}
\newcommand{\RNum}[1]{\uppercase\expandafter{\romannumeral #1\relax}}
\newcommand{\Indicator}{\mathds{1}}
\newcommand{\indicator}[1]{{\mathds{1}{\left\{{#1}\right\}}}}
\newcommand{\globinfo}{\mathcal{I}}
\newcommand{\indinfo}{\mathcal{J}}

\newcommand{\sbm}{\textnormal{SBM}}

\newcommand{\oone}{\left(1+o(1)\right)}

\newcommand{\init}{{\texttt{init}}}            
\allowdisplaybreaks

\begin{document}
\title{\titlename}

\author[a]{Shuxiao Chen\thanks{Email:
\href{mailto:shuxiaoc@wharton.upenn.edu}{shuxiaoc@wharton.upenn.edu}}}
\author[b]{Sifan Liu\thanks{Email: 
\href{mailto:sfliu@stanford.edu}{sfliu@stanford.edu}}}
\author[a]{Zongming Ma\thanks{Email: 
\href{mailto:zongming@wharton.upenn.edu}{zongming@wharton.upenn.edu}}}
\affil[a]{\textit{University of Pennsylvania}}
\affil[b]{\textit{Stanford University}}

\date{}
\maketitle

\begin{abstract}
In network applications, it has become increasingly common to obtain datasets in the form of multiple networks observed on the same set of subjects, where each network is obtained in a related but different experiment condition or application scenario.
Such datasets can be modeled by multilayer networks where each layer is a separate network itself while different layers are associated and share some common information.
The present paper studies community detection in a stylized yet informative inhomogeneous multilayer network model. 
In our model, layers are generated by different stochastic block models, the community structures of which are (random) perturbations of a common global structure while the connecting probabilities in different layers are not related.
Focusing on the symmetric two block case, we establish minimax rates for both \emph{\globest} of the common structure and \emph{\indest} of layer-wise community structures. 
Both minimax rates have sharp exponents. 
In addition, we provide an efficient algorithm that is simultaneously asymptotic minimax optimal for both estimation tasks under mild conditions. 
The optimal rates depend on the \emph{parity} of the number of most informative layers, a phenomenon that is caused by inhomogeneity across layers.
The method is extended to handle multiple and potentially asymmetric community cases. We demonstrate its effectiveness on both simulated examples and a real multi-modal single-cell dataset.
\end{abstract}

\addtocontents{toc}{\protect\setcounter{tocdepth}{2}}
\tableofcontents

\section{Introduction}  

Network data is among the most common types of relational data.
As a fundamental task in network data analysis \citep{wasserman1994social,goldenberg2010survey}, 
\emph{community detection}
refers to the problem of partitioning the nodes of a network into clusters so that intra-cluster nodes are connected in a different way from inter-cluster nodes, usually more densely.
\emph{Stochastic block model} (SBM) \cite{holland1983stochastic}
is a canonical model for studying community detection.
In an SBM, $n$ nodes are partitioned into $k$ disjoint subsets.
Each unordered pair of nodes are connected independently with probability $p$ if they come from the same community and with a different probability $q$ otherwise.
The observed connection pattern is encoded in an $n\times n$ symmetric adjacency matrix $A$.
Here the goal of community detection is to, upon observing $A$, estimate the partitioning of nodes.  
The stochastic block model, albeit simple, has found its success in many fields of science (see, e.g., \cite{jackson2010social,lu2011link,jackson2011overview,garcia2018applications}).
It has also undergone a plethora of theoretical investigations \cite{abbe2017community}. 
With joint efforts from mathematics, statistics, and computer science, 
we not only have a large algorithmic toolbox for detecting communities in SBMs, but also know the {information-theoretic limits} of this task as well as
{which algorithms are optimal}. 

Despite its popularity, 
stochastic block model focuses only on
a {single} adjacency matrix. 
This is in sharp contrast to the widely recognized fact that real world networks are often superpositions of multiple networks (layers), each encoding a potentially different but correlated interaction pattern among the same set of nodes \cite{kivela2014multilayer,boccaletti2014structure}. 
For example, in social network data, different layers could correspond to different types of relationship that link the social entities, and the information presented in the friendship network, compared to that in the professional network, is different, but not completely unrelated \cite{dickison2016multilayer}. 
Another example is given by the network representation of human brains, where each layer corresponds to an individual person's functional brain network.
It is well known that the parcellation of brain regions into different functional units are different but strongly correlated among human beings \cite{de2017multilayer}. 

A natural attempt at generalizing SBMs to multilayer networks is as follows.
Let us focus on the symmetric two block case where in each layer all nodes are partitioned into two blocks of roughly equal sizes.
Instead of observing a single adjacency matrix, the data analyst is now presented with a collection of $L$ adjacency matrices $\{\AAA{\ell}\}_{\ell=1}^L$. 
To model that ``communities in different layers are different but correlated'', we take a hierarchical modeling approach. 
A \emph{global} community assignment vector $\bz^\star \in \{\pm 1\}^n$ is introduced in our model.
To ensure symmetry, we require $\bz^\star$ to have roughly equal numbers of $1$'s and $-1$'s.
We let the \emph{individual} community assignments $\{\zzz{\ell}\}_1^L$ be independent samples from the following distribution:
\begin{equation}
	\label{eq:label_flip}
	\forall\ell\in[L], i\in[n], \qquad \zzz{\ell}_i ~{\stackrel{ind}{\sim}}~  \bz^\star_i \times \big[2 \Bern(1-\rho) - 1\big]. 
\end{equation}
Here, $[L] = \{1, \hdots, L\}, [n] = \{1, \cdots, n\}$ and $\Bern(\cdot)$ is a Bernoulli random variable.
That is, in a specific layer $\ell\in[L]$, with probability $1-\rho$, the community membership of the $i$-th node agrees with the global one, $\bz^\star_i$, and with probability $\rho$, it ``flips'' to the opposite side $-\bz^\star_i$. The parameter $\rho$ controls the level of \emph{inhomogeneity} across layers. When $\rho = 0$, all layers share the same community structure, whereas when $\rho=1/2$, the community structures across layers are mutually uninformative.
Upon realizations of $\zzz{\ell}$'s, the adjacency matrices are independently generated by
\begin{equation}
	\label{eq:layerwise_sbm}
\begin{aligned}
	\AAA{\ell}_{ij} = \AAA{\ell}_{ji} \stackrel{ind}{\sim} \Bern(p_\ell)\cdot \indc{\zzz{\ell}_i = \zzz{\ell}_j} + \Bern(q_\ell) \cdot \indc{\zzz{\ell}_i \neq \zzz{\ell}_j},~~~~ & \\
	~~ \forall i\neq j\in[n],\,\ell\in [L],	&
\end{aligned}
\end{equation}
where $\indc{\cdot}$ is the indicator function, and all diagonal entries are zeros.
In other words, the $\ell$-th layer network is generated by an SBM with community partitioning specified by $\zzz{\ell}$, intra-community connection probability $p_\ell$ and inter-community connection probability $q_\ell$.
Connection probabilities across different layers are not linked in any way.
The foregoing generalization, to the best of our knowledge, was first introduced by \citet{paul2018random}, which they termed as the \emph{random effects stochastic block model}. 
In \cite{paul2018random}, the random effects \eqref{eq:label_flip} can take other forms. Hence, to avoid confusion, we term the model in \eqref{eq:label_flip}--\eqref{eq:layerwise_sbm} as the \emph{inhomogeneous multilayer stochastic block model} (\modelname).
Clearly, the model can be generalized in obvious ways to include more than two communities and unequal community sizes. 
However, the present manuscript shall focus on the foregoing symmetric two block case as it is the simplest nontrivial model that reveals key new phenomena of community detection in inhomogeneous multilayer networks.

The goal of community detection in an \modelname~is now two-fold---upon observing $\{\AAA{\ell}\}_1^L$, we are interested in:
\begin{enumerate}
	\item \emph{\Globest.} Estimating the global community assignment $\bz^\star$;
	\item \emph{\Indest.} Estimating each of the individual assignments $\{\zzz{\ell}\}_1^L$.
\end{enumerate}
Global estimation needs to aggregate connection patterns across layers to better infer the global consensus structure, an instance of \emph{integrative data analysis} \cite{curran2009integrative}. On the other hand, \indest~requires borrowing information from different layers to better estimate the layer-wise community structure, an example of \emph{multi-task learning} \cite{caruana1997multitask}. 

Theoretical understanding of community detection in an \modelname~is lacking. Partial results exist in the homogeneous case ($\rho=0$) where \globest~and \indest~coincide. Here, by ``homogeneity'' we mean the layer-wise community structures are shared across all layers, and the connecting probabilities are allowed to differ. Under such a setup, it has been proved by \citet{paul2016consistent} that the minimax rates for expected proportion of  misclustered nodes scales as
\begin{equation}
	\label{eq:rho_zero_rate}
	\exp\bigg\{ -\big(1+o(1)\big) \cdot \frac{n}{2}\sum_{\ell\in[L]} (\sqrt{p_\ell} - \sqrt{q_\ell})^2 \bigg\},
\end{equation}
provided that the exponent diverges to infinity as $n$ tends to infinity. They in fact established the rates for a more general setting than the symmetric two block case considered in this paper.
Later, a polynomial-time algorithm that achieves this rate was proposed by \citet{xu2020optimal}. Nonetheless, it is unclear how to generalize their results to the inhomogeneous setting.
From an algorithmic perspective, spectral clustering \cite{bhattacharyya2018spectral,lei2020tail,paul2020spectral} and least-square estimators \cite{lei2020consistent,wang2019multiway} have been proposed and justified to be consistent (i.e., achieving an $o(1)$ misclustering proportion with high probability) under homogeneity. 
However, it is unknown whether any of these methods attains the information-theoretic limit \eqref{eq:rho_zero_rate}.

Although the homogeneous case ($\rho=0$) is interesting in its own right, it is the inhomogeneous case ($\rho > 0$) that characterizes our inductive bias --- ``layers are different but correlated''. 
In \cite{paul2018random}, a few heuristic algorithms were introduced and their performances were assessed by simulations. 
To the best of our limited knowledge, in the inhomogeneous regime, no algorithm with provable guarantee for either global or individualized estimation is known in the literature, let alone any optimality statement.

\subsection{Main Contributions}

The main contributions of the present manuscript are two-fold.
First, we give precise characterization of information-theoretic limits of both global and individualized community detection in a {symmetric two block} \modelname~when $\rho =o(1)$;
Moreover, we provide a polynomial-time algorithm that simultaneously attains information-theoretic limits for both global and \indest~under mild conditions.
We reiterate that results in the present manuscript are obtained under the symmetric two block setting which has already posed highly nontrivial theoretical and algorithmic challenges.
We leave extensions to more general settings for future work.

To provide an overview of our main results, we start with several key information-theoretic quantities that will appear throughout this paper.  
For any $\ell\in[L]$ and $t\in[0, 1]$, define
\begin{align}
	\label{eq:layerwise_info}
	\III{\ell}_t & : = -\log\big[p_\ell^{1-t}q_\ell^t + (1-p_\ell)^{1-t}(1-q_\ell)^t\big]\big[p_\ell^tq_\ell^{1-t} + (1-p_\ell)^t(1-q_\ell)^{1-t}\big].
\end{align}
The quantity $\III{\ell}_t$ can be regarded as the \emph{signal strength} of the $\ell$-th layer. 
When $p_\ell\asymp q_\ell =o(1)$, one can show that 
$
	\III{\ell}_{1/2} = \oone (\sqrt{p_\ell} - \sqrt{q_\ell})^2,
$
and hence the minimax rate for community detection {(i.e., the worst-case misclustering proportion)} in an SBM with community assignment $\zzz{\ell}$, intra-community connection probability $p_\ell$ and the inter-community connection probability $q_\ell$ derived in \citet{zhang2016minimax} can be equivalently written as
{$e^{-\oone n\III{\ell}_{1/2}/2}$,}
as long as the exponent tends to infinity.
For any collection of layers $S\subseteq[L]$, let
\begin{equation}
	\label{eq:cum_gen_fun}
	\psi_S(t) := -\frac{n}{2}\sum_{\ell \in S} \III{\ell}_t, \qquad    \psi_S^\star(a) := \sup_{0\leq t \leq 1} at-\psi_S(t).
\end{equation}
The function $\psi_S^\star$
characterizes the \emph{collective signal strength for layers in $S$}.
Indeed, from the definition of $\III{\ell}_t$, one readily checks that 
$
	\psi_S^\star(0) = -\psi_S(1/2) =(n/2) \sum_{\ell\in S}\III{\ell}_{1/2},
$
and thus the minimax rate \eqref{eq:rho_zero_rate} for community detection in a homogeneous multilayer SBM can be expressed equivalently as 
$
	e^{-\oone \psi_{[L]}^\star(0)}. 
$
Intuitively, inhomogeneity ($\rho > 0$) 
introduces additional noises. 
To characterize the \emph{noise level}, define 
\begin{equation}
	\label{eq:J_rho}
	{J_\rho} :=  -\log 2\sqrt{\rho(1-\rho)}.
\end{equation}
Since $J_0= \infty$ and $J_{1/2} = 0$, one can effectively think of $J_\rho$ as a measure of proximity of the individual layer community assignments $\{\zzz{\ell}\}_1^L$ to the global assignment $\bz^\star$.
{Note that both $\III{\ell}_t$ and $J_\rho$ can be written as convex combinations of R\'enyi divergences \cite{renyi1961measures}}
between Bernoulli distributions. 
Specifically, we have
$\III{\ell}_t = (1-t) D_t(p_\ell\| q_\ell) + t D_{1-t}(p_\ell\| q_\ell)$, and $J_\rho = \frac{1}{2}D_{1/2}(\rho \| 1-\rho)$, where $D_{t}(p\| q)$ is the \emph{R\'enyi divergence of order $t$} between $\textnormal{Bern}(p)$ and $\textnormal{Bern}(q)$.

\subsubsection{Global estimation error}
With the foregoing definitions,
we first show that, under certain regularity conditions, the minimax rate for \globest, measured in terms of proportion of misclustered nodes, is given by
\begin{equation}
	\label{eq:resbm_glob_rate}
	\exp\big\{-\oone \min_{S\subseteq[L]} \globinfo_S\big\},
\end{equation}
where $\globinfo_{S}$ represents the \emph{signal-to-noise ratio (SNR) for \globest} for layers in $S$, defined as
\begin{equation}
	\label{eq:globinfo}
	\globinfo_S := 
	\begin{cases}
		|S^c| J_\rho + \psi_S^\star(0) & \textnormal{ if } |S^c| \textnormal{ is even},\\
		(|S^c| + 1) J_\rho + \psi^\star_S(-2 J_\rho) & \textnormal{ if } |S^c| \textnormal{ is odd}.
	\end{cases}
\end{equation}

The minimax error rate \eqref{eq:resbm_glob_rate} exhibits two intriguing properties. First, when $\rho = 0$, we have $J_\rho = \infty$, and thus the only way to make $\globinfo_S$ finite is to choose $S = [L]$, which gives
$
	\min_{S\subseteq[L]} \globinfo_S = \globinfo_{[L]} = \psi_{[L]}^\star(0).
$
As the result, \eqref{eq:resbm_glob_rate} recovers the minimax rate in a homogeneous multilayer SBM given in \eqref{eq:rho_zero_rate}. Second, the SNR for layers in $S$ takes different forms according to the \emph{parity} of $S^c$, a phenomenon induced by inhomogeneity across layers.

\subsubsection{Individualized estimation error}
In correspondence, the minimax rate for \indest~for the $\ell$-th layer, measured by proportion of misclustered nodes, is given by
\begin{equation}
	\label{eq:resbm_ind_rate}
	\exp\big\{-\oone \min_{S\subseteq[L]\setminus\{\ell\}} \globinfo_{S\cup\{\ell\}}\big\} + \exp\big\{ -\oone\indinfo_{\{\ell\}}\big\} ,
\end{equation}
where $\indinfo_{\{\ell\}}$ is a suitably defined quantity (see \eqref{eq:indinfo} for a precise definition) that measures the \emph{SNR for \indest} for the $\ell$-th layer. 
We briefly mention here that similar to 
\eqref{eq:resbm_glob_rate}, the last display can recover the minimax rate in a homogeneous multilayer SBM by setting $\rho = 0$, and it crucially depends on the parity of the ``most informative'' set $S^c$ as well. 
We refer readers to Sections \ref{sec:lower_bound} and \ref{sec:minimax_rate} for details.

\subsubsection{Algorithm}
We propose an algorithm that achieves the optimal rates in \eqref{eq:resbm_glob_rate} and \eqref{eq:resbm_ind_rate} simultaneously under mild conditions.
The idea stems from \emph{maximum a posteriori} (MAP) estimation. 
Note that \modelname~is a hierarchical model where individual community assignments $\{\zzz{\ell}\}_1^L$ are drawn from the prior distribution \eqref{eq:label_flip}. 
It is thus tempting to write out the posterior of $\{\zzz{\ell}\}_1^L$ given the observed data $\{\AAA{\ell}\}_1^L$ and maximize the posterior density with respect to the parameters $(\bz^\star, \{\zzz{\ell}\}_1^L)$. 
A naive implementation of this strategy is doomed to fail, due to the fact that the MAP objective function gives rise to a combinatorial optimization problem, whose search space has cardinality $2^{n(L+1)}$. 
To bypass the combinatorial search, we adopt a two-stage ``warm-start'' MAP algorithm. In the first stage, an initial estimator of $\bz^\star$ is obtained using spectral clustering on a trimmed version of the weighted average of layer-wise adjacency matrices. 
In the second stage, a refined estimator of $\bz^\star$ and estimators of $\{\zzz{\ell}\}$ are simultaneously obtained by optimizing a ``decoupled'' MAP objective function, which can be computed in linear time (in $n$ and $L$). 
Although the formal definition of the MAP refinement step requires knowledge of $\rho$, our numerical experiments later show that the outcome is not sensitive to misspecification of $\rho$.

While ``spectral clustering + refinement'' procedures have appeared in community detection in \sbm s and variants
(e.g., \cite{mossel2015consistency,gao2017achieving,yun2016optimal}),
our algorithm has novelties in both stages, especially in their technical analysis.
For Stage \RN{1}, compared to existing analyses of spectral clustering for homogeneous multilayer SBMs \cite{bhattacharyya2018spectral,lei2020tail,paul2020spectral}, 
our analysis is novel in that we establish a \emph{stability} result, asserting that inhomogeneity hurts spectral clustering error rate by at most an additive factor of $\textnormal{poly}(\rho)$. 
The proof is based on a new concentration inequality on the spectral norm of a weighted average of Bernoulli random matrices, which is derived via a nontrivial generalization of the graph decomposition approach in \citet{le2017concentration} to multilayer networks. 
The concentration inequality improves the ones used in the existing work (e.g., \cite{bhattacharyya2018spectral,paul2020spectral}) and could be of independent interest.
For Stage \RN{2}, due to presence of multiple layers, devising a refinement scheme with time complexity that is polynomial in the number of layers presents new challenges.
In addition, due to inhomogeneity, the analysis 
is considerably more involved. 
A key step towards establishing matching upper bounds lies in a novel application of Sion's minimax theorem \cite{sion1958general}.

\subsection{Related Work}
The past decade has witnessed a venerable line of work on the theoretical development of community detection for SBMs. 
Optimal algorithms have been developed under various criteria, including 
(1) weak recovery, where the best achievable goal is to cluster the nodes better than random guess \cite{decelle2011asymptotic,mossel2012stochastic,mossel2018proof,bordenave2015non,mossel2015reconstruction,chen2014statistical,neeman2014non,montanari2015finding,abbe2018proof}; 
(2) exact recovery, where the requirement is to reconstruct from data the ground truth up to relabeling \cite{mossel2015consistency,abbe2015exact,abbe2015community}; 
and more related to our formulation, 
(3) almost exact recovery, where the goal is to output a community assignment with vanishing misclustering error \cite{yun2014community,mossel2015consistency,abbe2015exact,gao2017achieving}. 
The study under the minimax framework was initiated by \cite{zhang2016minimax,abbe2015community,yun2016optimal} and was later extended to more general settings such as \cite{gao2018community,xu2020optimal}. 
The above list of work is by no means exhaustive and we refer the readers to the review papers \cite{abbe2017community,li2018convex,gao2018minimax} for a more systematic account.

In comparison, study of community detection in multilayer networks is still in its early stage. 
Initial works in this area have focused on algorithmic developments (see, e.g., \cite{mucha2010community,kumar2011co,de2013mathematical,de2015identifying,peixoto2015inferring}), and most theoretical studies are restricted to the homogeneous case where all layers share the same community structure \cite{paul2016consistent,bhattacharyya2018spectral,paul2018random,paul2020spectral,lei2020consistent,lei2020tail,zhang2020flexible}. 
There are a few exceptions, such as \cite{stanley2016clustering,le2018estimating,arroyo2019inference,jing2020community}, where consistency has been established under several inhomogeneity-aware variants of SBMs while optimality results are missing. 

In addition to estimating community structures, there is a related line of work aiming at testing whether community structures across layers are indeed correlated or equivalent. 
We refer interested readers to recent papers \cite{gao2019testing,gao2020clusterings,gao2019testingEquiv} and references therein.

\subsection{Paper Organization}
The rest of the paper is organized as follows. 
In Section \ref{sec:lower_bound}, we provide our construction of minimax lower bounds for both global and \indest. 
We present the two-stage algorithm in Section \ref{sec:alg}, and its theoretical analysis is given in Section \ref{sec:minimax_rate}. 
We conduct numerical experiments to corroborate our theoretical results in Section \ref{sec:exp}. 
{Section \ref{sec:multi_cluster} extends our algorithm to multi-cluster and asymmetric cases. 
We finally illustrate our algorithm in a multi-modal single-cell dataset in Section \ref{sec:real_data}.
For brevity, additional theoretical and numerical results, as well as the technical proofs are deferred to the supplementary material.}
\subsection{Notation} We conclude this section by introducing some notations that will be used throughout this paper. For a positive integer $n$, we write $[n]:=\{1, \hdots, n\}$. 
Given $a, b \in \bbR$, we denote $a\lor b := \max\{a, b\}$ and $a\land b := \min\{a, b\}$.
For a set $S$, we let $\Indc_S$ be its indicator function and we use $\# S$ and $|S|$ interchangeably to denote its cardinality. 
For two positive sequences $\{a_n\}$ and $\{b_n\}$, we write 
$a_n \lesssim b_n$ or $a_n = \calO(b_n)$ to denote $\limsup a_n/b_n < \infty$,  
and we let $a_n \gtrsim b_n$ or $a_n = \Omega(b_n)$ to denote $b_n \lesssim a_n$. 
Meanwhile, the notation $a_n\asymp b_n$ or $a_n = \Theta(b_n)$ means $a_n \lesssim b_n$ and $a_n \gtrsim b_n$ simultaneously.
Moreover, we write $a_n\ll b_n$ to mean $b_n/a_n\to\infty$ and $a_n\gg b_n$ to mean $b_n \ll a_n$.
For a vector $x$, we let $\|x\|_p$ denote its $\ell_p$ norm, and we write $\|x\|_2 = \|x\|$ when there is no ambiguity. 
For a matrix $A$, we let $\|A\|_F$ be its Frobenius norm and $\|A\|_{p\to q}$ be its $\ell_p$ to $\ell_q$ operator norm. We will write $\|A\|_{2\to 2} =  \|A\|_2 = \|A\|$ when there is no ambiguity.

\section{Fundamental Limits and Costs of Inhomogeneity}\label{sec:lower_bound}

In this section, we present minimax lower bounds for estimating both $\bz^\star$ and individual $\zzz{\ell}$'s. To start with, let us recall that a two-block \sname~parameterized by $(\bz^\star, \rho, \{p_\ell\}_1^L, \{q_\ell\}_1^L)$ is a probability measure on a multilayer network, whose adjacency matrices $\{\AAA{\ell}\}_1^L$ 
are generated according to \eqref{eq:label_flip}--\eqref{eq:layerwise_sbm}.

\paragraph{Parameter space and the loss function.}
Let $n^\star_\pm(\bz^\star) = \sum_{i\in[n]}\indc{\bz^\star = \pm1}$ be the sizes of the positive and negative clusters of $\bz^\star$, respectively. 
We propose to consider the following collection of \sname s:
\begin{align}
	\label{eq:param_space}
	\calP_n(\rho, \{p_\ell\}_1^L, \{q_\ell\}_1^L, \beta):= \bigg\{ \textnormal{\sname}(\bz^\star, \rho, \{p_\ell\}_1^L, \{q_\ell\}_1^L) : \ \frac{n}{2\beta} \leq n^\star_\pm(\bz^\star) \leq  \frac{n\beta}{2},  p_\ell > q_\ell \ \forall \ell \bigg\}.
\end{align}
As we focus on the symmetric case, the constant $\beta\geq 1$ is taken to be $1+o(1)$ as $n\to\infty$. The rest of the quantities appearing above, namely $L, \rho, \{p_\ell\}_1^L, \{q_\ell\}_1^L$, are all allowed to scale with $n$. For an estimator $\bhz^\star$ of $\bz^\star$, we evaluate its performance by the misclustering proportion, defined as
\begin{align}
	\label{eq:loss_func}
	\calL(\bhz^\star, \bz^\star) : = \frac{d_\ham(\bhz^\star, \bz^\star) \land d_\ham(-\bhz^\star, \bz^\star)}{n}, 
\end{align}
where $d_\ham(\bz, \bz') = \sum_{i\in[n]}\indc{\bz_i\neq \bz_i'}$ is the Hamming distance between $\bz$ and $\bz'$, and the minimum is taken because $\bhz^\star$ and $-\bhz^\star$ give rise to the same partitioning of nodes. 
Similarly, for an estimator $\hzzz{\ell}$ of $\zzz{\ell}$, we evaluate its performance by $\calL(\hzzz{\ell}, \zzz{\ell})$.

\paragraph{An idealized setup.}
To characterize the information-theoretic limits in estimating $\bz^\star$ and $\zzz{\ell}$'s, we will consider an idealized setup as follows. Suppose $n = 2m+1$ and the nodes are labeled as $0 , 1, \hdots, 2m$. Consider a global assignment vector $\bz^\star$ whose value on $i\in\{1, \hdots, 2m\}$ is known to us: 
\begin{equation}
\label{eq:idealized_setup}
\bz^\star_{i} = 
\begin{cases}
+1 & 1\leq i\leq m\\
-1 & m+1\leq i\leq 2m.
\end{cases}
\end{equation}
We further observe $\{\AAA{\ell}\}_1^L \sim \textnormal{\sname}(\bz^\star, \rho, \{p_\ell\}_1^L, \{q_\ell\}_1^L)$. And our goal is to estimate $\bz^\star_0$ as well as $\zzz{\ell}_0$'s, a substantially simplified task compared to the original one. 
Such a strategy of ``reducing'' the task of doing inference for the whole parameter vector to doing inference for each coordinate is an instance of the celebrated Assouad's method \cite{assouad1983deux}, and has been successfully used in many recent works on characterizing the fundamental limits of community detection in SBMs and variants (see, e.g., \cite{zhang2016minimax,gao2017achieving,gao2018community}).
Our discussion in the rest of this section will largely rely on the intuitions built upon this idealized setup, and we refer the readers to Section \ref{append:prf_lower_bound} for fully rigorous proofs.

\subsection{Minimax Lower Bound for \GlobEst}\label{subsec:lower_bound_z_star}
Consider the task of estimating $\bz^\star$. Under the idealized setup \eqref{eq:idealized_setup}, our goal is to differentiate between $\bz^\star_0 = +1$ and $\bz^\star_0 = -1$ based on the data $\{\AAA{\ell}\}_1^L$, which gives rise to a binary hypothesis testing problem:
\begin{equation}
	\label{eq:hypothesis_z_star_intractable}
	H_0: \bz^\star_0 = +1 \ \ \ \textnormal{v.s.} \ \ \ H_1: \bz^\star_0 = -1.
\end{equation}
By Neyman--Pearson lemma, in principle, we can characterize the difficulty of the above testing problem by calculating the error made by the likelihood ratio test. However, due to the complicated structure of the likelihood function, this strategy is analytically intractable, calling for further simplifications.

\subsubsection{The fundamental testing problem}
If $\zzz{\ell}_0$ is actually observed by us, then since $\zzz{\ell}_0\sim2 \textnormal{Bern}(1-\rho)-1$ under $H_0$ and $\zzz{\ell}_0\sim 2\textnormal{Bern}(\rho)-1$ under $H_1$, deciding the value of $\bz^\star_0$ from $\zzz{\ell}_0$ is equivalent to the problem of differentiating between
\begin{equation}
	\label{eq:hypothesis_z_star_observed_zl}
	H_0: 2 \textnormal{Bern}(1-\rho)-1 \ \ \ \textnormal{v.s.} \ \ \ H_1:  2\textnormal{Bern}(\rho)-1.
\end{equation}
{In reality, we need to estimate each $\zzz{\ell}_0$ from its corresponding $\AAA{\ell}$}. This is the community detection problem in a vanilla two-block SBM, whose fundamental difficulty is characterized by the following testing problem \citep{gao2017achieving}:
\begin{equation}
	\label{eq:hypothesis_z_star_vanilla_sbm}
	H_0: \bigotimes_{i\in[m]} \textnormal{Bern}(p_\ell) \otimes \bigotimes_{i\in[m]} \textnormal{Bern}(q_\ell)  \ \ \ \textnormal{v.s.} \ \ \ H_1: \bigotimes_{i\in[m]} \textnormal{Bern}(q_\ell) \otimes \bigotimes_{i\in[m]} \textnormal{Bern}(p_\ell),
\end{equation}
where $\otimes$ denotes the product of probability measures.

Intuitively, if the signal strength in the $\ell$-th layer is strong enough, then we are close to the case of known $\zzz{\ell}_0$, in which the error for testing \eqref{eq:hypothesis_z_star_intractable} is mainly captured by that of testing \eqref{eq:hypothesis_z_star_observed_zl}. On the other hand, if we have barely any signal in the $\ell$-th layer, then we are in the unknown $\zzz{\ell}_0$ case, in which the error for testing \eqref{eq:hypothesis_z_star_intractable} mainly comes from testing \eqref{eq:hypothesis_z_star_vanilla_sbm}. This intuition is formalized in the following lemma.

\begin{lemma}[The fundamental testing problem for \globest]
\label{lemma:testing_problem_z_star}
{Assume $\log L \ll n^{c}$ for some $c \in (0, 1)$}. Then for any sequence $\delta_n=o(1)$, there exists another sequence $\delta'_n = o(1)$ satisfying $(1+\delta_n')n/2\in \bbN$, such that for any $S\subseteq [L]$,  
\begin{equation}
\label{eq:lower_bound_by_testing_problem_z_star}
  \inf_{\bhz^\star} \sup_{\bz^\star\in \calP_n} \bbE \calL(\bhz^\star, \bz^\star) \gtrsim \delta_n \inf_{\phi}\bigg(\bbE_{H_0}[\phi] + \bbE_{H_1}[1-\phi] \bigg),
\end{equation}
where $\phi$ is a testing function of the following problem:
\begin{align}
  & H_0: \bigg(\bigotimes_{\ell\in S^c} [2 \textnormal{Bern}(1-\rho) - 1]\bigg) \otimes \bigg(\bigotimes_{\ell\in S} \bigotimes_{i = 1}^{(1+\delta'_n)n/2} \textnormal{Bern}(p_\ell)\otimes \textnormal{Bern}(q_\ell)\bigg)\nonumber\\
  \label{eq:fund_testing_problem_z_star}
  & \qquad \textnormal{v.s. } H_1: \bigg(\bigotimes_{\ell\in S^c} [2 \textnormal{Bern}(\rho) - 1]\bigg) \otimes \bigg(\bigotimes_{\ell\in S} \bigotimes_{i = 1}^{(1+\delta'_n)n/2} \textnormal{Bern}(q_\ell)\otimes \textnormal{Bern}(p_\ell)\bigg).
\end{align}
\end{lemma}
\begin{proof}
	See Section \ref{subappend:prf_testing_problem_z_star}.
\end{proof}

\begin{remark}
\label{rmk:opt_S_lower_bound_z_star}
Though the lower bound \eqref{eq:lower_bound_by_testing_problem_z_star} holds for an arbitrary $S\subseteq[L]$, by our previous intuition, it is the tightest if we choose $S$ such that $S^c$ corresponds to layers with high SNRs, where \eqref{eq:hypothesis_z_star_observed_zl} kicks in, and $S$ corresponds to layers with low SNRs, where the error from \eqref{eq:hypothesis_z_star_vanilla_sbm} dominates. 
\end{remark}

\subsubsection{Optimal testing error and parity of \texorpdfstring{${|S^c|}$}{Sc}} 
By Neyman--Pearson lemma, the test that gives the optimal Type-I plus Type-II error is the likelihood ratio test with a cutoff of $1$. 
{For the problem \eqref{eq:fund_testing_problem_z_star}, it can be shown (see Section \ref{subappend:prf_testing_problem_z_star} for a detailed derivation) that the optimal error is given by
\begin{align}
	\label{eq:error_of_LR_test_z_star}
	\bbP\bigg(\sum_{\ell\in S} \log \bigg(\frac{q_\ell(1-p_\ell)}{p_\ell(1-q_\ell)}\bigg) \cdot \sum_{i=1}^{(1+\delta'_n)n/2} (\XXX{\ell}_i-\YYY{\ell}_i) \geq \log \bigg(\frac{1-\rho}{\rho}\bigg) \cdot \sum_{\ell\in S^c}\ZZZ{\ell}\bigg),
\end{align}
where 
\begin{equation}
\label{eq:def_of_rvs_in_testing}
\{\XXX{\ell}_i\}_{i=1}^n \overset{\textnormal{i.i.d.}}{\sim} \textnormal{Bern}(p_\ell), \ \ \ \{\YYY{\ell}_i\}_{i=1}^n \overset{\textnormal{i.i.d.}}{\sim} \textnormal{Bern}(q_\ell), \ \ \ \{\ZZZ{\ell}\}_{\ell=1}^L  \overset{\textnormal{i.i.d.}}{\sim} 2\textnormal{Bern}(1-\rho) - 1,
\end{equation}
all of which are mutually independent.}

Readers with an expertise in {large deviation principles} may have noticed that \eqref{eq:error_of_LR_test_z_star} is the tail probability of a sum of independent random variables, and a tight characterization of this probability should involve the cumulant generating functions (CGFs) of those random variables as well as their rate functions (i.e., the Legendre transforms of CGFs). 
This intuition explains the appearance of the two key information-theoretic quantities, namely $\psi_S^\star$ and $J_\rho$, in the definition of $\globinfo_S$ in \eqref{eq:globinfo}. As one can check, $\psi_S(t)$ is precisely the CGF of the random variable that appears to the left of the ``$\geq$'' sign in \eqref{eq:error_of_LR_test_z_star} ({if we set $\delta_n'=0$}), and $-J_\rho$ is the CGF of $\log\big((1-\rho)/\rho\big)\cdot \ZZZ{\ell}$ evaluated at $1/2$. 

The following lemma gives the asymptotically optimal testing error for \eqref{eq:fund_testing_problem_z_star}.

\begin{lemma}[Optimal testing error for \globest]
\label{lemma:opt_test_error_z_star}
Assume $\rho = o(1)$ and that {there exist constants $C_1, C_2 > 1, c\in(0, 1)$ such that $C_1 q_\ell \leq p_\ell \leq (C_2 q_\ell) \land (1-c)$ for any $\ell\in[L]$.} 
Then there exists a sequence $\delta''_n=o(1)$ which is independent of $S$, such that the probability \eqref{eq:error_of_LR_test_z_star} is lower bounded by
\begin{equation}
  \label{eq:opt_test_error_z_star}
  C\cdot \exp\big\{- (1+\delta_n'')\globinfo_S\big\},
\end{equation}
where $C>0$ is an absolute constant and $\calI_S$ is defined in \eqref{eq:globinfo}. 
\end{lemma}
\begin{proof}
	See Section \ref{subappend:prf_opt_test_error_z_star}.
\end{proof}
\begin{remark} 
\label{rmk:parity}
The SNR $\calI_S$ for \globest, which appears on the exponent in the optimal testing error \eqref{eq:opt_test_error_z_star}, takes different forms according to the \emph{parity} of $|S^c|$. There is a fundamental reason for this.
It happens that the dominating term in the probability \eqref{eq:error_of_LR_test_z_star} is given by the part with $\sum_{\ell\in S^c}\ZZZ{\ell}$ being \emph{non-positive and closest to zero}. Since $\ZZZ{\ell}$'s are $\{\pm 1\}$-valued, such a requirement translates to $\sum_{\ell\in S^c}\ZZZ{\ell} = 0$ when $|S^c|$ is even, and gives $\sum_{\ell\in S^c}\ZZZ{\ell} = -1$ when $|S^c|$ is odd. 
\end{remark}


\subsubsection{Minimax lower bound for \globest} 
Equipped with Lemmas \ref{lemma:testing_problem_z_star} and \ref{lemma:opt_test_error_z_star}, we are ready to present the main result in this subsection.
\begin{theorem}[Minimax lower bound for \globest]
  \label{thm:lower_bound_z_star}
  Assume $\rho = o(1)$. Meanwhile, {assume there exist constants $C_1, C_2 > 1$ and $c_1, c_2 =(0, 1)$ such that $C_1 q_\ell \leq p_\ell \leq (C_2 q_\ell) \land (1-c_1), \forall \ell\in[L]$ and $\log L\ll  n^{c_2}$}. If 
  $
    \min_{S\subseteq[L]} \globinfo_S \to \infty,
  $
  then there exists a sequence $\underline{\delta}_n = o(1)$ such that 
  \begin{align}
    \label{eq:lower_bound_z_star_vanishing}
    & \inf_{\bhz^\star} \sup_{\bz^\star\in \calP_n} \bbE \calL(\bhz^\star, \bz^\star) \geq \exp\big\{-(1+\underline{\delta}_n) \min_{S\subseteq[L]}\globinfo_S\big\}.
  \end{align}
  On the other hand, if $\min_{S\subseteq[L]} \globinfo_S=\calO(1)$, then there exists some $c' >0$ such that
  \begin{equation}
    \label{eq:lower_bound_z_star_constant}
    \inf_{\bhz^\star} \sup_{\bz^\star\in \calP_n} \bbE \calL(\bhz^\star, \bz^\star) \geq c'.
  \end{equation}
\end{theorem}
\begin{proof}
	If $\min_{S\subseteq[L]} \globinfo_S\to\infty$, then we can find a sequence $\delta_n = o(1)$ such that 
	$
	\log (\delta_n^{-1}) \ll
	\min_{S\subseteq[L]} \globinfo_S.
	$	
	Then, invoking Lemmas \ref{lemma:testing_problem_z_star} and \ref{lemma:opt_test_error_z_star}, there exists $\delta_n' = o(1)$ such that
	\begin{align*}
		& \inf_{\bhz^\star} \sup_{\bz^\star\in \calP_n} \bbE \calL(\bhz^\star, \bz^\star) 
		 \gtrsim \delta_n \exp\big\{-(1+\delta_n')\min_{S\subseteq[L]} \globinfo_S\big\}
		= \exp\big\{-(1+{\delta_n' + \delta_n''}) \min_{S\subseteq[L]} \globinfo_S \big\},
	\end{align*}
	where
	$
		\delta_n'' = {\log({\delta_n}^{-1})}/{\min_{S\subseteq[L]} \globinfo_S} = o(1)
	$
	by construction. Choosing $\underline{\delta}_n = \delta_n' + \delta_n''$ gives \eqref{eq:lower_bound_z_star_vanishing}. 

	On the other hand, if $\min_{S\subseteq[L]} \globinfo_S=\calO(1)$, then repeating the above arguments gives $\inf_{\bhz^\star} \sup_{\bz^\star\in \calP_n} \bbE \calL(\bhz^\star, \bz^\star) \gtrsim \delta_n$ for any $o(1)$ sequence $\delta_n$. 
	If $\inf_{\bhz^\star} \sup_{\bz^\star\in \calP_n} \bbE \calL(\bhz^\star, \bz^\star)$ is itself $o(1)$, then we would have
	$$
		\inf_{\bhz^\star} \sup_{\bz^\star\in \calP_n} \bbE \calL(\bhz^\star, \bz^\star)  \gtrsim \sqrt{\inf_{\bhz^\star} \sup_{\bz^\star\in \calP_n} \bbE \calL(\bhz^\star, \bz^\star)},
	$$
	a contradiction. Hence \eqref{eq:lower_bound_z_star_constant} follows.
\end{proof}

As an immediate corollary, we have the following result for the homogeneous case $(\rho=0)$.
\begin{corollary}[Minimax lower bound for \globest~under homogeneity]
	\label{cor:lower_bound_z_star_hom}
	Under the setup of Theorem \ref{thm:lower_bound_z_star}, assume in addition that $\rho=0$. If 
	$
		n\sum_{\ell\in[L]} \III{\ell}_{1/2} \to \infty,
	$
	then there exists a sequence $\underline{\delta}_n = o(1)$ such that
	\begin{align}
		\label{eq:lower_bound_z_star_vanishing_hom}
		& \inf_{\bhz^\star} \sup_{\bz^\star\in \calP_n} \bbE \calL(\bhz^\star, \bz^\star) \geq \exp\big\{-(1+\underline{\delta}_n) \frac{n}{2}\sum_{\ell\in[L]} \III{\ell}_{1/2}\big\}.
	\end{align}
	On the other hand, if $n\sum_{\ell\in[L]} \III{\ell}_{1/2}  = \calO(1)$, then there exists some $c' >0$ such that
	\begin{equation}
		\label{eq:lower_bound_z_star_constant_hom}
		\inf_{\bhz^\star} \sup_{\bz^\star\in \calP_n} \bbE \calL(\bhz^\star, \bz^\star) \geq c'.
	\end{equation}
\end{corollary}
\begin{proof}
	This follows from the fact that if $\rho = 0$, then $J_\rho=\infty$ and thus the set $S$ that minimizes $\globinfo_S$ is $S=[L]$. 
\end{proof}


\subsection{Minimax Lower Bound for \IndEst}
We now derive minimax lower bound for estimating individual $\zzz{\ell}$'s. Let us again consider the idealized setup \eqref{eq:idealized_setup}.

\subsubsection{Two testing problems from two sources of errors} 
Suppose that we additionally know the value of $\bz^\star_0$, say $\bz^\star_0 = +1$. 
Since $\zzz{\ell}_0$ is independent of $\zzz{-\ell}_0:= \{\zzz{r}_0: r\neq \ell\}$, the only information that's useful in determining $\zzz{\ell}_0$ comes from the following ``label sampling'' model:
\begin{equation}
	\label{eq:rho_flipping_model}
	\zzz{\ell}_i \sim  \bz^\star_i\times \big[2\Bern(1-\rho)-1\big], \qquad \AAA{\ell}\sim \textnormal{SBM}(\zzz{\ell}, p_\ell, q_\ell).
\end{equation}
Now, for any estimator $\hzzz{\ell}_0$ of $\zzz{\ell}_0$, the error probability reads
$$
	\bbP(\hzzz{\ell}_0 \neq \zzz{\ell}_0) = (1-\rho) \cdot \bbP(\hzzz{\ell}_0 = -1 ~|~ \zzz{\ell}_0 = +1) + \rho \cdot \bbP(\hzzz{\ell}_0  = +1 ~|~ \zzz{\ell}_0 = -1),
$$
which can be regarded as the $(1-\rho)\times \textnormal{type-I error} + \rho \times \textnormal{type-II error}$ of testing $H_0: \zzz{\ell}_0 = +1$ v.s. $H_1: \zzz{\ell}_0 = -1$ in a vanilla two-block SBM, which is almost equivalent to \eqref{eq:hypothesis_z_star_vanilla_sbm}. This intuition is formalized by the following lemma.
\begin{lemma}[The fundamental testing problem for \indest, Part \RN{1}]
  \label{lemma:testing_problem_z_l_part_I}
  {Assume there exist constants $c_1, c_2\in(0, 1)$ such that $\rho \leq 1/2-c_1$ and $\log L \ll n^{c_2}$.} Then there exists sequence $\delta_n = o(1)$ satisfying $(1+\delta_n)n/2 \in \bbN$, such that for any $\ell\in[L]$, 
  \begin{equation}
  \label{eq:lower_bound_by_testing_problem_z_l_part_I}
  \inf_{\hzzz{\ell}} \sup_{\bz^\star\in \calP_n} \bbE \calL(\hzzz{\ell}, \zzz{\ell}) \gtrsim \inf_{\phi}\bigg((1-\rho)\cdot\bbE_{H_0}[\phi] + \rho\cdot\bbE_{H_1}[1-\phi] \bigg),
\end{equation}
where $\phi$ is a testing function of the following problem:
\begin{align}
  \label{eq:fund_testing_problem_z_l_part_I}
  & H_0:  \bigotimes_{i = 1}^{(1+\delta_n)n/2} \textnormal{Bern}(p_\ell)\otimes \textnormal{Bern}(q_\ell)
  \ \  \
  \textnormal{v.s. } \ \ \ 
  H_1:  \bigotimes_{i = 1}^{(1+\delta_n)n/2} \textnormal{Bern}(q_\ell)\otimes \textnormal{Bern}(p_\ell).
\end{align}
\end{lemma}
\begin{proof}
	See Section \ref{subappend:prf_testing_problem_z_l_part_I}.
\end{proof}

Recall that Lemma \ref{lemma:testing_problem_z_l_part_I} reflects the situation when $\bz^\star_0$ is known to us. In practice, we need to estimate $\bz^\star_0$ from the data, giving rise to a testing problem similar to the one presented in Lemma \ref{lemma:testing_problem_z_star}.
\begin{lemma}[The fundamental testing problem for \indest, Part \RN{2}]
  \label{lemma:testing_problem_z_l_part_II}
  {Assume there exist constants $c_1, c_2 \in (0, 1)$ such that $\rho \leq 1/2-c_1$ and $\log L \ll n^{c_2}$.} Then for any sequence $\delta_n = o(1)$, {there exists another sequence $\delta_n' = o(1)$ satisfying $(1+\delta_n')n/2 \in \bbN$}, such that for any $\ell\in[L]$ and any $S\subseteq[L]\setminus\{\ell\}$, we have
  \begin{equation}
  \label{eq:lower_bound_by_testing_problem_z_l_part_II}
  \inf_{\hzzz{\ell}} \sup_{\bz^\star\in \calP_n} \bbE \calL(\hzzz{\ell}, \zzz{\ell}) \gtrsim \delta_n \inf_{\phi}\bigg(\bbE_{H_0}[\phi] + \bbE_{H_1}[1-\phi] \bigg),
\end{equation}
where $\phi$ is a testing function of the following problem:
\begin{align}
  & H_0: \bigg(\bigotimes_{r\in (S\cup\{\ell\})^c} [2 \textnormal{Bern}(1-\rho) - 1]\bigg) \otimes \bigg(\bigotimes_{r\in S\cup\{\ell\}} \bigotimes_{i = 1}^{(1+\delta'_n)n/2} \textnormal{Bern}(p_\ell)\otimes \textnormal{Bern}(q_\ell)\bigg)\nonumber\\
  \label{eq:fund_testing_problem_z_l_part_II}
  & \qquad \textnormal{v.s. } H_1: \bigg(\bigotimes_{r\in (S\cup\{\ell\})^c} [2 \textnormal{Bern}(\rho) - 1]\bigg) \otimes \bigg(\bigotimes_{r\in S\cup\{\ell\}} \bigotimes_{i = 1}^{(1+\delta'_n)n/2} \textnormal{Bern}(q_\ell)\otimes \textnormal{Bern}(p_\ell)\bigg).
\end{align}
\end{lemma}
\begin{proof}
	See Section \ref{subappend:prf_testing_problem_z_l_part_II}.
\end{proof}

The testing problem in \eqref{eq:fund_testing_problem_z_l_part_II} differs from the one in \eqref{eq:fund_testing_problem_z_star} in that the layer $\ell$ is never involved in the term regarding $\textnormal{Bern}(1-\rho)$ and $\textnormal{Bern}(\rho)$. This makes sense, because according to our intuition in Section \ref{subsec:lower_bound_z_star}, this term reflects the case when $\zzz{\ell}$ is (nearly) known to us, which can never happen since $\zzz{\ell}$ itself is the estimating target.

To characterize the optimal testing error for the two testing problems given in Lemmas \ref{lemma:testing_problem_z_l_part_I} and \ref{lemma:testing_problem_z_l_part_II}, apart from the SNR for \globest~$\globinfo_S$ defined in \eqref{eq:globinfo}, we additionally define the corresponding \emph{SNR for \indest}:
\begin{equation}
	\label{eq:indinfo}
	\indinfo_S:= 
	\begin{cases}
		|S| J_\rho + \psi^\star_S(0) & \textnormal{ if } |S| \textnormal{ is even}, \\
		(|S| + 1) J_\rho + \psi^\star_S(-2J_\rho) & \textnormal{ if } |S| \textnormal{ is odd}. 
	\end{cases}
\end{equation}

A careful analysis on the error incurred by the likelihood ratio test gives the following result.
\begin{lemma}[Optimal testing error for \indest]
\label{lemma:opt_test_error_z_l}
Assume $\rho = o(1)$ and that {there exist constants $C_1, C_2 > 1, c\in (0, 1)$ such that $C_1 q_\ell \leq p_\ell \leq (C_2 q_\ell)\land (1-c), \forall\ell\in[L]$.} 
Then there exists a sequence $\delta''_n=o(1)$ such that for any {$\ell\in[L], S \subseteq[L]\setminus \{\ell\}$}, the optimal $(1-\rho)\times \textnormal{type-I error} + \rho\times \textnormal{type-II error}$ of the testing problem in \eqref{eq:fund_testing_problem_z_l_part_I} is lower bounded by
\begin{equation}
  \label{eq:opt_test_error_z_l_part_I}
  C\cdot \exp\big\{ - (1+\delta_n'')\indinfo_{\{\ell\}}\big\} ,
\end{equation}
and the optimal type-I plus type-II error of the testing problem in \eqref{eq:fund_testing_problem_z_l_part_II} is lower bounded by
\begin{equation}
  \label{eq:opt_test_error_z_l_part_II}
  C\cdot \exp\big\{-(1+\delta_n'')\globinfo_{S\cup\{\ell\}}\big\} ,
\end{equation}
where $C>0$ is an absolute constant.
\end{lemma}
\begin{proof}
	See Section \ref{subappend:prf_opt_test_error_z_l}.
\end{proof}
For the same reason as explained in Remark \ref{rmk:parity}, the optimal testing error for \eqref{eq:fund_testing_problem_z_l_part_II} depends on the parity of $|(S\cup\{\ell\})^c|$.


\subsubsection{Minimax lower bound for \indest} 
We are now ready to present the main result in this subsection.
\begin{theorem}[Minimax lower bound for \indest]
	\label{thm:lower_bound_z_l}
	Under the same setup as Theorem \ref{thm:lower_bound_z_star},
	if for a fixed $\ell\in[L]$, it holds that 
	$
		\min_{S\subseteq[L]\setminus\{\ell\}} \globinfo_{S\cup\{\ell\}} \land \indinfo_{\{\ell\}} \to \infty,
	$
	then there exists a sequence $\underline{\delta}_n = o(1)$, independent of $\ell$, such that 
	\begin{align}
		\label{eq:lower_bound_z_l_vanishing}
		& \inf_{\hzzz{\ell}} \sup_{\bz^\star\in \calP_n} \bbE \calL(\hzzz{\ell}, \zzz{\ell}) \geq \exp\big\{ -(1+\underline{\delta}_n) \min_{S\subseteq[L]\setminus\{\ell\}} \globinfo_{S\cup\{\ell\}} \big\} + \exp\big\{-(1+\underline{\delta}_n) \indinfo_{\{\ell\}}\big\}.
	\end{align}
	On the other hand, if $\min_{S\subseteq[L]\setminus\{\ell\}} \globinfo_{S\cup\{\ell\}} \land \indinfo_{\{\ell\}}=\calO(1)$, then there exists $c'>0$ such that
	\begin{equation}
		\label{eq:lower_bound_z_l_constant}
		\inf_{\hzzz{\ell}} \sup_{\bz^\star\in \calP_n} \bbE \calL(\hzzz{\ell}, \zzz{\ell})\geq c'.
	\end{equation}
\end{theorem}
\begin{proof}
	Given Lemma \ref{lemma:testing_problem_z_l_part_I}, \ref{lemma:testing_problem_z_l_part_II} and \ref{lemma:opt_test_error_z_l}, the proof is essentially the same as the proof of Theorem \ref{thm:lower_bound_z_star}, and we omit the details.
\end{proof}

Under homogeneity, we have $\zzz{\ell}=\bz^\star ,\forall\ell\in[n]$, and the lower bound in the above theorem should coincide with  \eqref{eq:lower_bound_z_star_vanishing_hom}. 
Indeed, 
when $\rho = 0$, the only way to make the exponent finite is to choose $S = [L]\setminus\{\ell\}$, in which case we have 
$$
\min_{S\subseteq [L]\setminus\{\ell\}}\globinfo_{S\cup\{\ell\}} \land \indinfo_{\{\ell\}} = \globinfo_{[L]}\land \indinfo_{\{\ell\}} = \psi_{[L]}^\star(0)\land \infty = \psi_{[L]}^\star(0),
$$
and hence Corollary \ref{cor:lower_bound_z_star_hom} can be alternatively derived from Theorem \ref{thm:lower_bound_z_l}. 

Based on the intuitions built from Lemmas \ref{lemma:testing_problem_z_l_part_I} and \ref{lemma:testing_problem_z_l_part_II}, the interpretations of the two terms in the lower bound \eqref{eq:lower_bound_z_l_vanishing} should be clear: $\indinfo_{\{\ell\}}$ is the error incurred by the label sampling model \eqref{eq:rho_flipping_model}, which we cannot avoid even if we know the ground truth $\bz^\star$, whereas $\min_{S\subseteq[L]\setminus\{\ell\}}\globinfo_{S\cup\{\ell\}}$ represents the error incurred by empirically estimating $\bz^\star$. 

\section{A Two-Stage Algorithm}\label{sec:alg}
Recall that the \modelname~is a hierarchical model, where the individual assignments $\zzz{\ell}$'s are independent realizations from the ``prior'' distribution \eqref{eq:label_flip} which is parametrized by $\bz^\star$ and $\rho$.
We start by writing down the posterior density, which is proportional to
\begin{align}
	& \prod_{\ell\in[L]} \prod_{i\in[n]} (1-\rho)^{\indicator{\zzz{\ell}_i = \bz^{\star}_i}} \rho^{\indicator{\zzz{\ell}_i \neq \bz^{\star}_i}} \nonumber\\
	\label{eq:posterior}
	& ~~\times \prod_{\ell\in[L]} \prod_{i< j} \bigg(p_\ell^{\AAA{\ell}_{ij}} (1-p_\ell)^{1 - \AAA{\ell}_{ij}} \cdot \indicator{\zzz{\ell}_i = \zzz{\ell}_j}\\ 
	& \hskip 8em + q_\ell^{\AAA{\ell}_{ij}} (1-q_\ell)^{1-\AAA{\ell}_{ij}}  \cdot \indicator{\zzz{\ell}_i \neq \zzz{\ell}_j} \bigg).\nonumber
\end{align}
Computing the vanilla MAP estimator requires searching over a discrete set with cardinality $2^{n(L+1)}$, a hopeless task for even moderately-sized $n$ and $L$. 

{Now, supposed that for a fixed $i\in[n]$, we are given a collection of estimators $\{\tzzz{\ell}_{j}: \ell\in[L], j\neq i\}$ for the individual assignments $\{\zzz{\ell}_j: \ell\in[L], j\neq i\}$.} 
On the event that $\tzzz{\ell}_j$'s ($j\neq i$) agree with the ground truth parameters, the posterior density given in \eqref{eq:posterior}, as a function of $(\bz^\star_i, \{\zzz{\ell}_i\}_{\ell=1}^L)$, reduces to a constant multiple of 
\begin{align}
	& \prod_{\ell\in[L]} (1-\rho)^{\indicator{\zzz{\ell}_i = \bz^{\star}_i}} \rho^{\indicator{\zzz{\ell}_i \neq \bz^{\star}_i}} \nonumber\\
	\label{eq:posterior_reduced}
	& ~~\times \prod_{\ell\in[L]} \prod_{j\neq i} \bigg(p_\ell^{\AAA{\ell}_{ij}} (1-p_\ell)^{1 - \AAA{\ell}_{ij}} \cdot \indicator{\zzz{\ell}_i = \tzzz{\ell}_j}\\
	&  \hskip 8em + q_\ell^{\AAA{\ell}_{ij}} (1-q_\ell)^{1-\AAA{\ell}_{ij}}  \cdot \indicator{\zzz{\ell}_i \neq \tzzz{\ell}_j} \bigg). \nonumber
\end{align}
With some algebra, one finds that maximizing the above display over $(\zzz{\star}_i, \{\zzz{\ell}_i\}_{\ell=1}^L)$ is equivalent to maximizing the following objective function:
\begin{align}
	\label{eq:log_posterior_reduced}
	\sum_{\ell\in[L]} \bigg\{\log\bigg(\frac{1-\rho}{\rho}\bigg) \cdot \indicator{\zzz{\ell}_i = \bz^{\star}_i} + \sum_{\substack{j\neq i :\tzzz{\ell}_j = \zzz{\ell}_i}} \bigg[ \log\bigg(\frac{p_\ell(1-q_\ell)}{q_\ell(1-p_\ell)}\bigg) \AAA{\ell}_{ij} + \log \bigg(\frac{1-p_\ell}{1-q_\ell}\bigg)\bigg]\bigg\}.
\end{align}
This is already simpler than the original one of maximizing \eqref{eq:posterior}, because the search space now has cardinality $2^{L+1} \ll 2^{n(L+1)}$. 

A closer look at \eqref{eq:log_posterior_reduced} reveals that this function can be maximized in linear (in $L$) time. 
Indeed, if we fix $\bz^{\star}_i$, the problem of seeking for optimal $\zzz{\ell}$'s is \emph{decoupled} into $L$ subproblems. That is, it suffices to maximize
\begin{align}
	\label{eq:log_posterior_decoupled}
	\log\bigg(\frac{1-\rho}{\rho}\bigg) \cdot \indicator{\zzz{\ell}_i = \zzz{\star}_i} + \sum_{\substack{j\neq i :\tzzz{\ell}_j = \zzz{\ell}_i}}\bigg[ \log\bigg(\frac{p_\ell(1-q_\ell)}{q_\ell(1-p_\ell)}\bigg) \AAA{\ell}_{ij} + \log \bigg(\frac{1-p_\ell}{1-q_\ell}\bigg)\bigg]
\end{align}	
for each $\ell\in[L]$. Note that each subproblem can be efficiently solved, since one only needs to search over a space with cardinality two (i.e., $\zzz{\ell}_i\in\{\pm 1\}$). Thus, to obtain the global maximizer of \eqref{eq:log_posterior_reduced}, one can proceed as follows:
\begin{enumerate}
	\item Solve $L$ subproblems \eqref{eq:log_posterior_decoupled} with $\bz^{\star}_i = + 1$, and record the objective value of \eqref{eq:log_posterior_reduced};
	\item Repeat Step 1 with $\bz^{\star}_i = -1$;
	\item Obtain the global maximizer of \eqref{eq:log_posterior_reduced} by comparing the two objective values in the previous two steps.
\end{enumerate}	

The foregoing discussion shows that the MAP estimator of $(\bz^{\star}_i, \{\zzz{\ell}_i\}_{\ell = 1}^L)$ can be efficiently computed, provided \emph{the remaining parameters are given}.
This observation motivates the main algorithm of this paper, which is a two-stage procedure that first obtains initial estimators of $\{\zzz{\ell}\}_{1}^L$ via spectral clustering, and then refines the initial estimators in a node-wise fashion using MAP estimation.

\subsection{Stage \texorpdfstring{\RN{1}}{I}: Initialization via Spectral Clustering}
While our analysis in Section \ref{sec:minimax_rate} reveals that any consistent initialization would work, we will focus on a specific initialization scheme in this subsection: spectral clustering. 

If $\rho$ is of order $o(1)$, then the proportion of flips in $\zzz{\ell}$ from $\bz^{\star}$ will also be of order $o(1)$ with high probability. 
Hence, as long as a consistent estimator $\btz^{\star}$ of the global assignment $\bz^{\star}$ is given, consistent \indest~is automatic by setting $\tzzz{\ell} = \btz^{\star}$. In the rest of this subsection, we restrict ourselves to \globest.


Let $\weightvec = (\weight_1, \hdots, \weight_L)$ be an arbitrary positive (i.e., $\omega_\ell > 0,\forall \ell\in[L]$) weight vector, and let us consider the following weighted adjacency matrix 
\begin{equation}
	\label{eq:wted_adjacency}
	\bar A := \sum_{\ell\in [L]} \weight_\ell \AAA{\ell}.
\end{equation}
In the case of $\rho = 0$, one readily checks that all the information in $\bz^{\star}$ is contained in the top two eigenvectors of $\bbE[\bar A]$.
A natural proposal is then to take top two eigenvectors of $\bar A$, and apply $k$-means clustering to them.

In the case of a small $\rho > 0$, we expect spectral clustering to continue to work well for estimating $\bz^{\star}$, provided it exhibits a certain level of \emph{stability} to the additional ``noise'' induced by $\rho$. Our later analysis in Section \ref{subsec:init} shows that this is indeed the case.

\begin{algorithm}[t]
\KwIn{Adjacency matrices $\{\AAA{\ell}\}_1^L$, weight vector $\weightvec$, intra-cluster connecting probabilities $\{\pell\}_1^L$, trimming parameter $\gamma > 1$}
\KwOut{Initial global estimator $\btz^{\star}$}
Identify nodes $I\subseteq[n]$ such that $\sum_{j\in[n]}\barA_{ij} > \gamma n \sum_{\ell\in[L]}\weight_\ell p_\ell , \forall i\in I$\;
For any $i\in  (I\times[n])\cup ([n]\times I)$, set $\barA_{ij}$ to zero and call the resulting matrix $\tau(\bar A)$\;
Compute $U\in\R^{n\times2}$, the first two eigenvectors of $\tau(\bar A)$\;
Solve the $(1+\ep)$-approximate $k$-means objective \eqref{equ: approx kmeans} on $U$ to get $\hat Z$ and set 
$$
\btz^\star_i =  \Indc\{\hat Z_{i,1} = 1\} - \Indc\{\hat Z_{i,2} = 1\}
$$ 
for any $i\in[n]$\;
\Return{$\btz^\star$}\;
\caption{Stage \RN{1}: initialization via spectral clustering}
\label{alg: specc}
\end{algorithm}

The overall initialization scheme is detailed in Algorithm \ref{alg: specc}. 
There are two subtleties in this algorithm. First, instead of applying spectral clustering to $\barA$, we apply it to a {trimmed version}, $\tau(\barA)$, which is obtained by setting the ``larger-than-average'' entries of $\barA$ to zero. 
As shown in Section \ref{subsec:init}, such a trimming operation can significantly improve the concentration of $\bar A$, especially when the signal-to-noise ratio is low. 
{When $\{p_\ell\}_{\ell=1}^L$ are unknown, one can replace them with conservative estimators; see Section \ref{subappend:est_deg} for details. Alternatively, one can replace them with the sample average connecting probabilities but with a larger $\gamma$.}
The second subtlety is a computational one: since exactly solving the $k$-means objective is NP-hard, we instead find the solution of an $(1+\ep)$-approximation of it \citep{kumar2004simple}. Specifically, letting $U\in \bbR^{n\times 2}$ be the top two eigenvectors of $\tau(\bar A)$, we seek for $(\hat Z, \hat X)$ such that
\begin{align}
\|\hat Z\hat X- U\|_F^2\leq(1+\ep)\min_{Z,\, X}\|ZX-U\|_F^2,
\label{equ: approx kmeans}
\end{align}
where the the minimum is over all $n\times 2$ assignment matrix $Z$ (i.e., each row of $Z$ is a canonical basis of $\bbR^2$) and all $2\times 2$ matrix $X$. The initial estimator $\btz^{\star} = \tzzz{\ell}$ is then taken to be the clustering induced by $\hat Z$.

\subsection{Stage \texorpdfstring{\RN{2}}{II}: Node-Wise Refinement via MAP Estimation}\label{subsec:refine}

According to our discussion in the previous subsection, once an initial global estimator $\btz^{\star}$ is given, we can also take that to be the initial individualized estimator. Now, in view of the MAP objective functions \eqref{eq:posterior}---\eqref{eq:log_posterior_decoupled}, we propose to solve
\begin{equation}
	\label{eq:joint_refinement}
	(\bhz^\star_i, \hzzz{1}_i, ..., \hzzz{L}_i) = \underset{\substack{s_\star\in\{\pm 1\},  \\ s_\ell \in \{\pm 1\} ~\forall\ell\in[L]}}{\arg\max} \sum_{\ell\in[L]} f^{(\ell)}_i(s_\star, s_\ell, \btz^{\star}),
\end{equation}
where
\begin{equation}
	\label{eq:local_map_objective}
	f^{(\ell)}_i(s_\star, s_\ell, \btz^{\star}) = \log\bigg(\frac{1-\rho}{\rho}\bigg) \cdot \indicator{s_\ell = s_\star} + \sum_{j\neq i: \btz^{\star}_j = s_\ell} \bigg[\log\bigg(\frac{p_\ell(1-q_\ell)}{q_\ell(1-p_\ell)}\bigg) \AAA{\ell}_{ij} + \log \bigg(\frac{1-p_\ell}{1-q_\ell}\bigg)\bigg],
\end{equation}
By our discussion at the beginning of this section, the above optimization problem can be solved in linear time. A detailed description is given in Algorithm \ref{alg: generic_refinement}. 
{When $(p_\ell, q_\ell)_{\ell=1}^L$ are unknown, one can replace them with their estimators; see Section \ref{subappend:est_params} for details.}

\begin{algorithm}[t]
\SetNoFillComment
\KwIn{Initial global estimator $\btz^{\star}$, adjacency matrices $\{\AAA{\ell}\}_1^L$, connecting probabilities $\{(p_\ell, q_\ell)\}_1^L$, flipping probability $\rho$}
\KwOut {Global estimator $\bhz^\star$, individualized estimator $\{\hzzz{\ell}\}_1^L$}
\For{$i = 1, \hdots, n$}{
	\For{$\ell = 1, \hdots, L$}{
		$\sfz(\ell, +1) \leftarrow \argmax_{s\in\{\pm1\}} f^{(\ell)}_{i}(+1, s, \btz^{\star})$\tcp*[r]{$f^{(\ell)}_i$ is defined in \eqref{eq:local_map_objective}}  
		$\sfz(\ell, -1)\leftarrow \argmax_{s\in\{\pm1\}} f^{(\ell)}_{i}(-1, s, \btz^{\star})$\; 
	}
	$\bhz^\star_i \leftarrow \argmax_{s_\star\in\{\pm 1\}} \sum_{\ell\in [L]} f^{(\ell)}_i\big(s_\star, \sfz(\ell, s_\star), \btz^{\star}\big)$\tcp*[r]{final global estimator}
	\For{$\ell=1, \hdots, L$}{
		$\hzzz{\ell}_i \leftarrow \sfz(\ell, \bhz^\star_i)$\tcp*[r]{final individualized estimators}
	}
}
\Return{$\bhz^\star, \{\hzzz{\ell}\}_1^L$}\;
\caption{Stage \RN{2}: Node-wise refinement via MAP estimation}
\label{alg: generic_refinement}
\end{algorithm} 

We conclude this section by remarking that our proposed algorithm is naturally a {distributed} one: the two for loops in Algorithm \ref{alg: generic_refinement} can be easily parallelized.

\section{Performance of the Two-Stage Algorithm}
\label{sec:minimax_rate}

In this section, we present theoretical results on the two-stage algorithm introduced in Section \ref{sec:alg}. Specifically, the performance of spectral clustering is presented in Section \ref{subsec:init}, followed by an analysis of  MAP-based refinement in Section \ref{subsec:theory_refine}. The minimax optimality of the two-stage algorithm is proved in Section \ref{subsec:minimax_optimality}. 
Throughout this section, the high probability error bounds are uniform with respect to probability measures defined in the parameter space \eqref{eq:param_space}. In particular, the ``$\bbP$'' symbol represents the probability after marginalizing over the realizations of the $\zzz{\ell}$'s.

\subsection{Performance of Spectral Clustering}\label{subsec:init}

In this subsection, we analyze theoretical properties of Algorithm \ref{alg: specc}.
An important degree of freedom in Algorithm \ref{alg: specc} is the choice of the weight vector $\weightvec$, and it is restricted by the following assumption.
\begin{assump}[Balanced weights across layers]  
\label{assump:balanced_wts}
Assume $\omega_\ell > 0 ,\forall \ell\in[L]$ and $\sum_{\ell\in[L]}\weight_\ell = 1$. Moreover, assume that there exist two absolute constants $c_0 > 0$ and $c_1 \geq 1$ such that the following two inequalities hold:
\begin{align}
     \max_{\ell\in[L]} \{\weight_\ell\} \cdot \sum_{\ell\in[L]} \weight_\ell p_\ell \leq c_0 \sum_{\ell\in[L]} \weight_\ell^2 p_\ell , \qquad  \max_{\ell\in[L]} \{\weight_\ell\} \cdot \sum_{\ell\in[L]} p_\ell \leq c_1 \sum_{\ell} \weight_\ell q_\ell.
\end{align}  
\end{assump}

The above assumption essentially states that $\weightvec$ should be relatively balanced across layers. In particular, if $\weight_\ell \asymp 1/L ,\forall \ell\in[L]$, then this assumption holds.

We are now ready to state the main theorem of this subsection. We emphasize that the following theorem does not require $\rho = o(1)$ or $\beta = 1+o(1)$.
{\begin{theorem}[Performance of spectral clustering]
\label{thm:specc}
Let Assumption \ref{assump:balanced_wts} hold with $c_0>0, c_1\geq 1$. Let the input to Algorithm \ref{alg: specc} be an instance generated by an $\textnormal{\modelname}\in \calP_n(\rho, \{p_\ell\}_1^L, \{q_\ell\}_1^L, \beta)$ and assume
\begin{equation}
      \label{eq:spectral_gap}
      \frac{\bar p}{\bar p - \bar q} \leq c\cdot \frac{2(1-2\rho)^2}{\beta - \beta^{-1} + 4n^{-1}} + 2(\rho - \rho^2),
\end{equation}      
where $c \in [0, 1)$ is an absolute constant and $\bar p = \sum_{\ell\in[L]} \weight_\ell p_\ell, \bar q = \sum_{\ell\in[L]} \weight_\ell q_\ell$ are the weighted averages of connecting probabilities. 
Fix any $r\geq 1$ and choose the regularization parameter to be $\gamma > e^{c_1}$. Then, there exist constants $c_2 = c_2(\gamma)$ and $C = C(c, c_0 , c_1, r)$ such that with probability at least $1 - 11n^{-r} - c_2^{-n}$, the output $\btz^{\star}$ of Algorithm \ref{alg: specc} satisfies
\begin{equation}
      \label{eq:error_specc}
      \calL(\btz^{\star}, \bz^{\star}) \leq  \frac{C(2+\ep)  (\Delta_1 + \Delta_2)}{n^2(1-2\rho)^4 (\bar p - \bar q)^2} ,
\end{equation}
where
\begin{align}
      \label{eq:error_specc_1}
      \Delta_1 & = n\sum_{\ell\in[L]} \weight_\ell^2 p_\ell,\\
      \label{eq:error_specc_2}
      \Delta_2& = \max_{\ell\in[L]} \{\weight_\ell^2 (p_\ell-q_\ell)^2\} \cdot  [L^2\rho^2 + n^2L\rho + n\log n + (\log L)^2] \\
      & \qquad  + \rho^2(1-\rho)^2 (\bar p - \bar q)^2.\nonumber
\end{align}
\end{theorem}}
\begin{proof}
      See Section \ref{prf:thm:specc}.
\end{proof}

\begin{remark}
      Under our working assumption that $\beta = 1+o(1)$, the first term on the righthand side of \eqref{eq:spectral_gap} tends to infinity, and so the inequality holds if $p_\ell / (p_\ell - q_\ell)$ is uniformly bounded for all $\ell\in[L]$.
\end{remark}

{
In the upper bound \eqref{eq:error_specc}, the two terms $\Delta_{1}$ and $\Delta_2$ come from 
the fact that in our proof, we relate the misclustering error to the deviation (in spectral norm) of the trimmed weighted adjacency matrix $\tau(\bar A)$ from the expectation of $\barA$. More explicitly, $\Delta_1$ is induced by the concentration of $\tau(\bar A)$ around $\bbE[\barA ~|~ \{\zzz{\ell}\}_1^L]$, the \emph{conditional mean} of $\barA$, conditioning on the realization of $\zzz{\ell}$'s, whereas $\Delta_2$ is induced by the concentration of $\bbE[\bar A ~|~ \{\zzz{\ell}\}_1^L]$ around the \emph{marginal mean} $\bbE[\bar A]$.}

It turns out that bounding $\|\tau(\barA) - \bbE [\barA]\|$ is closely related to bounding $\|\tau(\bar B) - \bbE[\bar B]\|$, where $\bar B = \sum_{\ell\in[L]}\omega_\ell \BBB{\ell}$ and $\BBB{\ell}$'s are $n\times n$ independent Bernoulli random matrices with independent $\Bern(\ppp{\ell}_{ij})$ entries.
In order to have a tight control of $\|\bar B - \bbE[\bar B]\|$, we give a non-trivial generalization of the results in \cite{le2017concentration} to the multilayer setup in Appendix \ref{append:concentration}, which roughly states the following: if the weight vector is sufficiently ``balanced'', then with high probability, for the trimmed version of $\bar B$, we have
\begin{align}
      \label{eq:main_concentration_reg}
      \|\tau(\bar B) - \bbE[\bar B]\| \lesssim \sqrt{n\sum_{\ell\in[L]}\weight_\ell^2 \max_{i, j}\ppp{\ell}_{ij}},
\end{align}
and without any trimming operation, we have
\begin{align}
      \label{eq:main_concentration_no_reg}
      \|\bar B - \bbE[\bar B]\| \lesssim \sqrt{n\sum_{\ell\in[L]}\weight_\ell^2 \max_{i, j}\ppp{\ell}_{ij}} + \max_{\ell\in[L]}\{\weight_\ell\} \cdot \sqrt{\log n}.
\end{align}


{
If $\rho = 0$, then the conditional mean $\bbE[\bar A~|~\{\zzz{\ell}\}_1^L]$ coincides with the marginal mean $\bbE[\barA]$, and thus  \eqref{eq:error_specc} holds with $\Delta_2 = 0$.}

{
Curious readers may wonder why the expression of $\Delta_2$ given in \eqref{eq:error_specc_2} does not vanish as $\rho$ tends to zero. In particular, there is an additive term of $n\log n$. This is related to an interesting phenomenon regarding the concentration of Bernoulli random matrices. It happens that the problem of bounding $\big\|\bbE[\bar A~|~\{\zzz{\ell}\}_1^L]- \bbE[\barA]\big\|$ can be related to bounding $\|B - \bbE[B]\|$, where $B\in\{0, 1\}^{n\times n}$ has i.i.d. $\Bern(\rho)$ entries. The expression of $\Delta_2$ in \eqref{eq:error_specc_2} is based on that
$
      \|B - \bbE [B]\| \lesssim \sqrt{n\rho} + \sqrt{\log n}
$
with high probability, which does not vanish as $\rho$ tends to zero. 
In fact, such a ``discontinuity at zero'' is unavoidable: it has been shown in \cite{krivelevich2003largest} that if $\rho \lesssim 1/n$, then with probability tending to one, 
$
      \|B - \bbE[B]\| = \oone \sqrt{{\log n}/{\log\log n}},
$
which diverges as $n$ tends to infinity.}

To have a better understanding on the magnitude of the bound \eqref{eq:error_specc}, let us choose $\omega_\ell = 1/L , \forall \ell\in[L]$ and consider the following scaling of the connecting probabilities:
\begin{equation}
      \label{eq:hom_scaling_connecting_prob}
      p_\ell = \frac{a\log n}{nL}, \qquad q_\ell = \frac{b\log n}{nL}, \qquad \forall \ell\in[L],
\end{equation}
where $a>b>0$ are two constants. 
{With some algebra, it follows that
\begin{align*}
      \Delta_1 = \frac{a\log n}{L^2}, \qquad \Delta_2 \asymp \frac{(a-b)^2 (\log n)^2}{n^2L^2} \cdot \bigg(\rho^2 + \frac{n^2\rho}{L} + \frac{n\log n}{L^2} + \frac{(\log L)^2}{L^2}\bigg).
\end{align*}
It is clear that $\Delta_1 + \Delta_2\asymp \frac{\log n}{L^2} \cdot (1+ \frac{\rho \log n}{L})$, and from \eqref{eq:error_specc} we arrive at 
\begin{equation}
      \label{eq:error_specc_hom_scaling}
      \calL(\btz^\star, \bz^\star) \lesssim \frac{1}{(1-2\rho)^4 (a-b)^2} \cdot \bigg(\frac{1}{\log n} + \frac{\rho}{L}\bigg).
\end{equation}
In summary, consistent estimation of $\bz^\star$ by spectral clustering is possible when the connecting probabilities are as small as $\Omega\big(\frac{\log n}{nL}\big)$.}


Since $\calL(\btz^{\star}, \zzz{\ell}) \leq \calL(\btz^{\star}, \bz^{\star}) + \calL(\bz^{\star}, \zzz{\ell})$, we can obtain performance guarantees of Algorithm \ref{alg: specc} for \indest~by bounding the number of flips at each layer, as detailed in the following corollary.
\begin{corollary}[Spectral clustering for \indest]
      \label{cor:specc_ind_est}
      Under the setup of Theorem \ref{thm:specc}, for any $\rho' \in(0, 1-\rho)$, with probability at least $1 - 11n^{-r} - c_2^{-n} - L e^{-n D_\kl(\rho+\rho'\|\rho)}$, we have
      \begin{equation}
            \label{eq:error_specc_ind_est}
            \max_{\ell\in[L]} \calL(\btz^{\star}, \zzz{\ell}) \leq  \frac{C(2+\ep)  (\Delta_1 + \Delta_2)}{n^2(1-2\rho)^4 (\bar p - \bar q)^2} + \rho + \rho'.
      \end{equation}
      where $D_\kl(p\|q)$ is the Kullback-Leibler divergence between $\textnormal{Bern}(p)$ and $\textnormal{Bern}(q)$. 
\end{corollary}
\begin{proof}
      Classical Chernoff--Hoeffding bound \cite{chernoff1952measure,hoeffding1963probability} gives that for each fixed $\ell\in[L]$, 
      \begin{equation}
            \label{eq:large_dev_bern_sum}
            \bbP\bigg(\calL(\bz^{\star}, \zzz{\ell}) \geq \rho + \rho'\bigg) \leq e^{-n D_\kl(\rho + \rho' \| \rho)},
      \end{equation}      
      and the desired result follows by taking a union bound over $\ell\in[L]$.
\end{proof}

It is well-known that \eqref{eq:large_dev_bern_sum}, which is obtained by computing the rate function of Bernoulli random variables, is asymptotically tight (see, e.g, \cite{varadhan1984large}). If the goal is merely to ensure consistency (i.e., $\rho' = o(1)$), then we can use the following standard weakening of \eqref{eq:large_dev_bern_sum}:
$$
      \bbP\bigg( \calL(\bz^\star, \zzz{\ell}) \geq \rho + \rho'\bigg) \leq e^{-2n(\rho')^2}.
$$
{If $\log L \ll n^c$ for some $c\in(0, 1)$, then we can choose $\rho' = n^{-(1-c)/2}$, so that \eqref{eq:error_specc_ind_est} holds with $\rho' = o(1)$ with probability at least $1-11n^{-r} - c_2^{-n} - e^{-2{n}^c + \log L}$, and one can bound $e^{-2{n}^c + \log L}\leq e^{-c_3 {n}^c}$ for some absolute constant $c_3 > 0$ because $\log L \ll n^{c}$.}

\subsection{Performance of MAP-Based Refinement}\label{subsec:theory_refine}
The refinement procedure as introduced in Section \ref{subsec:refine}, in its current form (Algorithm \ref{alg: generic_refinement}), is highly flexible in that no assumption is imposed on the initial estimator $\btz^{\star}$ other than consistency. 
While such a flexibility is favored in practice, it brings some unnecessary complications to its theoretical analysis. 
In addition, for a fixed $i$, the initial estimators $\btz^{\star}_{-i}$ may have arbitrary dependence structures with $\{\AAA{\ell}_{ij}: j\neq i, \ell\in[L]\}$, which makes the analysis intractable.

\begin{algorithm}[t]
\SetNoFillComment
\SetKwFunction{Initialize}{Initialize}
\KwIn{Adjacency matrices $\{\AAA{\ell}\}_1^L$, connection probabilities $\{(p_\ell, q_\ell)\}_1^L$, flipping probability $\rho$}
\KwOut {Global estimator $\bhz^\star$, individualized estimator $\{\hzzz{\ell}\}_1^L$}
\tcc{\textcolor{blue}{Stage \RN{1}: Leave-one-out initialization}}
\For{$i = 1, \hdots, n$}{
	$\tzzz{\star, -i} \leftarrow \boldsymbol{0}_n$\;
	$\tzzz{\star, -i}_{-i}\leftarrow$ \Initialize\big($\{\AAA{\ell}_{-i, -i}, p_\ell, q_\ell\}_{\ell=1}^L, \rho$\big)\tcp*[r]{initial estimator of $\bhz^\star_{-i}$}
}
Set $\tzzz{\ell, -i} \leftarrow \tzzz{\star,-i}, \forall \ell\in[L]$\;
\tcc{\textcolor{blue}{Stage \RN{2}: MAP-based refinement}}
\For{$i = 1, \hdots, n$}{
	\For{$\ell = 1, \hdots [L]$}{
		$\sfz(\ell, +1) \leftarrow \argmax_{s\in\{\pm1\}} f^{(\ell)}_{i}(+1, s, \tzzz{\star, -i})$\; 
		$\sfz(\ell, -1) \leftarrow \argmax_{s\in\{\pm1\}} f^{(\ell)}_{i}(-1, s, \tzzz{\star, -i})$\;
	}
	\tcp{final estimator of $\bhz^\star_i$, not aligned}
	$\tzzz{\star, -i}_{i} \leftarrow \argmax_{s_\star\in\{\pm 1\}} \sum_{\ell\in [L]} f^{(\ell)}_i \big(s_\star, \sfz(\ell, s_\star), \tzzz{\star, -i}\big)$\;
	\For{$\ell = 1, \hdots, [L]$}{
		$\tzzz{\ell, -i}_i \leftarrow \sfz(\ell, \tzzz{\star, -i}_i)$\tcp*[r]{final estimator of $\zzz{\ell}_i$, not aligned}
	}
}
\tcc{\textcolor{blue}{Stage \RN{3}: Alignment}}
Set $\bhz^{\star}_1 \leftarrow \tzzz{\star, -1}_{1}$ and $\hzzz{\ell}_1 \leftarrow \tzzz{\ell, -1}_i , \forall \ell\in[L]$\;
\For{$i =2,\hdots, n$}{
	$\bhz^\star_i = \argmax_{s_\star\in\{\pm 1\}} \# \bigg\{ \{j\in[n]: \tzzz{\star, -1}_j = s_\star\} \bigcap \{j\in[n]: \tzzz{\star, -i}_j = \tzzz{\star, -i}_i\} \bigg\}$\; 
	\For{$\ell = 1, \hdots ,L$}{
		$\hzzz{\ell}_i = \argmax_{s\in\{\pm 1\}} \# \bigg\{\{j\in[n]: \tzzz{\ell, -1}_j = s\} \bigcap \{j\in[n]: \tzzz{\ell, -i}_j = \tzzz{\ell, -i}_i\} \bigg\}$\; 
	}
}
\Return{$\bhz^\star, \{\hzzz{\ell}\}_1^L$}\;
\caption{A provable version of Algorithm \ref{alg: generic_refinement}}
\label{alg: provable_refinement}
\end{algorithm}

To facilitate the analysis, we propose a modified version as shown in Algorithm \ref{alg: provable_refinement}. Instead of taking an arbitrary initial estimator as input (as done in Algorithm \ref{alg: generic_refinement}), we consider a {leave-one-out} initialization scheme.
In Stage \RN{1}, for each fixed $i$, the initial estimator $\tzzz{\star, -i}$ of $\bz^{\star}_{-i}$ are computed using only $\{\AAA{\ell}_{-i}\}_1^L:=\{\AAA{\ell}_{jk}: j, k\neq i, \ell \in[L]\}$, which ensures the conditional (on the realization of $\{\zzz{\ell}\}_1^L$) independence between $\{\tzzz{\ell, -i}\}_1^L$ and $\{\AAA{\ell}_{ij}: j\neq i, \ell\in[L]\}$, thus simplifying the analysis, though the final analysis still turns out to be highly nontrivial. 

In Stage \RN{2}, for each $i\in[n]$, we conduct MAP-based refinement using the initial estimators $\{\tzzz{\ell, -i}_{-i}\}_{\ell=1}^L$ (which are all equal to $\tzzz{\star, -i}_{-i}$) of $\zzz{\ell}_{-i}$, and the ``diagonal slots'' $(\tzzz{\star, -i}_i, \{\tzzz{\ell, -i}_i\}_{\ell=1}^L)$  are all zeros before Stage \RN{2} by our construction.
These ``diagonal slots'' are then filled in by the refined estimators of $(\bz^\star_i, \{\zzz{\ell}_i\}_{\ell=1}^L)$.

After Stage \RN{2}, it is temping to directly output $\bhz^\star_i = \tzzz{\star, -i}_i$ and $\hzzz{\ell}_i = \tzzz{\ell, -i}_i$ as the final estimators. 
However, a subtlety arises due to the leave-one-out initialization. Since the initial estimators $\{\tzzz{\star ,-i}_{-i}\}_{i=1}^n$ are not necessarily {aligned}, the refined estimators $\{\tzzz{\star ,-i}_{i}\}_{i=1}^n$ and $\{\tzzz{\ell, -i}_{i}: i\in[n], \ell\in[L]\}$ can have different {orientations}. 
For example, it could happen that $\tzzz{\star, -1}_1$ is estimating $\bz^{\star}_1$, but $\tzzz{\star, -2}_2$ is estimating $-\bz^{\star}_{2}$. 
This is where the extra Stage \RN{3} of Algorithm \ref{alg: provable_refinement} comes into play. By using an {alignment} procedure, all coordinates of $\bhz^\star$ and $\{\hzzz{\ell}\}_1^L$ will have the same orientation with high probability. 

We shall remark that Algorithm \ref{alg: provable_refinement} is mostly of theoretical interest, and similar strategies have appeared in \cite{gao2017achieving,gao2018community}. 
Our simulation in Section \ref{sec:exp} indicates that the estimation accuracy of Algorithm \ref{alg: generic_refinement} is indistinguishable from that of Algorithm \ref{alg: provable_refinement}, while Algorithm \ref{alg: generic_refinement} is much faster in speed. 
Such a near perfect match in accuracy between the two algorithms is itself an interesting phenomenon, which we leave for future work. A promising approach for analyzing Algorithm \ref{alg: generic_refinement} is the ``leave-one-out'' analysis such as that used in \cite{ma2018implicit}.

Before we present the main result of this subsection, we introduce the following assumption on consistent initialization.
\begin{assump}[Consistent initialization]
	\label{assump:consistent_init}
	 Assume the \textnormal{\texttt{Initialize}} procedure used in Algorithm \ref{alg: provable_refinement} takes an instance generated by an $\textnormal{\modelname} \in \calP_n(\rho, \{p_\ell\}_1^L, \{q_\ell\}_1^L, \beta)$ as its input and outputs a $\btz^{\star}$ satisfying
	\begin{equation}
		\label{eq:consistent_init} 
		\bbP\big(\calL(\btz^\star, \bz^\star) \geq \delta_{\init,n} \big) \lesssim n^{-(1+\ep_\init)}
	\end{equation}
	for some $\delta_{\init,n}=o(1)$ and $\ep_\init > 0$.
\end{assump}

\subsubsection{Performance for \globest}
The performance of Algorithm \ref{alg: provable_refinement} for \globest~is given by the following theorem.
\begin{theorem}[Performance of MAP-based refinement for \globest]
  \label{thm:refine_glob}
  Let the input to Algorithm \ref{alg: provable_refinement} be an instance generated by an $\textnormal{\modelname} \in \calP_n(\rho, \{p_\ell\}_1^L, \{q_\ell\}_1^L, \beta)$ satisfying $\rho = o(1)$, {$q_\ell < p_\ell \leq (C q_\ell) \land (1-c), \forall \ell\in[L]$, $\beta = 1+o(1)$ and $\log L\ll n^{c'}$, where $C > 1$ and $c, c' \in (0, 1)$ are absolute constants.} 
  Let Assumption \ref{assump:consistent_init} hold and assume that for any $\delta_n = o(1)$, the following holds:
  \begin{align}
    \label{eq:diverging_glob_snr_sum}
    \lim_{n\to\infty } \sum_{S\subseteq[L]} e^{-(1-\delta_n) \globinfo_S} =0,
  \end{align}
  where $\globinfo_S$ is defined in \eqref{eq:globinfo}.
  Then, there exist two sequences $\overline{\delta}_n, \overline{\delta}_n' = o(1)$ such that
  \begin{equation}
    \label{eq:refine_glob}
    \lim_{n\to \infty} \inf_{\bz^\star \in \calP_n} \bbP\bigg[ \calL(\bhz^\star , \bz^\star) \leq \bigg(\sum_{S\subseteq[L]} e^{-(1-\overline{\delta}_n)\globinfo_S}\bigg)^{1-\overline{\delta}_n'} \bigg]  = 1.
  \end{equation}
\end{theorem}
\begin{proof}
	See Section \ref{prf:thm:refine_glob}.
\end{proof}

Note that the lower bound given in \eqref{eq:lower_bound_z_star_vanishing} takes the form of the {maximum} of $2^L$ terms indexed by $S\subseteq[L]$, whereas the upper bound given in \eqref{eq:refine_glob} is a {summation} of $2^L$ terms. Our later analysis in Section \ref{subsec:minimax_optimality} shows that under slightly stronger conditions on the SNR, the upper and lower bounds match asymptotically.

Under homogeneity ($\rho = 0$), we have $\globinfo_S = \infty$ for every $S$ but $S = [L]$. So we have 
\begin{equation}
	\label{eq:refine_glob_hom_informal}
	\sum_{S\subseteq[L]} e^{-(1-\delta_n)\globinfo_S} = e^{-(1-\delta_n)\globinfo_{[L]}} = \exp\big\{-(1-\delta_n) \frac{n}{2}\sum_{\ell\in[L]} \III{\ell}_{1/2}\big\},
\end{equation}	
and the upper bound in \eqref{eq:refine_glob} matches the lower bound provided by Corollary \ref{cor:lower_bound_z_star_hom}. 

However, the derivation of \eqref{eq:refine_glob_hom_informal} is not fully rigorous, because the layer-wise objective function $\fff{\ell}_i$ defined in \eqref{eq:local_map_objective} becomes infinity when $\rho = 0$, which makes the optimization problem in \eqref{eq:joint_refinement} ill-defined. 
To address this issue, let us note that when $\rho = 0$, the ``regularization term'' in $\fff{\ell}_i$, namely $\log\big((1-\rho)/\rho\big)\cdot \indc{s_\ell = s_\star}$, essentially requires $s_\ell$ to \emph{exactly agree} with $s_\star$. Thus, we can shift from solving \eqref{eq:joint_refinement} to solving the following problem:
\begin{align}
	\label{eq:mle_refine}
	\bhz^\star_i = \argmax_{s_\star \in\{\pm 1\}} \sum_{\ell\in[L]} \sum_{j\neq i: \btz^{\star}_j = s_\star} \bigg[\log\bigg(\frac{p_\ell(1-q_\ell)}{q_\ell(1-p_\ell)}\bigg) \AAA{\ell}_{ij} + \log \bigg(\frac{1-p_\ell}{1-q_\ell}\bigg)\bigg].
\end{align}
With the above optimization formulation, Algorithm \ref{alg: provable_refinement} can be modified in a \emph{mutatis mutandis} fashion, and the upper bound in \eqref{eq:refine_glob_hom_informal} can be made rigorous, as detailed in the following corollary.
\begin{corollary}[Performance of MAP-based refinement for \globest~under homogeneity]
  \label{cor:refine_glob_hom}
  Consider again Algorithm \ref{alg: provable_refinement}, except that we change its Stage \RN{2} from MAP-based refinement \eqref{eq:joint_refinement} to maximum-likelihood-based refinement \eqref{eq:mle_refine}.  
  Let the input be an instance generated by an $\textnormal{\modelname} \in \calP_n(\rho, \{p_\ell\}_1^L, \{q_\ell\}_1^L, \beta)$ satisfying $\rho = 0$, {$q_\ell < p_\ell \leq (C q_\ell) \land (1-c), \forall \ell\in[L]$, $\beta = 1+o(1)$ and $\log L\ll n^{c'}$, where $C > 1$ and $c, c' \in (0, 1)$ are absolute constants.}
  Let Assumption \ref{assump:consistent_init} hold and assume $n\sum_{\ell\in[L]}\III{\ell}_{1/2} \to \infty$.  
  Then, there exists a sequence $\overline{\delta}_n= o(1)$ such that
  \begin{equation}
    \label{eq:refine_glob_hom}
    \lim_{n\to \infty} \inf_{\bz^\star\in\calP_n}\bbP\bigg( \calL(\bhz^\star , \bz^\star) \leq \exp\big\{-(1-\overline{\delta}_n) \frac{n}{2}\sum_{\ell\in[L]} \III{\ell}_{1/2}\big\}  \bigg)  = 1.
  \end{equation}
\end{corollary}
\begin{proof}
	The proof is a straightforward adaptation of the proof of Theorem \ref{thm:refine_glob}, and we omit the details.
\end{proof}

Minimax optimal algorithms for community detection in a homogeneous \mlsbm~have appeared in the literature \cite{paul2016consistent,xu2020optimal}. The procedure in \cite{paul2016consistent} is based on exactly solving the maximum likelihood objective, which is computationally infeasible. The algorithm in \cite{xu2020optimal} is computable in polynomial-time and it operates on a variant of \sbm, called {weighted \sbm}, of which the homogeneous multilayer SBM is a special case. The corollary above gives another polynomial-time minimax optimal algorithm for community detection in homogeneous multilayer SBMs.




\subsubsection{Performance for \indest}
The performance guarantee of Algorithm \ref{alg: provable_refinement} for \indest~is given by the following theorem.

\begin{theorem}[Performance of MAP-based refinement for \indest]
  \label{thm:refine_ind}
  Let the input to Algorithm \ref{alg: provable_refinement} be an instance generated by an $\textnormal{\modelname} \in \calP_n(\rho, \{p_\ell\}_1^L, \{q_\ell\}_1^L, \beta)$ satisfying $\rho = o(1)$, {$q_\ell < p_\ell \leq (C q_\ell) \land (1-c), \forall \ell\in[L]$, $\beta = 1+o(1)$ and $\log L\ll n^{c'}$, where $C > 1$ and $c, c' \in (0, 1)$ are absolute constants.}
  Let Assumption \ref{assump:consistent_init} hold and assume that for a fixed $\ell\in[L]$ and for any $\delta_n = o(1)$, the following holds:
  \begin{align}
    \label{eq:diverging_ind_snr_sum}
    \lim_{n\to\infty } \sum_{S\subseteq[L]\setminus\{\ell\}} \Big( e^{-(1-\delta_n) \globinfo_{S\cup\{\ell\}}} + e^{-(1-\delta_n)\indinfo_{S\cup\{\ell\}}} \Big) =0,
  \end{align}
  where $\globinfo_{S\cup\{\ell\}}$ and $\indinfo_{S\cup\{\ell\}}$ are defined in \eqref{eq:globinfo} and \eqref{eq:indinfo}, respectively. 
  Then, there exist two sequences $\overline{\delta}_n, \overline{\delta}_n' = o(1)$, independent of $\ell$, such that
  \begin{equation}
    \label{eq:refine_ind}
    \lim_{n\to \infty}\inf_{\bz^\star\in\calP_n}\bbP\bigg[ \calL(\hzzz{\ell} , \zzz{\ell}) \leq \bigg(\sum_{S\subseteq[L]\setminus\{\ell\}} \Big[ e^{-(1-\overline{\delta}_n)\globinfo_{S\cup\{\ell\}}} + e^{-(1-\overline{\delta}_n) \indinfo_{S\cup\{\ell\}}} \Big]
    \bigg)^{1-\overline{\delta}_n'}
     \bigg]  = 1.
  \end{equation}
\end{theorem}
\begin{proof}
	See Section \ref{prf:thm:refine_ind}.
\end{proof}
In this upper bound, the terms involving $\indinfo_{S\cup\{\ell\}}$'s come from estimating $\zzz{\ell}$ {given the knowledge of $\bz^\star$} (i.e., error from the label sampling model defined in \eqref{eq:rho_flipping_model}), whereas the terms involving $\globinfo_{S\cup\{\ell\}}$'s come from {empirically estimating $\bz^\star$}. 

	Similar to Theorem \ref{thm:refine_glob}, the bound given in this theorem is a {summation} of $2\times 2^{L-1}$ terms, whereas the corresponding lower bound in \eqref{eq:lower_bound_z_l_vanishing} is the {maximum} of $2^{L-1} + 1$ terms. We will show in Section \ref{subsec:minimax_optimality} that the two bounds asymptotically coincide under slightly stronger assumptions on both global and individualized SNRs.

	If $\rho = 0$, then $\globinfo_{S\cup\{\ell\}}$ is infinity for every $S$ but $S = [L]\setminus \{\ell\}$. On the other hand, $\indinfo_{S\cup\{\ell\}}$ is infinity for any $S \subseteq[L]\setminus\{\ell\}$. It follows that the upper bound in \eqref{eq:refine_ind} becomes
	$$
		e^{-\big(1-o(1)\big)\globinfo_{[L]}} = \exp\big\{-\big(1-o(1)\big) \frac{n}{2}\sum_{\ell\in[L]} \III{\ell}_{1/2}\big\},
	$$
	which agrees with the upper bound in \eqref{eq:refine_glob_hom}. 
	This makes sense as global and \indest~coincide under homogeneity.

\subsection{Minimax Optimality}\label{subsec:minimax_optimality}
Recall that the two upper bounds in Theorems \ref{thm:refine_glob} and \ref{thm:refine_ind} are both {summations} of {exponentially many} (in $L$) terms indexed by some subset $S\subseteq[L]$, whereas the corresponding lower bounds in Theorems \ref{thm:lower_bound_z_star} and \ref{thm:lower_bound_z_l} are both {maxima} of that many terms. Thus, a priori, there is no reason to believe that the upper and lower bounds should match, especially when $L$ tends to infinity with $n$. 
However, in this subsection, we show that this is indeed the case under mild regularity conditions, establishing {asymptotic minimaxity} of Algorithm \ref{alg: provable_refinement} for both global and individualized estimation.

\subsubsection{Minimax optimality for \globest}
Based on \eqref{eq:refine_glob}, a naive argument would upper bound $\sum_{S\subseteq[L]}e^{-(1-\overline{\delta}_n)\globinfo_S}$ by $2^L \exp\{-(1-\overline{\delta}_n)\min_{S\subseteq[L]}\globinfo_S\}$. 
In order to match the lower bound \eqref{eq:lower_bound_z_star_vanishing}, we need to assume that
$
L \ll \min_{S\subseteq[L]}\globinfo_S.
$
It turns out such a requirement on the growth rate of $L$ can be substantially relaxed, as detailed in the next theorem. 
\begin{theorem}[Minimax optimality of MAP-based refinement for \globest]
\label{thm:opt_glob_est}
Consider the setup of Theorem \ref{thm:refine_glob}, except that instead of assuming \eqref{eq:diverging_glob_snr_sum}, we now assume
\begin{equation}
	\label{eq:req_for_opt_glob_est}
	\log L + L e^{-c'J_\rho} \ll \min_{S\subseteq[L]} \globinfo_S \to \infty \qquad \textnormal{as } n\to\infty
\end{equation}
for some $c'\in(0, 1)$. Then, there exists a sequence $\overline{\delta}_n = o(1)$ such that
\begin{equation}
	\label{eq:tight_upper_bound_z_star}
	\lim_{n\to \infty}\inf_{\bz^\star\in\calP_n}\bbP\bigg( \calL(\bhz^\star , \bz^\star) \leq \exp\big\{-(1-\overline{\delta}_n) \min_{S\subseteq[L]}\globinfo_S\big\}  \bigg)  = 1.
\end{equation}
\end{theorem}
\begin{proof}
	See Section \ref{prf:thm:opt_glob_est}.
\end{proof}

Consider the case where $\rho = n^{-c''}$ for some constant $c''>0$. In this case, we have $J_\rho = \oone \frac{1}{2}\log \frac{1}{\rho} = \oone \frac{c''}{2}\log n$, and thus the requirement in \eqref{eq:req_for_opt_glob_est} becomes $\log L + L n^{-\oone c'c''/2}\ll \min_{S\subseteq[L]}\globinfo_S$, a vast improvement over the naive requirement of $L \ll \min_{S\subseteq[L]}\globinfo_S$.
In the homogeneous case of $\rho = 0$, the requirement in \eqref{eq:req_for_opt_glob_est} becomes $\log L \ll \globinfo_{[L]} = \frac{n}{2}\sum_{\ell\in[L]} \III{\ell}_{1/2}$.

In the proof, we need to identify the optimal $S\subseteq[L]$ such that $\globinfo_S$ is minimized. While it is easy to do so when $\rho = 0$, this task turns out to be challenging in the presence of inhomogeneity, and our proof is based on a nontrivial application of a generalization of Von Neumann's minimax theorem \cite{neumann1928theorie} due to \citet{sion1958general}.



\subsubsection{Minimax optimality for \indest.}
Similar to the case of \globest, in order to prove the tightness of \eqref{eq:refine_ind}, a naive argument would require $L \ll \big(\min_{S\subseteq[L]\setminus\{\ell\}}\globinfo_{S\cup\{\ell\}}\big)\land \big(\min_{S\subseteq[L]\setminus\{\ell\}}\indinfo_{S\cup\{\ell\}}\big)$, and this is relaxed in the following theorem.
\begin{theorem}[Minimax optimality of MAP-based refinement for \indest]
\label{thm:opt_ind_est}
Consider the setup of Theorem \ref{thm:refine_ind}, except that instead of assuming \eqref{eq:diverging_ind_snr_sum}, we now assume
\begin{equation}
	\label{req_for_ind_est}
	\log L + L e^{-c' J_\rho} \ll 
	\min_{S\subseteq[L]\setminus\{\ell\}} \globinfo_{S\cup\{\ell\}} 
	\land \indinfo_{\{\ell\}}\to \infty \qquad \textnormal{as } n \to \infty
\end{equation}
for some $c'\in(0, 1)$. Then, there exists a sequence $\overline{\delta}_n = o(1)$ such that
\begin{align}
	\label{eq:tight_upper_bound_z_l}
	\lim_{n\to \infty}\inf_{\bz^\star\in\calP_n}\bbP\bigg( \calL(\hzzz{\ell} , \zzz{\ell}) \leq \exp\big\{-(1-\overline{\delta}_n) \min_{S\subseteq[L]\setminus\{\ell\}}\globinfo_{S\cup\{\ell\}}\big\} + \exp\big\{-(1-\overline{\delta}_n)\indinfo_{\{\ell\}}\big\} 
	\bigg)  = 1.
\end{align}
\end{theorem}
\begin{proof}
	See Section \ref{prf:thm:opt_ind_est}.
\end{proof}

In the proof, we identify that the optimal $S$ that minimizes $\indinfo_{S\cup\{\ell\}}$ is precisely the empty set. On the other hand, the identification of the set that minimizes $\globinfo_{S\cup\{\ell\}}$ is done in a similar fashion as in the proof of Theorem \ref{thm:opt_glob_est}.

\section{Numerical Experiments}\label{sec:exp}

In this section, we conduct simulation studies to corroborate our theoretical results. Since we focus on the symmetric case, we can without loss of generality assume $\bz^\star_i = +1$ for $1\leq i\leq \lfloor n/2\rfloor$ and $\bz^\star_i = -1$ for $\lfloor n/2\rfloor + 1\leq i\leq n$.
The $L$ layers are divided into three disjoint groups:
\begin{enumerate}
	\item \emph{Weak layers.} For $1\leq \ell\leq \floor{0.3L}$, we let $p_\ell = {c}/(nL), q_\ell = 1/(nL)$ for some constant $c$ that controls the amount of information. Since $\III{\ell}_{1/2} \asymp 1/(nL)$, in view of the lower bound \eqref{eq:lower_bound_z_star_vanishing}, these layers, even when pooled together, cannot consistently estimate $\bz^\star$. 
	
	\item \emph{Intermediate layers.} For $\floor{0.3L}+1\leq \ell \leq \floor{0.95 L}$, we let $p_\ell = c(\log n)/(nL), q_\ell = (\log n)/(nL)$ where $c$ is the same constant as that appears in the weak layers. Note that $\III{\ell}_{1/2} \asymp \log(n)/(nL)$. Thus, while each individual layer does not contain sufficient information for consistent estimation of its own $\zzz{\ell}$, consistent estimation of $\bz^\star$ becomes possible if information is aggregated across these layers.
	
	\item \emph{Strong layers.} For $\floor{0.95 L} + 1\leq \ell \leq L$, we let $p_\ell = c\log(n)/n, q_\ell = \log(n)/n$ again for the same $c$ as above. As $\III{\ell}_{1/2}\asymp \log(n)/n$, these layers are capable of consistently estimating their $\zzz{\ell}$'s, even when treated individually without aggregation. 
\end{enumerate}

The rationale behind the above partition is to simulate the behaviors that are likely to appear in real world multilayer networks. For example, let us consider the case where layers are distinct 
``participants'' 
collaborating with each other, with the hope that they can borrow information from others to better estimate their own $\zzz{\ell}$'s. Such a setup is also known as ``federated learning'' in the machine learning literature \cite{kairouz2019advances}. The intermediate layers are participants with the most incentive in the collaboration, as ``united they stand, divided they fall''. In comparison, the weak layers may not be as incentivized as the intermediate layers, because they would ``fall even when united''. Nevertheless, they may still want to participate as ``hitchhikers''. Finally, the strong layers are participants that would ``stand even when divided'', and the only reason for them to participate is the hope for even more accurate estimation of $\zzz{\ell}$'s. A small proportion of strong layers (0.05 in our case) reflects our belief that the strong layers are relatively scarce. 

In this section, we will always use Algorithm \ref{alg: specc} as the initialization scheme. Our experiment in Section \ref{subappend:weight_choice} shows that using uniform weights (i.e., $\weight_\ell = 1/L,\forall \ell\in[L]$) and setting the trimming threshold $\gamma = 5$ work well in a wide range of scenarios, and we will always use such a hyperparameter choice in the following discussion. 

\begin{figure}[t]
		\begin{subfigure}{.43\textwidth}
			\centering
			\includegraphics[width=1\linewidth]{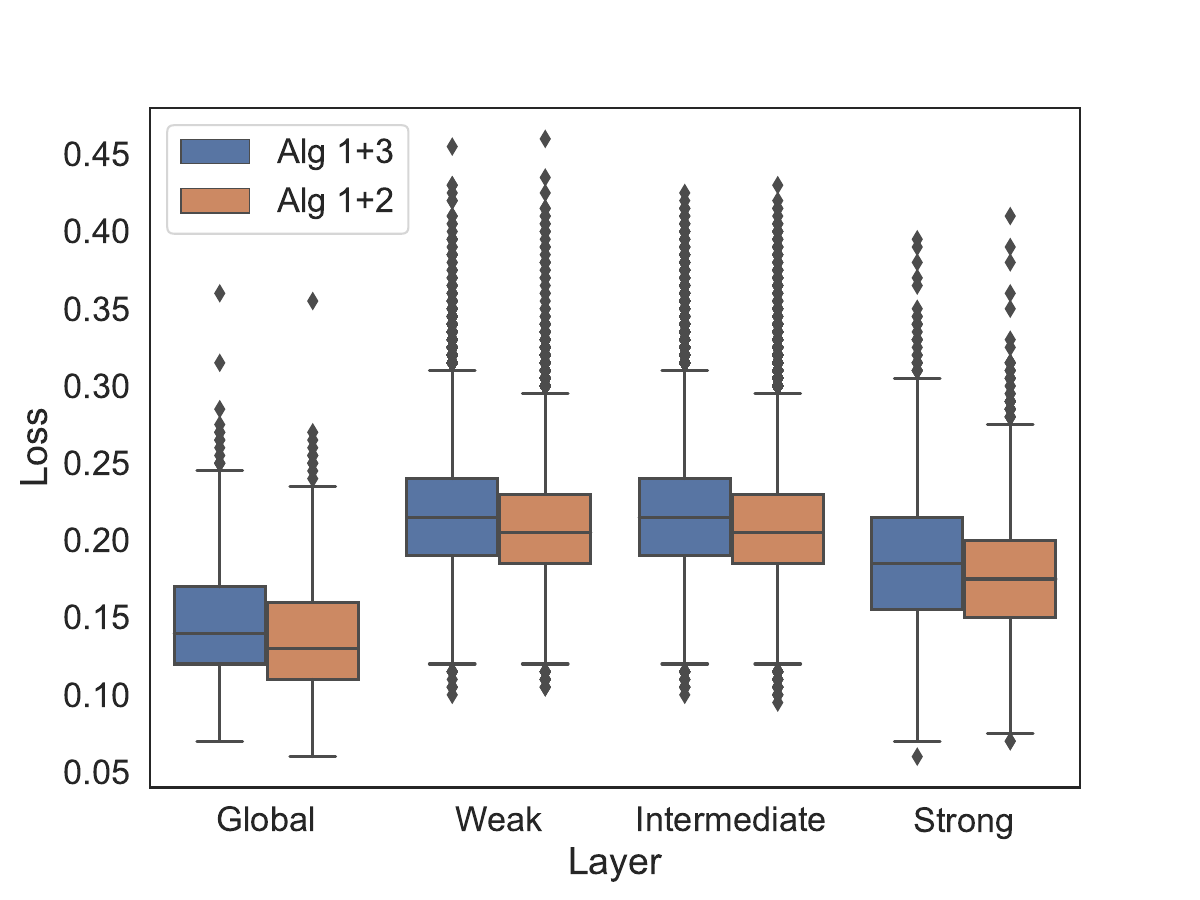}
		\end{subfigure}
		\begin{subfigure}{.43\textwidth}
			\centering
			\includegraphics[width=1\linewidth]{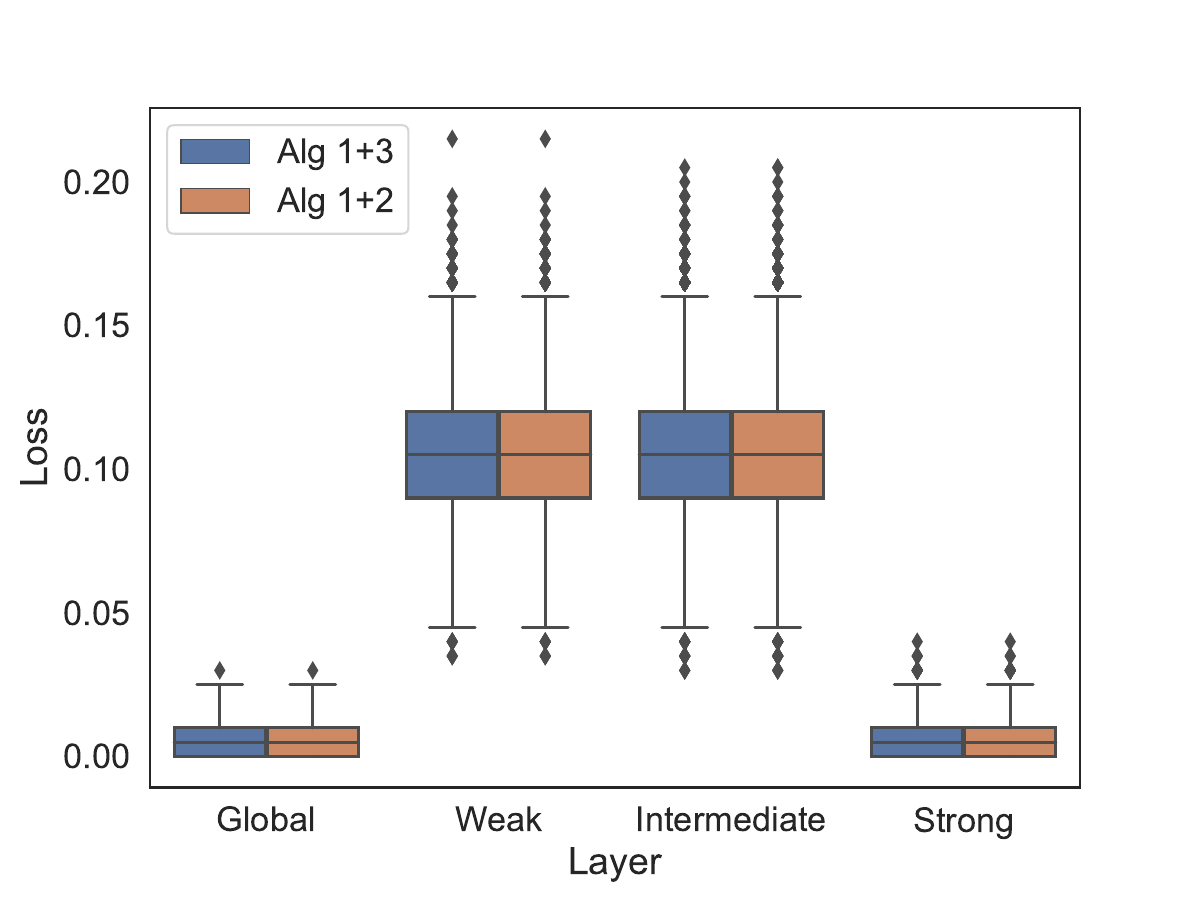}
		\end{subfigure}
        \caption{\small Comparison between estimation accuracies of the generic version (Algorithm \ref{alg: generic_refinement}) and the provable version (Algorithm \ref{alg: provable_refinement}) when $c = 2$ (left) and $c = 5$ (right). In each panel, the four groups from left to right correspond to global estimation and individualized estimation in weak, intermediate, and strong layers, respectively.}
		\label{fig:compare_algs}
\end{figure}

\subsection{Comparison between Algorithms \ref{alg: generic_refinement} and \ref{alg: provable_refinement}} 
Recall that we have developed two versions of the same algorithm: Algorithm \ref{alg: generic_refinement} is fast but we were not able to establish any theoretical guarantees, whereas Algorithm \ref{alg: provable_refinement} is slower but provably optimal. 
We have argued that these two versions should perform similarly in terms of estimation accuracy, and we now empirically justify this claim. 
We set $n = 200, L = 100, \rho = 0.1$ and let $c$ be either $2$ or $5$. 
We then run the two algorithms over $500$ instances of the model (assuming $\{p_\ell\}, \{q_\ell\}, \rho$ are known) and record the misclustering proportions for \globest~(\texttt{Layer=Global}), \indest~in weak layers (\texttt{Layer=Weak}), intermediate layers (\texttt{Layer=Intermediate}) and strong layers (\texttt{Layer=Strong}).

The results are presented in Figure \ref{fig:compare_algs}. We see that the performances of the two versions are indeed similar, and they even become indistinguishable when $c=5$. Thus, in the rest of this section, we always use Algorithm \ref{alg: generic_refinement}.

\begin{figure}[t]
\centering
		\begin{subfigure}{.4\textwidth}
			\centering
			\includegraphics[width=1.1\linewidth]{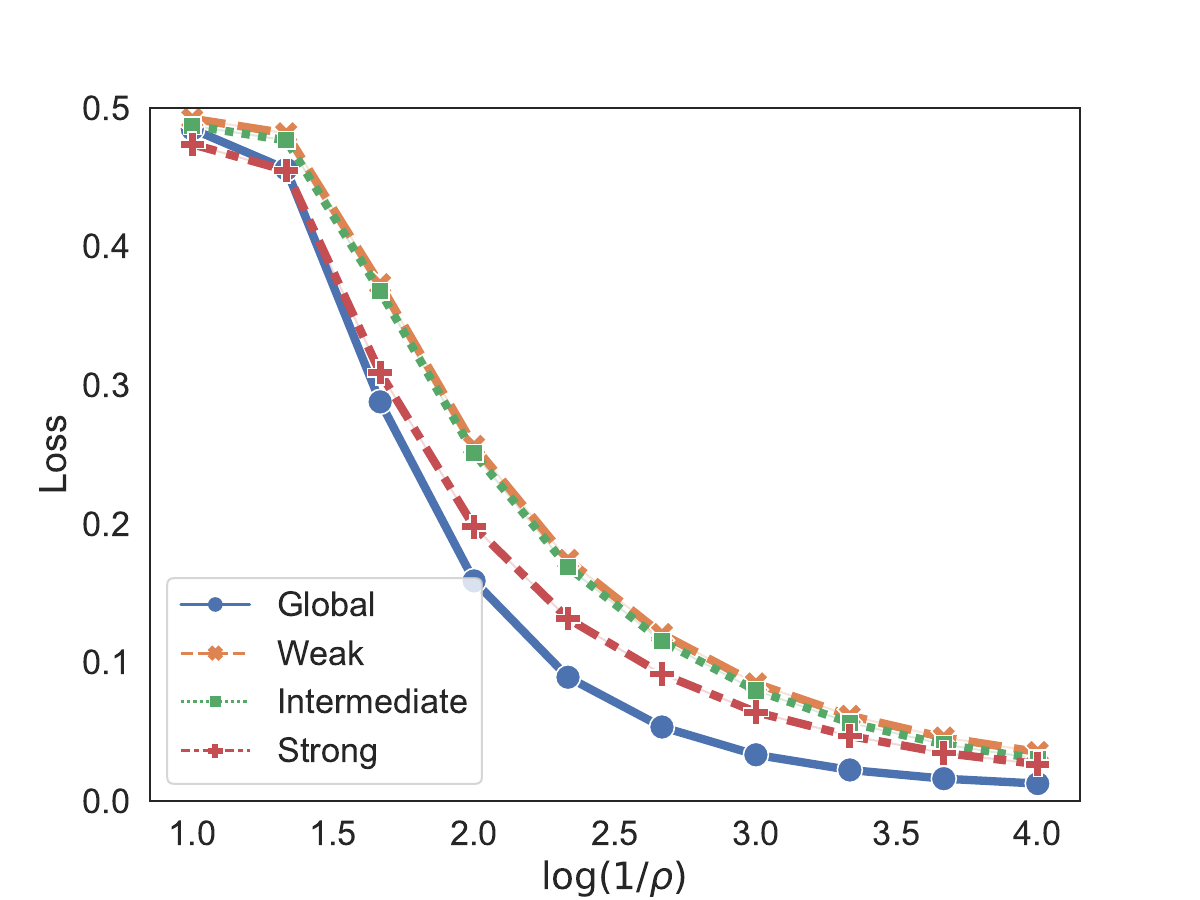}
		\end{subfigure}
		\begin{subfigure}{.4\textwidth}
			\centering
			\includegraphics[width=1.1\linewidth]{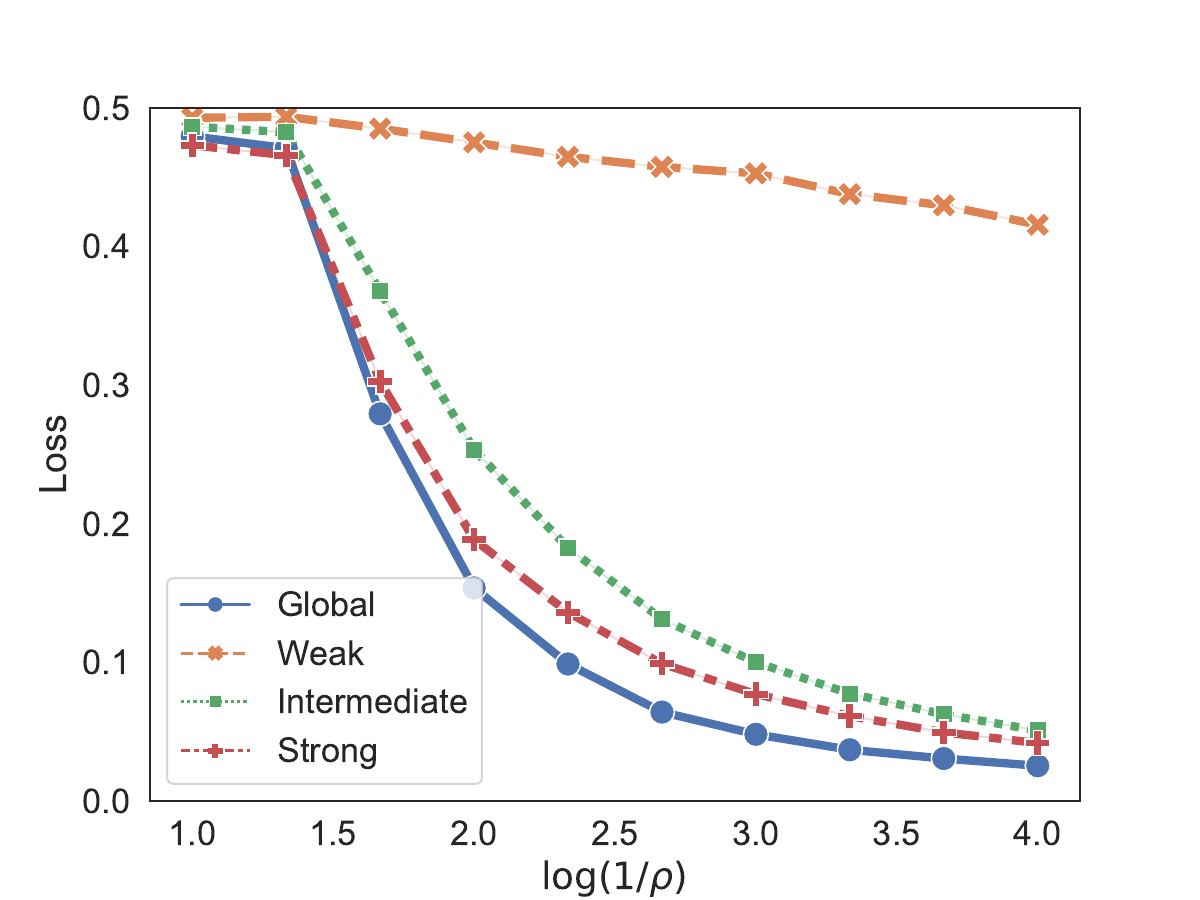}
		\end{subfigure}
        \vskip -0.12\baselineskip
		\begin{subfigure}{.4\textwidth}
			\centering
			\includegraphics[width=1.1\linewidth]{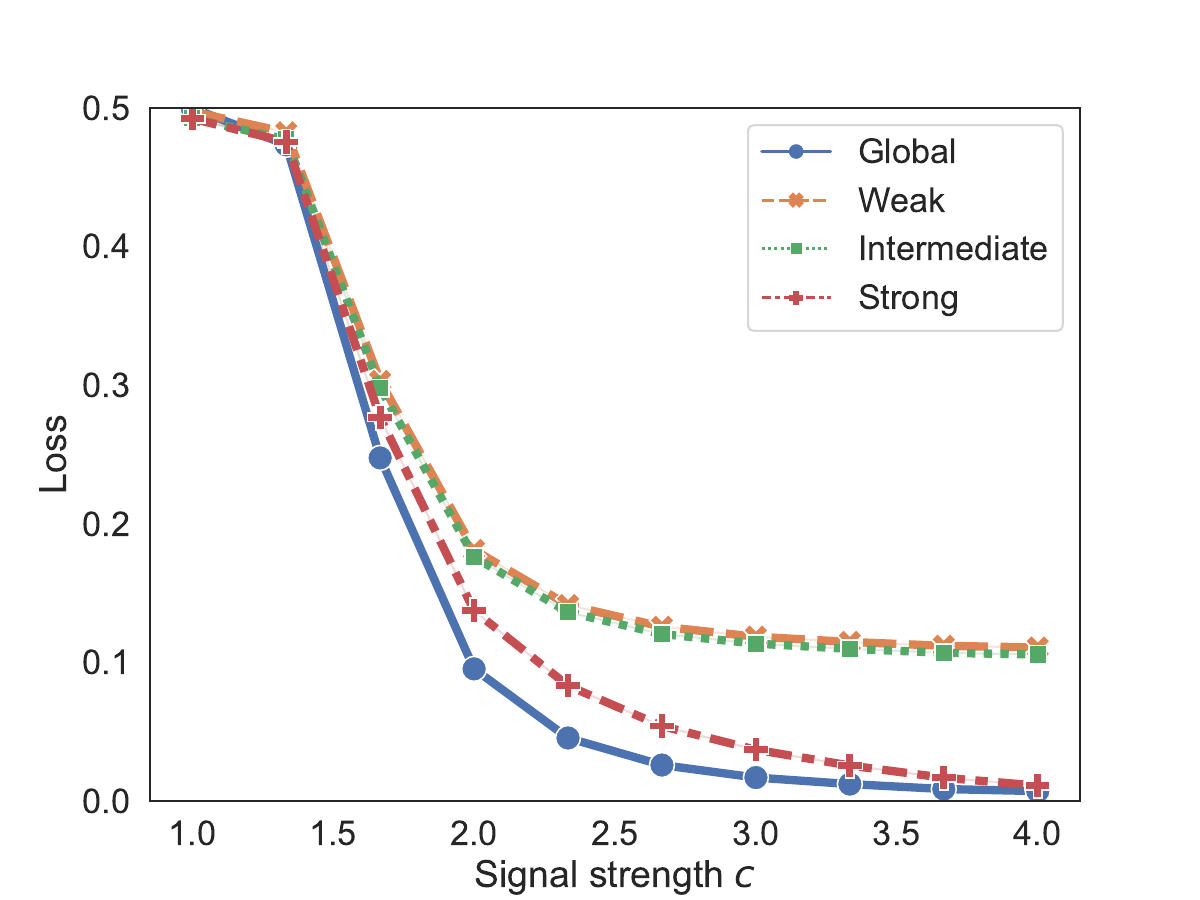}
			\caption{Our method: Algorithm \ref{alg: specc}+\ref{alg: generic_refinement}}
			\label{fig:our_method}
		\end{subfigure}
		\begin{subfigure}{.4\textwidth}
			\centering
			\includegraphics[width=1.1\linewidth]{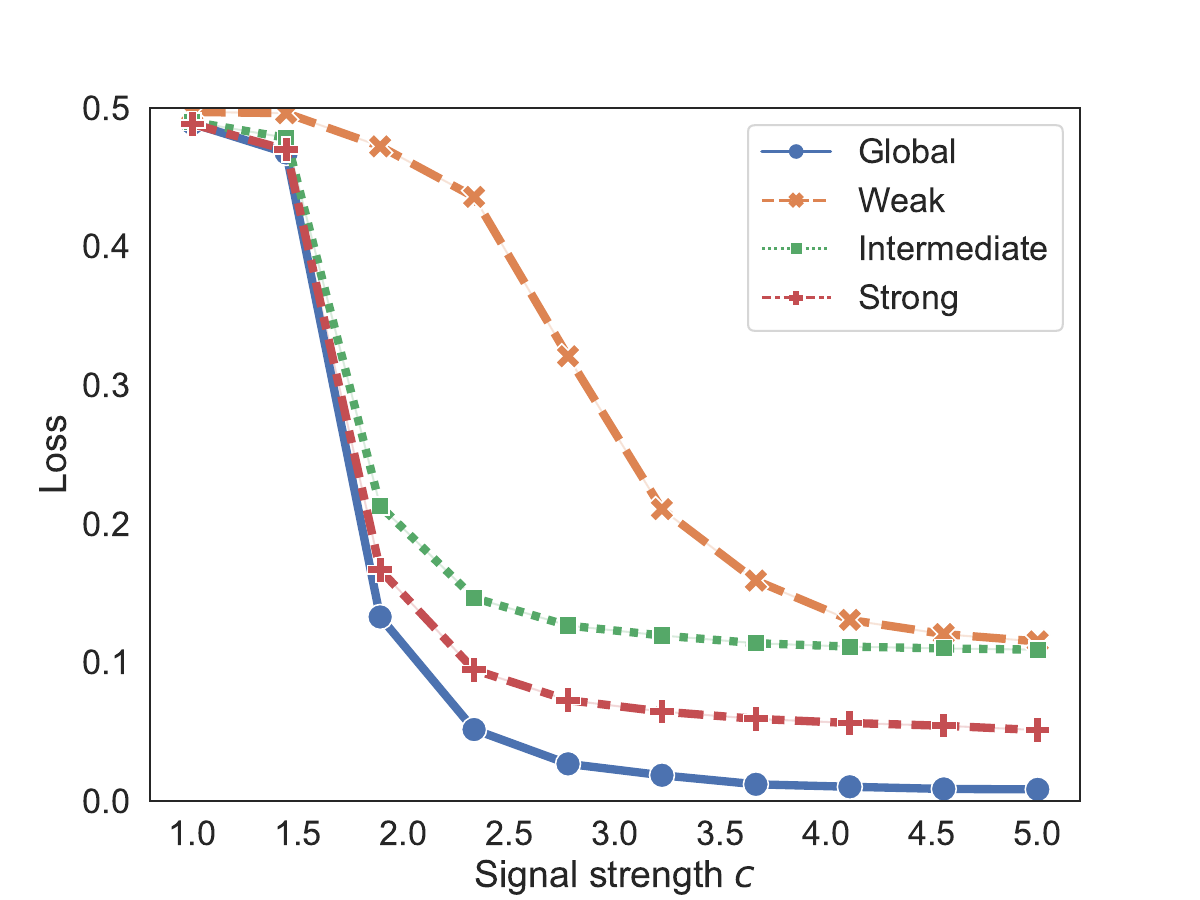}
			\caption{Co-regularized spectral clustering}
			\label{fig:PC20}
		\end{subfigure}
        \caption{\small Performance of our method (left panel) and co-regularized spectral clustering (right panel) \citep{kumar2011co,paul2020spectral} under different SNRs. The top two figures consider the setup where $c = 2$ is fixed and $\log(1/\rho)$ varies, and the bottom two figures consider the setup where $\rho=0.1$ is fixed and $c$ varies.
        }
		\label{fig:effect_of_snr}
\end{figure}

\subsection{Effects of SNRs and comparison with co-regularized spectral clustering} \label{subsec:effect_of_snr}
Recall that the minimax rates for \globest~and \indest~both rely on two information theoretic quantities: $J_\rho = \oone \cdot\frac{1}{2}\log(1/\rho)$ and $\III{\ell}_t =  \oone\cdot (c^t -1)(c^{1-t}-1)\times(\textnormal{scaling of } q_\ell)$. We now experiment on how these two quantities influence the performance of our proposed algorithm. We set $n = 1000, L = 100$, and we either fix $\rho = 0.1$ and vary $c$, or fix $c = 2$ and vary $\log(1/\rho)$. 
We run Algorithm \ref{alg: generic_refinement} over $500$ instances of the model (again assuming $\{p_\ell\}, \{q_\ell\}, \rho$ are known) and record the misclustering proportions for both \globest~and \indest. 
As a comparison, we implement the \emph{co-regularized spectral clustering} algorithm, a popular algorithm for clustering in multilayer networks originally proposed by \cite{kumar2011co} and later shown to be consistent in the $\rho=0$ case by \cite{paul2020spectral}. 
Since co-regularized spectral clustering requires running multiple ``coordinate ascent'' steps, each involving computing the eigen-decomposition of an $n\times n$ matrix, it is substantially slower than our method, and we only run this algorithm for {$50$} instances of the model. 
We refer the readers to Section \ref{subappend:co_reg} for details on this algorithm.

Figure \ref{fig:effect_of_snr} shows the results of this simulation. We see that our method significantly outperforms co-regularized spectral clustering in all scenarios considered. 
By the top-left plot in Figure \ref{fig:effect_of_snr}, the misclustering proportions of our method for both \globest~and \indest~tend to zero if we fix $c=2$ and increase the value of $J_\rho$. 
In contrast, by the bottom-left plot in Figure \ref{fig:effect_of_snr},
for fixed $\rho = 0.1$, misclustering proportions for $\bz^\star$ tend to zero as $c$ becomes large, which is as expected. 
However, for \indest, while errors of strong layers still tend to zero, 
errors of intermediate and weak layers both tend to $\rho = 0.1$. 
This behavior actually is well explained by our theory. 
Note that the minimax rate \eqref{eq:tight_upper_bound_z_l} for \indest~consists of two terms, where the first term represents the error from label sampling \eqref{eq:rho_flipping_model} and scales as $e^{-(1+o(1))\indinfo_{\{\ell\}}} = \rho^{1+o(1)}\times e^{-(1+o(1))\psi^\star_{\{\ell\}} (-2J_\rho)}$, whereas the second term comes from empirically estimating $\bz^\star$, which tends to zero much faster than the first term in the current setting. 
Different behaviors of individualized estimation errors in strong layers and in intermediate/weak layers occur since $\psi^\star_{\{\ell\}}(-2J_\rho)$ is quite large when the $\ell$-th layer is strong while it is nearly zero when the layer only has intermediate or weak signal.

\begin{figure}[t]
		\centering
		\includegraphics[width=0.5\textwidth]{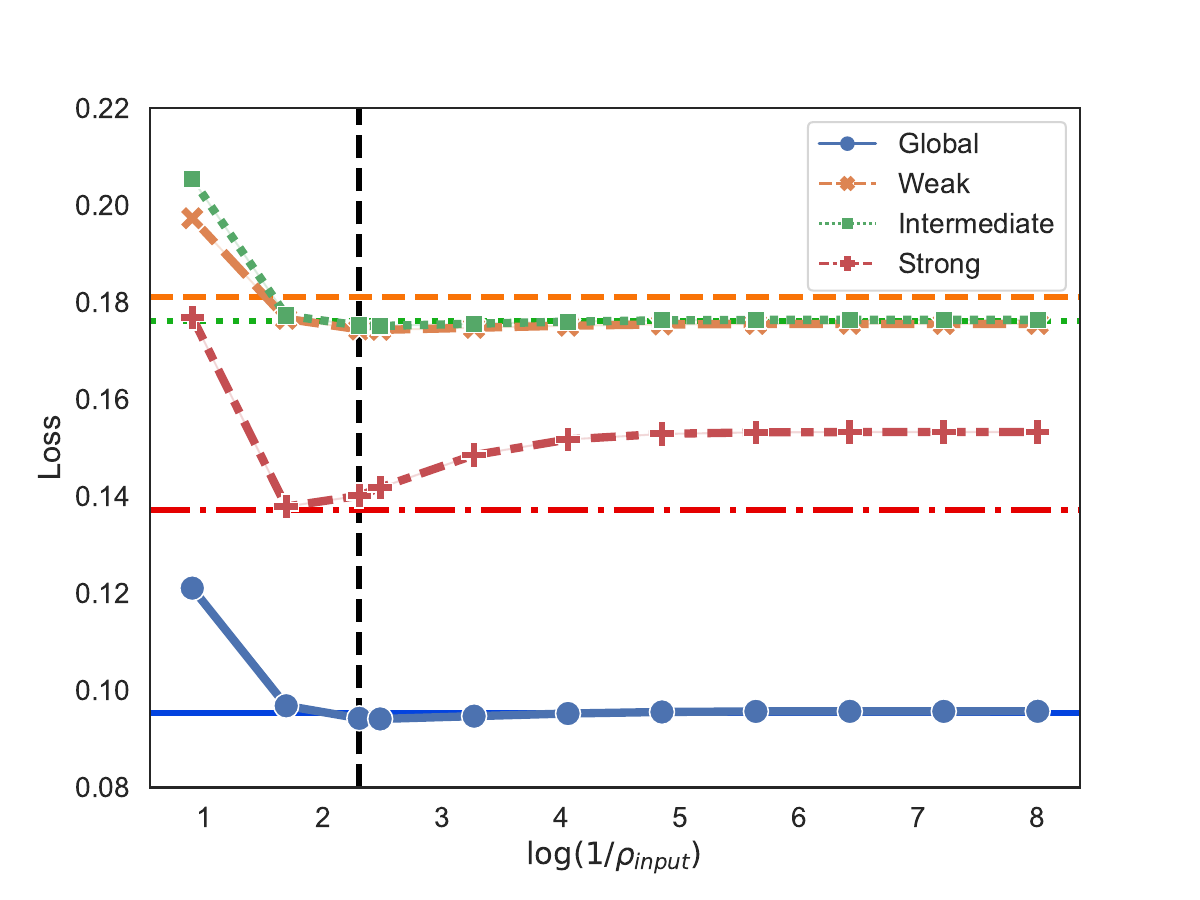}
        \caption{\small Average misclustering proportions against $\log(1/\rho_{\textnormal{input}})$ when $\{p_\ell\}, \{q_\ell\}$ are estimated from data. 
		The black dashed vertical line represents the location of the true $\rho$ that generates the data, and the horizontal lines represent the errors made by Algorithm \ref{alg: generic_refinement} when $\{p_\ell\}, \{q_\ell\}, \rho$ are all known. 
        }
		\label{fig:sens2params}
\end{figure}

\subsection{Sensitivity to inexact parameter specifications} 
The optimality of our proposed algorithm has been established assuming knowledge of the true $\{p_\ell\}, \{q_\ell\}$ and $\rho$. 
In practice, we need to estimate them from data. 
Estimating $\{p_\ell\}$ and $\{q_\ell\}$ is relatively easy --- we could first obtain a crude estimate of $p_\ell$'s using method of moment (see Section \ref{subappend:est_deg} for details), then input that to Algorithm \ref{alg: specc} to obtain an initial estimator $\tilde \bz$ of $\bz^\star$, and finally compute the \emph{intra-cluster} and \emph{inter-cluster} average of edges in each layer, which will be our final estimator of $\{p_\ell\}$ and $\{q_\ell\}$. 

Estimating $\rho$ is, however, a nontrivial task. 
Alternatively, we could treat the input $\rho$ (denoted as $\rho_\textnormal{input}$) to Algorithm \ref{alg: generic_refinement} as a hyperparameter. In this simulation, we examine the sensitivity of the algorithm to the estimated (hence inexact) $\{p_\ell\}, \{q_\ell\}$ and $\rho$. We again set $n = 1000, L = 100$, and we fix $c = 2, \rho = 0.1$. 
We run Algorithm \ref{alg: generic_refinement} with estimated $\{p_\ell\}$ and $\{q_\ell\}$ over $500$ instances of the model with different (misspecified) input values of $\rho$ and plot misclustering proportions for both \globest~and \indest~in Figure \ref{fig:sens2params}.

From Figure \ref{fig:sens2params}, we see that our algorithm is robust to inexact parameters. 
As long as we do not set $\rho_{\textnormal{input}}$ to be too large, the performance of our algorithm with estimated $\{p_\ell\}$ and $\{q_\ell\}$ only slightly degrades compared to when true parameter values are used.

\section{Extension to Multi-Cluster and Asymmetric Cases} 
\label{sec:multi_cluster}

\begin{algorithm}[t]
\SetNoFillComment
\KwIn{Number of communities $K$, initial global estimator $\btz^{\star}$, adjacency matrices $\{\AAA{\ell}\}_1^L$, connecting probabilities $\{(p_\ell, q_\ell)\}_1^L$, flipping probability $\rho$}
\KwOut {Global estimator $\bhz^\star$, individualized estimator $\{\hzzz{\ell}\}_1^L$}
\For{$i = 1, \hdots, n$}{
  \For{$\ell = 1, \hdots, L$}{
    \For{$k=1,\ldots,K$}{
    $\sfz(\ell, k) \leftarrow \argmax_{s\in[K]} f^{(\ell)}_{i}(k, s, \btz^{\star})$
    \tcp*[r]{$f^{(\ell)}_i$ is defined in \eqref{eq:local_map_objective multiclass}}  
  }
  }
  $\bhz^\star_i \leftarrow \argmax_{s_\star\in[K]} \sum_{\ell\in [L]} f^{(\ell)}_i\big(s_\star, \sfz(\ell, s_\star), \btz^{\star}\big)$\tcp*[r]{final global estimator}
  \For{$\ell=1, \hdots, L$}{
    $\hzzz{\ell}_i \leftarrow \sfz(\ell, \bhz^\star_i)$\tcp*[r]{final individualized estimators}
  }
}
\Return{$\bhz^\star, \{\hzzz{\ell}\}_1^L$}\;
\caption{Node-wise refinement via MAP estimation for multi-class IMLSBM}
\label{alg: generic_refinement multiclass}
\end{algorithm}
{In this section, we adapt Algorithm \ref{alg: generic_refinement} to accommodate cases where the number of communities could be more than two and the community sizes could be considerably different.}

We start by describing a canonical generalization of two-block IMLSBM (i.e., the model described by \eqref{eq:label_flip} and \eqref{eq:layerwise_sbm}) below.
Suppose the global community assignment vector is $\bz^\star\in\{1,2,\ldots,K\}^n$, where $K\geq 2$ is the number of communities. 
For each layer $\ell \in [L]$ and each node $i \in [L]$, the individual community assignment $\bz^{(\ell)}_i$ independently follows the multinomial distribution
\begin{align*}
\mathbb{P}\left(\bz^{(\ell)}_i=k\right)=(1-\rho)\cdot \indc{\bz^\star_i=k} + \frac{\rho}{K-1} \cdot \indc{\bz^\star_i\neq k},\quad 1\leq k\leq K
\end{align*}
for some $\rho \in (0, 1)$.
That is, $\bz^{(\ell)}_i$ agrees with $\bz^\star_i$ with probability $1-\rho$, and flips to other communities with equal probabilities. The layer-wise adjacency matrices are still generated according to \eqref{eq:layerwise_sbm}.

\begin{figure}
\centering 
\begin{subfigure}{.45\textwidth}
\includegraphics[width=\textwidth]{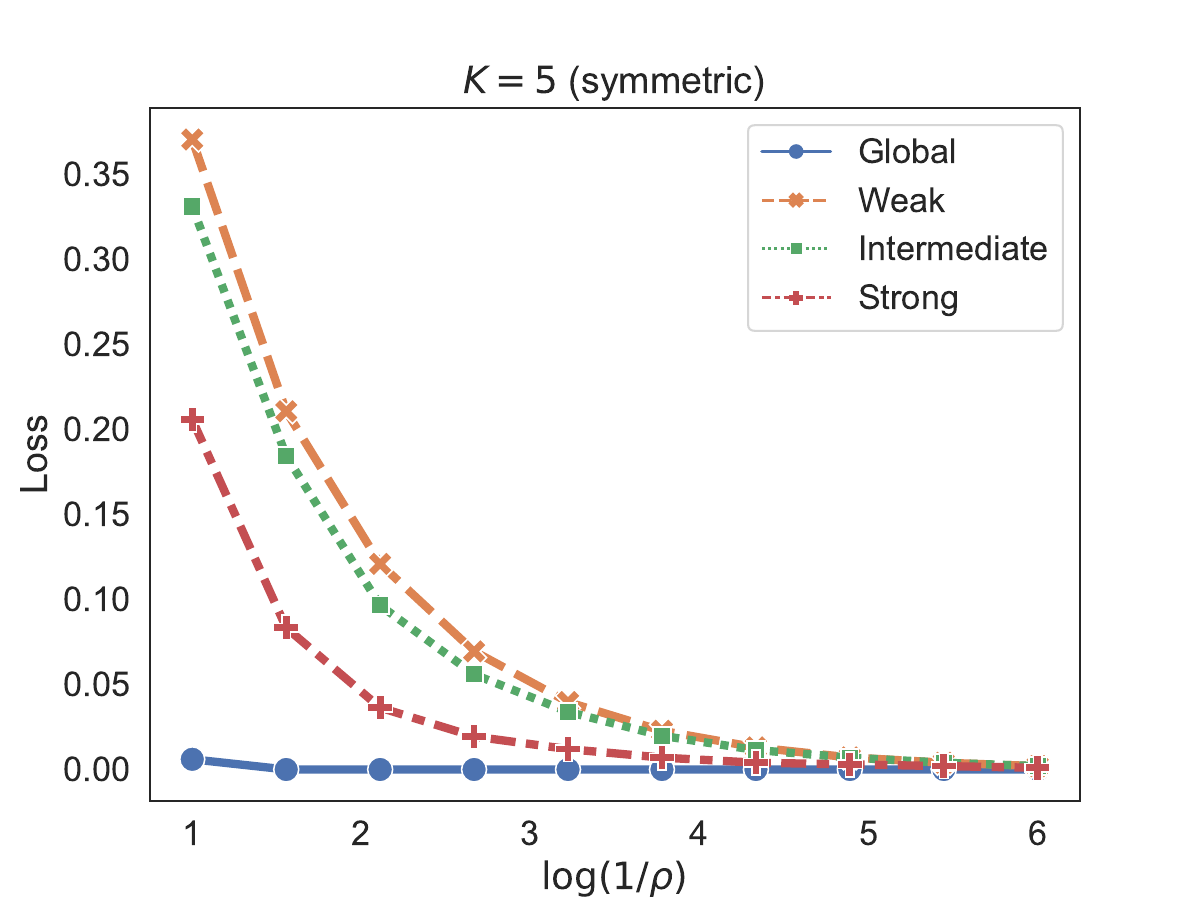}
\end{subfigure}
\begin{subfigure}{.45\textwidth}
\includegraphics[width=\textwidth]{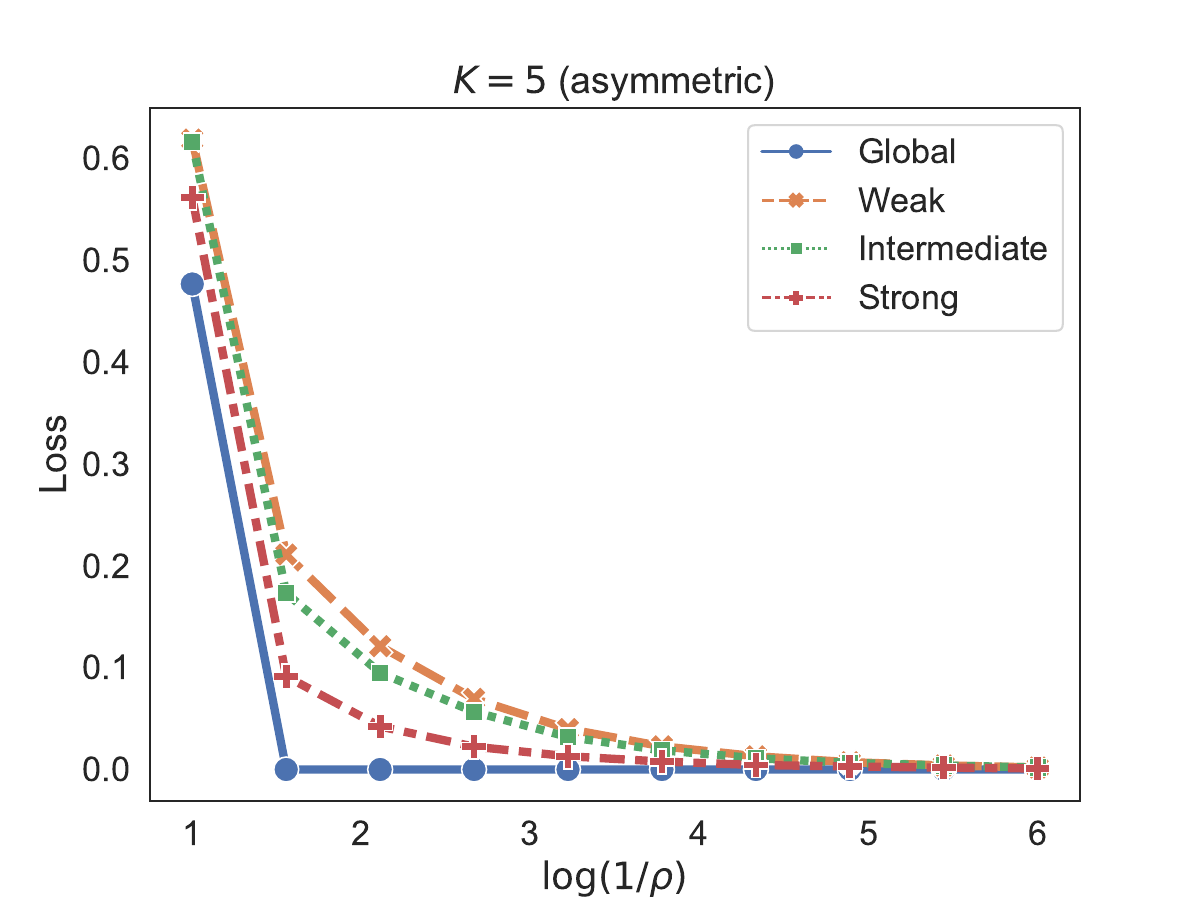}
\end{subfigure}
\begin{subfigure}{.45\textwidth}
\includegraphics[width=\textwidth]{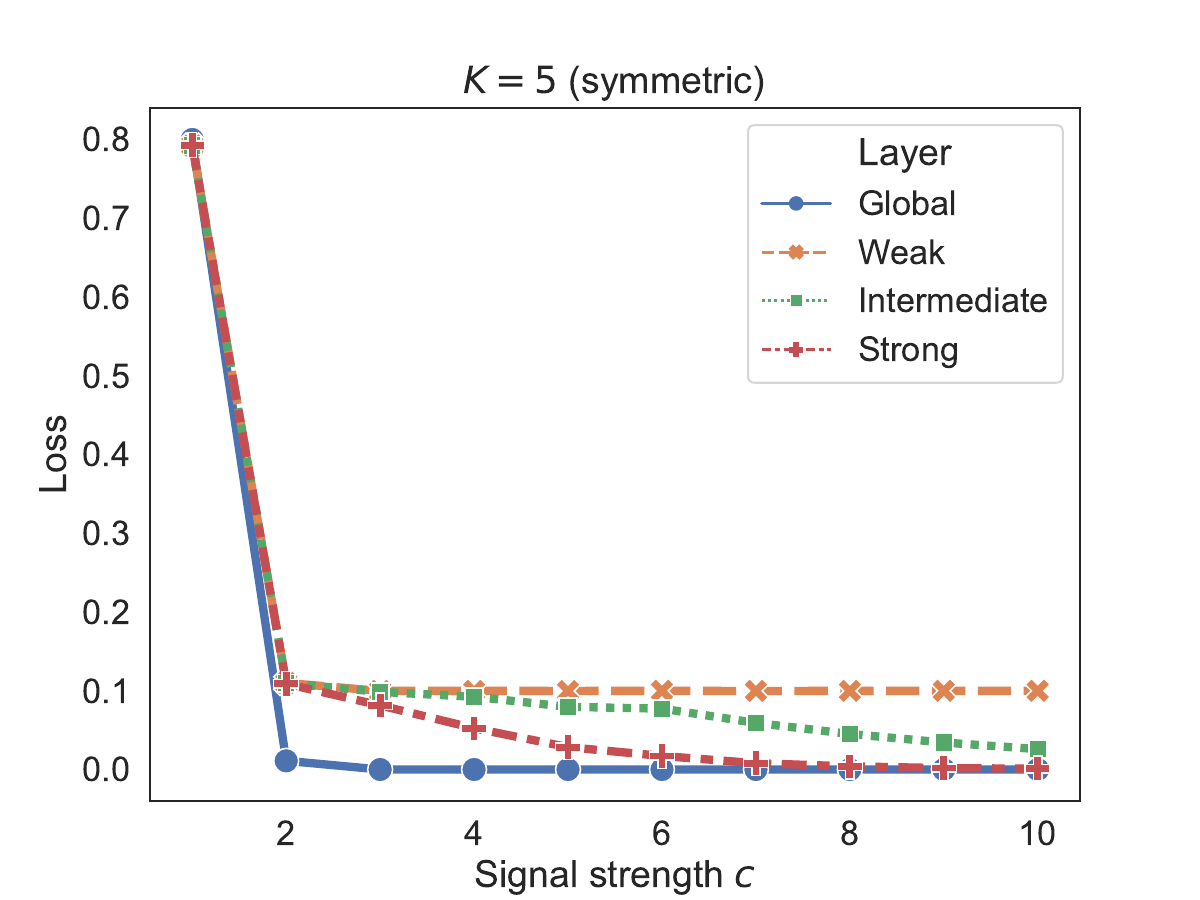}
\end{subfigure}
\begin{subfigure}{.45\textwidth}
\includegraphics[width=\textwidth]{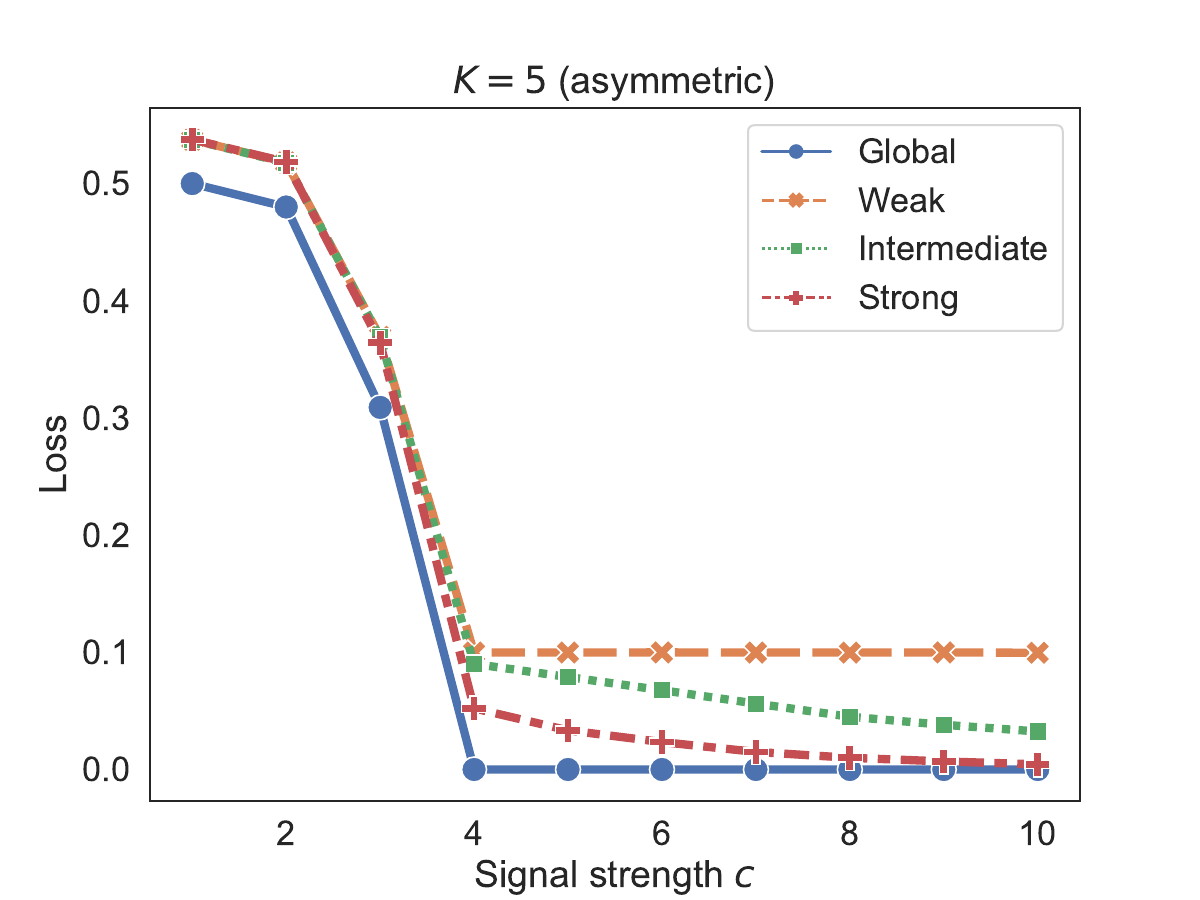}
\end{subfigure}
\caption{Performance of Algorithm \ref{alg: generic_refinement multiclass} for $K=5$ under different SNRs. The top two figures fix $c=5$ and vary $\log(1/\rho)$ from 1 to 6. The bottom two figures fix $\rho=0.1$ and vary $c$ from 1 to 10. The left two figures have symmetric communities, meaning that all 5 communities have the same size in $\bz^\star$. The right two figures have asymmetric communities, where the proportions of the 5 communities are $\frac{1}{2},\frac{1}{4},\frac{1}{12},\frac{1}{12},\frac{1}{12}$.
}
\label{fig: simu multiclass}
\end{figure}

Algorithm \ref{alg: generic_refinement} seamlessly generalizes to the current setting --- one can simply write down the MAP objective function and perform node-wise refinement. All we need to do is to modify \eqref{eq:joint_refinement} to
\begin{align}
  f^{(\ell)}_i(s_\star, s_\ell, \btz^{\star}) & = \log\bigg(\frac{1-\rho}{\rho/(K-1)}\bigg) \cdot \indicator{s_\ell = s_\star} \nonumber\\
  \label{eq:local_map_objective multiclass}
  & \qquad + \sum_{j\neq i: \btz^{\star}_j = s_\ell} \bigg[\log\bigg(\frac{p_\ell(1-q_\ell)}{q_\ell(1-p_\ell)}\bigg) \AAA{\ell}_{ij} + \log \bigg(\frac{1-p_\ell}{1-q_\ell}\bigg)\bigg],
\end{align}
See Algorithm \ref{alg: generic_refinement multiclass} for a detailed description.

We conduct a small scale simulation study to verify the effectiveness of Algorithm \ref{alg: generic_refinement multiclass}. 
In the simulation shown in Figure \ref{fig: simu multiclass}, we take $n=1000$, $L=100$, and $K=5$. We still divide the layers into weak, intermediate and strong according to the description in Section \ref{sec:exp}, except that for intermediate layers, we set $p_\ell = 35 c(\log n)/nL$ and $q_\ell = 35 (\log n) / (nL)$. 
We either fix $c=5$ and vary $\rho$ (top), or fix $\rho=0.1$ and vary $c$ (bottom). 
We consider both symmetric and asymmetric community sizes. 
Let $n_k=\sum_{i=1}^n\indc{\bz^\star_i=k}$ denote the size the $k$-th community in the global assignment. 
For the symmetric case, all $5$ communities in $\bz^\star$ have the same size (i.e., $n_1=\ldots=n_5=200$). 
For the asymmetric case, we take $n_1=500$, $n_2=250$, and the other three communities have sizes $83$, $83$ and $84$.

In this simulation, we observe similar patterns as those in Figure \ref{fig:effect_of_snr} which focused on the symmetric two-block case. 
When we fix $c$ and send $\rho$ to zero, both the loss in the global estimator and the loss in the individualized estimators tend to zero. When $\rho$ is fixed at $0.1$ and $c$ grows, the loss in the individualized estimators in the weak layers approaches $0.1$. 
Our algorithm also performs well in the asymmetric case.

\section{A Real Data Example}\label{sec:real_data}
{Development of biotechnologies have enabled biologists to simultaneously profile multiple modalities at single-cell resolution. 
Here we illustrate our algorithm on a multi-modal single cell data  \citep{hao2021integrated} obtained via the CITE-seq technology \citep{stoeckius2017simultaneous}.}

This dataset contains 211000 human peripheral blood mononuclear cells (PBMCs). For each cell, $20729$ RNA markers and $224$ protein antibody panels were measured and their expression levels were recorded. 
The dataset also contains expert-annotated labels that categorize each cell into one of eight cell types: Monocyte (\texttt{Mono}), \texttt{CD4 T} cell, \texttt{CD8 T} cell, \texttt{other T} cell, Natural killer cell (\texttt{NK}), \texttt{B} cell, Dendritic cell (\texttt{DC}), and \texttt{other} (undefined).

We randomly subsampled $5000$ cells and removed the two smallest clusters: \texttt{DC} (size$=127$) and \texttt{other} (size=$104$). Among the remaining $4769$ cells, we aggregated \texttt{CD4 T}, \texttt{CD8 T}, and \texttt{other T} cells as a single \texttt{T} cell group. 
Thus, we obtained two data matrices $X^{(1)}\in \bbR^{4769\times 20,729}, X^{(2)}\in \bbR^{4769\times 224}$ with RNA and protein information, respectively. 
The rows (cells) are classified into four clusters: \texttt{Mono} (size$=1547$), \texttt{T} (size$=2224$), \texttt{NK} (size$=581$), and \texttt{B} (size$=417$). 
For RNA data, we retained the top $50$\% variable columns (according to standard deviation) and computed a pairwise distance matrix $D^{(1)}\in\bbR^{(4769\times 4769)}$, 
where $D^{(1)}_{i, j}$ is one minus the Pearson correlation coefficient between the $i$-th row and the $j$-th row of $X^{(1)}$. 
We then obtained a  matrix $A^{(1)}_{\texttt{row}}$ by dichotomizing rows of $D^{(1)}$: setting the smallest $3\%$ entries of each row to be one and the rest zero. 
Similarly, we obtain $A^{(1)}_{\texttt{col}}$ by dichotomizing columns of $D^{(1)}$. 
The final adjacency matrix for RNA information is given by $A^{(1)} = (A^{(1)}_{ij})$ where $A^{(1)}_{ii}=0$ and for any $i$ and $A^{(1)}_{ij} = \max\{ (A^{(1)}_{\texttt{row}}\odot A^{(1)}_{\texttt{col}})_{ij}, (A^{(1)}_{\texttt{row}}\odot A^{(1)}_{\texttt{col}})_{ji}\}$ for any $i\neq j$. 
Here $\odot$ stands for element-wise product. 
The adjacency matrix $A^{(2)}$ for protein information was similarly constructed, except that we picked the top $95\%$ variable columns due to limited number of protein panels.

We applied the multi-class version of our algorithm (i.e., Algorithm \ref{alg: specc}+\ref{alg: generic_refinement multiclass}) to the two-layer network $(A^{(1)}, A^{(2)})$ with $K= 4$ and $\rho=0$. Note that $\rho=0$ reveals our prior knowledge that the two individual label vectors should exactly agree. The misclustering proportion of the resulting global estimator is $5.89\%$. 
In contrast, spectral clustering with local likelihood refinement applied to RNA network alone and protein network alone gave estimators whose misclustering proportions are $8.68\%$ and $10.99\%$, respectively.
The lower misclustering proportion achieved by applying our algorithm to the two-layer network reveals that our algorithm successfully integrated the two sources of information.

To illustrate our algorithm with non-zero $\rho$, we randomly sampled an index set $\calS$ with $|\calS| = 1906$ among the $4769$ cells. 
We then randomly partitioned $\calS$ into two equally-sized disjoint subsets $\calS_1 \cup \calS_2$ (i.e., $|\calS_1| = |\calS_2| = 953$), and let $\tilde X^{(1)} = X^{(1)}_{\calS_1\cup \calS^c, \bigcdot}\in \bbR^{3816\times 20,729}$ and $\tilde X^{(2)} = X^{(2)}_{\calS_2\cup \calS^c, \bigcdot} \in \bbR^{3816 \times 224}$.
After the operation, about $20\%$ of the cells have different cluster labels between the two modalities. Such a mismatch could happen, for example, when two modalities are not row-wise aligned a priori and a matching algorithm has been applied to partially align the cells \citep{haghverdi2018batch,barkas2019joint,stuart2019comprehensive,zhu2021robust}. 
We then applied the same dichotomization procedure to $(\tilde X^{(1)}, \tilde X^{(2)})$ and obtain a two-layer network $(\tilde A^{(1)}, \tilde A^{(2)})$, to which we applied our algorithm with $K=4$ and $\rho = 0.2$. 
The misclustering proportions of the two individualized estimators (for RNA labels and protein labels) are $8.12\%$ and $5.71\%$, respectively. 
In contrast, the misclustering proportions of the ``spectral clustering + local refinement'' estimators obtained by separately working with RNA network and  protein network are $8.49\%$ and $11.45\%$, respectively. 
Such a noticeable difference suggests that our algorithm produces individualized estimators that not only borrow information across multiple layers, but also retain modality-specific characteristics.

\small{
\setlength{\bibsep}{0.2pt plus 0.3ex}
\bibliographystyle{abbrvnat}
\bibliography{references}}

\newpage
\appendixtitleon
\appendixtitletocon
\addtocontents{toc}{\protect\setcounter{tocdepth}{1}}
\begin{appendices}

\section{Additional Theoretical Results} \label{append:additional_theories}

\subsection{Non-Vanishing Label Sampling Probability} \label{subappend:const_rho}

In this section, we derive theoretical results when the label sampling probability $\rho$ does not vanish as $n$ tends to infinity.

\paragraph{Lower bounds.}
The following lower bound is the counterpart of Theorem \ref{thm:lower_bound_z_star}.

\begin{theorem}[Minimax lower bound for \globest~with non-vanishing $\rho$]
  \label{thm:lower_bound_z_star_const_rho}
  Assume there exist constants $C_1, C_2 > 1$ and $c_1, c_2, c_3 \in (0, 1)$ such that $1\lesssim \rho \leq 1-c_1$, $C_1 q_\ell \leq p_\ell \leq (C_2 q_\ell) \land (1-c_2), \forall \ell\in[L]$ and $\log L\ll  n^{c_3}$. If 
  $
    \min_{S\subseteq[L]} \{|S^c|(J_\rho+\log\sqrt{2}) + \psi_S^\star(0)\} \to \infty,
  $
  then there exists a sequence $\underline{\delta}_n = o(1)$ such that 
  \begin{align}
    \label{eq:lower_bound_z_star_vanishing_const_rho}
    & \inf_{\bhz^\star} \sup_{\bz^\star\in \calP_n} \bbE \calL(\bhz^\star, \bz^\star) \geq \exp\big\{-(1+\underline{\delta}_n) \min_{S\subseteq[L]} \{|S^c|(J_\rho+\log\sqrt{2}) + \psi_S^\star(0)\}\big\}.
  \end{align}
  On the other hand, if $\min_{S\subseteq[L]} \{|S^c|(J_\rho+\log \sqrt{2}) + \psi_S^\star(0)\}=\calO(1)$, then there exists some $c' >0$ such that
  \begin{equation}
    \label{eq:lower_bound_z_star_constant_const_rho}
    \inf_{\bhz^\star} \sup_{\bz^\star\in \calP_n} \bbE \calL(\bhz^\star, \bz^\star) \geq c'.
  \end{equation}
\end{theorem}
\begin{proof}
  See Section \ref{prf:thm:lower_bound_z_star_const_rho}.
\end{proof}

Compared to the lower bound in the case of $\rho = o(1)$, the exponent in the above lower bound does not differ according to the parity of $|S^c|$.

The following lower bound is the counterpart of Theorem \ref{thm:lower_bound_z_l}.

\begin{theorem}[Minimax lower bound for \indest~with non-vanishing $\rho$]
  \label{thm:lower_bound_z_l_const_rho}
  Under the same setup as Theorem \ref{thm:lower_bound_z_star_const_rho},
  if for a fixed $\ell\in[L]$, it holds that 
  $
    \psi_{\{\ell\}}^\star(0) \to \infty,
  $
  then there exists a sequence $\underline{\delta}_n = o(1)$, independent of $\ell$, such that 
  \begin{align}
    \label{eq:lower_bound_z_l_vanishing_const_rho}
    & \inf_{\hzzz{\ell}} \sup_{\bz^\star\in \calP_n} \bbE \calL(\hzzz{\ell}, \zzz{\ell}) \geq \exp\big\{-(1+\underline{\delta}_n) \psi_{\{\ell\}}^\star(0)\big\}.
  \end{align}
  On the other hand, if $\psi_{\{\ell\}}^\star(0)=\calO(1)$, then there exists $c'>0$ such that
  \begin{equation}
    \label{eq:lower_bound_z_l_constant_non_vanish_rho}
    \inf_{\hzzz{\ell}} \sup_{\bz^\star\in \calP_n} \bbE \calL(\hzzz{\ell}, \zzz{\ell})\geq c'.
  \end{equation}
\end{theorem}
\begin{proof}
  See Section \ref{prf:thm:lower_bound_z_l_const_rho}.
\end{proof}

Note that in the constant $\rho$ regime, the lower bound for individualized estimation agrees with the minimax rate for a single-layer SBM. Intuitively, this is because even when the global label $\bz^\star$ is exactly known, this knowledge does not help if the goal is to consistently estimate $\zzz{\ell}$ --- the knowledge of $\bz^\star$ can only identify a constant proportion (namely $1-\rho$) of $\zzz{\ell}$.

\paragraph{Upper bounds.}
We now derive theoretical guarantees for Algorithm \ref{alg: provable_refinement} under the constant $\rho$ regime.
Note that in this regime, the difference between the global parameter $\zzz{\star}$ and the individualized parameters $\zzz{\ell}$ is not negligible. As the result, in Stage \RN{1} of Algorithm \ref{alg: provable_refinement}, na\"{i}vely setting $\tzzz{\ell, -i}\leftarrow \tzzz{\star, -i}$ does not give rise to a consistent estimator of $\zzz{\ell}$. To establish meaningful performance guarantees for Algorithm \ref{alg: provable_refinement}, we need to modify its Stage \RN{1} so that the following assumption holds.

\begin{assump}[Consistent layer-wise initialization]
  \label{assump:consistent_layerwise_init}
   For each $i\in[n], \ell \in[L]$, let Stage \RN{1} of Algorithm \ref{alg: provable_refinement} produce $\tzzz{\ell, -i} \leftarrow \mathtt{Initialize}_\ell(\{\AAA{\ell'}_{-i, -i}, p_{\ell'}, q_{\ell'}\}_{\ell'=1}^L, \rho)$. Assume those initial estimators satisfy
    \begin{equation}
    \label{eq:consistent_layerwise_init} 
    \max_{i\in[n]}\bbP\big(\calL(\tzzz{\ell, -i}, \zzz{\ell}) \geq \delta_{\init,n}, \forall \ell\in[L] \big) \lesssim n^{-(1+\ep_\init)}
  \end{equation}
  for some $\delta_{\init,n}=o(1)$ and $\ep_\init > 0$.
\end{assump}

The above assumption can be satisfied if the information in each layer is strong enough for consistent estimation of $\zzz{\ell}$.

The following theorem is the counterpart of Theorem \ref{thm:refine_glob}.
\begin{theorem}[Performance of MAP-based refinement for \globest~with non-vanishing $\rho$]
  \label{thm:refine_glob_const_rho}
  Consider the variant of Algorithm \ref{alg: provable_refinement} whose Stage \RN{1} is modified to satisfy Assumption \ref{assump:consistent_layerwise_init}.
  Let the input to this algorithm be an instance generated by an $\textnormal{\modelname} \in \calP_n(\rho, \{p_\ell\}_1^L, \{q_\ell\}_1^L, \beta)$ satisfying $1\lesssim \rho < 1/2$, $q_\ell < p_\ell \leq (C q_\ell) \land (1-c), \forall \ell\in[L]$, $\beta = 1+o(1)$ and $\log L\ll n^{c'}$, where $C > 1$ and $c, c' \in (0, 1)$ are absolute constants. 
  Assume that for any $\delta_n = o(1)$, the following holds:
  \begin{align}
    \label{eq:diverging_glob_snr_sum_const_rho}
    \lim_{n\to\infty } \sum_{S\subseteq[L]} e^{-(1-\delta_n) [|S^c|J_\rho + \psi_{S}^\star(0)]} =0.
  \end{align}
  Then, there exist two sequences $\overline{\delta}_n, \overline{\delta}_n' = o(1)$ such that
  \begin{equation}
    \label{eq:refine_glob_const_rho}
    \lim_{n\to \infty} \inf_{\bz^\star \in \calP_n} \bbP\bigg[ \calL(\bhz^\star , \bz^\star) \leq \bigg(\sum_{S\subseteq[L]} e^{-(1-\overline{\delta}_n)[|S^c|J_\rho + \psi_{S}^\star(0)]}\bigg)^{1-\overline{\delta}_n'} \bigg]  = 1.
  \end{equation}
\end{theorem}
\begin{proof}
  See Section \ref{prf:thm:refine_glob_const_rho}.
\end{proof}

Similar to the derivation of Theorem \ref{thm:opt_glob_est}, one can give sufficient conditions under which the above upper bound can be strengthened to
\begin{equation}
\label{eq:refine_glob_const_rho_strong}
  \exp\bigg\{-(1-o(1))\min_{S\subseteq[L]} \bigg(|S^c|J_\rho + \psi_S^\star(0)\bigg) \bigg\}.
\end{equation}
We sketch a calculation below.
The error of global estimation can be written as
\begin{align*}
  &\bigg[\prod_{\ell\in[L]} \bigg(e^{-(1 - \overline{\delta}_n)J_\rho} + e^{-(1-\overline{\delta}_n)n\III{\ell}_{1/2}/2}\bigg)\bigg]^{1-\overline{\delta}_n'}.
\end{align*}
Note that
\begin{align*}
  & \prod_{\ell\in[L]} \bigg(e^{-(1 - \overline{\delta}_n)J_\rho} + e^{-(1-\overline{\delta}_n)n\III{\ell}_{1/2}/2}\bigg)\\
  & = \exp\bigg\{-(1-\overline{\delta}_n)\min_{S\subseteq[L]}\bigg(|S^c|J_\rho + \psi_{S}^\star(0)\bigg)\bigg\} 
  \cdot \prod_{\ell\in[L]} \bigg[1 + \exp\bigg\{-(1-\overline{\delta}_n)\bigg( J_\rho \lor \psi^\star_{\{\ell\}}(0) - J_\rho \land \psi^\star_{\{\ell\}}(0) \bigg)\bigg\}\bigg].
\end{align*}
A naive upper bound is
$$
    2^L \cdot \exp\bigg\{-(1-\overline{\delta}_n) \min_{S\subseteq[L]}\bigg(|S^c|J_\rho + \psi_{S}^\star(0)\bigg)\bigg\}.
$$
This naive upper bound is tight in the special case of $n\III{\ell}_{1/2}/2 = J_\rho$. Recall the lower bound is
$$
\exp\bigg\{-(1-\underline{\delta}_n)\min_{S\subseteq[L]}\bigg(|S^c|(J_\rho+\log\sqrt{2}) + \psi_{S}^\star(0)\bigg)\bigg\} .
$$
To match the lower bound as much as possible, we need $J_\rho$ and $\psi^\star_{\{\ell\}}(0) = n\III{\ell}_{1/2}/2$ to be as ``separate'' as possible. 

For example, let us consider the case where individualized estimation is possible for all the layers. According to the lower bound, this translates to $n\III{\ell}_{1/2}/2 \to \infty$ for any $\ell\in[L]$. Then we have
\begin{align*}
  & \prod_{\ell\in[L]} \bigg[1 + \exp\bigg\{-(1-\overline{\delta}_n)\bigg( J_\rho \lor \psi^\star_{\{\ell\}}(0) - J_\rho \land \psi^\star_{\{\ell\}}(0) \bigg)\bigg\}\bigg]\\
  & =\prod_{\ell\in[L]} \bigg[1 + \exp\bigg\{-(1-\overline{\delta}_n)\bigg( \psi^\star_{\{\ell\}}(0) - J_\rho \bigg)\bigg\}\bigg]  \\
  & =\prod_{\ell\in[L]} \bigg[1 + \exp\bigg\{-(1-\overline{\delta}_n')\psi^\star_{\{\ell\}}(0)\bigg\}\bigg]  \\
  & \leq \exp\bigg\{\sum_{\ell\in[L]}e^{-(1-\overline{\delta}_n')\psi^\star_{\{\ell\}}(0)}\bigg\}\\
  & \leq \exp\bigg\{L e^{-(1-\overline{\delta}_n')\min_{\ell\in[L]} \psi^\star_{\{\ell\}}(0)}\bigg\}
\end{align*}
Meanwhile,
$$
  \min_{S\subseteq[L]} \bigg(|S^c|J_\rho + \psi_S^\star(0)\bigg) = LJ_\rho.
$$
Thus, we get an upper bound of the form
$$
  \exp\bigg\{-(1-o(1))\min_{S\subseteq[L]} \bigg(|S^c|J_\rho + \psi_S^\star(0)\bigg) \bigg\}
$$
as long as
$$
  e^{-(1-\overline{\delta}_n')\min_{\ell\in[L]} \psi^\star_{\{\ell\}}(0)} \ll J_\rho,
$$
which holds by design. In general, we would need assumptions like
$$
  \sum_{\ell\in[L]} \exp\bigg\{-(1-\overline{\delta}_n)\bigg( J_\rho \lor \psi^\star_{\{\ell\}}(0) - J_\rho \land \psi^\star_{\{\ell\}}(0) \bigg)\bigg\} \ll \min_{S\subseteq[L]} \bigg(|S^c|J_\rho + \psi_S^\star(0)\bigg).
$$
Note that this such assumptions cannot not hold in the worst case, e.g., $J_\rho = \psi^\star_{\{\ell\}}(0)$.

{Depending on the choice of optimal $S$ as well as the relationship between $|S^c|(J_\rho + \log \sqrt{2})$ and $\psi_S^\star(0)$, the upper bound \eqref{eq:refine_glob_const_rho_strong} may or may not match the lower bound in Theorem \ref{thm:lower_bound_z_star_const_rho}.}

The following theorem is the counterpart of Theorem \ref{thm:refine_ind}.

\begin{theorem}[Performance of MAP-based refinement for \indest~with non-vanishing $\rho$]
  \label{thm:refine_ind_const_rho}
  Consider the variant of Algorithm \ref{alg: provable_refinement} whose Stage \RN{1} is modified to satisfy Assumption \ref{assump:consistent_layerwise_init}.
  Let the input to this algorithm be an instance generated by an $\textnormal{IMLSBM} \in \calP_n(\rho, \{p_\ell\}_1^L, \{q_\ell\}_1^L, \beta)$ satisfying $1\lesssim \rho < 1/2$, $q_\ell < p_\ell \leq (C q_\ell) \land (1-c), \forall \ell\in[L]$, $\beta = 1+o(1)$ and $\log L\ll n^{c'}$, where $C > 1$ and $c, c' \in (0, 1)$ are absolute constants. 
  Assume that for a fixed $\ell\in[L]$ and for any $\delta_n = o(1)$, the following holds:
  \begin{align}
    \label{eq:diverging_ind_snr_sum_const_rho}
    \lim_{n\to\infty } \sum_{S\subseteq[L]\setminus\{\ell\}} e^{-(1-\delta_n)\psi_{S\cup\{\ell\}}^\star(0)}\cdot  \Big( e^{-(1-\delta_n)|(S\cup \{\ell\})^c|J_\rho} + e^{-(1-\delta_n) |S\cup\{\ell\}|J_\rho} \Big) =0.
  \end{align}
  Then, there exist two sequences $\overline{\delta}_n, \overline{\delta}_n' = o(1)$, independent of $\ell$, such that
  \begin{align}
    & \lim_{n\to \infty}\inf_{\bz^\star\in\calP_n}\bbP\bigg[ \calL(\hzzz{\ell} , \zzz{\ell}) \leq\nonumber \\
    \label{eq:refine_ind_const_rho}
    & \bigg(\sum_{S\subseteq[L]\setminus\{\ell\}} e^{-(1-\overline{\delta}_n)\psi^\star_{S\cup\{\ell\}}(0)} \Big[ e^{-(1-\overline{\delta}_n)|(S\cup\{\ell\})^c|J_\rho} + e^{-(1-\overline{\delta}_n) |S\cup\{\ell\}|J_\rho} \Big]
    \bigg)^{1-\overline{\delta}_n'}
     \bigg]  = 1.
  \end{align}
\end{theorem}
\begin{proof}
  See Appendix \ref{prf:thm:refine_ind_const_rho}.
\end{proof}

Similar to the derivation of Theorem \ref{thm:opt_ind_est}, one can give sufficient conditions under which the above upper bound can be written in the form of 
$$
  \exp\{-(1+o(1))\psi^\star_{\{\ell\}}(0)\},
$$
which matches the lower bound in Theorem \ref{thm:lower_bound_z_l_const_rho}. We give a sketch of calculations below.
The error of individualized estimation can be written as the superposition of two terms. The first term is given by
\begin{align*}
  & \sum_{S\subseteq[L]\setminus\{\ell\}} \exp\bigg\{-(1-\overline{\delta}_n) \bigg(|(S\cup\{\ell\})^c| J_\rho + n\sum_{\ell'\in S\cup\{\ell\}}\III{\ell'}_{1/2}/2\bigg)\bigg\}\\
  & = \exp\bigg\{-(1-\overline{\delta}_n)\psi^\star_{\{\ell\}}(0)\bigg\} 
  \prod_{\ell' \in [L]\setminus\{\ell\}}\bigg(e^{-(1-\overline{\delta}_n)J_\rho} + e^{-(1-\overline{\delta}_n)\psi^\star_{\{\ell'\}}(0)}\bigg).
\end{align*}  
To match the lower bound, it suffices to require
$$
  \sum_{\ell'\in[L]\setminus\{\ell\}}\log \bigg(e^{-(1-\overline{\delta}_n)J_\rho} + e^{-(1-\overline{\delta}_n)\psi^\star_{\{\ell'\}}(0)}\bigg) \ll \psi^\star_{\{\ell\}}(0).
$$
The second term is given by
\begin{align*}
  & \sum_{S\subseteq[L]\setminus\{\ell\}} \exp\bigg\{-(1-\overline{\delta}_n)\bigg(|S\cup\{\ell\}|J_\rho + n \sum_{\ell'\in S\cup\{\ell\}}\III{\ell}_{1/2}/2\bigg)\bigg\}\\
  & = \exp\bigg\{-(1-\overline{\delta}_n)(J_\rho + \psi^\star_{\{\ell\}}(0))\bigg\}
  \prod_{\ell'\in[L]\setminus\{\ell\}} \bigg(1 + e^{-(1-\overline{\delta}_n)[J_\rho + \psi^\star_{\{\ell'\}}(0)]}\bigg).
\end{align*}
To match the lower bound, it suffices to require
$$
  \sum_{\ell'\in[L]\setminus\{\ell\}} \log\bigg(1 + e^{-(1-\overline{\delta}_n)[J_\rho + \psi^\star_{\{\ell'\}}(0)]}\bigg) \ll \psi^\star_{\{\ell\}}(0).
$$
So as long as the SNR of the other layers are strong enough, our algorithm is minimax optimal for individualized estimation, and the rate is the same as if we only observe a single layer SBM (we get a multiplier of $e^{-J_\rho} = \calO(1)$).

\subsection{Asymmetric Partition} \label{subappend:asym_partition}
In this section, we consider the performance of Algorithm \ref{alg: provable_refinement} when the two clusters are asymmetric, i.e., $\beta \neq 1 + o(1)$.
To quantify the effect of asymmetry, we introduce
$$
  \tilde{I}^{(\ell)}_t := \III{\ell}_t - (\beta - \beta^{-1})\cdot\bigg|\log \frac{p_\ell^t q_\ell^{1-t} + (1-p_\ell)^t (1-q_\ell)^{1-t}}{p_\ell^{1-t}q_\ell^t + (1-p_\ell)^{1-t} (1-q_\ell)^t}\bigg|.
$$
Then, for any $S\subseteq[L]$, we let
$$
  \tilde \psi_S(t) := -\frac{n}{2}\sum_{\ell \in S} \tilde{I}^{(\ell)}_t, \qquad    \tilde\psi_S^\star(a) := \sup_{0\leq t \leq 1} at-\tilde\psi_S(t).
$$
We finally define
\begin{align*}
  \tilde\globinfo_S & := 
  \begin{cases}
    |S^c| J_\rho + \tilde \psi_S^\star(0) & \textnormal{ if } |S^c| \textnormal{ is even},\\
    (|S^c| + 1) J_\rho + \tilde \psi^\star_S(-2 J_\rho) & \textnormal{ if } |S^c| \textnormal{ is odd},
  \end{cases}\\
    \tilde\indinfo_S & := 
  \begin{cases}
    |S| J_\rho + \tilde \psi^\star_S(0) & \textnormal{ if } |S| \textnormal{ is even}, \\
    (|S| + 1) J_\rho + \tilde \psi^\star_S(-2J_\rho) & \textnormal{ if } |S| \textnormal{ is odd}. 
  \end{cases}
\end{align*}  

The following theorem summarizes the performance of Algorithm \ref{alg: provable_refinement} when $\beta \neq 1 + o(1)$.
\begin{theorem}[Performance of MAP-based refinement under asymmetry]
\label{thm:refine_glob_and_ind_cluster_imbalance}
Theorems \ref{thm:refine_glob} and \ref{thm:refine_ind} still holds when $\beta$ is not $1+o(1)$ (but is of constant order), after replacing any occurrence of $\globinfo_S$ and $\indinfo_S$ with $\tilde\globinfo_S$ and $\tilde \indinfo_S$.
\end{theorem}
\begin{proof}
  See Section \ref{prf:thm:refine_glob_and_ind_cluster_imbalance}.
\end{proof}

Compared to the upper bounds in the symmetric case, the above theorem is weaker in the sense that $\tilde{I}^{\ell}_t$ is smaller than $\III{\ell}_t$ by an additive factor of $(\beta - \beta^{-1})\cdot\bigg|\log \frac{p_\ell^t q_\ell^{1-t} + (1-p_\ell)^t (1-q_\ell)^{1-t}}{p_\ell^{1-t}q_\ell^t + (1-p_\ell)^{1-t} (1-q_\ell)^t}\bigg|$. Note that when $t=1/2$ (which corresponds to the $|S^c|$ even case in the definition of $\calI_S$ and $|S|$ even case in the definition of $\calJ_S$), this additive factor is zero. Thus, in certain cases, the upper bounds is the same as the upper bounds in the symmetric case.

\subsection{MAP-Based Refinement with Estimated Parameters} \label{subappend:est_params}
In this section, we derive performance guarantees for Algorithm \ref{alg: provable_refinement} when the nuisance parameters are unknown and need to be estimated from the data. 

To start with, let us note that it is possible to complete the leave-one-out initialization stage in Algorithm \ref{alg: provable_refinement} without prior knowledge on $(p_\ell, q_\ell)_{\ell=1}^L$ and $\rho$. For example, one can take the average of the adjacency matrices and apply the un-regularized version of spectral clustering (i.e., Lines 3 and 4 in Algorithm \ref{alg: specc}) on it. One can derive similar (but slightly worse) consistency results to those in Theorem \ref{thm:specc} by the concentration of the un-regularized average adjacency matrix (see Corollary \ref{cor:concentration_w/o_reg}).

Once Stage \RN{1} of Algorithm \ref{alg: provable_refinement} is done, we can use the initial estimators to estimate the connection probabilities $(p_\ell, q_\ell)_{\ell=1}^L$ as follows.
For each $i\in[n]$, let
$$
  \hat\calC^{(-i)}_{\pm} = \{j \neq i: \tzzz{\star, -i}_j = \pm 1\}
$$
be the two clusters according to the leave-$i$-out initial estimator $\tzzz{\star, -i}$.
Let $\hat \calE^{(\ell, -i)}_{\textsf{l}, \textsf{l'}}$ be the edges between $\hat\calC^{(-i)}_{\textsf{l}}$ and $\hat\calC^{(-i)}_{\textsf{l'}}$ for $\textsf{l, l'} \in \{+, -\}$ in the $\ell$-th layer. The leave-$i$-out estimators of $p_\ell$ and $q_\ell$ are given by
\begin{equation}
  \label{eq:estimated_connect_prob}
  \hat p^{(-i)}_\ell = \frac{2 |\hat\calE^{(-\ell, i)}_{++}|}{|\hat \calC^{(-i)}_{+}|(|\hat \calC^{(-i)}_{+}| - 1)} \land \frac{2 |\hat\calE^{(-\ell, i)}_{--}|}{|\hat \calC^{(-i)}_{-}|(|\hat \calC^{(-i)}_{-}| - 1)},
  \qquad 
  \hat q^{(-i)}_\ell = \frac{\hat \calE^{(\ell, -i)}_{+-}}{|\hat \calC^{(-i)}_{+}||\hat \calC^{(-i)}_{-}|},
\end{equation}  
respectively. 

In contrast to the estimation of the connection probabilities, the estimation of the label sampling probability $\rho$ is a considerably harder task. Here, we consider a misspecified (but fixed) $\rho^\dagger$ which potentially differs from the ground truth $\rho$. To quantify the cost of misspecification of $\rho$, we introduce
$$
  J_\rho^\dagger = J_\rho - \frac{1}{2}\bigg|\log {\frac{(1-\rho)/(1-\rho^\dagger)}{\rho / \rho^\dagger}}\bigg|.
$$

The following theorem gives the performance guarantee for a variant of Algorithm \ref{alg: provable_refinement} where $(p_\ell, q_\ell)_{\ell=1}^L$ are estimated and $\rho$ is misspecified.

\begin{theorem}[MAP refinement for global estimation with estimated parameters]
  \label{thm:refine_glob_with_est_param}
  Consider the variant of Algorithm \ref{alg: provable_refinement} where in Stage \RN{2}, we change the definition of $f^{(\ell)}_i$ from \eqref{eq:local_map_objective} to
\begin{align*}
  & f^{(\ell)}_i(s_\star, s_\ell, \btz^{\star})  = \log\bigg(\frac{1-\rho^\dagger}{\rho^\dagger}\bigg) \cdot \indicator{s_\ell = s_\star} \\
  \label{eq:local_map_objective_est_param}
  & ~~ + \sum_{j\neq i: \btz^{\star}_j = s_\ell} \bigg[\log\bigg(\frac{\hat p_\ell^{(-i)}(1-\hat q_\ell^{(-i)})}{\hat q_\ell^{(-i)}(1-\hat p_\ell^{(-i)})}\bigg) \AAA{\ell}_{ij} 
  + \log \bigg(\frac{ \hat p_\ell^{(-i)} \Big(\frac{\hat q_\ell^{(-i)} (1-\hat p_\ell^{(-i)})}{\hat p_\ell^{(-i)} (1- \hat q_\ell^{(-i)})}\Big)^{1/2} +  1-\hat p_\ell^{(-i) }}{\hat q_\ell^{(-i)} \Big(\frac{\hat p_\ell^{(-i)} (1 - \hat q_\ell^{(-i)})}{\hat q_\ell^{(-i)}(1-\hat p_\ell^{(-i)})}\Big)^{1/2} + 1-\hat q_\ell^{(-i)}}\bigg)\bigg].\numberthis
  \end{align*}
  Let the input to this algorithm be an instance generated by an $\textnormal{\modelname} \in \calP_n(\rho, \{p_\ell\}_1^L, \{q_\ell\}_1^L, \beta)$ satisfying $\rho = o(1), q_\ell < p_\ell \leq (C q_\ell) \land (1-c), \forall \ell\in[L]$, $\beta = 1+o(1)$ and $ L \lesssim n^{c'}$, where $C > 1, c \in (0, 1), c' \geq 0$ are absolute constants. 
  Let Assumption \ref{assump:consistent_init} hold and assume that for any $\delta_n = o(1)$, the following holds:
    \begin{align}
    \label{eq:diverging_glob_snr_sum_est_param}
    \lim_{n\to\infty } \sum_{S\subseteq[L]} e^{-(1-\delta_n) [|S^c|J_\rho^\dagger + \psi_{S}^\star(0)]} =0.
  \end{align}
  Then, there exist two sequences $\overline{\delta}_n, \overline{\delta}_n' = o(1)$ such that
  \begin{equation}
    \label{eq:refine_glob_est_param}
    \lim_{n\to \infty} \inf_{\bz^\star \in \calP_n} \bbP\bigg[ \calL(\bhz^\star , \bz^\star) \leq \bigg(\sum_{S\subseteq[L]} e^{-(1-\overline{\delta}_n)[|S^c|J_\rho^\dagger + \psi_{S}^\star(0)]}\bigg)^{1-\overline{\delta}_n'} \bigg]  = 1.
  \end{equation}
\end{theorem}
\begin{proof}
  See Section \ref{prf:thm:refine_glob_with_est_param}.
\end{proof}

Note that in the above theorem, we need $L \lesssim n^{c'}$, a stronger condition than what was assumed in earlier results (i.e., $\log L \ll n^{c'}$ for some $c'\in(0, 1)$). This is because in the proof of the above theorem, we show that for each $\ell$, we can produce consistent estimators of $p_\ell, q_\ell$ with probability at least $1 - n^{-r}$ for some $r$. We then apply a union bound over $[L]$ to show that consistent estimation of all $(p_\ell, q_\ell)_{\ell=1}^L$ is possible with probability at least $1- L n^{-r}\geq 1 - n^{-(r-c)}$.

The corresponding result for individualized estimation is given below.
\begin{theorem}[MAP refinement for individualized estimation with estimated parameters]
\label{thm:refine_ind_with_est_param}
Consider the variant of Algorithm \ref{alg: provable_refinement} discussed in Theorem \ref{thm:refine_glob_with_est_param}. 
Let the input to this algorithm be an instance generated by an $\textnormal{\modelname} \in \calP_n(\rho, \{p_\ell\}_1^L, \{q_\ell\}_1^L, \beta)$ satisfying $\rho = o(1), q_\ell < p_\ell \leq (C q_\ell) \land (1-c), \forall \ell\in[L]$, $\beta = 1+o(1)$ and $ L \lesssim n^{c'}$, where $C > 1, c \in (0, 1), c' \geq 0$ are absolute constants. 
Let Assumption \ref{assump:consistent_init} hold and assume that for a fixed $\ell\in[L]$ and for any $\delta_n = o(1)$, the following holds:
  \begin{align}
    \label{eq:diverging_ind_snr_sum_with_est_param}
    \lim_{n\to\infty } \sum_{S\subseteq[L]\setminus\{\ell\}} e^{-(1-\delta_n)\psi_{S\cup\{\ell\}}^\star(0)}\cdot  \Big( e^{-(1-\delta_n)|(S\cup \{\ell\})^c|J_\rho^\dagger} + e^{-(1-\delta_n) |S\cup\{\ell\}|J_\rho^\dagger} \Big) =0.
  \end{align}
  Then, there exist two sequences $\overline{\delta}_n, \overline{\delta}_n' = o(1)$, independent of $\ell$, such that
  \begin{align}
    & \lim_{n\to \infty}\inf_{\bz^\star\in\calP_n}\bbP\bigg[ \calL(\hzzz{\ell} , \zzz{\ell}) \leq\nonumber \\
    \label{eq:refine_ind_with_est_param}
    & \bigg(\sum_{S\subseteq[L]\setminus\{\ell\}} e^{-(1-\overline{\delta}_n)\psi^\star_{S\cup\{\ell\}}(0)} \Big[ e^{-(1-\overline{\delta}_n)|(S\cup\{\ell\})^c|J_\rho^\dagger} + e^{-(1-\overline{\delta}_n) |S\cup\{\ell\}|J_\rho^\dagger} \Big]
    \bigg)^{1-\overline{\delta}_n'}
     \bigg]  = 1.
  \end{align}
\end{theorem}
\begin{proof}
  See Section \ref{prf:thm:refine_ind_with_est_param}.
\end{proof}

The above two theorems are weaker than their counterparts when $\rho, p_\ell, q_\ell$ are known in the sense that (1) $J_\rho^\dagger \leq J_\rho$ and (2) the exponents do not differ according to the parity of the most informative layers. 

\begin{figure}[t]
    \centering
    \includegraphics[width=0.6\textwidth]{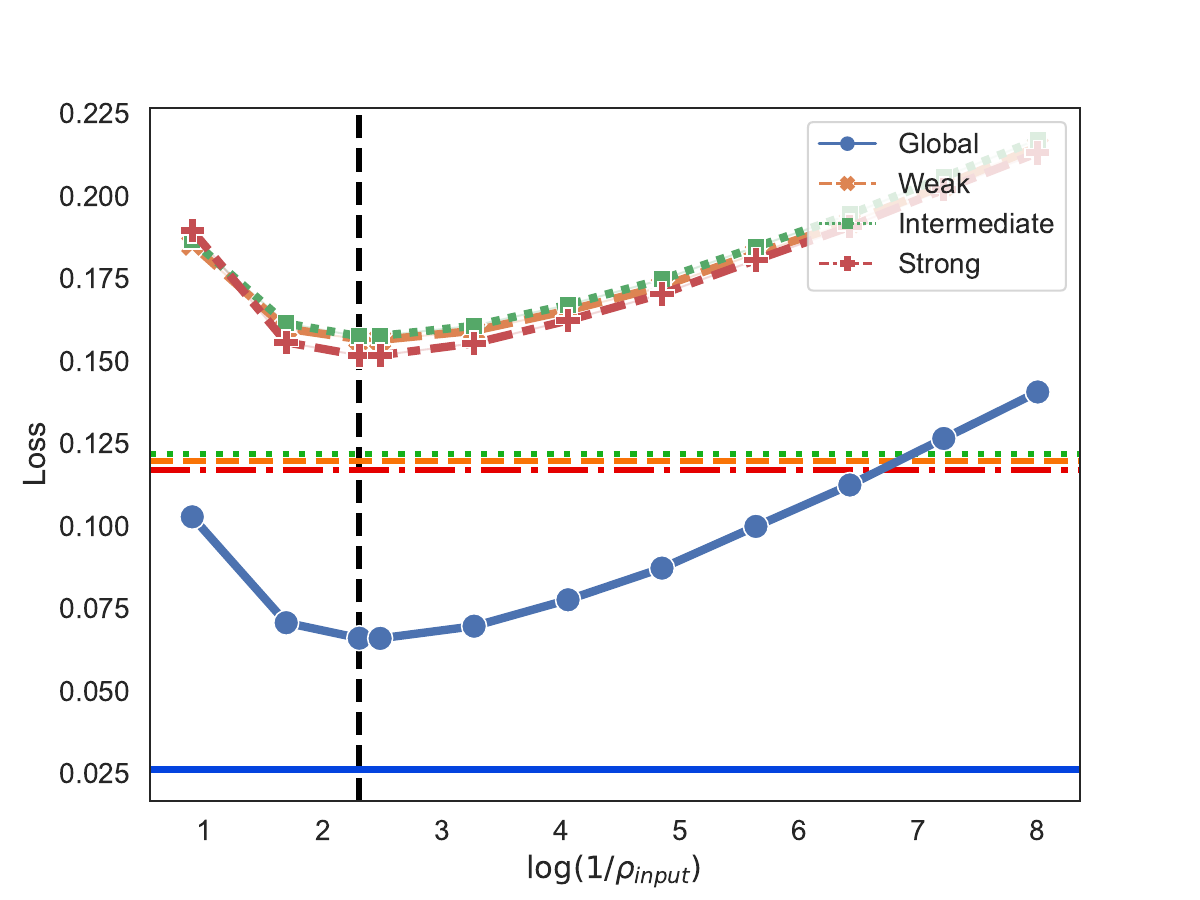}
        \caption{\small Average misclustering proportions against $\log(1/\rho_{\textnormal{input}})$ for the variant of Algorithm \ref{alg: generic_refinement} considered in Theorem \ref{thm:refine_glob_with_est_param}. 
        The black dashed vertical line represents the location of the true $\rho$ that generates the data, and the horizontal lines represent the errors made by Algorithm \ref{alg: generic_refinement} when $\{p_\ell\}, \{q_\ell\}, \rho$ are all known. 
        }
    \label{fig:sens2params_provable}
\end{figure}

To test the performance of the variant of Algorithm \ref{alg: provable_refinement}, we do a simulation similar to that in Figure \ref{fig:sens2params}, except that we now take $n=200, L = 100$ and we set $c=3$. Figure \ref{fig:sens2params_provable} shows the average misclustering proportions of the algorithm described in Theorem \ref{thm:refine_glob_with_est_param} against $\log(1/\rho_{\textnormal{input}})$. One can see that there is a relatively small, yet non-trivial performance gap between the adaptive algorithm with the oracle algorithm (i.e., Algorithm \ref{alg: generic_refinement}).

\subsection{Fixed \texorpdfstring{$n$}{n} Large \texorpdfstring{$L$}{L} Asymptotics} \label{subappend:large_L}
In this section, we consider a different asymptotic regime, where we assume $n = \calO(1)$ and $L\to \infty$.

By construction, if the mis-clustering error drops below $1/n$, then one gets perfect recovery. Under the current asymptotic regime, we have $1/n \gtrsim 1$, which means that as long as the mis-clustering error vanishes as $L\to\infty$, it becomes exactly zero. This further reveals that once we get consistent initial estimators (say, via spectral clustering), the MAP refinement step (i.e., Stage \RN{2} of Algorithm \ref{alg: provable_refinement}) cannot give further improvement, as the mis-clustering error is already zero for large $L$. Hence, in the following, we only give performance guarantees for spectral clustering.

\begin{theorem}[Performance of spectral clustering]
\label{thm:specc_large_L}
Consider the variant of Algorithm \ref{alg: specc}, where we omit the trimming step (i.e., Lines 1 and 2) so that $\tau(\bar A) = \bar A$.
Let the input to this algorithm be an instance generated by an $\textnormal{\modelname}\in \calP_n(\rho, \{p_\ell\}_1^L, \{q_\ell\}_1^L, \beta)$ with $n = \calO(1), L\to\infty$ and assume \eqref{eq:spectral_gap} holds.
If 
\begin{align}
  \label{eq:assump_specc_large_L}
  \lim_{L\to \infty} \frac{\max_\ell \omega_\ell }{\sqrt{ \sum_\ell \omega_\ell^2 q_\ell }} 
  \lor \frac{\max_\ell \omega_\ell(p_\ell - q_\ell) }{\sqrt{\rho\sum_{\ell} \omega_\ell^2 (p_\ell-q_\ell)^2}} 
  \lor \frac{\sum_\ell \omega_\ell^2 p_\ell^2}{(1-2\rho)^4 (\bar p - \bar q)^2} 
  \lor \frac{\rho\sum_\ell \omega_\ell^2 (p_\ell - q_\ell)^2}{(1-2\rho)^4 (\bar p - \bar q)^2} & = 0,
\end{align}
then there exists a sequence $\ep_L \to 0$ as $L\to\infty$, such that with probability $1 - \ep_L$, we have $\calL(\tilde\bz^\star, \bz^\star) = 0$ for large $L$.
\end{theorem}
\begin{proof}
  See Appendix \ref{prf:thm:specc_large_L}.
\end{proof}

Consider the simplified setup where $p_\ell = q, q_\ell = q$ for any $\ell \in [L]$ and $\weight_\ell = 1/L$. Then, the condition in \eqref{eq:assump_specc_large_L} becomes
$$
  \frac{1}{\sqrt{qL}} \lor \frac{1}{\sqrt{\rho L}} \lor \frac{p^2}{(1-2\rho)^2L(p-q)^2} \lor \frac{\rho}{(1-2\rho)^4 L} \to 0.
$$
When $\rho < 1/2 - c$ for some absolute constant $c\in(0, 1/2)$, the above condition is further simplified to 
$$
  qL \land \rho L \land \frac{(p-q)^2}{p^2} L \to \infty.
$$

\section{Proofs of Lower Bounds}\label{append:prf_lower_bound}
\subsection{Proof of Lemma \ref{lemma:testing_problem_z_star}}\label{subappend:prf_testing_problem_z_star}
Fix any $\btz^\star \in \calP_n$ with $n_+^\star(\btz^\star) = \lfloor n/2 \rfloor$ and $n_-^\star(\btz^\star) = n - \lfloor n/2\rfloor$. Let
$$
  C_+(\btz^\star):= \{i\in[n]: \btz^\star_i = +1\} ,\qquad C_-(\btz^\star):= \{i\in[n]: \btz^\star_i = -1\}.
$$
Now fix any $\delta_n = o(1)$ and choose any $\tilde C_+ \subset C_+(\btz^\star), \tilde C_- \subset C_-(\btz^\star)$ such that 
$$
  |\tilde C_+| = |\tilde C_-| = \lfloor n/2\rfloor - \lceil \delta_n n/2 \rceil.
$$
Denoting $T = \tilde C_+ \cup \tilde C_-$, we define
\begin{equation}
  \label{eq:clusters_in_sub_param_space_z_star}
  \calZ_T := \{\bz\in \{\pm 1\}^n: \bz_i = \btz^\star_i , \forall i \in T\}. 
\end{equation}  
Further, we define
\begin{equation}
  \label{eq:sub_param_space_z_star}
  \calP_n^0 := \bigg\{ \textnormal{\sname}(\bz^\star, \rho, \{p_\ell\}_1^L, \{q_\ell\}_1^L): \bz^\star \in \calZ_T, p_\ell > q_\ell ~ \forall \ell\in[L]\bigg\} \subseteq \calP_n.
\end{equation}  
Note that $\inf_{\bhz^*} \sup_{\bz^* \in \calP_n} \bbE \calL(\bhz^*, \bz^*) \geq \inf_{\bhz^*} \sup_{\bz^* \in \calP_n^0} \bbE \calL(\bhz^*, \bz^*)$. For any two $\bz, \bz' \in \calP_n^0$, we have
\begin{align*}
  \calL(\bz, \bz') = \frac{1}{n} \sum_{i\in[n]} \indc{\bz_i \neq \bz'_i},
\end{align*}
because 
$$
  \frac{1}{n} \sum_{i\in[n]} \indc{\bz_i \neq \bz'_i} \leq \frac{n - |T|}{n} = \frac{n - 2 \lfloor n/2 \rfloor + 2 \lceil \delta_n n/2 \rceil}{n} \leq \frac{4}{n} + \delta_n  < 1/2
$$
for large $n$. Hence, we can proceed by
\begin{align*}
\inf_{\bhz^*} \sup_{\bz^* \in \calP_n} \bbE \calL(\bhz^*, \bz^*) &  \geq \inf_{\bhz^*} \sup_{\bz^* \in \calZ_T} \frac{1}{n} \sum_{i\in[n]} \bbP(\bhz^*_i \neq \bz^*_i)\\
& = \frac{|T^c|}{n} \inf_{\bhz^*} \sup_{\bz^* \in \calZ_T} \frac{1}{|T^c|} \sum_{i\in T^c} \bbP(\bhz^*_i \neq \bz^*_i) \\
& \geq \frac{1}{n} \sum_{i\in T^c}  \inf_{\bhz^*_i} \underset{\bz^*\in \calZ_T}{\textnormal{ave}} \ \bbP(\bhz^*_i \neq \bz^*_i),
\end{align*}
where 
``$\textnormal{ave}$'' denotes the expectation if we assume $\bz^\star$ has a uniform distribution over $\calZ_T$. We are to show that all $\inf_{\bhz^*_i} \underset{\bz^*\in \calZ_T}{\textnormal{ave}} \ \bbP(\bhz^*_i \neq \bz^*_i)$'s are lower bounded by the same quantity, which is the type-I plus type-II error of the testing problem \eqref{eq:fund_testing_problem_z_star}, so that for any $i\in T^c$ we would have
\begin{equation}
  \label{eq:same_error_across_i_lower_bound_z_star}
  \inf_{\bhz^*} \sup_{\bz^* \in \calP_n} \bbE \calL(\bhz^*, \bz^*) \geq \frac{|T^c|}{n}\cdot \inf_{\bhz^*_i} \underset{\bz^*\in \calZ_T}{\textnormal{ave}} \ \bbP(\bhz^*_i \neq \bz^*_i) \gtrsim \delta_n \cdot \inf_{\bhz^*_i} \underset{\bz^*\in \calZ_T}{\textnormal{ave}} \ \bbP(\bhz^*_i \neq \bz^*_i),
\end{equation}  
which is the desired result. 

Now let us fix any $i\in T^c$ and $S\subseteq[L]$. We then have
\begin{align*}
  & \inf_{\bhz^\star_i} \underset{\bz^\star\in \calZ_T}{\textnormal{ave}} \ \bbP(\bhz^\star_i \neq \bz^\star_i) \\
  & = \inf_{\bhz^\star_i} \underset{\bz^\star_{-i}}{\textnormal{ave}} \ \underset{\bz^\star_{i}}{\textnormal{ave}}\sum_{\xi\in \{\pm 1\}^S} \bbP( \bhz^\star_i \neq \bz^\star \ | \ \zzz{\ell}_i = \xi_\ell \bz_i^\star \ \forall \ell\in S) \cdot \bbP(\zzz{\ell}_i = \xi_\ell \bz_i^\star \ \forall \ell\in S) \\
  & \geq \underset{\bz^\star_{-i}}{\textnormal{ave}}\inf_{\bhz^\star_i} \sum_{\xi\in \{\pm 1\}^S} \underset{\bz^\star_{i}}{\textnormal{ave}} \ \bbP( \bhz^\star_i \neq \bz^\star \ | \ \zzz{\ell}_i = \xi_\ell \bz_i^\star \ \forall \ell\in S) \cdot \bbP(\zzz{\ell}_i = \xi_\ell \bz_i^\star \ \forall \ell\in S) \\
  & = \frac{1}{2} \cdot  \underset{\bz^\star_{-i}}{\textnormal{ave}}\inf_{\bhz^\star_i} \sum_{\xi\in \{\pm 1\}^S} \bigg(\bbP( \bhz^\star_i =-1 \ | \ \bz_i^\star = 1,  \zzz{\ell}_i = \xi_\ell \bz_i^\star \ \forall \ell\in S) \cdot \bbP(\zzz{\ell}_i = \xi_\ell \bz_i^\star \ \forall \ell\in S \ | \ \bz_i^\star = 1) \\
  & \qquad  \qquad  \qquad  + \bbP( \bhz^\star_i =1 \ | \ \bz_i^\star = -1,  \zzz{\ell}_i = \xi_\ell \bz_i^\star \ \forall \ell\in S) \cdot \bbP(\zzz{\ell}_i = \xi_\ell \bz_i^\star \ \forall \ell\in S \ | \ \bz_i^\star = -1)  \bigg).
\end{align*}
Because
\begin{align*}
  \bbP(\zzz{\ell}_i = \xi_\ell \bz_i^\star \ \forall \ell\in S \ | \ \bz_i^\star = 1)  & = \bbP(\zzz{\ell}_i = \xi_\ell \bz_i^\star \ \forall \ell\in S \ | \ \bz_i^\star = 1) = \bbP(\zzz{\ell}_i = \xi_\ell \bz_i^\star \ \forall \ell\in S),
\end{align*}
we can write
\begin{align*}
  & \inf_{\bhz^\star_i} \underset{\bz^\star}{\textnormal{ave}} \ \bbP(\bhz^\star_i \neq \bz^\star_i) \\
  & \geq \frac{1}{2} \cdot  \underset{\bz^\star_{-i}}{\textnormal{ave}}\sum_{\xi\in \{\pm 1\}^S}  \bbP(\zzz{\ell}_i = \xi_\ell \bz_i^\star \ \forall \ell\in S)\\
  & \qquad \times  \inf_{\bhz^\star_i}\bigg(\bbP( \bhz^\star_i =-1 \ | \ \bz_i^\star = 1,  \zzz{\ell}_i = \xi_\ell \bz_i^\star \ \forall \ell\in S)  + \bbP( \bhz^\star_i =1 \ | \ \bz_i^\star = -1,  \zzz{\ell}_i = \xi_\ell \bz_i^\star \ \forall \ell\in S)  \bigg).
\end{align*}
Note that the term inside the parenthesis is the type-I plus type-II error of the following binary hypothesis testing problem:
\begin{equation}
\label{eq: test_without_knowing_community_in_each_layer}
H_0: \bz^\star_i = 1 \ \ \ \textnormal{v.s.} \ \ \ H_1: \bz^\star_i = -1,
\end{equation}
under the \emph{conditional law} of $\{\AAA{\ell}\} \ | \ \{\zzz{\ell}_i = \xi_\ell z_i^\star \ \forall \ell\in S\}$. Hence, this probability can be written as (with $d_{TV}$ denoting the total variation distance)
\begin{align*}
  & 1 - d_{TV} \bigg[ \textnormal{law}\bigg(\{\AAA{\ell}\} \ \bigg| \{\bz^\star_i = 1, \zzz{\ell}_i = \xi_\ell \bz_i^\star \ \forall \ell\in S\}\bigg),  \textnormal{law}\bigg(\{\AAA{\ell}\} \ \bigg| \{\bz^\star_i = -1, \zzz{\ell}_i = \xi_\ell \bz_i^\star \ \forall \ell\in S\}\bigg)\bigg] \\
  & \geq 1 - d_{TV} \bigg[ \textnormal{law}\bigg(\{\AAA{\ell}, \zzz{\ell}\} \ \bigg| \{\bz^\star_i = 1, \zzz{\ell}_i = \xi_\ell \bz_i^\star \ \forall \ell\in S\}\bigg) , \textnormal{law}\bigg(\{\AAA{\ell}, \zzz{\ell}\} \ \bigg| \{\bz^\star_i = -1, \zzz{\ell}_i = \xi_\ell \bz_i^\star \ \forall \ell\in S\}\bigg)\bigg],
\end{align*}
where the inequality is by data-processing inequality for total variation: for any random variables $X, Y, X', Y'$, we have
$$
  d_{TV} \bigg(\textnormal{law}(X), \textnormal{law}(X')\bigg) \leq d_{TV} \bigg(\textnormal{law}(X, Y), \textnormal{law}(X', Y')\bigg).
$$
Now the lower bound is the type-I plus type-II error of the same binary hypothesis testing problem, but under the \emph{joint law} of $\{\AAA{\ell}, \zzz{\ell}\} \ | \ \{\zzz{\ell}_i = \xi_\ell \bz_i^\star \ \forall \ell\in S\}$. 

We now begin to lower bound the optimal testing error of \eqref{eq: test_without_knowing_community_in_each_layer} under the joint law. For $j\neq i$, either (1) $j\in T$, so that $\bz^\star_j = \pm 1$ depending on their membership of $\calC_{\pm}$; or (2) $j\in T^c$, so that $\bz^\star_j$ is unknown to us. Let $m_\pm$ be the number of positive (resp. negative) nodes apart from $i$. By construction, most of the positive nodes come from $\calC_+$ and most of the negative nodes come from $\calC_-$, where the word ``most'' is justified by $m_- = (1+o(1))n/2$ and $m_+ = (1+o(1))n/2$. Let $m = m_- \lor m_+ = (1+o(1))n/2$. By data-processing inequality for total variation, the testing error of \eqref{eq: test_without_knowing_community_in_each_layer} is further lower bounded by the testing error of the same test, but \emph{with $m$ positive nodes and $m$ negative nodes} (excluding $i$). 

For notational simplicity, we consider the following equivalent setup: we have $2m +1$ nodes in total, where nodes $1, \hdots, m$ are labeled as $+1$, nodes $m +1,\hdots, 2m+ 1$ are labeled as $-1$, and the node labeled as $0$ (which is originally labeled as $i$) is the node whose community is to be decided. This is exactly the idealized setup in \eqref{eq:idealized_setup}.

Under the current notations, the joint density of $\{\AAA{\ell}, \zzz{\ell}\}$ is given by
\begingroup
\allowdisplaybreaks
\begin{align*}
  &(1-\rho)^{\#\{\ell\in[L], 0\leq i\leq 2m: \zzz{\ell}_i = \bz^\star_i\}} \rho^{nL - \#\{\ell\in[L], 0\leq i\leq 2m: \zzz{\ell}_i = \bz^\star_i\}}  \\
  & \qquad \times \prod_{\ell\in[L]} \prod_{i\neq j} p_\ell^{\AAA{\ell}_{ij}} (1-p_\ell)^{1 - \AAA{\ell}_{ij}} \indc{\zzz{\ell}_i = \zzz{\ell}_j} + q_\ell^{\AAA{\ell}_{ij}} (1-q_\ell)^{1 - \AAA{\ell}_{ij}} \indc{\zzz{\ell}_i \neq \zzz{\ell}_j} \\
  & = \rho^{nL} \bigg(\frac{1-\rho}{\rho}\bigg)^{\#\{\ell\in[L], 1\leq i\leq 2m: \zzz{\ell}_i = \bz^\star_i\}} \cdot \bigg(\frac{1-\rho}{\rho}\bigg)^{\#\{\ell\in S: \zzz{\ell}_0 = \bz^\star_0\} + \#\{\ell\in S^c: \zzz{\ell}_0 = \bz^\star_0\}} \\
  & \qquad \times \prod_{\ell\in[L]} \prod_{\substack{i\neq j \\ i\neq 0 \\ j\neq 0}} p_\ell^{\AAA{\ell}_{ij}} (1-p_\ell)^{1 - \AAA{\ell}_{ij}} \indc{\zzz{\ell}_i = \zzz{\ell}_j} + q_\ell^{\AAA{\ell}_{ij}} (1-q_\ell)^{1 - \AAA{\ell}_{ij}} \indc{\zzz{\ell}_i \neq \zzz{\ell}_j} \\
  & \qquad \times \prod_{\ell\in S} \prod_{\substack{j\neq 0 \\  \zzz{\ell}_j = \zzz{\ell}_0 }} p_\ell^{\AAA{\ell}_{0j}} (1-p_\ell)^{1 - \AAA{\ell}_{0j}}  \prod_{\substack{j\neq 0 \\  \zzz{\ell}_j \neq \zzz{\ell}_0 }} q_\ell^{\AAA{\ell}_{0j}} (1-q_\ell)^{1 - \AAA{\ell}_{0j}}\\
  & \qquad \times \prod_{\ell\in S^c} \prod_{\substack{j\neq 0 \\  \zzz{\ell}_j = \zzz{\ell}_0 }} p_\ell^{\AAA{\ell}_{0j}} (1-p_\ell)^{1 - \AAA{\ell}_{0j}}  \prod_{\substack{j\neq 0 \\  \zzz{\ell}_j \neq \zzz{\ell}_0 }} q_\ell^{\AAA{\ell}_{0j}} (1-q_\ell)^{1 - \AAA{\ell}_{0j}}.
\end{align*}
\endgroup
Hence, conditional on $\{\zzz{\ell}_0 = \xi_\ell z_0^\star \ \forall \ell\in S\}$, the density of $\{\AAA{\ell}, \zzz{\ell}\}$ becomes
\begingroup
\allowdisplaybreaks
\begin{align*}
  & \rho^{nL} \bigg(\frac{1-\rho}{\rho}\bigg)^{\#\{\ell\in[L], 1\leq i\leq 2m: \zzz{\ell}_i = \bz^\star_i\}} \cdot \bigg(\frac{1-\rho}{\rho}\bigg)^{\#\{\ell\in S: \xi_\ell = 1\} + \#\{\ell\in S^c: \zzz{\ell}_0 = \bz^\star_0\}} \\
  & \qquad \times \prod_{\ell\in[L]} \prod_{\substack{i\neq j \\ i\neq 0 \\ j\neq 0}} p_\ell^{\AAA{\ell}_{ij}} (1-p_\ell)^{1 - \AAA{\ell}_{ij}} \indc{\zzz{\ell}_i = \zzz{\ell}_j} + q_\ell^{\AAA{\ell}_{ij}} (1-q_\ell)^{1 - \AAA{\ell}_{ij}} \indc{\zzz{\ell}_i \neq \zzz{\ell}_j} \\
  & \qquad \times \prod_{\ell\in S} \prod_{\substack{j\neq 0 \\  \zzz{\ell}_j = \xi_\ell \bz^\star_0 }} p_\ell^{\AAA{\ell}_{0j}} (1-p_\ell)^{1 - \AAA{\ell}_{0j}}  \prod_{\substack{j\neq 0 \\  \zzz{\ell}_j  = - \xi_\ell \bz^\star_0}} q_\ell^{\AAA{\ell}_{0j}} (1-q_\ell)^{1 - \AAA{\ell}_{0j}}\\
  & \qquad \times \prod_{\ell\in S^c} \prod_{\substack{j\neq 0 \\  \zzz{\ell}_j = \zzz{\ell}_0 }} p_\ell^{\AAA{\ell}_{0j}} (1-p_\ell)^{1 - \AAA{\ell}_{0j}}  \prod_{\substack{j\neq 0 \\  \zzz{\ell}_j  \neq \zzz{\ell}_0 }} q_\ell^{\AAA{\ell}_{0j}} (1-q_\ell)^{1 - \AAA{\ell}_{0j}}.
\end{align*}
\endgroup
Let $L_0$ be the likelihood under the null and let $L_1$ be the likelihood under the alternative. The likelihood ratio is given by
\begin{align*}
  \frac{L_0}{L_1} & =   \bigg(\frac{1-\rho}{\rho}\bigg)^{\#\{\ell\in S^c: \zzz{\ell}_0 = 1 \} - \#\{\ell\in S^c: \zzz{\ell}_0 = -1 \}} \\
  & \qquad \times \prod_{\ell\in S} \prod_{\substack{j\neq 0 \\  \zzz{\ell}_j = \xi_\ell }} p_\ell^{\AAA{\ell}_{0j}} (1-p_\ell)^{1 - \AAA{\ell}_{0j}}  \prod_{\substack{j\neq 0 \\  \zzz{\ell}_j  = - \xi_\ell }} q_\ell^{\AAA{\ell}_{0j}} (1-q_\ell)^{1 - \AAA{\ell}_{0j}} \\
  & \qquad \times \bigg(\prod_{\ell\in S} \prod_{\substack{j\neq 0 \\  \zzz{\ell}_j = -\xi_\ell }} p_\ell^{\AAA{\ell}_{0j}} (1-p_\ell)^{1 - \AAA{\ell}_{0j}}  \prod_{\substack{j\neq 0 \\  \zzz{\ell}_j  = \xi_\ell }} q_\ell^{\AAA{\ell}_{0j}} (1-q_\ell)^{1 - \AAA{\ell}_{0j}}\bigg)^{-1}.
\end{align*}
By Neyman-Pearson lemma, 
\begin{align*}
  & \inf_{\bhz^\star_0}\bigg(\bbP( \bhz^\star_0 =-1 \ | \ \bz_0^\star = 1,  \zzz{\ell}_0 = \xi_\ell \bz_0^\star \ \forall \ell\in S)  + \bbP( \bhz^\star_0 =1 \ | \ \bz_0^\star = -1,  \zzz{\ell}_0 = \xi_\ell \bz_0^\star \ \forall \ell\in S)  \bigg) \\
  & = \bbP\bigg( \frac{L_0}{L_1}\leq 1 \ \bigg| \ \bz_0^\star = 1,  \zzz{\ell}_0 = \xi_\ell \bz_0^\star \ \forall \ell\in S\bigg)  + \bbP\bigg( \frac{L_0}{L_1}\geq 1 \ | \ \bz_0^\star = -1,  \zzz{\ell}_0 = \xi_\ell \bz_0^\star \ \forall \ell\in S\bigg).
\end{align*}
By symmetry, the two terms in the right-hand side above are equal to each other, so we focus on the first term. For notational simplicity, we let $\bbP_{H_0, S, \xi}$ to denote the conditional law of $\{\AAA{\ell}, \zzz{\ell}\} \ | \ \{\bz_0^\star = 1,  \zzz{\ell}_0 = \xi_\ell \bz_0^\star \ \forall \ell\in S\}$. We then have
\begingroup
\allowdisplaybreaks
\begin{align*}
  \bbP_{H_0, S, \xi} \bigg(\frac{L_0}{L_1} \leq 1\bigg) 
  & = \bbP_{H_0, S, \xi}\bigg[ \log\bigg(\frac{\rho}{1-\rho}\bigg)\cdot \sum_{\ell \in S^c} \bigg(\indc{\zzz{\ell}_0 = 1} - \indc{\zzz{\ell}_0 = -1} \bigg)\\
  & \qquad\qquad \qquad  + \sum_{\ell \in S} \sum_{\substack{j\neq 0 \\ \zzz{\ell}_j = \xi_\ell}} \AAA{\ell}_{0j} \log \bigg(\frac{q_\ell(1-p_\ell)}{p_\ell(1-q_\ell)}\bigg) + \log \bigg(\frac{1-q_\ell}{1-p_\ell}\bigg)\\
  & \qquad\qquad \qquad  + \sum_{\ell \in S} \sum_{\substack{j\neq 0 \\ \zzz{\ell}_j = -\xi_\ell}} \AAA{\ell}_{0j} \log \bigg(\frac{p_\ell(1-q_\ell)}{q_\ell(1-p_\ell)}\bigg) + \log \bigg(\frac{1-p_\ell}{1-q_\ell}\bigg)\geq 0\bigg] \\
  & = \bbE_{\{\zzz{\ell}_{-0}: \ell\in S\}} \bigg\{\bbP_{H_0, S, \xi}\bigg[ \log\bigg(\frac{\rho}{1-\rho}\bigg)\cdot \sum_{\ell \in S^c}\bigg( \indc{\zzz{\ell}_0 = 1} - \indc{\zzz{\ell}_0 = -1} \bigg)\\
  & \qquad \qquad \qquad + \sum_{\ell \in S} \sum_{\substack{j\neq 0 \\ \zzz{\ell}_j = \xi_\ell}} \AAA{\ell}_{0j} \log \bigg(\frac{q_\ell(1-p_\ell)}{p_\ell(1-q_\ell)}\bigg) + \log \bigg(\frac{1-q_\ell}{1-p_\ell}\bigg)\\
  & \qquad\qquad \qquad  + \sum_{\ell \in S} \sum_{\substack{j\neq 0 \\ \zzz{\ell}_j = -\xi_\ell}} \AAA{\ell}_{0j} \log \bigg(\frac{p_\ell(1-q_\ell)}{q_\ell(1-p_\ell)}\bigg) + \log \bigg(\frac{1-p_\ell}{1-q_\ell}\bigg)\geq 0 \ \bigg| \ \{\zzz{\ell}_{-0}\}\bigg]\bigg\}.
\end{align*}
\endgroup
The conditional probability $\bbP_{H_0, S, \xi}(\cdot \ | \ \{\zzz{\ell}_{-0}: \ell\in S\})$ in the right-hand side above is then equal to
\begin{align}
  & \bbP\bigg( \log\bigg(\frac{\rho}{1-\rho}\bigg)\cdot \sum_{\ell \in S^c} \ZZZ{\ell} + \sum_{\ell \in S} \sum_{i = 1}^{\mmm{\ell}_1} \XXX{\ell}_i \cdot \log \bigg(\frac{q_\ell(1-p_\ell)}{p_\ell(1-q_\ell)}\bigg) + \sum_{\ell \in S}\sum_{i = 1}^{\mmm{\ell}_2} \YYY{\ell}_i \cdot \log \bigg(\frac{p_\ell(1-q_\ell)}{q_\ell(1-p_\ell)}\bigg)   \nonumber\\
  \label{eq:prf_lr_test_z_star}
  &\qquad  - \sum_{\ell \in S} (\mmm{\ell}_1- \mmm{\ell}_2) \cdot \log \bigg(\frac{1-p_\ell}{1-q_\ell}\bigg) \geq 0\bigg),
\end{align}
where
\begin{align*}
  & \XXX{\ell}_i \overset{\textnormal{i.i.d.}}{\sim} \Bern(p_\ell), \ \ \ \YYY{\ell}_i \overset{\textnormal{i.i.d.}}{\sim} \Bern(q_\ell),  \ \ \ \ZZZ{\ell} \overset{\textnormal{i.i.d.}}{\sim} 2\Bern(1-\rho)-1, \\
  & \mmm{\ell}_1 = \#\{j\neq 0: \zzz{\ell}_j = \xi_\ell\} ,  \ \ \  \mmm{\ell}_2= \#\{j\neq 0: \zzz{\ell}_j = -\xi_\ell\},
\end{align*}
and $\XXX{\ell}_i$'s, $\YYY{\ell}_i$'s and $\ZZZ{\ell}$'s are jointly independent. 
Note that $\mmm{\ell}_1, \mmm{\ell}_2$ are treated as fixed when we condition on $\{\zzz{\ell}_{-0}: \ell\in S\}$.

With some algebra, one recognizes that \eqref{eq:prf_lr_test_z_star} is the type-I plus type-II error incurred by the likelihood ratio test of the following binary hypothesis testing problem:
\begin{align*}
  H_0' & : W \sim \bigg(\bigotimes_{\ell \in S^c} 2\Bern(1-\rho) -1  \bigg) \otimes \bigg(\bigotimes_{\ell \in S} \bigotimes_{i = 1}^{\mmm{\ell}_1}\Bern(p_\ell)  \bigg) \otimes \bigg(\bigotimes_{\ell \in S} \bigotimes_{i = 1}^{\mmm{\ell}_2}\Bern(q_\ell)  \bigg)  \\
  H_1' & : W \sim \bigg(\bigotimes_{\ell \in S^c} 2\Bern(\rho) -1  \bigg) \otimes \bigg(\bigotimes_{\ell \in S} \bigotimes_{i = 1}^{\mmm{\ell}_1}\Bern(q_\ell)  \bigg) \otimes \bigg(\bigotimes_{\ell \in S} \bigotimes_{i = 1}^{\mmm{\ell}_2}\Bern(p_\ell)  \bigg) .
\end{align*}
By data processing equality for the total variation distance, we know that the optimal testing error of $H_0'$ v.s. $H_1'$ is lower bounded by the optimal testing error of the following testing problem:
\begin{align*}
  H_0'' & : W \sim \bigg(\bigotimes_{\ell \in S^c} 2\Bern(1-\rho) -1  \bigg) \otimes \bigg(\bigotimes_{\ell \in S} \bigotimes_{i = 1}^{\mmm{\ell}}\Bern(p_\ell)  \bigg) \otimes \bigg(\bigotimes_{\ell \in S} \bigotimes_{i = 1}^{\mmm{\ell}}\Bern(q_\ell)  \bigg)  \\
  H_1'' & : W \sim \bigg(\bigotimes_{\ell \in S^c} 2\Bern(\rho) -1  \bigg) \otimes \bigg(\bigotimes_{\ell \in S} \bigotimes_{i = 1}^{\mmm{\ell}}\Bern(q_\ell)  \bigg) \otimes \bigg(\bigotimes_{\ell \in S} \bigotimes_{i = 1}^{\mmm{\ell}}\Bern(p_\ell)  \bigg),
\end{align*}
where we have let
$$
  \mmm{\ell} = \mmm{\ell}_1 \lor \mmm{\ell}_2. 
$$
This means that the probability in \eqref{eq:prf_lr_test_z_star} can be lower bounded by
\begin{align}
  & \bbP\bigg( \log\bigg(\frac{\rho}{1-\rho}\bigg)\cdot \sum_{\ell \in S^c} \ZZZ{\ell} + \sum_{\ell \in S} \sum_{i = 1}^{\mmm{\ell}} \XXX{\ell}_i \cdot \log \bigg(\frac{q_\ell(1-p_\ell)}{p_\ell(1-q_\ell)}\bigg) + \sum_{\ell \in S}\sum_{i = 1}^{\mmm{\ell}} \YYY{\ell}_i \cdot \log \bigg(\frac{p_\ell(1-q_\ell)}{q_\ell(1-p_\ell)}\bigg)   \geq 0\bigg)\nonumber,
\end{align}
Thus, we can lower bound $\inf_{\bhz^\star_0} \textnormal{ave}_{\bz^\star} \bbP(\bhz^\star_0 \neq \bz^\star_0)$ by taking the expectation w.r.t. $\{\zzz{\ell}: \ell \in S\}$:
\begingroup
\allowdisplaybreaks
\begin{align*}
  \inf_{\bhz^\star_0} \underset{\bz^\star}{\textnormal{ave}}~ \bbP(\bhz^\star_0 \neq \bz^\star_0) 
  & \geq \bbE_{\{\zzz{\ell}_0: \ell \in S\}} \bbE_{\{\zzz{\ell}_{-0}: \ell \in S\}} \bigg[ \bbP\bigg( \log\bigg(\frac{\rho}{1-\rho}\bigg)\cdot \sum_{\ell \in S^c} \ZZZ{\ell} + \sum_{\ell \in S} \sum_{i = 1}^{\mmm{\ell}_1} \XXX{\ell}_i \cdot \log \bigg(\frac{q_\ell(1-p_\ell)}{p_\ell(1-q_\ell)}\bigg) \\
  & \qquad \qquad + \sum_{\ell \in S}\sum_{i = 1}^{\mmm{\ell}_2} \YYY{\ell}_i \cdot \log \bigg(\frac{p_\ell(1-q_\ell)}{q_\ell(1-p_\ell)}\bigg)  
  - \sum_{\ell \in S} (\mmm{\ell}_1- \mmm{\ell}_2) \cdot \log \bigg(\frac{1-p_\ell}{1-q_\ell}\bigg) \geq 0\bigg) \bigg] \\
  & \geq \bbE_{\{\zzz{\ell}_0: \ell \in S\}} \bbE_{\{\zzz{\ell}_{-0}: \ell \in S\}} \bigg[ \bbP\bigg( \log\bigg(\frac{\rho}{1-\rho}\bigg)\cdot \sum_{\ell \in S^c} \ZZZ{\ell} + \sum_{\ell \in S} \sum_{i = 1}^{\mmm{\ell}} \XXX{\ell}_i \cdot \log \bigg(\frac{q_\ell(1-p_\ell)}{p_\ell(1-q_\ell)}\bigg) \\
  & \qquad \qquad + \sum_{\ell \in S}\sum_{i = 1}^{\mmm{\ell}} \YYY{\ell}_i \cdot \log \bigg(\frac{p_\ell(1-q_\ell)}{q_\ell(1-p_\ell)}\bigg)  
  \geq 0\bigg) \bigg] \\
  & \geq  \bbE_{\{\zzz{\ell}: \ell \in S\}} \bigg[\Indc_{E_\ep} \cdot \bbP\bigg( \log\bigg(\frac{\rho}{1-\rho}\bigg)\cdot \sum_{\ell \in S^c} \ZZZ{\ell} + \sum_{\ell \in S} \sum_{i = 1}^{\mmm{\ell}} \XXX{\ell}_i \cdot \log \bigg(\frac{q_\ell(1-p_\ell)}{p_\ell(1-q_\ell)}\bigg) \\
  & \qquad \qquad + \sum_{\ell \in S}\sum_{i = 1}^{\mmm{\ell}} \YYY{\ell}_i \cdot \log \bigg(\frac{p_\ell(1-q_\ell)}{q_\ell(1-p_\ell)}\bigg)  
  \geq 0\bigg) \bigg],
\end{align*}
\endgroup
where the event $E_\ep$ is defined as
\begin{align*}
  E_\ep:= \bigg\{ \bigg|\frac{\mmm{\ell}_1}{m} - 1\bigg| \lor \bigg|\frac{\mmm{\ell}_2}{m} - 1\bigg| \leq \ep \ \forall \ell \in S \bigg\}.
\end{align*}
Note that for a fixed $\ell \in [L]$, if $\zzz{\ell}_0 = \bz^\star_0 = 1$ (i.e., $\xi_\ell = 1$), then we have
\begin{align*}
  \mmm{\ell}_1 = \sum_{j = 1}^m \indc{\zzz{\ell}_j = 1} + \sum_{j = m+1}^{2m}\indc{\zzz{\ell}_j = 1},
\end{align*}
where for $1\leq j \leq m$, we have $\indc{\zzz{\ell}_j = 1} \overset{\textnormal{i.i.d.}}{\sim} \Bern(1-\rho)$, and for $m+1\leq j\leq 2m$, we have $\indc{\zzz{\ell}_j = 1} \overset{\textnormal{i.i.d.}}{\sim} \Bern(\rho)$. This gives
$
  \bbE[\mmm{\ell}_1] = m,
$
and by Chernoff bound, we have
$$
  \bbP( |\mmm{\ell}_1 - m |\geq m\ep) \leq 2e^{-\ep^2 m}.
$$
Since $\mmm{\ell}_1 + \mmm{\ell}_2 = 2m$, if $|\mmm{\ell}_1 - m |\leq m\ep$, then we automatically have $|\mmm{\ell}_2 - m |\leq m\ep$. Hence, we get
$$
  \bbP( |\mmm{\ell}_1 - m | \lor |\mmm{\ell}_2 - m | \leq m\ep) \geq 1- 2e^{-\ep^2 m}.
$$
The above inequality also holds for the case of $\xi_\ell = -1$. Now, taking a union bound over $L$ and recalling $\log L\ll n^c$ for some $c\in (0, 1)$, we have
\begin{equation}
  \label{eq:concentration_of_number_of_pos_and_neg_nodes}
  \bbP(E_\ep) \geq 1- 2L e^{-\ep^2 m} = 1 - 2 e^{- \Theta(\ep^2 n) + o(n^c)}.
\end{equation}  
Choosing $\ep = n^{-(1-c)/2} = o(1)$, we have $\bbP(E_\ep) = 1-2e^{-\Theta(n^c) + o(n^c)} = 1 -o(1)$. That is, with probability tending to one, we have $\mmm{\ell} = (1+o(1))m$ uniformly over $\ell \in S$. Hence, by further applying data-processing inequality if necessary, we can lower bound $\inf_{\bhz^\star_0} \textnormal{ave}_{\bz^\star} \bbP(\bhz^\star_0 \neq \bz^\star_0)$ by
\begin{align*}
  &\bbE_{\{\zzz{\ell}: \ell \in S\}} \bigg[\Indc_{E_\ep} \cdot \bbP\bigg( \log\bigg(\frac{\rho}{1-\rho}\bigg)\cdot \sum_{\ell \in S^c} \ZZZ{\ell} + \sum_{\ell \in S} \sum_{i = 1}^{m'} \XXX{\ell}_i \cdot \log \bigg(\frac{q_\ell(1-p_\ell)}{p_\ell(1-q_\ell)}\bigg) \\
  & \qquad \qquad + \sum_{\ell \in S}\sum_{i = 1}^{m'} \YYY{\ell}_i \cdot \log \bigg(\frac{p_\ell(1-q_\ell)}{q_\ell(1-p_\ell)}\bigg)  
   \geq 0\bigg) \bigg],
\end{align*}
where $m' = \max_{\ell\in S} \mmm{\ell} = (1+o(1))m = (1+o(1))n/2$. In the above display, the randomness of $\{\zzz{\ell}: \ell\in S\}$ only appears in the event $E_\ep$. So the above display is equal to 
\begin{align*}
  &\bbP(E_\ep)  \cdot \bbP\bigg( \log\bigg(\frac{\rho}{1-\rho}\bigg)\cdot \sum_{\ell \in S^c} \ZZZ{\ell} + \sum_{\ell \in S} \sum_{i = 1}^{m'} \XXX{\ell}_i \cdot \log \bigg(\frac{q_\ell(1-p_\ell)}{p_\ell(1-q_\ell)}\bigg) + \sum_{\ell \in S}\sum_{i = 1}^{m'} \YYY{\ell}_i \cdot \log \bigg(\frac{p_\ell(1-q_\ell)}{q_\ell(1-p_\ell)}\bigg)  \geq 0\bigg) \\
  & \gtrsim \bbP\bigg( \log\bigg(\frac{\rho}{1-\rho}\bigg)\cdot \sum_{\ell \in S^c} \ZZZ{\ell} + \sum_{\ell \in S} \sum_{i = 1}^{m'} \XXX{\ell}_i \cdot \log \bigg(\frac{q_\ell(1-p_\ell)}{p_\ell(1-q_\ell)}\bigg) + \sum_{\ell \in S}\sum_{i = 1}^{m'} \YYY{\ell}_i \cdot \log \bigg(\frac{p_\ell(1-q_\ell)}{q_\ell(1-p_\ell)}\bigg)  \geq 0\bigg),
\end{align*}
where the inequality is by $\bbP(E_\ep) = 1-o(1)$. Finally, we conclude the proof by recalling   \eqref{eq:same_error_across_i_lower_bound_z_star} and noting that the right-hand side above is the error incurred by the likelihood ratio test for the testing problem \eqref{eq:fund_testing_problem_z_star} with $m' = (1+\delta_n')n/2$.

\subsection{Proof of Lemma \ref{lemma:opt_test_error_z_star}}\label{subappend:prf_opt_test_error_z_star}
We will give a lower bound for
\begin{align}
  \label{eq:target_of_lower_bound_z_star}
  \bbP\bigg(\sum_{\ell\in S} \log \bigg(\frac{q_\ell(1-p_\ell)}{p_\ell(1-q_\ell)}\bigg) \cdot \sum_{i=1}^{m} (\XXX{\ell}_i-\YYY{\ell}_i) \geq \log \bigg(\frac{1-\rho}{\rho}\bigg) \cdot \sum_{\ell\in S^c}\ZZZ{\ell}\bigg),
\end{align}
where $m = (1+\delta'_n)n/2 = (1+o(1))n/2$. We begin by decomposing the above probability as
\begin{align*}
  \bbE_{\ZZZ{S^c}}\bigg[\bbP\bigg(\sum_{\ell \in S} \sum_{i = 1}^{m} \XXX{\ell}_i \cdot \log \bigg(\frac{q_\ell(1-p_\ell)}{p_\ell(1-q_\ell)}\bigg) + \YYY{\ell}_i \cdot \log \bigg(\frac{p_\ell(1-q_\ell)}{q_\ell(1-p_\ell)}\bigg)  \geq \log\bigg(\frac{1-\rho}{\rho}\bigg)\cdot \sum_{\ell\in S^c}\ZZZ{\ell} \ \bigg| \ \ZZZ{S^c} \bigg)\bigg].
\end{align*}
Consider the moment generating function
\begin{align*}
  \phi_S(t) & = \bbE\bigg[\exp\bigg\{ t \times \sum_{\ell \in S} \sum_{i = 1}^{m} \XXX{\ell}_i \cdot \log \bigg(\frac{q_\ell(1-p_\ell)}{p_\ell(1-q_\ell)}\bigg) + \YYY{\ell}_i \cdot \log \bigg(\frac{p_\ell(1-q_\ell)}{q_\ell(1-p_\ell)}\bigg) \bigg\}\bigg]\\
  & = \prod_{\ell \in S} \prod_{i\in[m]} \bigg[ p_\ell \bigg(\frac{q_\ell(1-p_\ell)}{p_\ell (1-q_\ell)}\bigg)^t + (1-p_\ell)\bigg] \bigg[q_\ell \bigg(\frac{p_\ell(1-q_\ell)}{q_\ell(1-p_\ell)}\bigg)^t + (1-q_\ell)\bigg] \\
  \label{eq:mgf_opt_test_z_star}
  & = \prod_{\ell\in S} \bigg(p_\ell^{1-t}q_\ell^t + (1-p_\ell)^{1-t}(1-q_\ell)^{t}\bigg)^m\bigg(p_\ell^t q_\ell^{1-t} + (1-p_\ell)^t (1-q_\ell)^{1-t}\bigg)^m.\numberthis
\end{align*}
The information-theoretic quantity $\psi_S(t)$ defined in \eqref{eq:cum_gen_fun} is the corresponding cumulant generating function (with $m$ replaced by $n/2$):
\begin{align}
  \label{eq:cgf_opt_test_z_star}
  \psi_S(t) &  = \sum_{\ell \in S} \frac{n}{2} \cdot \bigg[\log\bigg(p_\ell^{1-t}q_\ell^t + (1-p_\ell)^{1-t}(1-q_\ell)^{t}\bigg) + \log\bigg(p_\ell^t q_\ell^{1-t} + (1-p_\ell)^t (1-q_\ell)^{1-t}\bigg)\bigg].
\end{align} 
And we recall that 
\begin{align*}
  \psi_S^\star(a) = \sup_{0\leq t\leq 1} \{ at - \psi_S(t)\}.
\end{align*}
Now, treating $\ZZZ{S^c}$ as fixed, we let 
\begin{align}
  \label{eq:original_law}
  \mu &= \textnormal{Law} \bigg[ \sum_{\ell \in S} \sum_{i = 1}^{m} \XXX{\ell}_i \cdot \log \bigg(\frac{q_\ell(1-p_\ell)}{p_\ell(1-q_\ell)}\bigg) + \YYY{\ell}_i \cdot \log \bigg(\frac{p_\ell(1-q_\ell)}{q_\ell(1-p_\ell)}\bigg) \bigg] \\
  \label{eq:tilted_law}
  \tilde \mu_t & =  \Conv_{\ell\in S} \Conv_{i = 1}^{m} \textnormal{Law} (\tXXX{\ell}_t \conv \tYYY{\ell}_t),
\end{align} 
where $\conv$ denotes the convolution of random variables and 
\begin{equation}
\label{eq:tilted_rvs}
  \tXXX{\ell}_t + \tYYY{\ell}_t = 
  \begin{cases}
    \log \frac{q_\ell(1-p_\ell)}{p_\ell(1-q_\ell)} & \textnormal{w.p. } \frac{[p_\ell(1-q_\ell)]^{1-t} [q_\ell(1-p_\ell)]^t }{p_\ell q_\ell + (1-p_\ell)(1-q_\ell) + [p_\ell(1-q_\ell)]^{1-t} [q_\ell(1-p_\ell)]^t + [p_\ell(1-q_\ell)]^{t} [q_\ell(1-p_\ell)]^{1-t}} \\
    \log \frac{p_\ell(1-q_\ell)}{q_\ell(1-p_\ell)} & \textnormal{w.p. } \frac{[p_\ell(1-q_\ell)]^{t} [q_\ell(1-p_\ell)]^{1-t} }{p_\ell q_\ell + (1-p_\ell)(1-q_\ell) + [p_\ell(1-q_\ell)]^{1-t} [q_\ell(1-p_\ell)]^t + [p_\ell(1-q_\ell)]^{t} [q_\ell(1-p_\ell)]^{1-t}}\\
    0 & \textnormal{w.p. } \frac{p_\ell q_\ell + (1-p_\ell)(1-q_\ell)}{p_\ell q_\ell + (1-p_\ell)(1-q_\ell) + [p_\ell(1-q_\ell)]^{1-t} [q_\ell(1-p_\ell)]^t + [p_\ell(1-q_\ell)]^{t} [q_\ell(1-p_\ell)]^{1-t}}.
  \end{cases}
\end{equation}  
Here, the random variables $\tXXX{\ell}_t$ and $\tYYY{\ell}_t$ are exponentially tilted version of $\XXX{\ell}_i \cdot \log \frac{q_\ell(1-p_\ell)}{p_\ell (1-q_\ell)}$ and $\YYY{\ell}_i \cdot \log \frac{p_\ell(1-q_\ell)}{q_\ell(1-p_\ell)}$, respectively. As discussed in Remark \ref{rmk:parity}, the lower bound for \eqref{eq:target_of_lower_bound_z_star} depends on the parity of $|S^c|$.

\subsubsection{The case of even \texorpdfstring{$|S^c|$}{Sc}}
Assume $|S^c|$ is even. If $S^c = \varnothing$, then we have
\begin{align*}
  &\bbE_{\ZZZ{S^c}}\bigg[\bbP\bigg(\sum_{\ell \in S} \sum_{i = 1}^{m} \XXX{\ell}_i \cdot \log \bigg(\frac{q_\ell(1-p_\ell)}{p_\ell(1-q_\ell)}\bigg) + \YYY{\ell}_i \cdot \log \bigg(\frac{p_\ell(1-q_\ell)}{q_\ell(1-p_\ell)}\bigg)  \geq \log\bigg(\frac{1-\rho}{\rho}\bigg)\cdot \sum_{\ell\in S^c}\ZZZ{\ell} \ \bigg| \ \ZZZ{S^c} \bigg)\bigg]\\
  & = \bbP\bigg(\sum_{\ell \in S} \sum_{i = 1}^{m} \XXX{\ell}_i \cdot \log \bigg(\frac{q_\ell(1-p_\ell)}{p_\ell(1-q_\ell)}\bigg) + \YYY{\ell}_i \cdot \log \bigg(\frac{p_\ell(1-q_\ell)}{q_\ell(1-p_\ell)}\bigg)  \geq 0 \bigg).
\end{align*}
If $S^c \neq \varnothing$, we can write
\begin{align*}
  &\bbE_{\ZZZ{S^c}}\bigg[\bbP\bigg(\sum_{\ell \in S} \sum_{i = 1}^{m} \XXX{\ell}_i \cdot \log \bigg(\frac{q_\ell(1-p_\ell)}{p_\ell(1-q_\ell)}\bigg) + \YYY{\ell}_i \cdot \log \bigg(\frac{p_\ell(1-q_\ell)}{q_\ell(1-p_\ell)}\bigg)  \geq \log\bigg(\frac{1-\rho}{\rho}\bigg)\cdot \sum_{\ell\in S^c}\ZZZ{\ell} \ \bigg| \ \ZZZ{S^c} \bigg)\bigg]\\
  & = \sum_{x \in \{-|S^c|+ 2j: 0\leq j\leq |S^c|\}}\binom{|S^c|}{\frac{|S^c| + x}{2}} (1-\rho)^{\frac{|S^c| + x}{2}} \rho^{\frac{|S^c| - x}{2}}\\
  & \qquad \times \bbP\bigg(\sum_{\ell \in S} \sum_{i = 1}^{m} \XXX{\ell}_i \cdot \log \bigg(\frac{q_\ell(1-p_\ell)}{p_\ell(1-q_\ell)}\bigg) + \YYY{\ell}_i \cdot \log \bigg(\frac{p_\ell(1-q_\ell)}{q_\ell(1-p_\ell)}\bigg)  \geq \log\bigg(\frac{1-\rho}{\rho}\bigg)\cdot x \bigg)\\
  & \geq \binom{|S^c|}{\frac{|S^c|}{2}} \bigg((1-\rho)\rho\bigg)^{\frac{|S^c|}{2}}\times \bbP\bigg(\sum_{\ell \in S} \sum_{i = 1}^{m} \XXX{\ell}_i \cdot \log \bigg(\frac{q_\ell(1-p_\ell)}{p_\ell(1-q_\ell)}\bigg) + \YYY{\ell}_i \cdot \log \bigg(\frac{p_\ell(1-q_\ell)}{q_\ell(1-p_\ell)}\bigg)  \geq 0\bigg)\\
  & \geq 2^{\frac{|S^c|}{2}} \cdot \exp\bigg\{-(1+o(1))|S^c|J_\rho \bigg\}\times \bbP\bigg(\sum_{\ell \in S} \sum_{i = 1}^{m} \XXX{\ell}_i \cdot \log \bigg(\frac{q_\ell(1-p_\ell)}{p_\ell(1-q_\ell)}\bigg) + \YYY{\ell}_i \cdot \log \bigg(\frac{p_\ell(1-q_\ell)}{q_\ell(1-p_\ell)}\bigg)  \geq 0\bigg)\\
  & \geq \exp\bigg\{-(1+o(1))|S^c|J_\rho \bigg\}\times \bbP\bigg(\sum_{\ell \in S} \sum_{i = 1}^{m} \XXX{\ell}_i \cdot \log \bigg(\frac{q_\ell(1-p_\ell)}{p_\ell(1-q_\ell)}\bigg) + \YYY{\ell}_i \cdot \log \bigg(\frac{p_\ell(1-q_\ell)}{q_\ell(1-p_\ell)}\bigg)  \geq 0\bigg),
\end{align*}
where we have used the fact that $|S^c|$ is even, $\binom{n}{k}\geq (n/k)^k$, and 
$$
J_\rho := -\log2\sqrt{\rho(1-\rho)} = -(1+o(1))\log\sqrt{\rho(1-\rho)}
$$
when $\rho = o(1)$. Thus, in both cases,   \eqref{eq:target_of_lower_bound_z_star} can be lower bounded by
$$
\exp\bigg\{-(1+o(1))|S^c|J_\rho \bigg\}\times \bbP\bigg(\sum_{\ell \in S} \sum_{i = 1}^{m} \XXX{\ell}_i \cdot \log \bigg(\frac{q_\ell(1-p_\ell)}{p_\ell(1-q_\ell)}\bigg) + \YYY{\ell}_i \cdot \log \bigg(\frac{p_\ell(1-q_\ell)}{q_\ell(1-p_\ell)}\bigg)  \geq 0\bigg).
$$
Hence, we focus on lower bounding the following probability:
\begin{equation}
\label{eq:target_of_lower_bound_z_star_even}
  \bbP\bigg(\sum_{\ell \in S} \sum_{i = 1}^{m} \XXX{\ell}_i \cdot \log \bigg(\frac{q_\ell(1-p_\ell)}{p_\ell(1-q_\ell)}\bigg) + \YYY{\ell}_i \cdot \log \bigg(\frac{p_\ell(1-q_\ell)}{q_\ell(1-p_\ell)}\bigg)  \geq 0\bigg) = \mu([0, \infty)).
\end{equation}  

By a standard exponential tilting argument (also known as the Cram\'er-Chernoff argument, a technique commonly used in proving large deviation principles), we have
\begin{align*}
  \mu([ 0,  \infty))& = e^{-(1+o(1))\psi_S(t)} \cdot \bbE_{W \sim  \tilde\mu_t}\bigg[e^{-tW} \cdot \indc{W \geq 0}\bigg]\\
  & \geq \exp\bigg\{-(1+o(1))\psi_S(t)-t\xi\bigg\}  \times \bbP_{W\sim \tilde\mu_t}(0\leq W\leq \xi),
\end{align*}
where the $(1+o(1))$ term comes from $m = (1+o(1))n/2$.
Choosing $t = 1/2$, we arrive at the following lower bound:
\begin{align*}
  \mu([ 0,  \infty))&  \geq \exp\bigg\{-(1+o(1))\sum_{\ell\in S}m\III{\ell}_{1/2}-\xi/2\bigg\}  \times \bbP_{W\sim \tilde\mu_{1/2}}(0\leq W\leq \xi).
\end{align*}
We need the following lemma.
\begin{lemma}
  \label{lemma: exp_and_var_of_tilted_bernoulli}
  Assume there exist constants $C > 1, c\in(0, 1)$ such that $q_\ell < p_\ell \leq (C q_\ell) \land (1-c)$ for any $\ell\in[L]$.
  Then for any $t\in[0, 1]$, we have
  \begin{align*}
    \bbE[(\tXXX{\ell}_t + \tYYY{\ell}_t)^2]
    \asymp \Var(\tXXX{\ell}_t + \tYYY{\ell}_t) \asymp \III{\ell}_{1/2}.
  \end{align*}
\end{lemma}
\begin{proof}
The first two moments of $\tXXX{\ell}_t + \tYYY{\ell}_t$ are given by
\begin{align*}
  & \bbE[\tXXX{\ell}_t + \tYYY{\ell}_t]\\
  & =  \bigg(\log \frac{p_\ell(1-q_\ell)}{q_\ell(1-p_\ell)}\bigg) 
  \cdot \frac{[p_\ell(1-q_\ell)]^{t} [q_\ell(1-p_\ell)]^{1-t}  - [p_\ell(1-q_\ell)]^{1-t} [q_\ell(1-p_\ell)]^t }{p_\ell q_\ell + (1-p_\ell)(1-q_\ell) + [p_\ell(1-q_\ell)]^{1-t} [q_\ell(1-p_\ell)]^t + [p_\ell(1-q_\ell)]^{t} [q_\ell(1-p_\ell)]^{1-t}},
\end{align*}
and
\begin{align*}
  & \bbE[(\tXXX{\ell}_t + \tYYY{\ell}_t)^2] \\
  & = \bigg(\log\frac{p_\ell(1-q_\ell)}{q_\ell(1-p_\ell)}\bigg)^2 \cdot \frac{[p_\ell(1-q_\ell)]^t[q_\ell(1-p_\ell)]^{1-t} + [p_\ell(1-q_\ell)]^{1-t}[q_\ell(1-p_\ell)]^{t}}{p_\ell q_\ell + (1-p_\ell)(1-q_\ell) + [p_\ell(1-q_\ell)]^{1-t} [q_\ell(1-p_\ell)]^t + [p_\ell(1-q_\ell)]^{t} [q_\ell(1-p_\ell)]^{1-t}},
\end{align*}
respectively.
Note that
\begin{align*}
  \log\frac{p_\ell(1-q_\ell)}{q_\ell(1-p_\ell)} 
  & = \log \bigg(1 + \frac{p_\ell- q_\ell}{q_\ell}\bigg) + \log \bigg(1 + \frac{p_\ell - q_\ell}{1-p_\ell}\bigg) \\
  & \leq (p_\ell - q_\ell) \cdot \bigg(\frac{1}{q_\ell} + \frac{1}{1-p_\ell}\bigg) \\
  & \leq  C \cdot \frac{p_\ell-q_\ell}{p_\ell} + \frac{p_\ell}{1-p_\ell} \cdot \frac{p_\ell - q_\ell}{p_\ell}\\
  & \leq (C + c^{-1})\cdot \frac{p_\ell-q_\ell}{p_\ell}, 
\end{align*}  
where the first inequality is by $\log(1+x)\leq x$ and the last two inequalities is by $q_\ell < p_\ell \leq (C q_\ell)\land (1-c)$.
Meanwhile, we have
\begin{align*}
  \log\frac{p_\ell(1-q_\ell)}{q_\ell(1-p_\ell)} 
  & \geq \log\bigg(1 + \frac{p_\ell-q_\ell}{q_\ell}\bigg) \geq \frac{(p_\ell-q_\ell)/q_\ell}{1 + (p_\ell-q_\ell)/q_\ell} = \frac{p_\ell-q_\ell}{p_\ell},
\end{align*}
where the second inequality is by $\log(1+x)\geq x/(1+x)$. Thus, we have
\begin{align}
  \label{eq: exp_and_var_of_tilted_bernoulli_intermediate_result}
  \log\frac{p_\ell(1-q_\ell)}{q_\ell(1-p_\ell)}  \asymp \frac{p_\ell-q_\ell}{p_\ell}.
\end{align}
Now, we claim that
$$
\frac{[p_\ell(1-q_\ell)]^t[q_\ell(1-p_\ell)]^{1-t} + [p_\ell(1-q_\ell)]^{1-t}[q_\ell(1-p_\ell)]^{t}}{p_\ell q_\ell + (1-p_\ell)(1-q_\ell) + [p_\ell(1-q_\ell)]^{1-t} [q_\ell(1-p_\ell)]^t + [p_\ell(1-q_\ell)]^{t} [q_\ell(1-p_\ell)]^{1-t}} \asymp p_\ell.
$$
Indeed, since $q_\ell < p_\ell < 1-c$, the denominator is $\Theta(1)$, and since $q_\ell < p_\ell \leq C q_\ell$, the numerator satisfies 
$$
  [p_\ell(1-q_\ell)]^t[q_\ell(1-p_\ell)]^{1-t} + [p_\ell(1-q_\ell)]^{1-t}[q_\ell(1-p_\ell)]^{t} = \Theta(p_\ell) \cdot [(1-q_\ell)^{t} (1-p_\ell)^{1-t} + (1-q_\ell)^{1-t}(1-p_\ell)^t] \asymp p_\ell.
$$
Thus, we have
\begin{align*}
  & \bbE[(\tXXX{\ell}_t + \tYYY{\ell}_t)^2] \asymp  \frac{(p_\ell-q_\ell)^2}{p_\ell^2} \cdot p_\ell =  \frac{(p_\ell-q_\ell)^2}{p_\ell} \asymp \III{\ell}_{1/2},
\end{align*}
where the last inequality is by Lemma \ref{lemma:asymp_equiv_of_I_t}. To show the asymptotic equivalence of the variance, from the formula for the first two moments, we have
\begin{align*}
  & \Var(\tXXX{\ell}_t + \tYYY{\ell}_t) \\
  & = \bigg(\log \frac{p_\ell(1-q_\ell)}{q_\ell(1-p_\ell)}\bigg)^2 \cdot 
  \bigg[
    \frac{[p_\ell(1-q_\ell)]^t[q_\ell(1-p_\ell)]^{1-t} + [p_\ell(1-q_\ell)]^{1-t}[q_\ell(1-p_\ell)]^{t}}{p_\ell q_\ell + (1-p_\ell)(1-q_\ell) + [p_\ell(1-q_\ell)]^{1-t} [q_\ell(1-p_\ell)]^t + [p_\ell(1-q_\ell)]^{t} [q_\ell(1-p_\ell)]^{1-t}}\\
    &\qquad - 
    \bigg(\frac{[p_\ell(1-q_\ell)]^{t} [q_\ell(1-p_\ell)]^{1-t}  - [p_\ell(1-q_\ell)]^{1-t} [q_\ell(1-p_\ell)]^t }{p_\ell q_\ell + (1-p_\ell)(1-q_\ell) + [p_\ell(1-q_\ell)]^{1-t} [q_\ell(1-p_\ell)]^t + [p_\ell(1-q_\ell)]^{t} [q_\ell(1-p_\ell)]^{1-t}}\bigg)^2
  \bigg]\\
  & = \bigg(\log \frac{p_\ell(1-q_\ell)}{q_\ell(1-p_\ell)}\bigg)^2 \cdot 
  \bigg(\frac{1}{p_\ell q_\ell + (1-p_\ell)(1-q_\ell) + [p_\ell(1-q_\ell)]^{1-t} [q_\ell(1-p_\ell)]^t + [p_\ell(1-q_\ell)]^{t} [q_\ell(1-p_\ell)]^{1-t}}\bigg)^2\\
  & \qquad \times 
  \bigg[
    \bigg([p_\ell(1-q_\ell)]^t[q_\ell(1-p_\ell)]^{1-t} + [p_\ell(1-q_\ell)]^{1-t}[q_\ell(1-p_\ell)]^{t}\bigg)\\
    & \qquad \qquad \cdot \bigg({p_\ell q_\ell + (1-p_\ell)(1-q_\ell) + [p_\ell(1-q_\ell)]^{1-t} [q_\ell(1-p_\ell)]^t + [p_\ell(1-q_\ell)]^{t} [q_\ell(1-p_\ell)]^{1-t}}\bigg) \\
    &\qquad\qquad - \bigg([p_\ell(1-q_\ell)]^{t} [q_\ell(1-p_\ell)]^{1-t}  - [p_\ell(1-q_\ell)]^{1-t} [q_\ell(1-p_\ell)]^t\bigg)^2
  \bigg]\\
  & = \bigg(\log \frac{p_\ell(1-q_\ell)}{q_\ell(1-p_\ell)}\bigg)^2 \cdot 
  \bigg(\frac{1}{p_\ell q_\ell + (1-p_\ell)(1-q_\ell) + [p_\ell(1-q_\ell)]^{1-t} [q_\ell(1-p_\ell)]^t + [p_\ell(1-q_\ell)]^{t} [q_\ell(1-p_\ell)]^{1-t}}\bigg)^2\\
  & \qquad \times 
  \bigg[
    \bigg([p_\ell(1-q_\ell)]^t[q_\ell(1-p_\ell)]^{1-t} + [p_\ell(1-q_\ell)]^{1-t}[q_\ell(1-p_\ell)]^{t}\bigg)\bigg(p_\ell q_\ell + (1-p_\ell)(1-q_\ell)\bigg)\\
    & \qquad \qquad + \bigg([p_\ell(1-q_\ell)]^t[q_\ell(1-p_\ell)]^{1-t} + [p_\ell(1-q_\ell)]^{1-t}[q_\ell(1-p_\ell)]^{t}\bigg)^2  \\
    &\qquad\qquad - \bigg([p_\ell(1-q_\ell)]^{t} [q_\ell(1-p_\ell)]^{1-t}  - [p_\ell(1-q_\ell)]^{1-t} [q_\ell(1-p_\ell)]^t\bigg)^2
  \bigg]\\
  & \geq \bigg(\log \frac{p_\ell(1-q_\ell)}{q_\ell(1-p_\ell)}\bigg)^2 \cdot 
  \bigg(\frac{1}{p_\ell q_\ell + (1-p_\ell)(1-q_\ell) + [p_\ell(1-q_\ell)]^{1-t} [q_\ell(1-p_\ell)]^t + [p_\ell(1-q_\ell)]^{t} [q_\ell(1-p_\ell)]^{1-t}}\bigg)^2\\
  & \qquad \times 
    \bigg([p_\ell(1-q_\ell)]^t[q_\ell(1-p_\ell)]^{1-t} + [p_\ell(1-q_\ell)]^{1-t}[q_\ell(1-p_\ell)]^{t}\bigg)\bigg(p_\ell q_\ell + (1-p_\ell)(1-q_\ell)\bigg).
\end{align*}
There are four terms in the right-hand side above. We have shown that the first term is $\Theta\left(\frac{(p_\ell-q_\ell)^2}{p_\ell^2}\right)$, the second term is $\Theta(1)$, and the third term is $\Theta(p_\ell)$. It is clear that the fourth term is $\Omega(1)$. Thus, we have
$$
  \Var(\tXXX{\ell}_t + \tYYY{\ell}_t) \gtrsim \frac{(p_\ell-q_\ell)^2}{p_\ell} \asymp \III{\ell}_{1/2}.
$$
The reverse inequality is trivial as $\Var(\tXXX{\ell}_t + \tYYY{\ell}_t) \leq \bbE[(\tXXX{\ell}_t + \tYYY{\ell}_t)^2]$. The proof is concluded.
\end{proof}

Now, by lemma \ref{lemma: exp_and_var_of_tilted_bernoulli} and Chebyshev's inequality, we have
$$
  \bbP_{W\sim \tilde\mu_{1/2}} \bigg(|W|>c \sqrt{\frac{\sum_{\ell\in S} m\III{\ell}_{1/2}}{\delta}}\bigg) \leq \delta,
$$
where $c$ is some absolute constant. Note that with $t = 1/2$, the random variable $\tXXX{\ell}_{1/2} + \tXXX{\ell}_{1/2}$ is actually \emph{symmetric about zero}. This gives
$$
  \bbP_{W\sim \tilde \mu_{1/2}}\bigg(0\leq W\leq c \sqrt{\frac{\sum_{\ell\in S} m \III{\ell}_{1/2}}{\delta}}\bigg) \geq \frac{1-\delta}{2}. 
$$
Hence, by properly choosing $\delta$, we have
\begin{align*}
  \mu([0, \infty)) \gtrsim \bigg\{-(1+o(1)) \bigg(\sum_{\ell \in S} m\III{\ell}_{1/2} - c \sqrt{\sum_{\ell \in S}m\III{\ell}_{1/2}}\bigg)\bigg\}.
\end{align*}
This gives the following lower bound for   \eqref{eq:target_of_lower_bound_z_star}:
\begin{align*}
  & C \cdot \exp\bigg\{-\bigl(1+o(1)\bigr) \cdot \bigg(|S^c| J_\rho + \sum_{\ell \in S}m\III{\ell}_{1/2} + c\sqrt{\sum_{\ell\in S}m\III{\ell}_{1/2}}\bigg) \bigg\}.
\end{align*}
Consider the following three cases.
\begin{enumerate}
  \item In this case, we assume $S^c = \varnothing$. Now, if $\sum_{\ell \in S}m\III{\ell}_{1/2} = \sum_{\ell \in [L]}m\III{\ell}_{1/2}\to \infty$, then the lower bound for   \eqref{eq:target_of_lower_bound_z_star} becomes
  $$
     C\cdot\exp\bigg\{-\bigl(1+o(1)\bigr) \cdot \sum_{\ell \in [L]}m\III{\ell}_{1/2}\bigg\}.
  $$
  On the other hand, if $\sum_{\ell \in [L]}m\III{\ell}_{1/2} = O(1)$, then the corresponding lower bound is again
  $$
    C\cdot\exp\bigg\{-\bigl(1+o(1)\bigr) \cdot \sum_{\ell \in [L]}m\III{\ell}_{1/2}\bigg\},
  $$
  because $\sqrt{\sum_{\ell\in[L]} m\III{\ell}_{1/2}} = O(1)$.
  \item In this case, we assume $S^c \neq \varnothing$ and $\sum_{\ell \in S} m\III{\ell}_{1/2} \geq J_\rho$. Since $J_\rho \to\infty$, $\sum_{\ell\in S} m\III{\ell}_{1/2}\to\infty$. Now the lower bound for   \eqref{eq:target_of_lower_bound_z_star} becomes 
  \begin{align*}
  & c\cdot \exp\bigg\{-\bigl(1+o(1)\bigr)|S^c| J_\rho + \bigg(1+o(1) + \frac{c}{\sqrt{\sum_{\ell\in S}m\III{\ell}_{1/2}}}\bigg)\cdot \sum_{\ell \in S}m\III{\ell}_{1/2}\bigg] \bigg\}\\
  & =c\cdot \exp\bigg\{-\bigl(1+o(1)\bigr) \cdot \bigg(|S^c| J_\rho + \sum_{\ell \in S}m\III{\ell}_{1/2}\bigg) \bigg\}.
  \end{align*}
  \item If $S^c \neq \varnothing$ and $\sum_{\ell \in S}m\III{\ell}_{1/2} < J_\rho$, then the lower bound for   \eqref{eq:target_of_lower_bound_z_star} becomes
  \begin{align*}
   & C\cdot \exp\bigg\{-\bigg(1+o(1) + \frac{c\sqrt{\sum_{\ell\in S}m\III{\ell}_{1/2}}}{J_\rho}\bigg)|S^c| J_\rho + \bigl(1+o(1)\bigr) \cdot \sum_{\ell \in S}m\III{\ell}_{1/2}\bigg] \bigg\}\\
  & = C\cdot \exp\bigg\{-\bigl(1+o(1)\bigr) \cdot \bigg(|S^c| J_\rho + \sum_{\ell \in S}m\III{\ell}_{1/2} \bigg) \bigg\}.
  \end{align*}
\end{enumerate}
Hence, for any $|S^c|$ even, we have the following lower bound for   \eqref{eq:target_of_lower_bound_z_star}:
\begin{align*}
   & C\cdot \exp\bigg\{-\bigl(1+o(1)\bigr) \cdot \bigg(|S^c| J_\rho + \sum_{\ell \in S}mI_\ell \bigg) \bigg\} = C\cdot\exp\bigg\{-\bigl(1+o(1)\bigr) \cdot \bigg(|S^c| J_\rho + \psi_S^\star(0)\bigg) \bigg\},
\end{align*}
where we remark that the $o(1)$ term does not depend on $S$, thus proving the first part of \eqref{eq:opt_test_error_z_star}.

\subsubsection{The case of odd \texorpdfstring{$|S^c|$}{Sc}}\label{subsubappend:prf_opt_test_error_z_star_odd_Sc}
Now consider the case where $|S^c|$ is odd. Similar to the previous case, we start by writing
\begingroup
\allowdisplaybreaks
\begin{align*}
  &\bbE_{\ZZZ{S^c}}\bigg[\bbP\bigg(\sum_{\ell \in S} \sum_{i = 1}^{m} \XXX{\ell}_i \cdot \log \bigg(\frac{q_\ell(1-p_\ell)}{p_\ell(1-q_\ell)}\bigg) + \YYY{\ell}_i \cdot \log \bigg(\frac{p_\ell(1-q_\ell)}{q_\ell(1-p_\ell)}\bigg)  \geq \log\bigg(\frac{1-\rho}{\rho}\bigg)\cdot \sum_{\ell\in S^c}\ZZZ{\ell} \ \bigg| \ \ZZZ{S^c} \bigg)\bigg]\\
  & \geq \sum_{x \in \{-|S^c|+ 2j: 0\leq j\leq |S^c|\}}\binom{|S^c|}{\frac{|S^c| + x}{2}} (1-\rho)^{\frac{|S^c| + x}{2}} \rho^{\frac{|S^c| - x}{2}}\\
  & \qquad \times \bbP\bigg(\sum_{\ell \in S} \sum_{i = 1}^{m} \XXX{\ell}_i \cdot \log \bigg(\frac{q_\ell(1-p_\ell)}{p_\ell(1-q_\ell)}\bigg) + \YYY{\ell}_i \cdot \log \bigg(\frac{p_\ell(1-q_\ell)}{q_\ell(1-p_\ell)}\bigg)  \geq \log\bigg(\frac{1-\rho}{\rho}\bigg)\cdot x \bigg)\\
  & \geq \binom{|S^c|}{\frac{|S^c + 1|}{2}} \bigg(\rho(1-\rho)\bigg)^{\frac{|S^c|}{2}} \cdot \bigg(\frac{1-\rho}{\rho}\bigg)^{-1/2} \\
  & \qquad \times \bbP\bigg(\sum_{\ell \in S} \sum_{i = 1}^{m} \XXX{\ell}_i \cdot \log \bigg(\frac{q_\ell(1-p_\ell)}{p_\ell(1-q_\ell)}\bigg) + \YYY{\ell}_i \cdot \log \bigg(\frac{p_\ell(1-q_\ell)}{q_\ell(1-p_\ell)}\bigg)  \geq \log\bigg(\frac{\rho}{1-\rho}\bigg) \bigg)\\
  & \geq \exp\bigg\{- \bigl(1+o(1)\bigr) \bigg(|S^c| + 1\bigg) J_\rho \bigg\}\\
  & \qquad \times \bbP\bigg(\sum_{\ell \in S} \sum_{i = 1}^{m} \XXX{\ell}_i \cdot \log \bigg(\frac{q_\ell(1-p_\ell)}{p_\ell(1-q_\ell)}\bigg) + \YYY{\ell}_i \cdot \log \bigg(\frac{p_\ell(1-q_\ell)}{q_\ell(1-p_\ell)}\bigg)  \geq \log\bigg(\frac{\rho}{1-\rho}\bigg) \bigg).
\end{align*}
\endgroup
where the second inequality is by choosing $x = -1$ (note that we can do so because $|S^c|$ is odd). 
Hence, we focus on lower bounding the following probability:
\begin{equation}
\label{eq:target_of_lower_bound_z_star_odd}
  \bbP\bigg(\sum_{\ell \in S} \sum_{i = 1}^{m} \XXX{\ell}_i \cdot \log \bigg(\frac{q_\ell(1-p_\ell)}{p_\ell(1-q_\ell)}\bigg) + \YYY{\ell}_i \cdot \log \bigg(\frac{p_\ell(1-q_\ell)}{q_\ell(1-p_\ell)}\bigg)  \geq \log\bigg(\frac{\rho}{1-\rho}\bigg) \bigg). 
\end{equation}  
With some algebra, we have
\begin{align*}
  \sum_{\ell\in S} \sum_{i=1}^{m}\bbE[\tXXX{\ell}_0 + \tYYY{\ell}_0] = -\sum_{\ell\in S} m(p_\ell - q_\ell) \cdot \log\frac{p_\ell(1-q_\ell)}{q_\ell(1-p_\ell)} = - \Theta\bigg(\sum_{\ell\in S}m \III{\ell}_{1/2}\bigg),
\end{align*}
where the asymptotic equivalence is by \eqref{eq: exp_and_var_of_tilted_bernoulli_intermediate_result} and Lemma \ref{lemma:asymp_equiv_of_I_t}. We divide our discussion into three cases.

~\\
{\bf Case A.}
In this case, we assume
  $$
    \sum_{\ell\in S} m (p_\ell - q_\ell) \cdot \log\frac{p_\ell(1-q_\ell)}{q_\ell(1-p_\ell)} \leq \sqrt{J_\rho}. 
  $$
  From \eqref{eq: exp_and_var_of_tilted_bernoulli_intermediate_result} and Lemma \ref{lemma:asymp_equiv_of_I_t}, the left-hand side above is $\Theta(\sum_{\ell \in S} m \III{\ell}_{1/2})$. This means that in this case, we have
  $$
    \sum_{\ell \in S} m\III{\ell}_{1/2} \lesssim \sqrt{J_\rho}. 
  $$
  We now have
  \begingroup
  \allowdisplaybreaks
  \begin{align*}
    & \bbP\bigg(\sum_{\ell \in S} \sum_{i = 1}^{m} \XXX{\ell}_i \cdot \log \bigg(\frac{q_\ell(1-p_\ell)}{p_\ell(1-q_\ell)}\bigg) + \YYY{\ell}_i \cdot \log \bigg(\frac{p_\ell(1-q_\ell)}{q_\ell(1-p_\ell)}\bigg)  \geq \log\bigg(\frac{\rho}{1-\rho}\bigg) \bigg)\\
    &  = 1 - \bbP\bigg(\sum_{\ell \in S} \sum_{i = 1}^{m} \XXX{\ell}_i \cdot \log \bigg(\frac{q_\ell(1-p_\ell)}{p_\ell(1-q_\ell)}\bigg) + \YYY{\ell}_i \cdot \log \bigg(\frac{p_\ell(1-q_\ell)}{q_\ell(1-p_\ell)}\bigg)  <\log\bigg(\frac{\rho}{1-\rho}\bigg) \bigg)\\
    &  = 1 - \bbP\bigg(-\sum_{\ell \in S} \sum_{i = 1}^{m}\bigg(\tXXX{\ell}_0 + \tYYY{\ell}_0\bigg) > \log\bigg(\frac{1-\rho}{\rho}\bigg) \bigg)\\
    & \geq 1- \exp\bigg\{- \frac{1}{2}\log \bigg(\frac{1-\rho}{\rho}\bigg)\bigg\} \cdot \prod_{\ell\in S} \prod_{i=1}^{m} \bbE \exp\bigg\{ -\frac{\tXXX{\ell}_0 + \tYYY{\ell}_0}{2}\bigg\} \\
    & = 1 - e^{-\bigl(1+o(1)\bigr)J_\rho} \cdot \prod_{\ell\in S} \prod_{i =  1}^{m} \frac{\bigl(p_\ell(1-q_\ell)\bigr)^{3/2}}{\bigl(q_\ell(1-p_\ell)\bigr)^{1/2}} +  \frac{\bigl(q_\ell(1-p_\ell)\bigr)^{3/2}}{\bigl(p_\ell(1-q_\ell)\bigr)^{1/2}} + 1 - p_\ell(1-q_\ell) - q_\ell(1-p_\ell),
  \end{align*}
  \endgroup
  where the fourth line is by Markov's inequality.
  Note that
  \begin{align*}
  & \frac{\bigl(p_\ell(1-q_\ell)\bigr)^{3/2}}{\bigl(q_\ell(1-p_\ell)\bigr)^{1/2}} +  \frac{\bigl(q_\ell(1-p_\ell)\bigr)^{3/2}}{\bigl(p_\ell(1-q_\ell)\bigr)^{1/2}} + 1 - p_\ell(1-q_\ell) - q_\ell(1-p_\ell) \\
  & = 1 - p_\ell - q_\ell + 2p_\ell q_\ell +  \frac{\bigl(p_\ell(1-q_\ell)\bigr)^{3/2}}{\bigl(q_\ell(1-p_\ell)\bigr)^{1/2}} +  \frac{\bigl(q_\ell(1-p_\ell)\bigr)^{3/2}}{\bigl(p_\ell(1-q_\ell)\bigr)^{1/2}}\\
  & = 1 - \bigl(\sqrt{p_\ell} -\sqrt{q_\ell}\bigr)^2 - 2 \sqrt{p_\ell q_\ell} + 2p_\ell q_\ell +  \frac{\bigl(p_\ell(1-q_\ell)\bigr)^{3/2}}{\bigl(q_\ell(1-p_\ell)\bigr)^{1/2}} +  \frac{\bigl(q_\ell(1-p_\ell)\bigr)^{3/2}}{\bigl(p_\ell(1-q_\ell)\bigr)^{1/2}}.
  \end{align*}
  Now, 
  \begingroup
  \allowdisplaybreaks
  \begin{align*}
    & \log\bigg(\frac{\bigl(p_\ell(1-q_\ell)\bigr)^{3/2}}{\bigl(q_\ell(1-p_\ell)\bigr)^{1/2}} +  \frac{\bigl(q_\ell(1-p_\ell)\bigr)^{3/2}}{\bigl(p_\ell(1-q_\ell)\bigr)^{1/2}} + 1 - p_\ell(1-q_\ell) - q_\ell(1-p_\ell)\bigg) \\
    & \leq  - \bigl(\sqrt{p_\ell} -\sqrt{q_\ell}\bigr)^2 - 2 \sqrt{p_\ell q_\ell} + 2p_\ell q_\ell +  \frac{\bigl(p_\ell(1-q_\ell)\bigr)^{3/2}}{\bigl(q_\ell(1-p_\ell)\bigr)^{1/2}} +  \frac{\bigl(q_\ell(1-p_\ell)\bigr)^{3/2}}{\bigl(p_\ell(1-q_\ell)\bigr)^{1/2}}\\
    &  = \frac{1}{\sqrt{p_\ell q_\ell}} \cdot
    \bigg(
      -2p_\ell q_\ell + p_\ell^2 \cdot \frac{(1-q_\ell)^{3/2}}{(1-p_\ell)^{1/2}} + q_\ell^2 \cdot \frac{(1-p_\ell)^{3/2}}{(1-q_\ell)^{1/2}} + 2 (p_\ell q_\ell)^{3/2} 
    \bigg)\\
    & \leq \frac{1}{\sqrt{p_\ell q_\ell}} \cdot
    \bigg(
      (p_\ell-q_\ell)^2 + \bigg(\frac{(1-q_\ell)^{3/2}}{(1-p_\ell)^{1/2}} - 1\bigg) p_\ell^2+  2 (p_\ell q_\ell)^{3/2} 
    \bigg)\\
    & \lesssim \frac{1}{p_\ell}  \cdot
    \bigg(
      (p_\ell-q_\ell)^{2} + \calO(p_\ell^2)
    \bigg).
  \end{align*}
  \endgroup
  Since $p_\ell^2 / (p_\ell-q_\ell)^2 = \frac{(p_\ell/q_\ell)^2}{[(p_\ell/q_\ell) - 1]^2} \leq C_2^2/(C_1-1)^2 \lesssim 1$, the right-hand side above is $\calO\left(\frac{(p_\ell-q_\ell)^2}{p_\ell}\right) = \III{\ell}_{1/2}$,
  where we have invoked Lemma \eqref{lemma:asymp_equiv_of_I_t}.
  Hence, for some constant $c, c'>0$ we have
  \begingroup
  \allowdisplaybreaks
  \begin{align*}
    & \bbP\bigg(\sum_{\ell \in S} \sum_{i = 1}^{m} \XXX{\ell}_i \cdot \log \bigg(\frac{q_\ell(1-p_\ell)}{p_\ell(1-q_\ell)}\bigg) + \YYY{\ell}_i \cdot \log \bigg(\frac{p_\ell(1-q_\ell)}{q_\ell(1-p_\ell)}\bigg)  \geq \log\bigg(\frac{\rho}{1-\rho}\bigg) \bigg)\\
    & = 1 - e^{-\bigl(1+o(1)\bigr)J_\rho} \cdot \prod_{\ell\in S} \prod_{i =  1}^{m} e^{c\III{\ell}_{1/2}}\\
    & = 1 - \exp\bigg\{- \bigl( 1 + o(1)\bigr)\cdot \bigg(J_\rho - c\sum_{\ell\in S}m \III{\ell}_{1/2}\bigg)\bigg\} \\
    & \geq 1 - \exp\bigg\{- \bigl( 1 + o(1)\bigr)\cdot \bigg(J_\rho - c' \sqrt{J_\rho}\bigg)\bigg\} \\
    & = 1 - e^{-\bigl(1+o(1)\bigr)J_\rho}\\
    & \geq 1/2,
  \end{align*}
  \endgroup
  where the third line is by our assumption that $\sum_{\ell\in S}m \III{\ell}_{1/2} \lesssim \sqrt{J_\rho}$ and the last line is by $J_\rho \to\infty$.
  This means that   \eqref{eq:target_of_lower_bound_z_star} can be lower bounded by
  \begin{align*}
    \frac{1}{2}\cdot \exp\bigg\{- \bigl(1+o(1)\bigr) \bigg(|S^c| + 1\bigg) J_\rho \bigg\} 
    & = \exp\bigg\{- \bigl(1+o(1)\bigr) \bigg(|S^c| + 1\bigg) J_\rho \bigg\}\\
    & \geq \exp\bigg\{- \bigl(1+o(1)\bigr) \bigg(\bigl(|S^c| + 1\bigr) J_\rho + \psi_S^\star(-2J_\rho)\bigg)  \bigg\},
  \end{align*}
  where the last inequality is by $\psi_S^\star(-2J_\rho) \geq 0$, as shown in   \eqref{eq:ineq_for_cvx_conjugate}.

~\\
{\bf Case B.}
In this case, we assume
  $$
    \sqrt{J_\rho} < \sum_{\ell\in S} m (p_\ell - q_\ell) \cdot \log\frac{p_\ell(1-q_\ell)}{q_\ell(1-p_\ell)} \leq \log \frac{1-\rho}{\rho} = \bigl(1+o(1)\bigr) 2J_\rho. 
  $$
  From \eqref{eq: exp_and_var_of_tilted_bernoulli_intermediate_result} and Lemma \ref{lemma:asymp_equiv_of_I_t}, we have $ \sum_{\ell\in S} m (p_\ell - q_\ell) \cdot \log\frac{p_\ell(1-q_\ell)}{q_\ell(1-p_\ell)}  \asymp \sum_{\ell \in S} m \III{\ell}_{1/2}$, and thus
  $$
    \sqrt{J_\rho}\lesssim \sum_{\ell\in S} m\III{\ell}_{1/2}  \to \infty,
  $$
  and
  $$
    \sum_{\ell\in S} m\III{\ell}_{1/2} \asymp \Var \bigg(\sum_{\ell \in S}\sum_{i = 1}^{m} \tXXX{\ell}_0 + \tYYY{\ell}_0\bigg)  \to \infty.
  $$
  We now have
  \begin{align*}
    & \bbP\bigg(\sum_{\ell \in S} \sum_{i = 1}^{m} \XXX{\ell}_i \cdot \log \bigg(\frac{q_\ell(1-p_\ell)}{p_\ell(1-q_\ell)}\bigg) + \YYY{\ell}_i \cdot \log \bigg(\frac{p_\ell(1-q_\ell)}{q_\ell(1-p_\ell)}\bigg)  \geq \log\bigg(\frac{\rho}{1-\rho}\bigg) \bigg)\\
    & = \bbP\bigg(\sum_{\ell\in S}\sum_{i=1}^{m} \tXXX{\ell}_0 + \tYYY{\ell}_0 - \bbE[\tXXX{\ell} + \tYYY{\ell}]  \geq \log \bigg(\frac{\rho}{1-\rho}\bigg) + \sum_{\ell\in S} m (p_\ell - q_\ell) \cdot \log\frac{p_\ell(1-q_\ell)}{q_\ell(1-p_\ell)}\bigg) \\
    & \geq \bbP\bigg(\sum_{\ell\in S}\sum_{i=1}^{m} \tXXX{\ell}_0 + \tYYY{\ell}_0 - \bbE[\tXXX{\ell} + \tYYY{\ell}]  \geq 0\bigg). 
  \end{align*}
  Note that 
  \begin{align*}
    \bigg(\tXXX{\ell}_0 + \tYYY{\ell}_0 - \bbE[\tXXX{\ell} + \tYYY{\ell}] \bigg)^2 \leq  4 \bigg(\log \frac{p_\ell(1-q_\ell)}{q_\ell (1-p_\ell)}\bigg)^2 = \calO(1).
  \end{align*}
  Hence, by Lindeberg-Feller CLT, we have
  $$
    \frac{\sum_{\ell\in S}\sum_{i=1}^{m} \tXXX{\ell}_0 + \tYYY{\ell}_0 - \bbE[\tXXX{\ell} + \tYYY{\ell}]}{\Var \bigg(\sum_{\ell\in S}\sum_{i=1}^{m} \tXXX{\ell}_0 + \tYYY{\ell}_0 - \bbE[\tXXX{\ell} + \tYYY{\ell}]\bigg)} \Rightarrow \calN(0, 1).
  $$
  This gives
  \begin{align*}
    & \bbP\bigg(\sum_{\ell \in S} \sum_{i = 1}^{m} \XXX{\ell}_i \cdot \log \bigg(\frac{q_\ell(1-p_\ell)}{p_\ell(1-q_\ell)}\bigg) + \YYY{\ell}_i \cdot \log \bigg(\frac{p_\ell(1-q_\ell)}{q_\ell(1-p_\ell)}\bigg)  \geq \log\bigg(\frac{\rho}{1-\rho}\bigg) \bigg) \geq C 
  \end{align*}
  for large enough $n$ where $C > 0$ is some absolute constant. Hence, in this case, we have the following lower bound for   \eqref{eq:target_of_lower_bound_z_star}:
  \begin{align*}
    C\cdot \exp\bigg\{- \bigl(1+o(1)\bigr) \bigg(|S^c| + 1\bigg) J_\rho \bigg\}
    & \geq C\cdot\exp\bigg\{- \bigl(1+o(1)\bigr) \bigg(\bigl(|S^c| + 1\bigr) J_\rho + \psi_S^\star(-2J_\rho)\bigg)  \bigg\}.
  \end{align*}

~\\
{\bf Case C.}
In this case, we assume
  $$
    \sum_{\ell\in S} m (p_\ell - q_\ell) \cdot \log\frac{p_\ell(1-q_\ell)}{q_\ell(1-p_\ell)} > \log \frac{1-\rho}{\rho} = \bigl(1+ o(1)\bigr)2J_\rho. 
  $$
  We let
  $$
    \tilde \psi_S(t) = -m\sum_{\ell\in S} \III{\ell}_t = \frac{\psi_S(t)}{n/2} \cdot m = (1+o(1)) \psi_S(t).
  $$
  That is, we replace $n/2$ in the definition of $\psi_S(t)$ with $m = (1+o(1))n/2$. Meanwhile, we let $t^\star$ be the maximizer of
  $$
    \sup_{t\in[0, 1]} -t \log\frac{1-\rho}{\rho} - \tilde\psi_S(t).
  $$
  By construction, we have
  $$
    \sum_{\ell\in S} \sum_{i=1}^{m} \bbE[\tXXX{\ell}_{t} + \tYYY{\ell}_{t}] = \frac{d}{dt} \tilde\psi_S(t).
  $$
  and $\tilde\psi_S(t)$ is a convex function in $t$ (see Lemma \ref{eq:convexity_of_cum_gen_func} for details).  
  By convexity, the optimal $t^\star$ is such that
  \begin{align*}
    -\log \frac{1-\rho}{\rho} & = \sum_{\ell\in S} \sum_{i=1}^{m} \bbE[\tXXX{\ell}_{t^\star} + \tYYY{\ell}_{t^\star}],
  \end{align*}
  and the right-hand side is an increasing function in $t$.
  This gives
  $$
    - \sum_{\ell\in S} m (p_\ell - q_\ell) \cdot \log\frac{p_\ell(1-q_\ell)}{q_\ell(1-p_\ell)}  \leq \sum_{\ell\in S} \sum_{i=1}^{m} \bbE[\tXXX{\ell}_{t} + \tYYY{\ell}_{t}] \leq   0
  $$
  for $t \in [0, 1/2]$, where the left-hand side above is attained at $t = 0$ and right-hand side is attained at $t = 1/2$. Since $\sum_{\ell\in S}\sum_{i\in[m]}\bbE[\tXXX{\ell}_t + \tYYY{\ell}_t]$ is continuous and increasing in $t$ for $t\in[0, 1/2]$, it attains every value between the two sides in the above display. On the other hand, the optimal $t^\star$ is such that the expectation is exactly equal to $-\log\frac{1-\rho}{\rho}$, which is between the two sides in the above display under our current assumption. As a result, we have
  $$
    t^\star \in [0, 1/2].
  $$
  Now note that
  $$
    \sup_{t\in[0, 1]} -t \log\frac{1-\rho}{\rho} - \tilde\psi_S(t) = \sup_{t\in[0, 1]} (1+o(1))\bigg(-2tJ_\rho - \psi_S(t)\bigg) = (1+o(1))\psi_S^\star(-2J_\rho).
  $$

  Hence, with a standard tilting argument, we get
  \begin{align*}
    \mu\bigg(\bigg[  \log\frac{\rho}{1-\rho},  \infty\bigg)\bigg)& \geq \exp\bigg\{- (1+o(1))\psi_S^\star(-2J_\rho) - t^\star\xi\bigg\} \\
    & \qquad \times \bbP_{W\sim \tilde\mu_{t^\star}}\bigg(  \log\frac{\rho}{1-\rho}\leq W\leq \xi + \log\frac{\rho}{1-\rho}\bigg).
  \end{align*}
  By Lemma \ref{lemma: exp_and_var_of_tilted_bernoulli}, $\Var(W)\asymp \sum_{\ell\in S}m\III{\ell}_{1/2}\to \infty$, and with a similar argument as the previous case, an application of Lindeberg-Feller CLT gives 
  $$
    \frac{W - \log\frac{\rho}{1-\rho}}{\sqrt{\Var(W)}} \Rightarrow \calN(0 , 1).
  $$
  Hence, choosing $\xi = \sqrt{\Var(W)} \asymp \sqrt{\sum_{\ell\in S}m\III{\ell}_{1/2}}$, we get 
  \begin{align*}
    \mu\bigg(\bigg[  \log\frac{\rho}{1-\rho},  \infty\bigg)\bigg)& \geq C \cdot \exp\bigg\{- (1+o(1))\psi_S^\star(-2J_\rho) - ct^\star\sqrt{\sum_{\ell\in S}m\III{\ell}_{1/2}}\bigg\},
  \end{align*}
  and we arrive at the following lower bound for   \eqref{eq:target_of_lower_bound_z_star}:
  \begin{align*}
    C\cdot \exp\bigg\{- \bigl(1+o(1)\bigr) \bigg(\bigl(|S^c| + 1\bigr) J_\rho + \psi_S^\star(-2J_\rho) + ct^\star\sqrt{\sum_{\ell\in S}m\III{\ell}_{1/2}}\bigg)  \bigg\}.
  \end{align*}
  By the inequality in \eqref{eq:ineq_for_cvx_conjugate}, we have
  $$
    \bigl(|S^c| + 1\bigr) J_\rho + \psi_S^\star(-2J_\rho) \geq |S^c| J_\rho + \sum_{\ell \in S} m\III{\ell}_{1/2} \gg \sqrt{\sum_{\ell \in S}m\III{\ell}_{1/2}}.
  $$
  This gives
  $$
    \bigl(|S^c| + 1\bigr) J_\rho + \psi_S^\star(-2J_\rho) + ct^\star\sqrt{\sum_{\ell\in S}mI_\ell} = \bigl(1 + o(1)\bigr) \cdot \bigg(\bigl(|S^c| + 1\bigr) J_\rho + \psi_S^\star(-2J_\rho)\bigg).
  $$
  As a result, we get the following lower bound for   \eqref{eq:target_of_lower_bound_z_star}:
  \begin{align*}
    C\cdot \exp\bigg\{- \bigl(1+o(1)\bigr) \bigg(\bigl(|S^c| + 1\bigr) J_\rho + \psi_S^\star(-2J_\rho)\bigg)  \bigg\}.
  \end{align*}

~\\
{\bf Summary.}
Combining the above three cases and noting that the $o(1)$ term does not depend on $S$ proves the second part of \eqref{eq:opt_test_error_z_star}. Thus the proof of Lemma \ref{lemma:opt_test_error_z_star} is concluded.



\subsection{Proof of Lemma \ref{lemma:testing_problem_z_l_part_I}}\label{subappend:prf_testing_problem_z_l_part_I}

Fix $\btz^\star \in \calP_n$ with $n^\star_+(\btz^\star) = \lfloor n/2 \rfloor$ and $n_-^\star(\btz^\star) = n - \lfloor n/2 \rfloor$. Consider the following parameter space, which consists of a single clustering vector:
\begin{equation}
  \label{eq:sub_param_space_z_star_known}
  \calP_n^{1}:= \bigg\{ \textnormal{\sname}(\bz^\star, \rho, \{p_\ell\}_1^L, \{q_\ell\}_1^L): \bz^\star = \btz^\star, p_\ell > q_\ell \ \forall \ell\in[L] \bigg\}.
\end{equation}
We then have 
\begin{align*}
\inf_{\hzzz{\ell}} \sup_{\bz^\star \in \calP_n} \bbE\calL(\hzzz{\ell}, \zzz{\ell}) & \geq \inf_{\hzzz{\ell}} \sup_{\bz^\star \in \calP_n^1} \bbE\calL(\hzzz{\ell}, \zzz{\ell}) \geq \inf_{\hzzz{\ell}} \sup_{\bz^\star \in \calP_n^1} \bbE[\calL(\hzzz{\ell}, \zzz{\ell})\cdot \indc{E}],
\end{align*}
where the event $E$ is defined as
$$
  E = \bigg\{\frac{1}{n}\sum_{i\in[n]} \indc{\btz^\star_i \neq \zzz{\ell}_i} \leq \frac{1}{2} ~\forall \ell \in[L]\bigg\}.
$$
Since $\btz^\star$ is known to us, it is a legitimate estimator of $\zzz{\ell}$. Hence, for the optimal estimator $\hzzz{\ell}$ of $\zzz{\ell}$, on the event $E$, we necessarily have
$$
  \frac{1}{n}\sum_{i\in[n]} \indc{\hzzz{\ell}_i \neq \zzz{\ell}_i} \leq \frac{1}{n}\sum_{i\in[n]} \indc{\btz^\star_i \neq \zzz{\ell}_i} \leq \frac{1}{2}.
$$
This gives
\begin{align*}
  \inf_{\hzzz{\ell}} \sup_{\bz^\star \in \calP_n} \bbE\calL(\hzzz{\ell}, \zzz{\ell}) & \geq \inf_{\hzzz{\ell}} \bbE[ \frac{1}{n}\sum_{i\in[n]} \indc{\hzzz{\ell}_i \neq \zzz{\ell}_i}\cdot \indc{E}] \geq \frac{1}{n}\sum_{i\in[n]} \inf_{\hzzz{\ell}_i} \bbP(\hzzz{\ell}_i \neq \zzz{\ell}_i , E).
\end{align*}  
For some $\ep = o(1)$ whose value will be determined later, we define
$$
  E_i := \bigg\{\frac{1}{n}\sum_{j\neq i} \indc{\btz^\star_j \neq \zzz{\ell}_j}\leq \rho + \ep~\forall \ell \in [L]\bigg\} \subseteq E,
$$
where the inclusion is by
$$
  \frac{1}{n} \sum_{j\in[n]}\indc{\btz^\star_j \neq \zzz{\ell}_j} \leq \frac{1}{n} \sum_{j\neq i}\indc{\btz^\star_j \neq \zzz{\ell}_j}  + \frac{1}{n} \leq \rho + \ep + \frac{1}{n} \leq 1/2 - c_1 + o(1) \leq \frac{1}{2}
$$
for large $n$. 
By Hoeffding's inequality, we have
\begin{align*}
  \bbP\bigg(\frac{1}{n}\sum_{j\neq i} \indc{\btz^\star_j \neq  \zzz{\ell}_j}> \rho + \ep\bigg)
  \leq \bbP \bigg(\frac{1}{n-1}\sum_{j\neq i} [\Bern(\rho) - \rho] > \ep \bigg)
  \leq e^{-2(n-1)\ep^2},
\end{align*}
and thus
$$
  \bbP(E_i^c) \leq  e^{-\calO(n\ep^2) + \log L} = e^{-\calO(n\ep^2) + o(n^{c_2})}.
$$
By choosing $\ep = n^{-(1-c_2)/2} = o(1)$, we get $\bbP(E_i^c) = o(1)$.
In addition, let us define
$$
  F_i := \bigg\{ \bigg|\#\{j\neq i: \zzz{\ell}_j = \bz^\star_i\} - \frac{n}{2}\bigg| \lor\bigg|\#\{j\neq i: \zzz{\ell}_j = -\bz^\star_i\} - \frac{n}{2}\bigg| \leq \frac{n}{2}\cdot \ep ~\forall \ell\in[L]\bigg\}.
$$
By the same arguments as those that led to \eqref{eq:concentration_of_number_of_pos_and_neg_nodes}, we can choose $\ep = o(1)$ so that $\bbP(F_i^c) = 1-o(1)$. Hence, invoking an union bound, we get $\bbP(E_i\cap F_i) = 1 -o(1)$.
Now, we have 
\begin{align*}
  \inf_{\hzzz{\ell}} \sup_{\bz^\star \in \calP_n} \bbE\calL(\hzzz{\ell}, \zzz{\ell}) & \geq \frac{1}{n}\sum_{i\in[n]} \inf_{\hzzz{\ell}_i} \bbP(\hzzz{\ell}_i \neq \zzz{\ell}_i, E_i\cap F_i).
\end{align*}
Fix any $i\in[n]$. Without loss of generality we assume $\btz^\star_i = +1$. Since the event $E_i\cap F_i$ only depends on $\zzz{\ell}_{-i} := \{\zzz{\ell}_j: j\neq i\}$, we have
\begin{align*}
  &\inf_{\hzzz{\ell}_i} \bbP(\hzzz{\ell}_i \neq \zzz{\ell}_i, E_i\cap F_i) \\
  & \geq \inf_{\hzzz{\ell}_i} \sum_{\xi\in\{\pm 1\}^{n-1}} \bbP(\zzz{\ell}_j = \xi_j \bz^\star_j ~\forall j\neq i) \cdot \Indc\bigg\{\frac{1}{n}\sum_{j\neq i} \indc{\xi_j = -1} \leq \rho + \ep, |\mmm{\ell}_+ - n/2| \lor |\mmm{\ell}_- - n/2|  \leq \frac{n}{2}\cdot \ep\bigg\} \\
  & \qquad \times \bigg((1-\rho)\cdot  \bbP(\hzzz{\ell}_i =-1 ~|~ \zzz{\ell}_i = +1, \zzz{\ell}_j = \xi_j \bz^\star_j ~\forall j\neq i) 
  + \rho \cdot \bbP(\hzzz{\ell}_i = +1 ~|~ \zzz{\ell}_i = -1, \zzz{\ell}_j = \xi_j \bz^\star_j ~\forall j\neq i) \bigg) \\
  & \geq \sum_{\xi\in\{\pm 1\}^{n-1}} \bbP(\zzz{\ell}_j = \xi_j \bz^\star_j ~\forall j\neq i) \cdot \Indc\bigg\{\frac{1}{n}\sum_{j\neq i} \indc{\xi_j = -1} \leq \rho + \ep, |\mmm{\ell}_+ - n/2| \lor |\mmm{\ell}_- - n/2|  \leq \frac{n}{2}\cdot \ep\bigg\} \\
  & \qquad \times \inf_{\hzzz{\ell}_i}  \bigg((1-\rho)\cdot  \bbP(\hzzz{\ell}_i =-1 ~|~ \zzz{\ell}_i = +1, \zzz{\ell}_j = \xi_j \bz^\star_j ~\forall j\neq i) + \rho \cdot \bbP(\hzzz{\ell}_i = +1 ~|~ \zzz{\ell}_i = -1, \zzz{\ell}_j = \xi_j \bz^\star_j ~\forall j\neq i) \bigg),
\end{align*}
where
$$
  \mmm{\ell}_+ = \#\{j\neq i: \xi_j\bz^\star_j = + 1\}, \ \ \ \mmm{\ell}_- = \#\{j\neq i: \xi_j\bz^\star_j = -1\}.
$$
Since we know $\bz^\star = \btz^\star$, by independence, we can without loss of generality restrict ourselves to $\hzzz{\ell}_i$'s that are only functions of $\AAA{\ell}$ alone. Thus, the ``$\inf$'' term in the right-hand side above can be regarded as the $(1-\rho)\times \textnormal{type-I error} + \rho \times \textnormal{type-II error}$ of the following binary hypothesis testing problem:
$$
  H_0: \zzz{\ell}_i = +1 \ \ \ \textnormal{v.s.} \ \ \ H_1: \zzz{\ell}_i = -1,
$$
where the data is a single adjacency matrix $\AAA{\ell}$ sampled from a vanilla two-block SBM.

We now focus on lower bounding this ``$\inf$'' term. By Lemma \ref{lemma: optimal_test}, the optimal test is given by the likelihood ratio test with cutoff being $\rho/(1-\rho)$. Let $L_0$ and $L_1$ be the likelihood function under $H_0$ and $H_1$, respectively. With some algebra, we have
\begin{align*}
  \frac{L_0}{L_1} & = \prod_{\substack{j\neq i\\ \xi_j\bz^\star_j = 1}} p_\ell^{\AAA{\ell}_{ij}} (1-p_\ell)^{1-\AAA{\ell}_{ij}}\prod_{\substack{j\neq i\\ \xi_j\bz^\star_j = -1}} q_\ell^{\AAA{\ell}_{ij}} (1-q_\ell)^{1-\AAA{\ell}_{ij}} \\
  & \qquad \times \bigg(\prod_{\substack{j\neq i\\ \xi_j\bz^\star_j = -1}} p_\ell^{\AAA{\ell}_{ij}} (1-p_\ell)^{1-\AAA{\ell}_{ij}}\prod_{\substack{j\neq i\\ \xi_j\bz^\star_j= 1}} q_\ell^{\AAA{\ell}_{ij}} (1-q_\ell)^{1-\AAA{\ell}_{ij}}\bigg)^{-1}.
\end{align*}
Thus, the type-I error of the optimal test is given by
\begin{align*}
  &  \bbP\bigg(\sum_{i=1}^{\mmm{\ell}_+} \XXX{\ell}_i \cdot \log \frac{q_\ell(1-p_\ell)}{p_\ell(1-q_\ell)} + \sum_{i = 1}^{\mmm{\ell}_-} \YYY{\ell}_i \cdot \log\frac{p_\ell(1-q_\ell)}{q_\ell(1-p_\ell)}\geq \log \frac{1-\rho}{\rho}\bigg),
\end{align*}
and the type-II error of the optimal test is given by
\begin{align*}
  &  \bbP\bigg(\sum_{i=1}^{\mmm{\ell}_+} \XXX{\ell}_i \cdot \log \frac{q_\ell(1-p_\ell)}{p_\ell(1-q_\ell)} + \sum_{i = 1}^{\mmm{\ell}_-} \YYY{\ell}_i \cdot \log\frac{p_\ell(1-q_\ell)}{q_\ell(1-p_\ell)}\geq -\log \frac{1-\rho}{\rho}\bigg),
\end{align*}
where $\XXX{\ell}_i, \YYY{\ell}_i$'s are defined as 
\begin{align*}
  & \XXX{\ell}_i \overset{\textnormal{i.i.d.}}{\sim} \textnormal{Bern}(p_\ell), \ \ \ \YYY{\ell}_i \overset{\textnormal{i.i.d.}}{\sim} \textnormal{Bern}(q_\ell).
\end{align*}
So overall, the weighted testing error is given by
\begin{align*}
  &  \bbP\bigg(\sum_{i=1}^{\mmm{\ell}_+} \XXX{\ell}_i \cdot \log \frac{q_\ell(1-p_\ell)}{p_\ell(1-q_\ell)} + \sum_{i = 1}^{\mmm{\ell}_-} \YYY{\ell}_i \cdot \log\frac{p_\ell(1-q_\ell)}{q_\ell(1-p_\ell)}\geq \ZZZ{\ell}\log \frac{1-\rho}{\rho}\bigg),
\end{align*}
where $\ZZZ{\ell}\sim 2\textnormal{Bern}(1-\rho)-1$, which is independent of $\XXX{\ell}_i$ and $\YYY{\ell}_i$'s.
By Lemma \ref{lemma: general_data_processing}, the above probability can be lower bounded by
\begin{align*}
  &  \bbP\bigg(\sum_{i=1}^{m} \XXX{\ell}_i \cdot \log \frac{q_\ell(1-p_\ell)}{p_\ell(1-q_\ell)} + \sum_{i = 1}^{m} \YYY{\ell}_i \cdot \log\frac{p_\ell(1-q_\ell)}{q_\ell(1-p_\ell)}\geq \ZZZ{\ell}\log \frac{1-\rho}{\rho}\bigg),
\end{align*}
for $m = \mmm{\ell}_+ \lor \mmm{\ell}_-$, which is $(1+o(1))n/2$ under $F_i$. 

In summary, we have
\begin{align*}
  & \inf_{\hzzz{\ell}} \sup_{\bz^\star \in \calP_n} \bbE\calL(\hzzz{\ell}, \zzz{\ell}) \\
  & \geq \sum_{\xi\in\{\pm 1\}^{n-1}} \bbP(\zzz{\ell}_j = \xi_j \bz^\star_j ~\forall j\neq i) \cdot \Indc\bigg\{\frac{1}{n}\sum_{j\neq i} \indc{\xi_j = -1} \leq \rho + \ep, |\mmm{\ell}_+ - n/2| \lor |\mmm{\ell}_- - n/2|  \leq \frac{n}{2}\cdot \ep \bigg\} \\
  & \qquad \times \bbP\bigg(\sum_{i=1}^{m} \XXX{\ell}_i \cdot \log \frac{q_\ell(1-p_\ell)}{p_\ell(1-q_\ell)} + \sum_{i = 1}^{m} \YYY{\ell}_i \cdot \log\frac{p_\ell(1-q_\ell)}{q_\ell(1-p_\ell)}\geq \ZZZ{\ell}\log \frac{1-\rho}{\rho}\bigg)\\
  & = \bbP(E_i\cap F_i) \cdot \bbP\bigg(\sum_{i=1}^{m} \XXX{\ell}_i \cdot \log \frac{q_\ell(1-p_\ell)}{p_\ell(1-q_\ell)} + \sum_{i = 1}^{m} \YYY{\ell}_i \cdot \log\frac{p_\ell(1-q_\ell)}{q_\ell(1-p_\ell)}\geq \ZZZ{\ell}\log \frac{1-\rho}{\rho}\bigg) \\
  & \gtrsim \bbP\bigg(\sum_{i=1}^{m} \XXX{\ell}_i \cdot \log \frac{q_\ell(1-p_\ell)}{p_\ell(1-q_\ell)} + \sum_{i = 1}^{m} \YYY{\ell}_i \cdot \log\frac{p_\ell(1-q_\ell)}{q_\ell(1-p_\ell)}\geq \ZZZ{\ell}\log \frac{1-\rho}{\rho}\bigg),
\end{align*}
where the last line is by $\bbP(E_i \cap F_i ) = 1-o(1)$. We finish the proof by noting that the right-hand side above is the $(1-\rho)\times \textnormal{type-I error} + \rho \times \textnormal{type-II error}$ of the testing problem \eqref{eq:fund_testing_problem_z_l_part_I} with $(1+\delta_n)n/2 = m$ where $\delta_n = \ep = o(1)$.

\subsection{Proof of Lemma \ref{lemma:testing_problem_z_l_part_II}}\label{subappend:prf_testing_problem_z_l_part_II}
The proof has a similar flavor to the proof of Lemma \ref{lemma:testing_problem_z_star}. Recall the sub-parameter space $\calP_n^0$ defined in \eqref{eq:sub_param_space_z_star}, which consists of clustering vectors that agree with $\btz^\star$ on $T$. We have
\begin{align*}
  \inf_{\hzzz{\ell}}  \sup_{\bz^\star \in \calP_n} \bbE\calL(\hzzz{\ell}, \zzz{\ell}) & \geq \inf_{\hzzz{\ell}}  \sup_{\bz^\star \in \calP_n^0} \bbE\calL(\hzzz{\ell}, \zzz{\ell}).
\end{align*}
Since $\btz^\star$ is known to us, it is a legitimate estimator of $\zzz{\ell}$. Thus, for the optimal estimator $\hzzz{\ell}$ of $\zzz{\ell}$, it necessarily satisfies
$$
  \frac{1}{n} \sum_{i\in[n]} \indc{\hzzz{\ell}_i \neq \zzz{\ell}_i} \leq \frac{1}{n} \sum_{i\in[n]} \indc{\btz^\star_i \neq \zzz{\ell}_i} \leq \frac{1}{n} \sum_{i\in T} \indc{\bz^\star_i \neq \zzz{\ell}_i} + \frac{|T^c|}{n} \leq \frac{1}{n} \sum_{i\in T} \indc{\bz^\star_i \neq \zzz{\ell}_i} + \frac{4}{n} + \delta_n.
$$
Define the event $E$ to be 
$$
  E:= \bigg\{\frac{1}{n}\sum_{i\in [n]} \indc{\bz^\star_i \neq \zzz{\ell}_i} \leq \rho + \ep ~ \forall \ell\in[L] \bigg\}
$$
for some $\ep$ satisfying $1/n \ll \ep \ll 1$ whose value will be specified later.
On the event $E$, for large $n$ we have
$$
  \frac{1}{n} \sum_{i\in[n]} \indc{\hzzz{\ell}_i \neq \zzz{\ell}_i} \leq \rho + \ep + \frac{4}{n} + \delta_n \leq 1/2 - c_1 + o(1)\leq  \frac{1}{2},
$$
and hence
$$
  \calL(\hzzz{\ell}, \zzz{\ell}) = \frac{1}{n} \sum_{i\in[n]} \bbP(\hzzz{\ell}_i \neq \zzz{\ell}_i).
$$
Let us define
$$
  E_i = \bigg\{ \frac{1}{n}\sum_{j\neq i} \indc{\bz^\star_i \neq \zzz{\ell}_i} \leq \rho + \ep - \frac{1}{n} ~\forall \ell\in[L]\bigg\} \subseteq E
$$
for large $n$. In addition, we let
$$
  F_i := \bigg\{ \bigg|\#\{j\neq i: \zzz{\ell}_j = \bz^\star_i\} - \frac{n}{2}\bigg| \lor\bigg|\#\{j\neq i: \zzz{\ell}_j = -\bz^\star_i\} - \frac{n}{2}\bigg| \leq \frac{n}{2}\cdot \ep ~\forall\ell\in[L] \bigg\}.
$$
By nearly identical arguments as those appeared in the proof of Lemma \ref{lemma:testing_problem_z_l_part_I}, we can choose $\ep = n^{-(1-c_2)/2}$ so that $\bbP(E_i\cap F_i) = 1-o(1)$.
We then proceed by
\begin{align*}
  \inf_{\hzzz{\ell}}  \sup_{\bz^\star \in \calP_n} \bbE\calL(\hzzz{\ell}, \zzz{\ell}) & \geq \inf_{\hzzz{\ell}}  \sup_{\bz^\star \in \calP_n^0} \bbE[\calL(\hzzz{\ell}, \zzz{\ell})\cdot \indc{E}]\\
  & = \inf_{\hzzz{\ell}}  \sup_{\bz^\star \in \calP_n^0} \frac{1}{n} \sum_{i\in[n]}\bbP(\hzzz{\ell}_i \neq \zzz{\ell}_i, E)\\
  & = \inf_{\hzzz{\ell}}  \sup_{\bz^\star \in \calP_n^0} \frac{1}{n} \sum_{i\in[n]}\bbP(\hzzz{\ell}_i \neq \zzz{\ell}_i, E_i\cap F_i)\\
  & \geq \frac{|T^c|}{n} \inf_{\hzzz{\ell}} \sup_{\bz^\star \in \calZ_T} \frac{1}{|T^c|} \sum_{i\in T^c} \bbP(\hzzz{\ell}_i \neq \zzz{\ell}_i, E_i\cap F_i)\\
  & \geq \frac{1}{n} \sum_{i\in T^c} \inf_{\hzzz{\ell}_i} \underset{\bz^\star\in \calZ_T}{\textnormal{ave}}~ \bbP(\hzzz{\ell}_i \neq \zzz{\ell}_i, E_i\cap F_i).
\end{align*}
We are to show that all the summands in the right-hand side above are lower bounded by the same quantity, which is the type-I plus type-II error of the testing problem \eqref{eq:fund_testing_problem_z_l_part_II}, so that for any $i\in T^c$ we would have
$$
  \inf_{\hzzz{\ell}}  \sup_{\bz^\star \in \calP_n} \bbE\calL(\hzzz{\ell}, \zzz{\ell}) \gtrsim \delta_n \inf_{\hzzz{\ell}_i} \underset{\bz^\star\in \calZ_T}{\textnormal{ave}}~ \bbP(\hzzz{\ell}_i \neq \zzz{\ell}_i, E_i\cap F_i),
$$
which is the desired result.

In the following discussion, we without loss of generality assume $\ell = 1$. 
Now for any $i\in T^c$ and $S \subseteq\{2, \hdots, L\}$, we have 
\begin{align*}
  & \inf_{\hzzz{1}_i} \underset{\bz^\star\in \calZ_T}{\textnormal{ave}}~ \bbP(\hzzz{1}_i \neq \zzz{1}_i, E_i\cap F_i) \\
  & = \inf_{\hzzz{1}_i} \underset{\bz^\star_{-i}}{\textnormal{ave}}~\underset{\bz^\star_i}{\textnormal{ave}}~ \sum_{\xi\in \{\pm 1\}^S}  \bbP(\hzzz{1}_i \neq \zzz{1}_i , E_i\cap F_i~|~ \zzz{\ell}_i = \xi_\ell \bz^\star_i ~\forall \ell\in S)\cdot \bbP(\zzz{\ell}_i = \xi_\ell \bz^\star_i ~\forall \ell\in S)\\ 
  & \geq \underset{\bz^\star_{-i}}{\textnormal{ave}}~ \sum_{\xi\in \{\pm 1\}^S}\bbP(\zzz{\ell}_i = \xi_\ell \bz^\star_i ~\forall \ell\in S) \cdot \inf_{\hzzz{1}_i}\underset{\bz^\star_i}{\textnormal{ave}}~\bbP(\hzzz{1}_i \neq \zzz{1}_i, E_i\cap F_i ~|~ \zzz{\ell}_i = \xi_\ell \bz^\star_i ~\forall \ell\in S),
\end{align*}
where in the last inequality we can pull $\bbP(\zzz{\ell}_i = \xi_\ell \bz^\star_i ~\forall \ell\in S)$ in front of $\textnormal{ave}_{\bz^\star_i}$ because
$$
  \bbP(\zzz{\ell}_i = \xi_\ell \bz^\star_i ~\forall \ell\in S)  = \bbP(\zzz{\ell}_i = \xi_\ell \bz^\star_i ~\forall \ell\in S ~|~ \bz^\star_i = +1) = \bbP(\zzz{\ell}_i = \xi_\ell \bz^\star_i ~\forall \ell\in S ~|~ \bz^\star_i = -1).
$$
Since $E_i\cap F_i$ only depends on $\{\zzz{\ell}_{-i}\}$, we can decompose the error probability according to whether $\zzz{1}_j$ is flipped or not:
\begingroup
\allowdisplaybreaks
\begin{align*}
  & \inf_{\hzzz{1}_i} \underset{\bz^\star_i}{\textnormal{ave}} \  \bbP ( \hzzz{1}_i \neq \zzz{1}_i, E_i\cap F_i \ | \ \zzz{\ell}_i = \xi_\ell \bz^\star_i \  \forall \ell \in S ) \\
  & = \frac{1}{2} \inf_{\hzzz{1}_i}  \bbP ( \hzzz{1}_i \neq \zzz{1}_i , E_i\cap F_i\ | \ \bz^\star_i = 1,  \zzz{\ell}_i = \xi_\ell \bz^\star_i \  \forall \ell \in S ) + \bbP ( \hzzz{1}_i \neq \zzz{1}_i, E_i\cap F_i \ | \ \bz^\star_i = -1,  \zzz{\ell}_i = \xi_\ell \bz^\star_i \  \forall \ell \in S ) \\
  & \geq \frac{1}{2} \inf_{\hzzz{1}_0} \sum_{\zeta\in\{\pm 1\}^{n-1}} \bbP(\zzz{1}_j = \zeta_j \bz^\star_j ~\forall j\neq i) \cdot \Indc\bigg\{ \frac{1}{n}\sum_{j\neq i}\indc{\zeta_j = -1} \leq \rho + \ep -n^{-1}, |\mmm{1}_1- n/2|\lor |\mmm{1}_2 - n/2|\leq \frac{n}{2}\cdot\ep\bigg\} \\
  & \qquad \times\bigg( (1-\rho) \cdot \bbP ( \hzzz{1}_i \neq \zzz{1}_i \ | \ \bz^\star_i = 1, \zzz{1}_i = \bz^\star_i, \zzz{1}_j = \zeta_j \bz^\star_j ~\forall j\neq i,  \zzz{\ell}_i = \xi_\ell \bz^\star_i \  \forall \ell \in S ) \\
  & \qquad\qquad \qquad   + \rho \cdot \bbP ( \hzzz{1}_i \neq \zzz{1}_i\ | \ \bz^\star_i = 1, \zzz{1}_i = -\bz^\star_i, \zzz{1}_j = \zeta_j \bz^\star_j ~\forall j\neq i,  \zzz{\ell}_i = \xi_\ell \bz^\star_i \  \forall \ell \in S )  \\
  & \qquad\qquad \qquad + (1-\rho) \cdot \bbP ( \hzzz{1}_i \neq \zzz{1}_i \ | \ \bz^\star_i = -1, \zzz{1}_i = \bz^\star_i, \zzz{1}_j = \zeta_j \bz^\star_j ~\forall j\neq i,  \zzz{\ell}_i = \xi_\ell \bz^\star_i \  \forall \ell \in S ) \\
  & \qquad\qquad \qquad   + \rho \cdot \bbP ( \hzzz{1}_i \neq \zzz{1}_i \ | \ \bz^\star_i = -1, \zzz{1}_i = -\bz^\star_i, \zzz{1}_j = \zeta_j \bz^\star_j ~\forall j\neq i,  \zzz{\ell}_i = \xi_\ell \bz^\star_i \  \forall \ell \in S ) \bigg),
\end{align*}
\endgroup
where
$$
  \mmm{1}_1 = \#\{j\neq i: \zeta_j\bz^\star_j = + 1\}, \ \ \ \mmm{1}_2 = \#\{j\neq i: \zeta_j\bz^\star_j = -1\}.
$$
Here $\zeta_j$ indicates whether $\zzz{1}_j$ is flipped or not. We then have
\begin{align*}
  & \inf_{\hzzz{1}_i} \underset{\bz^\star_i}{\textnormal{ave}} \  \bbP ( \hzzz{1}_i \neq \zzz{1}_i, E_i\cap F_i \ | \ \zzz{\ell}_i = \xi_\ell \bz^\star_i \  \forall \ell \in S ) \\
  & \geq \sum_{\zeta\in\{\pm 1\}^{n-1}} \bbP(\zzz{1}_j = \zeta_j \bz^\star_j ~\forall j\neq i) \cdot \Indc\bigg\{ \frac{1}{n}\sum_{j\neq i}\indc{\zeta_j = -1} \leq \rho + \ep -n^{-1}, |\mmm{1}_1 - n/2|\lor |\mmm{1}_2 - n/2|\leq \frac{n}{2}\cdot \ep\bigg\}\\
  & \qquad \times \frac{1-\rho}{2} \cdot \inf_{\hzzz{1}_i}  \bigg\{\bbP ( \hzzz{1}_i = - 1 \ | \ \bz^\star_i = 1, \zzz{1}_i = \bz^\star_i, \zzz{1}_j = \zeta_j \bz^\star_j ~\forall j\neq i,  \zzz{\ell}_i = \xi_\ell \bz^\star_i \  \forall \ell \in S )  \\
  & \qquad \qquad\qquad\qquad+ \bbP ( \hzzz{1}_i = +1\ | \ \bz^\star_i = -1, \zzz{1}_i = \bz^\star_i, \zzz{1}_j = \zeta_j \bz^\star_j ~\forall j\neq i,  \zzz{\ell}_i = \xi_\ell \bz^\star_i \  \forall \ell \in S ) \bigg\} \\
  & \qquad + \frac{\rho}{2} \cdot \inf_{\hzzz{1}_i}  \bigg\{\bbP ( \hzzz{1}_i  = +1 \ | \ \bz^\star_i = 1, \zzz{1}_i = -\bz^\star_i, \zzz{1}_j = \zeta_j \bz^\star_j ~\forall j\neq i,  \zzz{\ell}_i = \xi_\ell \bz^\star_i \  \forall \ell \in S )  \\
  & \qquad \qquad \qquad \qquad +  \bbP ( \hzzz{1}_i = -1 \ | \ \bz^\star_i = -1, \zzz{1}_i = -\bz^\star_i, \zzz{1}_j = \zeta_j \bz^\star_j ~\forall j\neq i,  \zzz{\ell}_i = \xi_\ell \bz^\star_i \  \forall \ell \in S )\bigg\}.
\end{align*}
Let $\RN{1}, \RN{2}$ be the first and the second ``$\inf$'' term in the right-hand side above, respectively. 

We first deal with term $\RN{1}$. Note that in $\RN{1}$, $\hzzz{1}_i$ can be regarded as the testing function for
$$
  H_0: \bz^\star_i = 1, \qquad H_1: \bz^\star_i = -1
$$
under the conditional law of $\{\AAA{\ell}\} \ | \ \{\zzz{1}_i = \bz^\star_i, \zzz{1}_j = \zeta_j \bz^\star_j ~\forall j\neq i, \zzz{\ell}_i = \xi_\ell \bz^\star_i \ \forall \ell \in S\}$. By the same arguments as those in the proof of Lemma \ref{lemma:testing_problem_z_star}, $\RN{1}$ is lower bounded by the testing error of $H_0$ v.s. $H_1$ under the joint law of $\{\AAA{\ell}, \zzz{\ell}\} \ | \ \{\zzz{1}_i = \bz^\star_i,\zzz{1}_j = \zeta_j \bz^\star_j ~\forall j\neq i, \zzz{\ell}_i = \xi_\ell \bz^\star_i \ \forall \ell \in S\}$ with $m$ positive nodes and $m$ negative nodes (excluding node $i$), where $m = (1+o(1))n/2$. 

For notational simplicity, we again consider the following equivalent setup: we have $2m+1$ nodes in total, where nodes $1, \hdots, m$ are labeled as $+1$, nodes $m+1, \hdots, 2m$ are labeled as $-1$, and the node labeled as $0$ (which is originally labeled as $i$) is the node whose community is to be decided. 

Under the current notations, the density of $\{\AAA{\ell}, \zzz{\ell}\} \ | \ \{\zzz{1}_0 = \bz^\star_0, \zzz{1}_j = \zeta_j \bz^\star_j ~\forall j\neq 0, \zzz{\ell}_0 = \xi_\ell \bz^\star_0 \ \forall \ell \in S\}$ is given by
\begingroup
\allowdisplaybreaks
\begin{align*}
  & \rho^{nL} \bigg(\frac{1-\rho}{\rho}\bigg)^{\#\{\ell\in[L], 1\leq i\leq 2m: \zzz{\ell}_i = \bz^\star_i\}} \cdot \bigg(\frac{1-\rho}{\rho}\bigg)^{1+ \#\{\ell\in S: \xi_{\ell} = 1\} + \#\{\ell\notin S\cup\{1\}: \zzz{\ell}_0 = \bz^\star_0\}} \\
  & \qquad \times \prod_{\ell\in[L]} \prod_{\substack{i\neq j \\ i\neq 0 \\ j\neq 0}} p_\ell^{\AAA{\ell}_{ij}} (1-p_\ell)^{1 - \AAA{\ell}_{ij}} \indc{\zzz{\ell}_i = \zzz{\ell}_j} + q_\ell^{\AAA{\ell}_{ij}} (1-q_\ell)^{1 - \AAA{\ell}_{ij}} \indc{\zzz{\ell}_i \neq \zzz{\ell}_j} \\
  & \qquad \times \prod_{\substack{j\neq 0 \\  \zeta_j\bz^\star_j = \bz^\star_0 }} p_1^{\AAA{1}_{0j}} (1-p_1)^{1 - \AAA{1}_{0j}}  \prod_{\substack{j\neq 0 \\  \zeta_j \bz^\star_j \neq \bz^\star_0 }} q_1^{\AAA{1}_{0j}} (1-q_1)^{1 - \AAA{1}_{0j}}  \\
  & \qquad \times \prod_{\ell\in S} \prod_{\substack{j\neq 0 \\  \zzz{\ell}_j =\xi_\ell\bz^\star_0  }} p_\ell^{\AAA{\ell}_{0j}} (1-p_\ell)^{1 - \AAA{\ell}_{0j}}  \prod_{\substack{j\neq 0 \\  \zzz{\ell}_j \neq \xi_\ell \bz^\star_0 }} q_\ell^{\AAA{\ell}_{0j}} (1-q_\ell)^{1 - \AAA{\ell}_{0j}}\\
  & \qquad \times \prod_{\ell\notin S\cup\{1\}} \prod_{\substack{j\neq 0 \\  \zzz{\ell}_j = \zzz{\ell}_0 }} p_\ell^{\AAA{\ell}_{0j}} (1-p_\ell)^{1 - \AAA{\ell}_{0j}}  \prod_{\substack{j\neq 0 \\  \zzz{\ell}_j \neq \zzz{\ell}_0 }} q_\ell^{\AAA{\ell}_{0j}} (1-q_\ell)^{1 - \AAA{\ell}_{0j}}.
\end{align*}
\endgroup
The likelihood ratio is then given by
\begingroup
\allowdisplaybreaks
\begin{align*}
  \frac{L_0}{L_1} & =   \bigg(\frac{1-\rho}{\rho}\bigg)^{\#\{\ell\notin S\cup\{1\}: \zzz{\ell}_0 = 1\} - \#\{\ell\notin S\cup\{1\}: \zzz{\ell}_0 = -1\}} \\
  & \qquad \times \prod_{\substack{j\neq 0 \\  \zeta_j\bz^\star_j = 1 }} p_1^{\AAA{1}_{0j}} (1-p_1)^{1 - \AAA{1}_{0j}}  \prod_{\substack{j\neq 0 \\  \zeta_j\bz^\star_j =-1 }} q_1^{\AAA{1}_{0j}} (1-q_1)^{1 - \AAA{1}_{0j}}  \\
  & \qquad \times \bigg( \prod_{\substack{j\neq 0 \\  \zeta_j\bz^\star_j= -1 }} p_1^{\AAA{1}_{0j}} (1-p_1)^{1 - \AAA{1}_{0j}}  \prod_{\substack{j\neq 0 \\  \zeta_j\bz^\star_j =1 }} q_1^{\AAA{1}_{0j}} (1-q_1)^{1 - \AAA{1}_{0j}} \bigg)^{-1}\\
  & \qquad \times \prod_{\ell\in S} \prod_{\substack{j\neq 0 \\  \zzz{\ell}_j =\xi_\ell }} p_\ell^{\AAA{\ell}_{0j}} (1-p_\ell)^{1 - \AAA{\ell}_{0j}}  \prod_{\substack{j\neq 0 \\  \zzz{\ell}_j = - \xi_\ell }} q_\ell^{\AAA{\ell}_{0j}} (1-q_\ell)^{1 - \AAA{\ell}_{0j}} \\
  & \qquad \times \bigg( \prod_{\ell\in S} \prod_{\substack{j\neq 0 \\  \zzz{\ell}_j = - \xi_\ell }} p_\ell^{\AAA{\ell}_{0j}} (1-p_\ell)^{1 - \AAA{\ell}_{0j}}  \prod_{\substack{j\neq 0 \\  \zzz{\ell}_j =  \xi_\ell }} q_\ell^{\AAA{\ell}_{0j}} (1-q_\ell)^{1 - \AAA{\ell}_{0j}} \bigg)^{-1}.
\end{align*}
\endgroup
By Neyman-Pearson lemma, the term $\RN{1}$ is given by
\begin{align*}
  &  \bbP\bigg(\frac{L_0}{L_1} \leq  1 \ \bigg| \ \bz^\star_0 = 1, \zzz{1}_0 = \bz^\star_0, \zzz{1}_j = \zeta_j \bz^\star_j ~\forall j\neq 0, \zzz{\ell}_0 = \xi_\ell \bz^\star_0 \ \forall \ell \in S \bigg) \\
  & \qquad + \bbP\bigg(\frac{L_0}{L_1} \geq 1 \ \bigg| \ \bz^\star_0 = -1, \zzz{1}_0 = \bz^\star_0, \zzz{1}_j = \zeta_j \bz^\star_j ~\forall j\neq 0, \zzz{\ell}_0 = \xi_\ell \bz^\star_0 \ \forall \ell \in S \bigg).
\end{align*}  
By symmetry, the two terms above are equal to each other, and we calculate the first term. With a slight abuse of notation, let $\bbP_{H_0, S,\xi, \zeta}$ be the conditional law of $\{\AAA{\ell}, \zzz{\ell}\} \ | \ \{\bz^\star_0 = 1, \zzz{1}_0 = \bz^\star_0,\zzz{1}_j = \zeta_j \bz^\star_j ~\forall j\neq 0, \zzz{\ell}_0 = \xi_\ell \bz^\star_0 \ \forall \ell \in S\}$. Then the quantity of interest is 
\begingroup
\allowdisplaybreaks
\begin{align*}
  &\bbP_{H_0, S, \xi, \zeta} \bigg(\frac{L_0}{L_1} \leq 1\bigg)\\ 
  & = \bbP_{H_0, S, \xi,\zeta} \bigg[ \log \bigg(\frac{\rho}{1-\rho}\bigg) \cdot \sum_{\ell \notin S\cup\{1\}} \indc{\zzz{\ell}_0 = 1} - \indc{\zzz{\ell}_0 = -1} \\
  & \qquad \qquad + \sum_{\substack{j\neq 0 \\ \zeta_j\bz^\star_j= 1}}  \AAA{1}_{0j} \log \bigg(\frac{q_1 (1-p_1)}{p_1(1-q_1)}\bigg) + \log \bigg(\frac{1-q_1}{1-p_1}\bigg)\\
  &  \qquad \qquad + \sum_{\substack{j\neq 0 \\ \zeta_j\bz^\star_j = -1}}  \AAA{1}_{0j} \log \bigg(\frac{p_1 (1-q_1)}{q_1(1-p_1)}\bigg) + \log \bigg(\frac{1-p_1}{1-q_1}\bigg) \\
  & \qquad \qquad + \sum_{\ell \in S}\sum_{\substack{j\neq 0 \\ \zzz{\ell}_j = \xi_\ell}}  \AAA{\ell}_{0j} \log \bigg(\frac{q_\ell (1-p_\ell)}{p_\ell(1-q_\ell)}\bigg) + \log \bigg(\frac{1-q_\ell}{1-p_\ell}\bigg)\\
  & \qquad \qquad + \sum_{\ell \in S}\sum_{\substack{j\neq 0 \\ \zzz{\ell}_j = -\xi_\ell}}  \AAA{\ell}_{0j} \log \bigg(\frac{p_\ell (1-q_\ell)}{q_\ell(1-p_\ell)}\bigg) + \log \bigg(\frac{1-p_\ell}{1-q_\ell}\bigg) \geq 0 \bigg] \\
  & = \bbE_{\{\zzz{\ell}_{-0}: \ell\in S\}}\bigg\{ \bbP_{H_0, S, \xi, \zeta} \bigg[ \log \bigg(\frac{\rho}{1-\rho}\bigg) \cdot \sum_{\ell \notin S\cup\{1\}} \indc{\zzz{\ell}_0 = 1} - \indc{\zzz{\ell}_0 = -1} \\
  & \qquad \qquad\qquad\qquad + \sum_{\substack{j\neq 0 \\ \zeta_j\bz^\star_j = 1}}  \AAA{1}_{0j} \log \bigg(\frac{q_1 (1-p_1)}{p_1(1-q_1)}\bigg) + \log \bigg(\frac{1-q_1}{1-p_1}\bigg)\\
  &  \qquad \qquad\qquad\qquad + \sum_{\substack{j\neq 0 \\ \zeta_j\bz^\star_j = -1}}  \AAA{1}_{0j} \log \bigg(\frac{p_1 (1-q_1)}{q_1(1-p_1)}\bigg) + \log \bigg(\frac{1-p_1}{1-q_1}\bigg) \\
  & \qquad \qquad\qquad\qquad + \sum_{\ell \in S}\sum_{\substack{j\neq 0 \\ \zzz{\ell}_j = \xi_\ell}}  \AAA{\ell}_{0j} \log \bigg(\frac{q_\ell (1-p_\ell)}{p_\ell(1-q_\ell)}\bigg) + \log \bigg(\frac{1-q_\ell}{1-p_\ell}\bigg)\\
  & \qquad \qquad\qquad\qquad + \sum_{\ell \in S}\sum_{\substack{j\neq 0 \\ \zzz{\ell}_j = -\xi_\ell}}  \AAA{\ell}_{0j} \log \bigg(\frac{p_\ell (1-q_\ell)}{q_\ell(1-p_\ell)}\bigg) + \log \bigg(\frac{1-p_\ell}{1-q_\ell}\bigg) \geq 0  \ \bigg| \ \{\zzz{\ell}_{-0}: \ell\in S\}\bigg] \bigg\}.
\end{align*}
\endgroup
The conditional probability above is equal to
\begingroup
\allowdisplaybreaks
\begin{align*}
  & \bbP\bigg( \log\bigg(\frac{\rho}{1-\rho}\bigg)\cdot \sum_{\ell \notin S\cup\{1\}} \ZZZ{\ell} + \sum_{\ell \in S\cup\{1\}} \sum_{i = 1}^{\mmm{\ell}_1} \XXX{\ell}_i \cdot \log \bigg(\frac{q_\ell(1-p_\ell)}{p_\ell(1-q_\ell)}\bigg) + \sum_{\ell \in S\cup\{1\}}\sum_{i = 1}^{\mmm{\ell}_2} \YYY{\ell}_i \cdot \log \bigg(\frac{p_\ell(1-q_\ell)}{q_\ell(1-p_\ell)}\bigg)   \\
  &\qquad  - \sum_{\ell \in S\cup\{1\}} (\mmm{\ell}_1- \mmm{\ell}_2) \cdot \log \bigg(\frac{1-p_\ell}{1-q_\ell}\bigg) \geq 0\bigg),
\end{align*}
\endgroup
where
\begin{align*}
  & \XXX{\ell}_i \overset{\textnormal{i.i.d.}}{\sim} \Bern(p_\ell), \ \ \ \YYY{\ell}_i \overset{\textnormal{i.i.d.}}{\sim} \Bern(q_\ell),  \ \ \ \ZZZ{\ell} \overset{\textnormal{i.i.d.}}{\sim} 2\Bern(1-\rho)-1, \\
  & \mmm{\ell}_1 = \#\{j\neq 0: \zzz{\ell}_j = \xi_\ell\} ,  \ \ \  \mmm{\ell}_2= \#\{j\neq 0: \zzz{\ell}_j = -\xi_\ell\} \ \ \ \forall \ell \in S,\\
  & \mmm{1}_1 = \#\{j\neq 0: \zeta_j\bz^\star_j = 1\}, \ \ \ \mmm{1}_2 = \#\{j\neq 0: \zeta_j \bz^\star_j = -1\}.
\end{align*}
and $\XXX{\ell}_i$'s, $\YYY{\ell}_i$'s and $\ZZZ{\ell}$'s are jointly independent. By data-processing inequality for total variation, with $\mmm{\ell}:= \mmm{\ell}_1 \lor \mmm{\ell}_2$, the above probability can be further lower bounded by
\begin{align*}
  & \bbP\bigg( \log\bigg(\frac{\rho}{1-\rho}\bigg)\cdot \sum_{\ell \notin S\cup\{1\}} \ZZZ{\ell} + \sum_{\ell \in S\cup\{1\}} \sum_{i = 1}^{\mmm{\ell}} \XXX{\ell}_i \cdot \log \bigg(\frac{q_\ell(1-p_\ell)}{p_\ell(1-q_\ell)}\bigg) + \sum_{\ell \in S\cup\{1\}}\sum_{i = 1}^{\mmm{\ell}} \YYY{\ell}_i \cdot \log \bigg(\frac{p_\ell(1-q_\ell)}{q_\ell(1-p_\ell)}\bigg)  \geq 0\bigg).
\end{align*}
Thus, we have 
\begin{align*}
  \RN{1} & \geq \bbE_{\{\zzz{\ell}_{-0}: \ell\in S\}}\bbP\bigg( \log\bigg(\frac{\rho}{1-\rho}\bigg)\cdot \sum_{\ell \notin S\cup\{1\}} \ZZZ{\ell} + \sum_{\ell \in S\cup\{1\}} \sum_{i = 1}^{\mmm{\ell}} \XXX{\ell}_i \cdot \log \bigg(\frac{q_\ell(1-p_\ell)}{p_\ell(1-q_\ell)}\bigg) \\
  &\qquad \qquad \qquad  \qquad + \sum_{\ell \in S\cup\{1\}}\sum_{i = 1}^{\mmm{\ell}} \YYY{\ell}_i \cdot \log \bigg(\frac{p_\ell(1-q_\ell)}{q_\ell(1-p_\ell)}\bigg)  \geq 0\bigg).
\end{align*}
The term $\RN{2}$ is treated similarly. 

In summary, by further taking expectation w.r.t. $\{\zzz{\ell}_0: \ell\in S\cup\{1\}\}$ (note that the expression $(1-\rho)\cdot\RN{1} + \rho \cdot\RN{2}$ is taking the expectation w.r.t. $\zzz{1}_0$), we have 
\begin{align*}
  & \inf_{\hzzz{1}_0} \underset{\bz^\star\in \calZ_T}{\textnormal{ave}}~ \bbP(\hzzz{1}_0 \neq \zzz{1}_0, E_0\cap F_0) \\
  & \geq \bbE_{\{\zzz{\ell}: \ell\in S\cup\{1\}\}} \bigg[ \indc{E_0\cap F_0} \cdot \bbP\bigg( \log\bigg(\frac{\rho}{1-\rho}\bigg)\cdot \sum_{\ell \notin S\cup\{1\}} \ZZZ{\ell} + \sum_{\ell \in S\cup\{1\}} \sum_{i = 1}^{\mmm{\ell}} \XXX{\ell}_i \cdot \log \bigg(\frac{q_\ell(1-p_\ell)}{p_\ell(1-q_\ell)}\bigg) \\
  &\qquad \qquad \qquad \qquad  \qquad + \sum_{\ell \in S\cup\{1\}}\sum_{i = 1}^{\mmm{\ell}} \YYY{\ell}_i \cdot \log \bigg(\frac{p_\ell(1-q_\ell)}{q_\ell(1-p_\ell)}\bigg)  \geq 0\bigg)\bigg].
\end{align*}
The rest of the proof is exactly the same as the proof of Lemma \ref{lemma:testing_problem_z_star}. Since $\bbP(E_0\cap F_0) = 1-o(1)$, we can find some $m' = (1+o(1))n/2$ such that the right-hand side above is lower bounded by a constant multiple of
\begin{align*}
  \bbP\bigg( \log\bigg(\frac{\rho}{1-\rho}\bigg)\cdot \sum_{\ell \notin S\cup\{1\}} \ZZZ{\ell} + \sum_{\ell \in S\cup\{1\}} \sum_{i = 1}^{m'} \XXX{\ell}_i \cdot \log \bigg(\frac{q_\ell(1-p_\ell)}{p_\ell(1-q_\ell)}\bigg)+ \sum_{\ell \in S\cup\{1\}}\sum_{i = 1}^{m'} \YYY{\ell}_i \cdot \log \bigg(\frac{p_\ell(1-q_\ell)}{q_\ell(1-p_\ell)}\bigg)  \geq 0\bigg),
\end{align*}
and the above probability is the error incurred by the likelihood ratio test for the testing problem \eqref{eq:fund_testing_problem_z_l_part_II} with $m' = (1+\delta_n')n/2$.

\subsection{Proof of Lemma \ref{lemma:opt_test_error_z_l}}\label{subappend:prf_opt_test_error_z_l}
The optimal testing error for \eqref{eq:fund_testing_problem_z_l_part_II} (i.e.,   \eqref{eq:opt_test_error_z_l_part_II}) follows directly from Lemma \ref{lemma:opt_test_error_z_star}. 
On the other hand, as shown in the proof of Lemma \ref{lemma:testing_problem_z_l_part_I}, the optimal (weighted) testing error for $\eqref{eq:fund_testing_problem_z_l_part_I}$ is given by
\begin{align*}
  &  \bbP\bigg(\sum_{i=1}^{m} \XXX{\ell}_i \cdot \log \frac{q_\ell(1-p_\ell)}{p_\ell(1-q_\ell)} + \sum_{i = 1}^{m} \YYY{\ell}_i \cdot \log\frac{p_\ell(1-q_\ell)}{q_\ell(1-p_\ell)}\geq \ZZZ{\ell}\log \frac{1-\rho}{\rho}\bigg)\\
  & \geq \rho \cdot \bbP\bigg(\sum_{i=1}^{m} \XXX{\ell}_i \cdot \log \frac{q_\ell(1-p_\ell)}{p_\ell(1-q_\ell)} + \sum_{i = 1}^{m} \YYY{\ell}_i \cdot \log\frac{p_\ell(1-q_\ell)}{q_\ell(1-p_\ell)}\geq \log \frac{\rho}{1-\rho}\bigg).
\end{align*}
The probability in the right-hand side above is calculated in the proof of Lemma \ref{lemma:opt_test_error_z_star} (see Appendix \ref{subsubappend:prf_opt_test_error_z_star_odd_Sc},   \eqref{eq:target_of_lower_bound_z_star_odd} with $S = \{\ell\}$). Now recognizing $\rho = e^{-(1+o(1))2J_\rho}$ gives   \eqref{eq:opt_test_error_z_l_part_I}.

\subsection{Proof of Theorem \ref{thm:lower_bound_z_star_const_rho}}\label{prf:thm:lower_bound_z_star_const_rho}
Throughout the proof, we let
$$
  J'_\rho = - \log \sqrt{2\rho(1-\rho)} = J_\rho + \log\sqrt{2} \approx J_\rho + 0.35.
$$
The proof relies on the following lemma, which is the counterpart of Lemma \ref{lemma:opt_test_error_z_star}.
\begin{lemma}[Optimal testing error for \globest~with non-vanishing $\rho$]
\label{lemma:opt_test_error_z_star_const_rho}
Assume $1\lesssim \rho < \frac{1}{2}$ and that there exist constants $C_1, C_2 > 1, c\in(0, 1)$ such that $C_1 q_\ell \leq p_\ell \leq (C_2 q_\ell) \land (1-c)$ for any $\ell\in[L]$. 
Then, there exists a sequence $\delta''_n=o(1)$ which is independent of $S$, such that the probability \eqref{eq:error_of_LR_test_z_star} is lower bounded by
\begin{equation}
  \label{eq:opt_test_error_z_star_const_rho}
  C\cdot \exp\bigg\{- (1+\delta_n'') \cdot \bigg(|S^c|J_\rho' + \psi_S^\star(0) + C' [\psi_S^\star(0)]^{1/2}\bigg)\bigg\},
\end{equation}
where $C, C'>0$ are absolute constants and $\psi_S^\star(\cdot)$ is defined in \eqref{eq:cum_gen_fun}. 
\end{lemma}
\begin{proof}
  Following the proof of Lemma \ref{lemma:opt_test_error_z_star}, we can lower bound \eqref{eq:error_of_LR_test_z_star} by
  $$
  \exp\bigg\{-(1+o(1))|S^c|J_\rho \bigg\}\times \bbP\bigg(\sum_{\ell \in S} \sum_{i = 1}^{m} \XXX{\ell}_i \cdot \log \bigg(\frac{q_\ell(1-p_\ell)}{p_\ell(1-q_\ell)}\bigg) + \YYY{\ell}_i \cdot \log \bigg(\frac{p_\ell(1-q_\ell)}{q_\ell(1-p_\ell)}\bigg)  \geq 0\bigg),
  $$
  if $S^c = \varnothing$ and by
  \begin{align*}
    & \sum_{x \in \{-|S^c|+ 2j: 0\leq j\leq |S^c|\}}\binom{|S^c|}{\frac{|S^c| + x}{2}} (1-\rho)^{\frac{|S^c| + x}{2}} \rho^{\frac{|S^c| - x}{2}}\\
  & \qquad \times \bbP\bigg(\sum_{\ell \in S} \sum_{i = 1}^{m} \XXX{\ell}_i \cdot \log \bigg(\frac{q_\ell(1-p_\ell)}{p_\ell(1-q_\ell)}\bigg) + \YYY{\ell}_i \cdot \log \bigg(\frac{p_\ell(1-q_\ell)}{q_\ell(1-p_\ell)}\bigg)  \geq \log\bigg(\frac{1-\rho}{\rho}\bigg)\cdot x \bigg)
  \end{align*}
  otherwise.
  If $S^c \neq \varnothing$ and $|S^c|$ is even, by only keeping the $x = 0$ term, the above display can be further lower bounded by
  \begin{align*}
  & \binom{|S^c|}{\frac{|S^c|}{2}} \bigg((1-\rho)\rho\bigg)^{\frac{|S^c|}{2}}\times \bbP\bigg(\sum_{\ell \in S} \sum_{i = 1}^{m} \XXX{\ell}_i \cdot \log \bigg(\frac{q_\ell(1-p_\ell)}{p_\ell(1-q_\ell)}\bigg) + \YYY{\ell}_i \cdot \log \bigg(\frac{p_\ell(1-q_\ell)}{q_\ell(1-p_\ell)}\bigg)  \geq 0\bigg)\\
  & \geq 2^{|S^c|/2} \cdot \exp\bigg\{\frac{|S^c|}{2}\log(\rho(1-\rho))\bigg\} \times
   \bbP\bigg(\sum_{\ell \in S} \sum_{i = 1}^{m} \XXX{\ell}_i \cdot \log \bigg(\frac{q_\ell(1-p_\ell)}{p_\ell(1-q_\ell)}\bigg) + \YYY{\ell}_i \cdot \log \bigg(\frac{p_\ell(1-q_\ell)}{q_\ell(1-p_\ell)}\bigg)  \geq 0\bigg)\\
  & = e^{-|S^c|J_\rho'}\times  \bbP\bigg(\sum_{\ell \in S} \sum_{i = 1}^{m} \XXX{\ell}_i \cdot \log \bigg(\frac{q_\ell(1-p_\ell)}{p_\ell(1-q_\ell)}\bigg) + \YYY{\ell}_i \cdot \log \bigg(\frac{p_\ell(1-q_\ell)}{q_\ell(1-p_\ell)}\bigg)  \geq 0\bigg),
  \end{align*}
  where the first inequality is by $\binom{n}{k}\geq (n/k)^k$.
  If $S^c \neq \varnothing$ and $|S^c|$ is odd, by only keeping the $x = -1$ term, the lower bound reads
  \begin{align*}
    & \binom{|S^c|}{\frac{|S^c|-1}{2}} \bigg((1-\rho)\rho\bigg)^{\frac{|S^c|}{2}} \bigg(\frac{\rho}{1-\rho}\bigg)^{1/2} \times  \bbP\bigg(\sum_{\ell \in S} \sum_{i = 1}^{m} \XXX{\ell}_i \cdot \log \bigg(\frac{q_\ell(1-p_\ell)}{p_\ell(1-q_\ell)}\bigg) + \YYY{\ell}_i \cdot \log \bigg(\frac{p_\ell(1-q_\ell)}{q_\ell(1-p_\ell)}\bigg)  \geq 0\bigg)\\
    & \geq 2^{\frac{|S^c|-1}{2}} \exp\bigg\{\frac{|S^c|}{2}\log(\rho(1-\rho))\bigg\} \bigg(\frac{\rho}{1-\rho}\bigg)^{1/2} \\
    & \qquad \times \bbP\bigg(\sum_{\ell \in S} \sum_{i = 1}^{m} \XXX{\ell}_i \cdot \log \bigg(\frac{q_\ell(1-p_\ell)}{p_\ell(1-q_\ell)}\bigg) + \YYY{\ell}_i \cdot \log \bigg(\frac{p_\ell(1-q_\ell)}{q_\ell(1-p_\ell)}\bigg)  \geq 0\bigg)\\
    & \gtrsim e^{-|S^c|J_\rho'} \times  \bbP\bigg(\sum_{\ell \in S} \sum_{i = 1}^{m} \XXX{\ell}_i \cdot \log \bigg(\frac{q_\ell(1-p_\ell)}{p_\ell(1-q_\ell)}\bigg) + \YYY{\ell}_i \cdot \log \bigg(\frac{p_\ell(1-q_\ell)}{q_\ell(1-p_\ell)}\bigg)  \geq 0\bigg),
  \end{align*}
  where the first inequality is by $\binom{n}{k}\geq (n/k)^k$ and the second inequality is by $\rho \gtrsim 1$. Summarizing the above cases, \eqref{eq:error_of_LR_test_z_star} can be lower bounded by a constant multiple of
  $$
  e^{-|S^c|J_\rho'} \times  \bbP\bigg(\sum_{\ell \in S} \sum_{i = 1}^{m} \XXX{\ell}_i \cdot \log \bigg(\frac{q_\ell(1-p_\ell)}{p_\ell(1-q_\ell)}\bigg) + \YYY{\ell}_i \cdot \log \bigg(\frac{p_\ell(1-q_\ell)}{q_\ell(1-p_\ell)}\bigg)  \geq 0\bigg)
  $$
  Recall that we have shown in the proof of Lemma \ref{lemma:opt_test_error_z_star} that 
  \begin{align*}
    & \bbP\bigg(\sum_{\ell \in S} \sum_{i = 1}^{m} \XXX{\ell}_i \cdot \log \bigg(\frac{q_\ell(1-p_\ell)}{p_\ell(1-q_\ell)}\bigg) + \YYY{\ell}_i \cdot \log \bigg(\frac{p_\ell(1-q_\ell)}{q_\ell(1-p_\ell)}\bigg)  \geq 0\bigg)\\
    & \gtrsim \exp\bigg\{-(1+o(1))\bigg(\sum_{\ell\in S} m\III{\ell}_{1/2} + c \sqrt{\sum_{\ell\in S} m\III{\ell}_{1/2}}\bigg)\bigg\},
  \end{align*}  
  where $m = n/2$ and $c> 0$ is an absolute constant. The proof is concluded.
\end{proof}

We now finish the proof of Theorem \ref{thm:lower_bound_z_star_const_rho}.
If $\min_{S\subseteq[L]} \{|S^c|J_\rho' + \psi_S^\star(0)\}\to\infty$, then we can find a sequence $\delta_n = o(1)$ such that 
$
\log (\delta_n^{-1}) \ll
\min_{S\subseteq[L]} \{|S^c|J_\rho' + \psi_S^\star(0) + C'[\psi_S^\star(0)]^{1/2}\}.
$ 
Then, invoking Lemmas \ref{lemma:testing_problem_z_star} and \ref{lemma:opt_test_error_z_star_const_rho}, there exists $\delta_n' = o(1)$ such that
\begin{align*}
  & \inf_{\bhz^\star} \sup_{\bz^\star\in \calP_n} \bbE \calL(\bhz^\star, \bz^\star) 
   \gtrsim \delta_n 
   \exp\bigg\{-(1+\delta_n')\min_{S\subseteq[L]} \bigg(|S^c|J_\rho' + \psi_S^\star(0) + C' [\psi_S^\star(0)]^{1/2}\bigg)\bigg\}\\
  & = \exp\bigg\{-(1+{\delta_n' + \delta_n''}) \min_{S\subseteq[L]}\bigg(|S^c|J_\rho' + \psi_S^\star(0) + C' [\psi_S^\star(0)]^{1/2}\bigg)\bigg\},
\end{align*}
where
$
  \delta_n'' = {\log({\delta_n}^{-1})}/\min_{S\subseteq[L]} \{|S^c|J_\rho' + \psi_S^\star(0) + C'[\psi_S^\star(0)]^{1/2}\} = o(1)
$
by construction. The right-hand side above can be written as
\begin{align*}
  & \exp\bigg\{-(1+{\delta_n' + \delta_n''}) \min_{S\subseteq[L]}\bigg(|S^c|J_\rho' + \psi_S^\star(0)\bigg)\cdot \bigg(1 + \frac{C' [\psi_S^\star(0)]^{1/2}}{|S^c|J_\rho' + \psi_S^\star(0)}\bigg) \bigg\}\\
  & \geq 
  \exp\bigg\{-(1+{\delta_n' + \delta_n''}) \min_{S\subseteq[L]}\bigg(|S^c|J_\rho' + \psi_S^\star(0)\bigg)\cdot \bigg(1 + C' [|S^c|J_\rho' + \psi_S^\star(0)]^{-1/2}\bigg) \bigg\} \\
  & \geq 
  \exp\bigg\{-(1+{\delta_n' + \delta_n''}) \min_{S\subseteq[L]}\bigg(|S^c|J_\rho' + \psi_S^\star(0)\bigg)\cdot \bigg[1 + C' \bigg(\min_{S\subseteq[L]}|S^c|J_\rho' + \psi_S^\star(0)\bigg)^{-1/2}\bigg] \bigg\} \\
  & = \exp\bigg\{-(1+{\delta_n' + \delta_n''})\cdot (1+ o(1)) \min_{S\subseteq[L]}\bigg(|S^c|J_\rho' + \psi_S^\star(0)\bigg) \bigg\},
\end{align*}
where the last inequality is by $\min_{S\subseteq[L]} \{|S^c|J_\rho' + \psi_S^\star(0)\}\to\infty$. Thus, \eqref{eq:lower_bound_z_star_vanishing_const_rho} follows.

Now, if $\min_{S\subseteq[L]} \{|S^c|J_\rho' + \psi_S^\star(0)\}= \calO(1)$, we know that $\min_{S\subseteq[L]} \{|S^c|J_\rho' + \psi_S^\star(0) + C'[\psi_S^\star(0)]^{1/2}\}$ is also $\calO(1)$. Then Lemmas \ref{lemma:testing_problem_z_star} and \ref{lemma:opt_test_error_z_star_const_rho} give that
$\inf_{\bhz^\star} \sup_{\bz^\star\in \calP_n} \bbE \calL(\bhz^\star, \bz^\star) \gtrsim \delta_n$ for any $o(1)$ sequence $\delta_n$. 
If
$\inf_{\bhz^\star} \sup_{\bz^\star\in \calP_n} \bbE \calL(\bhz^\star, \bz^\star)$ is itself $o(1)$, then we would have
$$
  \inf_{\bhz^\star} \sup_{\bz^\star\in \calP_n} \bbE \calL(\bhz^\star, \bz^\star)  \gtrsim \sqrt{\inf_{\bhz^\star} \sup_{\bz^\star\in \calP_n} \bbE \calL(\bhz^\star, \bz^\star)},
$$
a contradiction. Hence \eqref{eq:lower_bound_z_star_constant_const_rho} follows.

\subsection{Proof of Theorem \ref{thm:lower_bound_z_l_const_rho}}\label{prf:thm:lower_bound_z_l_const_rho}
We first present the following counterpart of Lemma \ref{lemma:opt_test_error_z_l}.
\begin{lemma}[Optimal testing error for \indest~with non-vanishing $\rho$]
\label{lemma:opt_test_error_z_l_const_rho}
Assume there exist constants $C_1, C_2 > 1, c\in (0, 1)$ such that $1\lesssim \rho < 1/2$ and $C_1 q_\ell \leq p_\ell \leq (C_2 q_\ell)\land (1-c), \forall\ell\in[L]$. Then there exists a sequence $\delta''_n=o(1)$ such that for any $\ell\in[L], S \subseteq[L]$, the optimal $(1-\rho)\times \textnormal{type-I error} + \rho\times \textnormal{type-II error}$ of the testing problem in \eqref{eq:fund_testing_problem_z_l_part_I} is lower bounded by
\begin{equation}
  \label{eq:opt_test_error_z_l_part_I_const_rho}
  C\cdot \exp\bigg\{ - (1+\delta_n'')\bigg(\psi_{\{\ell\}}^\star(0) + C'[\psi_{\{\ell\}}^\star(0) ]^{1/2}\bigg)\bigg\} ,
\end{equation}
and the optimal type-I plus type-II error of the testing problem in \eqref{eq:fund_testing_problem_z_l_part_II} is lower bounded by
\begin{equation}
  \label{eq:opt_test_error_z_l_part_II_const_rho}
  C\cdot \exp\bigg\{-(1+\delta_n'')\bigg(|(S\cup\{\ell\})^c| J_\rho' + \psi_{S\cup\{\ell\}}^\star(0) + C'[\psi_{S\cup\{\ell\}}^\star(0) ]^{1/2}\bigg)\bigg\} ,
\end{equation}
where $C, C'>0$ are absolute constants.
\end{lemma}
\begin{proof}
The optimal testing error for \eqref{eq:fund_testing_problem_z_l_part_II} (i.e., \eqref{eq:opt_test_error_z_l_part_II_const_rho}) follows directly from Lemma \ref{lemma:opt_test_error_z_star_const_rho}. 
On the other hand, as shown in the proof of Lemma \ref{lemma:testing_problem_z_l_part_I}, the optimal (weighted) testing error for $\eqref{eq:fund_testing_problem_z_l_part_I}$ is given by
\begin{align*}
  &  \bbP\bigg(\sum_{i=1}^{m} \XXX{\ell}_i \cdot \log \frac{q_\ell(1-p_\ell)}{p_\ell(1-q_\ell)} + \sum_{i = 1}^{m} \YYY{\ell}_i \cdot \log\frac{p_\ell(1-q_\ell)}{q_\ell(1-p_\ell)}\geq \ZZZ{\ell}\log \frac{1-\rho}{\rho}\bigg)\\
  & \geq \rho \cdot \bbP\bigg(\sum_{i=1}^{m} \XXX{\ell}_i \cdot \log \frac{q_\ell(1-p_\ell)}{p_\ell(1-q_\ell)} + \sum_{i = 1}^{m} \YYY{\ell}_i \cdot \log\frac{p_\ell(1-q_\ell)}{q_\ell(1-p_\ell)}\geq \log \frac{\rho}{1-\rho}\bigg) \\
  & \gtrsim \bbP\bigg(\sum_{i=1}^{m} \XXX{\ell}_i \cdot \log \frac{q_\ell(1-p_\ell)}{p_\ell(1-q_\ell)} + \sum_{i = 1}^{m} \YYY{\ell}_i \cdot \log\frac{p_\ell(1-q_\ell)}{q_\ell(1-p_\ell)}\geq 0\bigg),
\end{align*}
where the last inequality is by $1\lesssim \rho < 1/2$.
The probability in the right-hand side above is calculated in the proof of Lemma \ref{lemma:opt_test_error_z_star_const_rho}, and the proof is concluded.
\end{proof}

To finish the proof, note that \eqref{eq:opt_test_error_z_l_part_II_const_rho} is always smaller than \eqref{eq:opt_test_error_z_l_part_I_const_rho}, because
$$
  |(S\cup\{\ell\})^c| J_\rho' + \psi_{S\cup\{\ell\}}^\star(0) + C'[\psi_{S\cup\{\ell\}}^\star(0) ]^{1/2} \geq \psi_{S\cup\{\ell\}}^\star(0) + C'[\psi_{S\cup\{\ell\}}^\star(0) ]^{1/2} \geq \psi^\star_{\{\ell\}}(0).
$$
Thus, the overall lower bound is dominated by \eqref{eq:opt_test_error_z_l_part_I_const_rho}. 
Given Lemma \ref{lemma:testing_problem_z_l_part_I}, \ref{lemma:testing_problem_z_l_part_II} and Equation \eqref{eq:opt_test_error_z_l_part_I_const_rho} in Lemma \ref{lemma:opt_test_error_z_l_const_rho}, the proof is essentially the same as the proof of Theorem \ref{thm:lower_bound_z_star_const_rho}, and we omit the details.

\subsection{Optimal Tests for Minimizing Weighted Type-\texorpdfstring{\RN{1}}{I} and \texorpdfstring{\RN{2}}{II} Errors}

The following lemma is a version of Neyman-Pearson lemma, and it allows us to characterize the optimal test that minimizes the \emph{weighted average} of type-\RN{1} and type-\RN{2} errors.
\begin{lemma}[Neyman--Pearson lemma]
\label{lemma: optimal_test}
  Consider testing $H_0: X\sim \bbP_0$ against $H_1: X\sim \bbP_1$, where $\bbP_0$ and $\bbP_1$ have densities $p_0(x), p_1(x)$ respectively against some dominating measure $\mu(x)$. For any $\omega_0, \omega_1 > 0$, the optimal test that minimizes $\omega_0 \cdot \textnormal{Type-I error} + \omega_1 \cdot \textnormal{Type-II error}$ is given by rejecting $H_0$ when
  $$
    \frac{p_0(x)}{p_1(x)} \leq \frac{\omega_1}{\omega_0}.
  $$
\end{lemma}
\begin{proof}
  This is a restatement of Problem 3.10 in \cite{lehmann2006testing}, and we provide a proof here for completeness.  
  It suffices to consider $w_0 = 1-\rho, w_1 = \rho$ for some $\rho \in (0, 1)$.
  The optimal error is given by
  \begin{align*}
    \inf_{\psi_0, \psi_1: \psi_0 + \psi_1 = 1}  \int (1-\rho) \cdot p_0(x) \psi_1(x) d\mu(x) + \rho \cdot p_1(x)\psi_0(x) d\mu(x). 
  \end{align*}
  When $(1-\rho)\cdot p_0(x)\geq \rho \cdot p_1(x)$, the integrand is lower bounded by
  $$
    \rho \cdot p_1(x) = (1-\rho) \cdot p_0(x) \land \rho \cdot p_1(x).
  $$
  When $(1-\rho)\cdot p_0(x)\leq \rho \cdot p_1(x)$, the integrand is lower bounded by
  $$
    (1-\rho) \cdot p_0(x) = (1-\rho) \cdot p_0(x) \land \rho \cdot p_1(x).
  $$
  This gives
  \begin{align*}
    & \inf_{\psi_0, \psi_1: \psi_0 + \psi_1 = 1}  \int (1-\rho) \cdot p_0(x) \psi_1(x) d\mu(x) + \rho \cdot p_1(x)\psi_0(x) d\mu(x) \geq \int (1-\rho) \cdot p_0(x) \land \rho \cdot p_1(x) d\mu(x).
  \end{align*}
  On the other hand, if we take $\psi_1(x) = \indc{(1-\rho)\cdot p_0(x)\leq \rho \cdot p_1(x)} $, then it is easy to check that the integrand is exactly equal to $(1-\rho) \cdot p_0(x) \land \rho \cdot p_1(x)$. This gives the desired result.
\end{proof}

The following lemma asserts that the testing error can only be smaller when extra information is present, and can be regarded as an instance of ``data-processing inequalities''.
\begin{lemma}
\label{lemma: general_data_processing}
  Under the setup of Lemma \ref{lemma: optimal_test}, let us additionally consider testing $H'_0: (X, Y)\sim \bbQ_0$ against $H'_1: (X, Y)\sim \bbQ_1$, where 
  the marginal of the first coordinate of $\bbQ_0$ (resp. $\bbQ_1$) agrees with $\bbP_0$ (resp. $\bbQ_1$), and there exist densities $q_0(x, y)$ and $q_1(x, y)$ such that 
  $$
  d \bbQ_0(x, y) = q_0(x, y) d\mu(x)d\mu(y), \ \ \ d \bbQ_1(x, y) = q_1(x, y) d\mu(x)d\mu(y).
  $$  
  Then we have
  $$
    \inf_{\psi}~ w_0 \cdot \bbE_{H_0} \psi + w_1 \cdot \bbE_{H_1} [1-\psi] \geq \inf_{\psi'}~ w_0 \cdot \bbE_{H'_0} \psi' + w_1 \cdot \bbE_{H'_1} [1-\psi'],
  $$
  where $\psi, \psi'$ are testing functions for $H_0$ v.s. $H_1$ and $H_0'$ v.s. $H_1'$, respectively.
\end{lemma}
\begin{proof}
  If $\omega_0 = \omega_1$, this lemma immediately follows from the data-processing inequality for total variation distances. But we need more work for the general case.
  We again without loss of generality assume $w_0 = 1-\rho$, $w_1 = \rho$ for some $\rho \in (0, 1)$. By the proof of Lemma \ref{lemma: optimal_test} and the existence of densities, it suffices to show
  $$
    \int (1-\rho) \cdot p_0 (x) \land \rho \cdot p_1(x) d\mu(x) \geq \iint (1-\rho)\cdot q_0(x, y) \land \rho \cdot q_1(x, y) d\mu(x) d\mu(y),
  $$
  which is implied by
  $$
    (1-\rho) \cdot p_0(x) \land \rho \cdot p_1(x) \geq \int (1-\rho) \cdot q_0(x, y) \land \rho \cdot q_1(x, y) d\mu(y).
  $$
  Note that
  \begin{align*}
    \int (1-\rho) \cdot q_0(x, y) \land \rho \cdot q_1(x, y) d\mu(y) & \leq \int (1-\rho) \cdot q_0(x, y) d\mu(y) \land  \int \rho \cdot q_1(x, y) d\mu(y) \\
    & = (1-\rho) \cdot p_0(x) \land \rho \cdot p_0(x),
  \end{align*}  
  which is the desired result.
\end{proof}

\section{Proofs of Upper Bounds}\label{append:prf_upper_bound}

\subsection{Proof of Theorem \ref{thm:specc}}\label{prf:thm:specc}

We start by stating a structural lemma, which relates the misclustering error to the deviation of the trimmed weighted adjacency matrix $\tau(\bar A)$ from the expectation of $\barA$.
\begin{lemma}\label{lem: structure lem}
Assume there exists a constant $c\in[0, 1)$ such that
\begin{equation}
      \label{eq:spectral_gap_structral_lemma}
      \frac{\bar{p}}{\bar{p}-\bar{q}} \leq c\,\frac{2(1-2\rho)^2}{\beta - \beta^{-1} + 4n^{-1}} + 2(\rho-\rho^2).
\end{equation}      
Then, for any instance generated by an $\textnormal{\modelname}\in \calP_n(\rho, \{p_\ell\}_1^L, \{q_\ell\}_1^L, \beta)$, the output of Algorithm \ref{alg: specc}, $\btz^\star$, will satisfy
\begin{equation}
      \label{eq: mistakes_of_spec_clustering}
      \calL(\btz^\star, \bz^\star) 
      \leq C\cdot \frac{ (2+\ep)}{n^2 (1-2\rho)^4}  \cdot \frac{\|\tau(\bar A) - \bbE \barA \|_2^2}{(\bar p-\bar q) ^2},
\end{equation}    
where $C$ is an absolute constant only depending on $c$.
\end{lemma}

The proof of the above lemma uses the following result, which relates the misclustering error to the geometry of the point cloud.
\begin{lemma}
\label{lemma: geometry}
Let $\bz^{\star}\in\{\pm1\}^n$ be the global parameter for an \modelname~in the parameter space defined in \eqref{eq:param_space}. Suppose there exists a matrix $V\in\bbR^{n\times 2}$ (to be thought as the ``ground truth Euclidean embedding'' of the nodes) and a constant $b>0$ such that
\[
\min_{\bz^{\star}_i\neq \bz^{\star}_j}\|V_{i\bigcdot}-V_{j\bigcdot}\|\geq 2b,
\]
where $V_{i\bigcdot}\in\bbR^2$ is the $i$-th row of $V$.
Then, for any estimator $\btz^{\star}$, any $\{\tilde v_1, \tilde v_2\} \subset \bbR^2$ (to be thought as the ``estimated centroids'' of the nodes), and any $\tilde V\in \bbR^{n\times 2}$ satisfying $\tilde V_{i\bigcdot} = \tilde v_{\btz^{\star}_i}$ (to be thought as the ``estimated Euclidean embedding'' of the nodes), we have
\begin{align*}
  d_\ham(\btz^\star, \bz^\star) \land d_\ham(-\btz^\star, \bz^\star) \leq C \cdot \# \calS,
\end{align*}
where $\calS:=\{i\in[n]:\|\tilde{V}_{i\bigcdot}-V_{i\bigcdot}\|\geq b \}$, $\# \calS$ is the cardinality of $\calS$, and $C$ is an absolute constant. 
\end{lemma}
\begin{proof}
  This is Lemma 5 in \cite{gao2018community}. See also the proof of Theorem 2 in \cite{chen2018convexified}.
\end{proof}

We now present the proof of Lemma \ref{lem: structure lem}.


\begin{proof}[Proof of Lemma \ref{lem: structure lem}]
Note that for $\bz_i^\star=\bz_j^\star$, marginally we have $\bbP({A^{(\ell)}_{ij}=1})=[\rho^2+(1-\rho)^2]\pell+2\rho(1-\rho)\qell=\pell-2(\pell-\qell)(\rho-\rho^2)$. Similarly, for $\bz_i^\star\neq\bz_j^\star$, we have $\bbP({A^{(\ell)}_{ij}=1})=\qell+2(\pell-\qell)(\rho-\rho^2)$. 
So marginally, we have
$$
A^{(\ell)}_{ij}\sim\Bern({\tilde p}_{\ell}\indc{\bz_i^\star=\bz_j^\star}+{\tilde q}_{\ell}\indc{\bz_i^\star\neq\bz_j^\star}),
$$
where 
$$
      \tpell = \pell-2(\pell-\qell)(\rho-\rho^2) , \qquad \tqell = \qell+2(\pell-\qell)(\rho-\rho^2).
$$
Let $Z$ be the $n\times 2$ assignment matrix such that $Z_{i,1}=\indicator{\bz^\star_i=+1}$ and $Z_{i, 2} = \indicator{\bz^\star_i = -1}$. Then $Z\tran Z=\mathrm{diag}(n^{\star}_+, n^{\star}_-)$, where $n^{\star}_\pm \in[\frac{n}{2\beta}, \frac{n\beta}{2}]$ are the sizes of the two communities. Define
\begin{align*}
B^{(\ell)}=\begin{pmatrix} 
\tpell & \tqell \\ \tqell & \tpell
\end{pmatrix}.
\end{align*}
One readily checks that $\bbE \AAA{\ell}=Z B^{(\ell)} Z\tran - \mathrm{diag}(\tpell)$, and thus $\bbE\bar A=Z \bar B Z\tran-\mathrm{diag}(\sum_\ell\omega_\ell\tpell)$, where $\bar B=\sum_\ell\omega_\ell B^{(\ell)}$. We let $P=Z\bar B Z\tran$, and $G=Z(Z\tran Z)^{-1/2}$. Then $G$ is orthonormal and $P=G X G\tran$, where $X= (Z\tran Z)^{1/2} \bar B (Z\tran Z)^{1/2}$. Let $X=WDW\tran$ be the eigen-decomposition of $X$. Then $P$ has the eigen-decomposition $P=VDV\tran$ with $V=GW=Z(Z\tran Z)^{-1/2}W$. Note that $W$ is a $2\times 2$ orthogonal matrix and $(Z\tran Z)^{-1/2}=\mathrm{diag}[(n^{\star}_+)^{-1/2},(n^{\star}_-)^{-1/2}]$.  So if $Z_{i\bigcdot}\neq Z_{j\bigcdot}$, we have $\|V_{i\bigcdot}-V_{j\bigcdot}\|=\|(n^{\star}_+)^{-1/2}W_{1\bigcdot}-(n^{\star}_-)^{-1/2}W_{2\bigcdot}\|=\sqrt{(n^{\star}_+)^{-1}+(n^{\star}_-)^{-1}}$.

We now invoke Lemma \ref{lem: structure lem} with the ``estimated embedding'' $\tilde V$ being the output of Algorithm \ref{alg: specc} (i.e., the solution of the $(1+\ep)$-approximate $k$-means clustering), and with the ``ground truth'' embedding being $V O$, where $O\in\bbR^{2\times 2}$ is an arbitrary orthonormal matrix. Since $\|(VO)_{i\bigcdot}-(VO)_{j\bigcdot}\|=\sqrt{(n^{\star}_+)^{-1}+(n^{\star}_-)^{-1}}$, we can take $b=\frac{1}{2}\sqrt{(n^{\star}_+)^{-1}+(n^{\star}_-)^{-1}}$.
We bound $\# \calS=\#\{i\in[n]:\|\tilde{V}_{i\bigcdot}-(VO)_{i\bigcdot}\|\geq b \}$ as follows:
\begin{align*}
\# \calS&\leq \frac{1}{b^2} \sum_{i\in\calS}\|\tilde{V}_{i\bigcdot}-(VO)_{i\bigcdot}\|^2 \leq \frac{1}{b^2}\sum_{i=1}^n\|\tilde{V}_{i\bigcdot}-(VO)_{i\bigcdot}\|^2=\frac{1}{b^2}\|\tilde{V}-VO\|_F^2
\end{align*}
Since $\tilde V$ solves the $(1+\ep)$ $k$-means clustering objective, the above display can be further bounded by
\begin{equation}
\label{equ: bound S}
      \# \calS \leq \frac{2}{b^2}(\|\tilde V-U\|_F^2+\|U-VO\|_F^2) \leq \frac{2}{b^2}(2+\ep)\|U-VO\|_F^2=\frac{8(2+\ep)}{(n^{\star}_+)^{-1}+(n^{\star}_-)^{-1}}\|U-VO\|_F^2,
\end{equation}
where we recall that $U$ consists of the top two eigenvectors of the trimmed weighted adjacency matrix $\tau(\barA)$. 
Since $V$ consists of the two leading eigenvectors of $P$, which are also the two leading eigenvectors of $\bbE[\bar A]$, we can invoke Davis-Kahan Sin-$\Theta$ theorem (specifically, the version proved in Lemma 5.1 of \cite{lei2015consistency}) to conclude that, if the smallest eigenvalue $\gamma_n$ of $\bbE[\bar A]$ is strictly positive, we will have
\begin{align*}
\# \calS\leq \frac{8(2+\ep)}{(n^{\star}_+)^{-1}+(n^{\star}_-)^{-1}}\|U-VO\|_F^2\leq 32(2+\ep)\frac{1}{(n^{\star}_+)^{-1}+(n^{\star}_-)^{-1}}\frac{\|\tau(\bar{A})- \bbE[\bar A]\|^2}{\gamma_n^2}
\numberthis
\label{equ: bound Frobenius norm}
\end{align*}
Note that $\gamma_n + \sum_{\ell} \weight_\ell \tilde p_\ell$ is also the smallest singular value of 
\[
X=(Z\tran Z)^{1/2}\bar B(Z\tran Z)^{1/2}=\begin{pmatrix}
n^{\star}_+\bar B_{11} & \sqrt{n^{\star}_+ n^{\star}_-}\bar B_{12}\\ \sqrt{n^{\star}_+ n^{\star}_-}\bar B_{21} & n^{\star}_-\bar B_{22}
\end{pmatrix}.
\]
Solving the characteristic polynomial, one finds that 
\begin{align*}
      2(\gamma_n + \sum_\ell \weight_\ell \tilde p_\ell) & = {n \sum_{\ell}\weight_\ell \tilde p_\ell - \sqrt{(n^{\star}_+ - n^{\star}_-)^2 (\sum_{\ell}\weight_\ell \tilde p_\ell)^2 + 4 n^{\star}_+ n^{\star}_- (\sum_{\ell} \weight_\ell \tilde q_\ell)^2}} \\
      & \geq  n \sum_{\ell}\weight_\ell \tilde p_\ell - |n^{\star}_+ - n^{\star}_-| \sum_{\ell}\weight_\ell \tilde p_\ell - n\sum_\ell \weight_\ell \tilde q_\ell,
\end{align*}      
where the last inequality follows from $\sqrt{a+b}\leq \sqrt{a}+\sqrt{b}$ for $a, b\geq 0$ and $n^{\star}_+n^{\star}_-\leq n^2/4$. Since $n^{\star}_\pm \in [\frac{n}{2\beta}, \frac{\beta n}{2}]$, we can further lower bound the right-hand side above by
\begin{align*}
      2(\gamma_n + \sum_\ell \weight_\ell \tilde p_\ell)& \geq  n \sum_\ell \weight_\ell (\tilde p_\ell - \tilde q_\ell) - \frac{n}{2} (\beta - \beta^{-1}) \sum_{\ell} \weight_\ell \tilde p_\ell,
\end{align*}      
which is equivalent to
$$
      2 \gamma \geq n \sum_\ell \weight_\ell (\tilde p_\ell - \tilde q_\ell) - \frac{n}{2}(\beta - \beta^{-1} + 4 n^{-1}) \sum_\ell \weight_\ell \tilde p_\ell.
$$
We claim that if   \eqref{eq:spectral_gap_structral_lemma} holds, then 
$$
      \frac{n}{2}(\beta -\beta^{-1} + 4n^{-1}) \sum_{\ell} \weight_\ell \tilde p_\ell \leq  c n \sum_{\ell} \weight_\ell (\tilde p_\ell   - \tilde q_\ell).
$$
Indeed, since 
$
      \sum_{\ell} \weight_\ell(\tilde p_\ell - \tilde q_\ell) = (\bar p - \bar q) - 4(\bar p - \bar q)(\rho - \rho^2) = (1-2\rho)^2 (\bar p - \bar q),
$
one readily checks that the above display is equivalent to 
$$
      (\beta - \beta^{-1} + 4n^{-1}) (\bar p - 2(\bar p - \bar q) (\rho-\rho^2)) \leq 2c(\bar p - \bar q) (1 -2\rho)^{2},
$$
and this is exactly   \eqref{eq:spectral_gap_structral_lemma} by rearranging terms. Thus, we arrive at
\begin{equation}
      \label{eq:smallest_eigenval}
      \gamma_n \geq (1-c) n\sum_{\ell} \weight_\ell (\tilde p_\ell - \tilde q_\ell) = (1-c) n (1-2\rho)^2(\bar p - \bar q).
\end{equation}      
Plugging \eqref{eq:smallest_eigenval} to \eqref{equ: bound Frobenius norm}, we get
$$
      \calL(\btz^{\star}, \bz^{\star}) \lesssim \frac{(2+\ep)\|\tau(\bar A) - \bbE \bar A\|^2}{(1-c)^2 n^2 (1-2\rho)^4 (\bar p - \bar q)^2},
$$
from which \eqref{eq: mistakes_of_spec_clustering} follows.
\end{proof}

In view of Lemma \ref{lem: structure lem}, what is left is to upper bound the deviation of $\tau(\barA)$ from $\bbE \bar A$. Consider the following decomposition:
\begin{equation}
      \label{eq: decomp_spec_clustering}
      \| \tau(\bar A) - \bbE\barA \|_2 \leq \|\tau(\bar A) - \bbE^{(1:L)} [\bar A] \|_2 + \| \bbE^{(1:L)}[\bar A] - \bbE[\bar A]\|_2, 
\end{equation}
where we let $\bbE$ be the marginal expectation, and $\bbE^{(1:L)}$ be the expectation \emph{conditional} on the realization of $\{\zzz{\ell}\}_1^L$.
The right-hand side of \eqref{eq: decomp_spec_clustering} is the superposition of two terms: the first term is the deviation of $\tau(\bar A)$ from the \emph{conditional mean} of $\bar A$, and the second term is the deviation induced by label sampling.

~\\
{\bf \noindent Bounding the deviation from the conditional mean.}
Conditional on the realization of $\zzz{\ell}$'s, the layer-wise adjacency matrices become symmetric Bernoulli random matrices with independent entries (on the upper-triangular part). Thus, we can invoke the technical tools we developed in Appendix \ref{append:concentration} to get the following result.
\begin{lemma}[Bounding the deviation from the conditional mean]
      \label{lemma:dev_from_cond_mean}
     Let Assumption \ref{assump:balanced_wts} hold with $c_0 > 0, c_1 \geq 1$ and fix two constants $r\geq 1, \gamma > e^{c_1}$. Define $I:= \{i\in[n]: \sum_{j\in[n]} \barA_{ij} > \gamma n \sum_{\ell\in[L]} \weight_\ell p_\ell\}$. We trim the entries of $\bar A$ in $\calE = (I\times[n])\cup ([n]\times I)$, so that the resulting matrix $\tau(\barA)$ is zero on $\calE$. Then with probability at least $1-3n^{-r} - c_2^{-n}$, we have
     $$
      \|\tau(\barA) - \bbE^{(1:L) }\bar A\| \leq C \cdot \sqrt{n \sum_{\ell\in[L]} \weight_\ell^2 p_\ell},
     $$
     where $c_2$ only depends on $\gamma$, and $C$ only depends on $c_0, c_1, r$. 
\end{lemma}
\begin{proof}
    This follows from Corollary \ref{cor:concentration_trimmed_mat} with $d_\ell = np_\ell$. 
\end{proof}

~\\
{\bf \noindent Bounding the deviation due to label sampling.} We have the following lemma.
\begin{lemma}[Bounding the deviation due to label sampling]
      \label{lemma:cluster_flips}
For any $r\geq 1, c>0$, with probability at least $1-6n^{-r} - 2n^{-c}$, we have
\begin{equation*}
      \| \bbE^{(1:L)} \bar A - \bbE \barA\| \leq C \cdot  \bigg[ \max_{\ell\in[L]}\{\weight_\ell (p_\ell - q_\ell)\} \cdot \bigg( L\rho + n\sqrt{L\rho} + \sqrt{n\log n} + \log L\bigg) + \rho(1-\rho) \sum_{\ell} \weight_\ell(p_\ell-q_\ell)\bigg],
\end{equation*}            
where $C$ is an absolute constant only depending on $r, c, c'$.
\end{lemma}
\begin{proof}
Let
$
      \MMM{\ell} := \bbE^{1:L}[\AAA{\ell}] - \bbE [\AAA{\ell}].
$
We are interested in bounding the spectral norm of 
$
      \bar M := \sum_{\ell\in[L]} \omega_\ell \MMM{\ell}.
$
Note that the diagonal element of $\MMM{\ell}$ is zero. Meanwhile, for $i\neq j$, we have
\begin{align*}
      \MMM{\ell}_{ij}& = (p_\ell - \tilde p_\ell) \indc{\zzz{\ell}_i = \zzz{\ell}_j, \bz^\star_i = \bz^\star_j} +  (q_\ell - \tilde q_\ell)\indc{\zzz{\ell}_i\neq \zzz{\ell}_j, \bz^\star_i \neq \bz^\star_j} \\
      & \qquad   + (p_\ell - \tilde q_\ell) \indc{\zzz{\ell}_i= \zzz{\ell}_j, \bz^\star_i \neq \bz^\star_j} + (q_\ell - \tilde p_\ell) \indc{\zzz{\ell}_i \neq \zzz{\ell}_j, \bz^\star_i = \bz^\star_j}\\
      & = (p_\ell-q_\ell) (2\rho - 2\rho^2) \indc{\zzz{\ell}_i = \zzz{\ell}_j, \bz^\star_i = \bz^\star_j} + (p_\ell - q_\ell)(-2\rho + 2\rho^2) \indc{\zzz{\ell}_i\neq \zzz{\ell}_j, \bz^\star_i \neq \bz^\star_j}\\
      & \qquad  + (p_\ell-q_\ell) (1-2\rho+2\rho^2)\indc{\zzz{\ell}_i= \zzz{\ell}_j, \bz^\star_i \neq \bz^\star_j} + (p_\ell - q_\ell)(-1+2\rho - 2\rho^2) \indc{\zzz{\ell}_i \neq \zzz{\ell}_j, \bz^\star_i = \bz^\star_j} \\
      & = (p_\ell-q_\ell) \bigg[\indc{\bz^\star_i = \bz^\star_j} \cdot \bigg(2\rho(1-\rho) - \indc{\zzz{\ell}_i\neq \zzz{\ell}_j}\bigg)  - \indc{\bz^\star_i\neq \bz^\star_j} \cdot \bigg( 2\rho(1-\rho) - \indc{\zzz{\ell}_i = \zzz{\ell}_j} \bigg) \bigg].
\end{align*}
Thus, one readily checks that
$$
      \MMM{\ell} = \frac{1}{2} (p_\ell - q_\ell) \bigg(\zzz{\ell}{\zzz{\ell}}^\top - (1-2\rho)^2\bz^\star {\bz^\star}^\top - 4\rho(1-\rho) I_n\bigg),
$$
where the $4\rho(1-\rho)I_n$ offsets the diagonal entries so that $\diag(\MMM{\ell}) = 0$. We then have 
\begin{align*}
      \|\bar M\| & = \bigg\|\frac{1}{2} \sum_{\ell\in[L]} \omega_\ell(p_\ell - q_\ell) \bigg(\zzz{\ell}{\zzz{\ell}}^\top - (1-2\rho)^2\bz^\star {\bz^\star}^\top - 4\rho(1-\rho)I_n\bigg)\bigg\| \\
      & \leq \frac{1}{2}\max_{\ell\in[L]} \{\weight_\ell (p_\ell -q_\ell)\} \cdot \bigg\| \sum_{\ell\in[L]}  \zzz{\ell}{\zzz{\ell}}^\top - (1-2\rho)^2\bz^\star {\bz^\star}^\top\bigg\| +  2\rho(1-\rho) \sum_{\ell\in[L]} \weight_\ell(p_\ell -q_\ell).
\end{align*}      
With some algebra, one can show that
\begin{align*}
      & \bigg\| \sum_{\ell\in[L]}  \zzz{\ell}{\zzz{\ell}}^\top - (1-2\rho)^2\bz^\star {\bz^\star}^\top\bigg\| \\
      & = \bigg\|\sum_{\ell\in[L]}  (\zzz{\ell}-\bbE \zzz{\ell})(\zzz{\ell}-\bbE \zzz{\ell})^\top + (\zzz{\ell}-\bbE \zzz{\ell})\bbE[\zzz{\ell}]^\top + \bbE[\zzz{\ell}](\zzz{\ell}-\bbE \zzz{\ell})^\top \bigg\| \\
      & \leq \underbrace{\bigg\|\sum_{\ell\in[L]}  (\zzz{\ell}-\bbE \zzz{\ell})(\zzz{\ell}-\bbE \zzz{\ell})^\top\bigg\|}_{\RN{1}} + \underbrace{2\sqrt{n} \bigg\|\sum_{\ell\in[L]} \zzz{\ell}- \bbE\zzz{\ell}\bigg\|}_{\RN{2}},
\end{align*}      
where in the last inequality we have used $\|\zzz{\ell}\|= \sqrt{n}$. We now bound the two terms in the right-hand side above separately.

To bound Term \RN{1}, let us introduce
\[
      Z := 
      \begin{pmatrix}
            (\zzz{1}-\bbE \zzz{1})^\top \\
            (\zzz{2}-\bbE \zzz{2})^\top \\
            \vdots \\
             (\zzz{L}-\bbE \zzz{L})^\top
      \end{pmatrix} \in \bbR^{L\times n}.
\]
Then, we can write Term \RN{1} as $\|Z^\top Z\|$. 
Let $\BBB{\ell}_i = (1-\zzz{\ell}_i\bz^\star_i)/2$, which is distributed as $\BBB{\ell}_i \sim \Bern(\rho)$. Under current notations, we have
\[
      \zzz{\ell}_i - \bbE \zzz{\ell}_i = -2 \bz^\star_i (\BBB{\ell}_i - \bbE \BBB{\ell}_i).
\]
Let $B$ denote the $L\times n$ matrix with $B_{\ell,i}=B_i^{(\ell)}-\bbE B_i^{(\ell)}$. Then $B$ is a matrix with i.i.d. centered $\Bern(\rho)$ entries. Because $Z_{\ell,i}=(\zzz{\ell}_i-\bbE \zzz{\ell}_i)=-2B_{\ell, i}\bz^\star_i $, we have
$
-Z/2= B \cdot \mathrm{diag}(\bz^\star).
$
Hence, we have
$$
      \| Z\| \leq 2 \|B-\bbE B \|. 
$$
We define the $(n+L)\times(n+L)$ matrix $\tB$ as
$$
      \tilde B = 
      \begin{pmatrix}
            \mathbf{0}_{L\times L} & B \\
            B^\top & \mathbf{0}_{n\times n}
      \end{pmatrix}.
$$
It is clear that $\|B - \bbE B \| = \| \tilde B - \bbE\tilde B \|$. Since $\tB$ is a symmetric square matrix with i.i.d. Bernoulli entries in its upper-triangular part, by Corollary \ref{cor:concentration_w/o_reg}, for any $r\geq 1$ and $c>0$, we have
$$
      \|\tilde B - \bbE \tilde B \| \leq C \bigg(\sqrt{(n+L)\rho} + \sqrt{\log(n+L)}\bigg)
$$
with probability at least $1-3n^{-r} - n^{-c}$, where $C$ is a constant only depending on $r$ and $c$.
Thus, on the same high probability event, we have
$$
\RN{1} \leq \|Z\|^2 \leq C' \bigg((n+L)\rho + \log (n+L)\bigg)
$$
for another $C'$ which only depends on $r$ and $c$. 

For Term \RN{2}, we use a similar trick:
$$
      \|\sum_{\ell} \zzz{\ell}-\bbE \zzz{\ell} \|_2
      =
      \bigg\|
      \begin{pmatrix}
            \mathbf{0}_{n\times n} & \sum_{\ell} (\zzz{\ell}-\bbE \zzz{\ell})  \\
            \sum_{\ell}  (\zzz{\ell}-\bbE \zzz{\ell})^\top & \mathbf{0}_{1\times 1}
      \end{pmatrix} \bigg\|,
$$
and the right-hand side above is bounded above by
$$
      2 \cdot \|\hat B \| := 2 \cdot \bigg\|
      \begin{pmatrix}
            \mathbf{0}_{n\times n} & \sum_{\ell} (\BBB{\ell}-\bbE \BBB{\ell})  \\
            \sum_{\ell} (\BBB{\ell}-\bbE \BBB{\ell})^\top & \mathbf{0}_{1\times 1}
      \end{pmatrix} \bigg\|,
$$
where again $\BBB{\ell}_i$'s are i.i.d. $\Bern(\rho)$ random variables. By Corollary \ref{cor:concentration_w/o_reg}, we have
$$
      \|\hat B \| \leq C'' \bigg(\sqrt{n L\rho } +  \sqrt{\log n}\bigg)
$$
with probability at least $1 - 3n^{-r} - n^{-c}$, where $C''$ only depends on $r$ and $c$. This means that on the same high probability event, we have 
$$
      \RN{2} \leq 4 C''\sqrt{n} \cdot (\sqrt{nL\rho} + \sqrt{\log n}).
$$

Combining the bound on $\RN{1}$ and $\RN{2}$, we conclude that with probability at least $1 - 6n^{-r} - 2n^{-c}$, 
\begin{align*}
      \|\bar M\| & \lesssim \max_{\ell\in[L]} \{\weight_\ell(p_\ell-q_\ell)\} \cdot (\RN{1} + \RN{2}) + \rho(1-\rho) \sum_{\ell\in[L]} \weight_\ell(p_\ell-q_\ell) \\
      & \lesssim \max_{\ell\in[L]} \{\weight_\ell(p_\ell-q_\ell)\} \cdot \bigg((n+L)\rho + n\sqrt{L\rho} + \log(n+L) + \sqrt{n \log n}\bigg) + \rho(1-\rho) \sum_{\ell\in[L]} \weight_\ell(p_\ell-q_\ell).
\end{align*}      
The desired result follows by noting that $n\rho \leq n\sqrt{L\rho}, \log (n+L)\leq \log n + \log L$, and $\log n \ll \sqrt{n\log n}$. 
\end{proof}

~\\
{\bf \noindent Finishing the proof of Theorem \ref{thm:specc}.}
Theorem \ref{thm:specc} is a direct consequence of Lemma \ref{lem: structure lem}, \ref{lemma:dev_from_cond_mean} and \ref{lemma:cluster_flips}.

\subsection{Proof of Theorem \ref{thm:refine_glob}} \label{prf:thm:refine_glob}
The high-level idea of this proof is that we can bound node-wise errors separately due to the additive form of the loss function. While such an idea has appeared in \cite{gao2017achieving,gao2018community}, the implementations of this idea is considerably more complicated in our case due to the combinatorial structure induced by the presence of inhomogeneity across layers. 

Note that in Stage \RN{2} of Algorithm \ref{alg: provable_refinement}, we modify the $i$-th coordinate of $\tzzz{\star, -i}$ and $\tzzz{\ell, -i}$. To avoid confusions, we let $\tzzz{\star, -i} = \tzzz{\ell, -i}$ be the initial estimators computed in Stage \RN{1} (but before Stage \RN{2}), whose $i$-th coordinates are zero by construction, and we let $\barzzz{\star, -i}, \barzzz{\ell, -i}$ be the estimators computed in Stage \RN{2}, whose $i$-th coordinates satisfy 
\begin{equation}
  \label{eq:provable_refinement_obj_func}
  (\barzzz{\star, -i}_i, \barzzz{1, -i}_i, \hdots, \barzzz{L, -i}_i) = \argmax_{\substack{s_\star \in\{\pm 1\} \\ s_\ell \in \{\pm 1\}~\forall \ell\in[L] }} \sum_{\ell\in[L]} \fff{\ell}_i(s_\star, s_\ell, \tzzz{\star, -i}),
\end{equation}
and whose rest of the coordinates agree with $\tzzz{\star, -i}_{-i}$.

We start by presenting two preliminary results.

\begin{lemma}
  \label{lemma:align}
  Fix any $\bz, \bz'\in\{\pm 1\}^n$ and assume there exists a constant $C\geq 1$ such that
  $$
    \min_{s\in\{\pm 1\}} \# \{i\in[n]: \bz_i = s\} \geq \frac{n}{2C}, ~~~ \min_{s\in\{\pm 1\}} \# \{i\in[n]: \bz_i' = s\} \geq \frac{n}{2C}, ~~~ \min_{\pi\in\{\pm 1\}} d_\ham(\bz,\pi \bz') \leq \frac{1}{2C}.
  $$
  Define $\vartheta: \{\pm 1\} \to \{\pm 1\}$ by
  \begin{equation}
    \label{eq:align_map}
    \vartheta(r) = \argmax_{s\in\{\pm 1\}} \#\bigg\{\{j\in[n]: \bz_j = s\} \bigcap \{j\in[n]: \bz_j' = r\}\bigg\}.
  \end{equation}    
  Then $\vartheta$ is a bijection and hence can be identified by $\vartheta\in\{\pm 1\}$ with $+ 1$ being the identity map. Moreover, we have
  $$
    d_\ham(\bz, \vartheta \bz') = \min_{\pi\in\{\pm 1\}} d_\ham(\bz, \pi \bz').
  $$
\end{lemma}
\begin{proof}
This is Lemma 4 in \cite{gao2017achieving}.     
\end{proof}

\begin{proposition}
  \label{prop:refine_glob_nodewise}
  Assume $\rho = o(1)$, $q _\ell < p_\ell \leq (Cq_\ell) \land (1-c), \forall \ell\in[L]$, $\beta= 1+o(1)$, and $\log L\ll n^{c'}$ for some constants $C > 1$ and $c, c' \in (0, 1)$. In addition, assume there exists a sequence $\delta_n=o(1)$ and constants $\ep_{\init}>0, C'>0$ such that $\forall i\in[n], \exists \pi_i\in\{\pm 1\}$ which makes the following holds:
  \begin{equation}
    \label{eq:ind_consistent_init}
    \bbP\bigg(d_\ham(\pi_i \tzzz{\star, -i}_{-i} ,\zzz{\ell}_{-i})\leq n\delta_n ~ \forall \ell\in[n]\bigg) \geq 1 - C' n^{-(1+\ep_{\init})},
  \end{equation}
  Then there exists another sequence $\delta_n'=o(1)$ and an absolute constant $C''>0$ such that for any $i\in[n]$, we have 
    \begin{align}
    \label{eq:upper_bound_z_star}
    \bbP(\pi_i\barzzz{\star, -i}_i \neq  \bz^\star_i) 
    & \leq  C''n^{-(1+\ep_{\init})}+ \sum_{S\subseteq[L]}e^{-(1-\delta_n') \globinfo_S}
  \end{align}
\end{proposition}
\begin{proof}
  See Appendix \ref{prf:prop:refine_glob_nodewise}. 
\end{proof}

By Assumption \ref{assump:consistent_init}, for any $i\in[n]$, there exists $\pi_i\in\{\pm 1\}$ such that 
\begin{equation}
  \label{eq:consist_init_leave_i_out}
  \bbP\bigg(d_\ham(\pi_i \tzzz{\star, -i}_{-i}, \bz^\star_{-i}) \leq n\eta_{\init, n-1}\bigg) \geq 1- \calO\left(n^{-(1+\ep_{\init})}\right).
\end{equation}  
Since $d_\ham(\bz^\star, \zzz{\ell})$ is the i.i.d. sum of $n$ Bernoulli random variables, an application of Chernoff bound gives $d_\ham(\bz^\star, \zzz{\ell})\geq n\rho + nt$ with probability at most $2e^{-\calO(nt^2)}$ for any $t > 0$. Choosing $t = n^{-(1-c')/2}$ and invoking a union bound over all layers, we conclude that
\begin{equation}
  \label{eq:vanishing_flips}
  \bbP\bigg(d_\ham(\bz^\star, \zzz{\ell}) \leq n\cdot (\rho + n^{-(1-c')/2}) ~ \forall \ell\in[L]\bigg) \geq 1 - e^{-\calO(n^{c'}) + \log L} \geq 1 -  \calO\left(n^{-(1+\ep_{\init})}\right),
\end{equation}  
where the last inequality holds by $\log L \ll n^{c'}$. In particular, on the union of the two high probability events in \eqref{eq:consist_init_leave_i_out} and \eqref{eq:vanishing_flips}, for any fixed $i\in[n]$ and uniformly over $\ell\in[L]$, we have
$$
  d_\ham(\pi_i\tzzz{\star, -i}_{-i}, \zzz{\ell}_{-i}) \leq d_\ham(\pi_i \tzzz{\star, -i}_{-i}, \bz^\star_{-i}) + d_\ham(\bz^\star_{-i}, \zzz{\ell}_{-i}) \leq n \cdot(\eta_{\init, n-1} + \rho + n^{-(1-c')/2} + n^{-1}).
$$
Hence, we can invoke Proposition \ref{prop:refine_glob_nodewise} to conclude that for any $i\in[n]$,   \eqref{eq:upper_bound_z_star} holds. 

In the rest of the proof, we assume $\pi_1 = +1$ without loss of generality.
Now, for each $i\in[n]\setminus\{1\}$, we define the map $\vartheta^\star_i:\{\pm 1\}\to \{\pm 1\}$ as in   \eqref{eq:align_map} with $\bz = \barzzz{\star, -1}$ and $\bz' = \barzzz{\star, -i}$. By construction we have
$$
  \bhz^\star_i = \vartheta^\star_i (\barzzz{\star, -i}_i). 
$$
Thus we have
\begin{align*}
  \bbP(\bhz^\star_i\neq \bz^\star_i) & = \bbP(\vartheta^\star_i(\barzzz{\star, -i}_i) \neq \bz^\star_i)\\
  \label{eq:nodewise_error_alignment}
  & \leq \bbP(\pi_i \barzzz{\star, -i}_i \neq \bz^\star_i) + \bbP(\vartheta^\star_i(\barzzz{\star, -i})\neq \pi_i \barzzz{\star, -i}_i).\numberthis
\end{align*}
On the following event:
\begin{equation}
  \label{eq:Ei_cap_E1}
  \bigg\{ d_\ham(\pi_i \barzzz{\star, -i}_{-i}, \bz^\star_{-i})\leq n\eta_{\init, n-1} \bigg\}   \bigcap \bigg\{ d_\ham(\barzzz{\star, -1}_{-1}, \bz^\star_{-1})\leq n\eta_{\init, n-1}\bigg\},
\end{equation}  
we have 
\begin{equation}
  \label{eq:two_layer_alignment}
  d_\ham(\barzzz{\star, -1}, \barzzz{\star, -i})\leq n\cdot (2\eta_{\init, n-1} + 2n^{-1}).
\end{equation}
Invoking Lemma \ref{lemma:align} gives that on the above event, $\vartheta^\star_i$ is a bijection, and 
$$
  \vartheta^\star_i = \argmin_{\pi\in\{\pm 1\}} d_\ham(\barzzz{\star, -1}, \pi \barzzz{\star ,-i}) = \pi_i
$$
where we have regarded $\vartheta^\star_i$ as a $\{\pm 1\}$-valued scaler, with $+ 1$ representing the identity map, and the last equality follows from \eqref{eq:two_layer_alignment}. In particular, we know that $\vartheta^\star_i(\barzzz{\star, -i}) = \pi_i \barzzz{\star, -i}_i$ on the event \eqref{eq:Ei_cap_E1}. Since this event happens with probability at least $1-\calO\left(n^{1+\ep_{\init}}\right)$, from \eqref{eq:nodewise_error_alignment} we get
\begin{align*}
  \bbP(\bhz^\star_i \neq \bz^\star_i) & \leq  C_1 n^{-(1+\ep_{\init})}+ \sum_{S\subseteq[L]} e^{-(1-\overline{\delta}_n) \globinfo_S},
\end{align*}  
where $C_1>0$ is an absolute constant and $\overline{\delta}_n=o(1)$. Let us set
$$
  \overline{\delta}_n' =  \bigg[\log \bigg(\frac{1}{\sum_{S\subseteq [L]} \exp\{-(1-\overline{\delta}_n\globinfo_S)\}}\bigg)\bigg]^{-1/2},
$$
which tends to zero as $n$ tends to infinity by   \eqref{eq:diverging_glob_snr_sum}. By Markov's inequality, we have
\begin{align*}
  & \bbP\bigg(\calL(\bhz^\star, \bz^\star) \geq \big(\sum_{S\subseteq[L]} e^{-(1-\overline{\delta}_n) \globinfo_S}\big)^{1-\overline{\delta}_n'}\bigg)\\
  & \leq \big(\sum_{S\subseteq[L]} e^{-(1-\overline{\delta}_n) \globinfo_S}\big)^{\overline{\delta}_n'-1} \cdot \frac{1}{n}\sum_{i\in[n]} \bbP(\bhz^\star_i \neq \bz^\star_i) \\
  & \leq \big(\sum_{S\subseteq[L]} e^{-(1-\overline{\delta}_n) \globinfo_S}\big)^{\overline{\delta}_n'} + C_1 n^{-(1+\ep_{\init})} \big(\sum_{S\subseteq[L]} e^{-(1-\overline{\delta}_n) \globinfo_S}\big)^{\overline{\delta}_n'-1}\\
  & = \exp\bigg\{- \sqrt{\log \bigg(\frac{1}{\sum_{S\subseteq[L]} e^{-(1-\overline{\delta}_n) \globinfo_S} }\bigg)}\bigg\} 
  + C_1 {n^{-(1+\ep_{\init})}} \cdot \exp\bigg\{ (1-\overline{\delta}_n')\log\bigg(\frac{1}{\sum_{S\subseteq[L]} e^{-(1-\overline{\delta}_n)\globinfo_S}}\bigg) \bigg\} 
\end{align*}
If 
$$
  \exp\bigg\{ (1-\overline{\delta}_n')\log\bigg(\frac{1}{\sum_{S\subseteq[L]} e^{-(1-\overline{\delta}_n)\globinfo_S}}\bigg) \bigg\}  \leq n^{1+\frac{\ep_{\init}}{2}},
$$
then we get
$$
  \bbP\bigg(\calL(\bhz^\star, \bz^\star) \geq \big(\sum_{S\subseteq[L]} e^{-(1-\overline{\delta}_n) \globinfo_S}\big)^{1-\overline{\delta}_n'}\bigg) \leq \exp\bigg\{- \sqrt{\log \bigg(\frac{1}{\sum_{S\subseteq[L]} e^{-(1-\overline{\delta}_n) \globinfo_S} }\bigg)}\bigg\} 
  + C_1 {n^{-\ep_{\init}/2}} \to 0
$$
as $n\to \infty$. Otherwise, we can proceed by
\begin{align*}
  \bbP\bigg(\calL(\bhz^\star, \bz^\star) \geq \big(\sum_{S\subseteq[L]} e^{-(1-\overline{\delta}_n) \globinfo_S}\big)^{1-\overline{\delta}_n'}\bigg)
  & \leq \bbP\bigg(\calL(\bhz^\star, \bz^\star) > 0\bigg) \\
  & \leq \sum_{i\in[n]} \bbP(\bhz^\star_i \neq \bz^\star_i)\\
  & \leq C_1n^{-\ep_{\init}} + n \exp\bigg\{-\log\bigg( \frac{1}{\sum_{S\subseteq[L]} e^{-(1-\overline{\delta}_n) \globinfo_S}}\bigg)\bigg\}\\
  & \leq C_1n^{-\ep_{\init}} + n \exp\bigg\{-(1-\overline{\delta}_n')\log\bigg( \frac{1}{\sum_{S\subseteq[L]} e^{-(1-\overline{\delta}_n) \globinfo_S}}\bigg)\bigg\}\\
  & \leq C_1 n^{-\ep_{\init}} + n^{-\ep_{\init}/2}\\
  & \to 0.
\end{align*}
Thus, in either case,   \eqref{eq:refine_glob} holds, and the proof is concluded.

\subsubsection{Proof of Proposition \ref{prop:refine_glob_nodewise}}\label{prf:prop:refine_glob_nodewise}
Fix $i\in[n]$ and we without loss of generality assume $\pi_i = 1$. 
Let
\begin{align}
  \label{eq:Ei}
  E_i & := \bigg\{d_\ham(\tzzz{\star, -i}_{-i}, \zzz{\ell}_{-i})\leq n\delta_n ~ \forall\ell\in[L]\bigg\}, 
\end{align}
which happens with probability at least $1- Cn^{-(1+\ep_{\init})}$ by assumption. For $\nnn{\ell, -i}_{\pm}:= \sum_{j\neq i}\indc{\zzz{\ell}_j = \pm 1}$, we define
\begin{equation}
  \label{eq:Fi}
  F_i  := \bigg\{\bigg|(\nnn{\ell, -i}_+ - \nnn{\ell, -i}_-) - (1-2\rho)(n^{\star}_+ - n^{\star}_-)\bigg| \leq n\cdot (n^{-(1-c')/2} + n^{-1}) ~ \forall \ell\in[L]\bigg\}.
\end{equation}  
We claim that the event $F_i$ also happens with high probability.
\begin{lemma}
  \label{lemma:high_prob_Fi}
  If $\log L\ll n^{c'}$, then there exists an absolute constant $c''>0$ such that
  $$
    \bbP\big(\bigcap_{i\in[n]} F_i\big) \geq  1 - e^{-c''n^{c'}}.
  $$
\end{lemma}
\begin{proof}
  Recall that 
  $$
    \nnn{\ell}_+ - \nnn{\ell}_- = \sum_{i\in[n]} \indc{\zzz{\ell}_i = +1} - \indc{\zzz{\ell}_i = -1} = \sum_{i: \bz^\star_i = +1} \textnormal{Rad}(1-\rho) + \sum_{i: \bz^\star_i = -1} \textnormal{Rad}(\rho),
  $$
  where $\textnormal{Rad}(p)$ is a Rademacher random variable with positive probability $p$. Since the above display has mean $(1-2\rho)(n^{\star}_+ - n^{\star}_-)$, by Hoeffding's inequality, for any $t>0$ we have
  $$
    \bbP\bigg((\nnn{\ell}_+ - \nnn{\ell}_-) - (1-2\rho)(n^{\star}_+ - n^{\star}_-) \geq nt\bigg) \leq 2 e^{-\calO(nt^2)}.
  $$
  Setting $t = n^{-(1-c')/2}$ and using a union bound over all $\ell\in[L]$, we get 
  $$
    \bbP\bigg((\nnn{\ell}_+ - \nnn{\ell}_-) - (1-2\rho)(n^{\star}_+ - n^{\star}_-) \leq n\cdot n^{-(1-c')/2} ~ \forall \ell\in[L]\bigg) \geq 1-2Le^{-\calO(n^{c'})} \geq 1-e^{-\calO(n^{c'})},
  $$
  where the last inequality is by $\log L\ll n^{c'}$. Note that by construction, we have
  $$
    \bigg|(\nnn{\ell, -i}_+ - \nnn{\ell, -i}_-) - (\nnn{\ell}_+ - \nnn{\ell}_-)\bigg| = 1
  $$
  for any $i\in[n]$. An application of the triangle inequality gives
  $$
    \bigg|(\nnn{\ell, -i}_+, \nnn{\ell, -i}_-) - (1-2\rho)(n^{\star}_+ - n^{\star}_-)\bigg| \leq  1 + n\cdot n^{-(1-c')/2} = n\cdot (n^{-(1-c')/2} + n^{-1})
  $$
  with probability at least $1-e^{-\calO(n^{c'})}$, and this is the desired result.
\end{proof}
The above lemma, along our assumption \eqref{eq:ind_consistent_init}, gives
\begin{equation}
  \label{eq:high_prob_Ei_Fi}
  \bbP(E_i\cap F_i) \geq 1 - \calO\left(n^{-(1+\ep_{\init})}\right) - e^{-\calO(n^{c'})} \geq 1- C'' n^{-(1+\ep_{\init})}
\end{equation}
for some constant $C''>0$. Thus, we have
\begin{equation}
  \label{eq:glob_nodewise_err_decomp}
  \bbP(\barzzz{\star, -i}_i\neq \bz^\star_i) = \bbP(\barzzz{\star, -i}_i = -\bz^\star_i) \leq \bbP(\barzzz{\star, -i}_i = -\bz^\star_i, E_i\cap F_i) + C''n^{-(1+\ep_{\init})}. 
\end{equation}
We can decompose the probability via
\begin{align}
  \label{eq:glob_nodewise_err_decomp_all_subsets}
  \bbP(\barzzz{\star, -i}_i = -\bz^\star_i, E_i\cap F_i)  & = \sum_{S \subseteq [L]} \bbP(\barzzz{\star, -i}_i = -\bz^\star_i , \hzzz{S}_i = -\zzz{S}_i, \hzzz{S^c}_i = \zzz{S^c}_i, E_i\cap F_i),
\end{align}
where we used the shorthand notation $\zzz{S}_i:= \{\zzz{\ell}_i: \ell \in S\}$ for any $S \subseteq[L]$. Occurrence of the event in the right-hand side above implies that 
$$
  \sum_{\ell\in S} \fff{\ell}_i(-\bz^\star_i, -\zzz{\ell}_i, \tzzz{\star, -i}) + \sum_{\ell \in S^c} \fff{\ell}_i(-\bz^\star_i, \zzz{\ell}_i , \tzzz{\star, -i}) \geq \sum_{\ell\in [L]} \fff{\ell}_i(\bz^\star_i, \zzz{\ell}_i, \tzzz{\star, -i}).
$$
Hence, we have
\begin{align*}
  & \bbP(\barzzz{\star, -i}_i = -\bz^\star_i, E_i\cap F_i)  \\
  & \leq \sum_{S \subseteq [L]} 
  \bbP\bigg\{ \log\bigg(\frac{1-\rho}{\rho}\bigg) \cdot \bigg(\#\{\ell\in S^c: \bz^{\star}_i = -\zzz{\ell}_i\} - \#\{\ell\in S^c: \bz^{\star}_i = \zzz{\ell}_i\}\bigg) \\
  & \qquad + \sum_{\ell\in S}\sum_{\substack{j\neq i:  \tzzz{\star, -i}_j = -\zzz{\ell}_i}} \bigg[\log\bigg(\frac{p_\ell (1-q_\ell)}{q_\ell(1-p_\ell)}\bigg) \AAA{\ell}_{ij} + \log\bigg(\frac{1-p_\ell}{1-q_\ell}\bigg) \bigg] \\
  & \qquad - \sum_{\ell\in S} \sum_{j\neq i : \tzzz{\star, -i}_j = \zzz{\ell}_i} \bigg[\log\bigg(\frac{p_\ell(1-q_\ell)}{q_\ell(1-p_\ell)}\bigg) \AAA{\ell}_{ij} + \log\bigg(\frac{1-p_\ell}{1-q_\ell}\bigg)\bigg] \geq  0  ~ \textnormal{and } E_i\cap F_i \bigg\}.
\end{align*}
Note that $E_i$ and $F_i$ are both independent of $\{\zzz{\ell}_i\}_{\ell=1}^L$. So we can decompose the above probability by conditioning on the value of $\#\{\ell\in S^c: \bz^\star_i = -\zzz{\ell}_i\} - \#\{\ell\in S^c: \bz^\star_i = \zzz{\ell}_i\}$:
\begin{align*}
  & \bbP(\barzzz{\star, -i}_i = -\bz^\star_i, E_i\cap F_i)  \\
  & \leq \sum_{S \subseteq [L]} \sum_{x \in\{ -|S^c|+2k: 0\leq k \leq |S^c|\}} \binom{|S^c|}{\frac{|S^c| + x}{2}} \bigg(\frac{1-\rho}{\rho}\bigg)^{{x}/{2}}\bigg(\rho(1-\rho)\bigg)^{|S^c|/2}\\
  & \qquad \times \bbP\bigg\{ -x\log\bigg(\frac{1-\rho}{\rho}\bigg) 
  + \sum_{\ell\in S}\sum_{\substack{j\neq i:  \tzzz{\star, -i}_j = -\zzz{\ell}_i}} \bigg[\log\bigg(\frac{p_\ell (1-q_\ell)}{q_\ell(1-p_\ell)}\bigg) \AAA{\ell}_{ij} + \log\bigg(\frac{1-p_\ell}{1-q_\ell}\bigg)\bigg] \\
  & \qquad \qquad - \sum_{\ell\in S} \sum_{j\neq i : \tzzz{\star, -i}_j = \zzz{\ell}_i} \bigg[\log\bigg(\frac{p_\ell(1-q_\ell)}{q_\ell(1-p_\ell)}\bigg) \AAA{\ell}_{ij} + \log\bigg(\frac{1-p_\ell}{1-q_\ell}\bigg)\bigg] \geq  0  ~ \textnormal{and } E_i\cap F_i \bigg\}.
\end{align*}
We further decompose the above probability according to the orientations of $\zzz{\ell}_i$'s for $\ell\in S$:
\begin{align*}
  & \bbP(\barzzz{\star, -i}_i = -\bz^\star_i, E_i\cap F_i)  \\
  & \leq \sum_{S \subseteq [L]} \sum_{x \in\{ -|S^c|+2k: 0\leq k \leq |S^c|\}} \binom{|S^c|}{\frac{|S^c| + x}{2}} \bigg(\frac{1-\rho}{\rho}\bigg)^{{x}/{2}}\bigg(\rho(1-\rho)\bigg)^{|S^c|/2} \sum_{\xi\in\{\pm 1\}^S}\bbP(\zzz{S}_i =  \bz^\star_i \xi )\\
  & \qquad \times \bbP\bigg\{ -x\log\bigg(\frac{1-\rho}{\rho}\bigg)  
  + \sum_{\ell\in S}\sum_{\substack{j\neq i:  \tzzz{\star, -i}_j = -\xi_\ell \bz^\star_i}} \bigg[\log\bigg(\frac{p_\ell (1-q_\ell)}{q_\ell(1-p_\ell)}\bigg) \AAA{\ell}_{ij} + \log\bigg(\frac{1-p_\ell}{1-q_\ell}\bigg) \bigg]\\
  & \qquad \qquad - \sum_{\ell\in S} \sum_{j\neq i : \tzzz{\star, -i}_j = \xi_\ell\bz^\star_i} \bigg[\log\bigg(\frac{p_\ell(1-q_\ell)}{q_\ell(1-p_\ell)}\bigg) \AAA{\ell}_{ij} + \log\bigg(\frac{1-p_\ell}{1-q_\ell}\bigg)\bigg]
  \geq  0  ~ \textnormal{and } E_i\cap F_i \bigg\} \\
  & = \sum_{S \subseteq [L]} \sum_{x \in\{ -|S^c|+2k: 0\leq k \leq |S^c|\}} \binom{|S^c|}{\frac{|S^c| + x}{2}} \exp\bigg\{ -|S^c|\log\frac{1}{\sqrt{\rho(1-\rho)}} + x\log \sqrt{\frac{1-\rho}{\rho}}\bigg\} \\
  & \qquad \times \sum_{\xi\in\{\pm 1\}^S}\bbP(\zzz{S}_i =  \bz^\star_i \xi ) \cdot \bbE_{\{\zzz{\ell}_{-i}\}_{\ell=1}^L}\bigg[\bbP\bigg(\sakura  ~\bigg|~ \{\zzz{\ell}_{-i}\}_{\ell=1}^L \bigg)\bigg],
\end{align*}
where
\begin{align*}
  \sakura &:=  \bigg\{-x\log\bigg(\frac{1-\rho}{\rho}\bigg)  
  + \sum_{\ell\in S}\sum_{\substack{j\neq i:  \tzzz{\star, -i}_j = -\xi_\ell \bz^\star_i}} \bigg[\log\bigg(\frac{p_\ell (1-q_\ell)}{q_\ell(1-p_\ell)}\bigg) \AAA{\ell}_{ij} + \log\bigg(\frac{1-p_\ell}{1-q_\ell}\bigg)\bigg] \\
  & \qquad - \sum_{\ell\in S} \sum_{j\neq i : \tzzz{\star, -i}_j = \xi_\ell\bz^\star_i} \bigg[\log\bigg(\frac{p_\ell(1-q_\ell)}{q_\ell(1-p_\ell)}\bigg) \AAA{\ell}_{ij} + \log\bigg(\frac{1-p_\ell}{1-q_\ell}\bigg)\bigg] \geq  0  ~ \textnormal{and } E_i\cap F_i\bigg\}.
\end{align*}
Invoking Markov's inequality, for $t = t_{S, x, \xi} \in [0, 1]$ whose value will be specified later, we can bound the conditional probability by
\begin{align*}
  & \bbP\bigg(\sakura~\bigg|~ \{\zzz{\ell}_{-i}\}_{\ell=1}^L\bigg)\\
  & \leq \bbE\bigg\{ \exp\bigg\{ -t x\log \bigg(\frac{1-\rho}{\rho}\bigg)
  + t\sum_{\ell\in S}\sum_{\substack{j\neq i:  \tzzz{\star, -i}_j = -\xi_\ell \bz^\star_i}} \bigg[\log\bigg(\frac{p_\ell (1-q_\ell)}{q_\ell(1-p_\ell)}\bigg) \AAA{\ell}_{ij} + \log\bigg(\frac{1-p_\ell}{1-q_\ell}\bigg)\bigg]\\
  & \qquad -t \sum_{\ell\in S} \sum_{j\neq i : \tzzz{\star, -i}_j = \xi_\ell\bz^\star_i} \bigg[\log\bigg(\frac{p_\ell(1-q_\ell)}{q_\ell(1-p_\ell)}\bigg) \AAA{\ell}_{ij} + \log\bigg(\frac{1-p_\ell}{1-q_\ell}\bigg)\bigg]\bigg\}  \cdot \Indc_{E_i\cap F_i}~ \bigg| ~ \{\zzz{\ell}_{-i}\}_{\ell=1}^L\bigg\}
\end{align*}  
Let us define
\begin{align}
  \label{eq:ml-i-}
  \mmm{\ell, -i}_-  &:= \#\{j\neq i: \tzzz{\star, -i}_j = -\xi_\ell \bz^\star_i\} ,\\
  \label{eq:tml-i-}
  \tmmm{\ell, -i}_- & := \# \{j\neq i: \tzzz{\star, -i}_j = -\xi_\ell \bz^\star_i, \tzzz{\star, -i}_j = \zzz{\ell}_j\} \\
  \label{eq:ml-i+}
  \mmm{\ell, -i}_+ & := \#\{j\neq i: \tzzz{\star, -i}_j = \xi_\ell \bz^\star_i\} \\
  \label{eq:tml-i+}
  \tmmm{\ell, -i} & := \#\{j\neq i: \tzzz{\star, -i}_j = \xi_\ell \bz^\star_i , \tzzz{\star, -i}_j = \zzz{\ell}_j\}.
\end{align}
Then, we have 
\begin{align*}
  & \bbP\bigg(\sakura~\bigg|~ \{\zzz{\ell}_{-i}\}_{\ell=1}^L\bigg)\\
  & \leq \exp\bigg\{-tx \log\bigg(\frac{1-\rho}{\rho}\bigg)\bigg\} \\
  & ~~\times \bbE\bigg[  
  \prod_{\ell \in S} \exp\bigg\{ t \log\bigg(\frac{p_\ell(1-q_\ell)}{q_\ell(1-p_\ell)}\bigg)\\
  & \qquad\qquad\qquad\qquad \times \bigg(
  \sum_{\substack{j\neq i\\ \tzzz{\star, -i}_j = -\xi_\ell \bz^\star_i \\ \tzzz{\star, -i}_{j} = \zzz{\ell}_j}} \AAA{\ell}_{ij}
  + \sum_{\substack{j\neq i\\ \tzzz{\star, -i}_j = -\xi_\ell \bz^\star_i \\ \tzzz{\star, -i}_{j} = -\zzz{\ell}_j}} \AAA{\ell}_{ij}
  - \sum_{\substack{j\neq i\\ \tzzz{\star, -i}_j = \xi_\ell \bz^\star_i \\ \tzzz{\star, -i}_{j} = \zzz{\ell}_j}} \AAA{\ell}_{ij}
  - \sum_{\substack{j\neq i\\ \tzzz{\star, -i}_j = \xi_\ell \bz^\star_i \\ \tzzz{\star, -i}_{j} = -\zzz{\ell}_j}} \AAA{\ell}_{ij}
  \bigg)\\
  & \qquad\qquad\qquad\qquad + t \log\bigg(\frac{1-p_\ell}{1-q_\ell}\bigg)(\mmm{\ell, -i}_- - \mmm{\ell, -i}_+)\bigg\}  \cdot \Indc_{E_i \cap F_i}
  ~\bigg|~ \{\zzz{\ell}_{-i}\}_{\ell=1}^L\bigg]\\
  & = \exp\bigg\{-tx \log\bigg(\frac{1-\rho}{\rho}\bigg)\bigg\}  \\
  & ~~ \times \prod_{\ell\in S}\bbE\bigg\{ \bigg[q_\ell \bigg(\frac{p_\ell(1-q_\ell)}{q_\ell(1-p_\ell)}\bigg)^t + (1-q_\ell)\bigg]^{\tmmm{\ell, -i}_-} 
  \cdot \bigg[p_\ell\bigg(\frac{p_\ell(1-q_\ell)}{q_\ell(1-p_\ell)}\bigg)^t + (1-p_\ell)\bigg]^{\mmm{\ell, -i}_- - \tmmm{\ell, -i}_-}  \\
  & \qquad \times \bigg[p_\ell\bigg(\frac{q_\ell(1-p_\ell)}{p_\ell(1-q_\ell)}\bigg)^t + (1-p_\ell)\bigg]^{\tmmm{\ell, -i}_+}
  \cdot \bigg[ q_\ell\bigg(\frac{q_\ell(1-p_\ell)}{p_\ell(1-q_\ell)}\bigg)^t + (1-q_\ell) \bigg]^{\mmm{\ell, -i}_+ - \tmmm{\ell, -i}_+}\\
  & \qquad \times \bigg( \frac{1-p_\ell}{1-q_\ell} \bigg)^{t(\mmm{\ell, -i}_- - \mmm{\ell, -i}_+)}
  \cdot \Indc_{E_i \cap F_i}
  ~\bigg|~ \{\zzz{\ell}_{-i}\}_{\ell=1}^L\bigg\}\\
  & = \exp\bigg\{-tx \log\bigg(\frac{1-\rho}{\rho}\bigg)\bigg\}  \\
  & ~~ \times \prod_{\ell\in S} \bbE\bigg\{
  \bigg(\frac{1-p_\ell}{1-q_\ell}\bigg)^{t(\mmm{\ell, -i}_- - \mmm{\ell,-i}_+)} 
  \cdot \bigg[ q_\ell \bigg(\frac{p_\ell(1-q_\ell)}{q_\ell(1-p_\ell)}\bigg)^t + (1-q_\ell) \bigg]^{\mmm{\ell, -i}_-}\\
  & \qquad \times \bigg[ p_\ell\bigg(\frac{q_\ell(1-p_\ell)}{p_\ell(1-q_\ell)}\bigg)^t + (1-p_\ell) \bigg]^{\mmm{\ell, -i}_+}
  \cdot \bigg[ \frac{p_\ell \bigg(\frac{p_\ell(1-q_\ell)}{q_\ell(1-p_\ell)}\bigg)^t + (1-p_\ell)}{q_\ell \bigg(\frac{p_\ell(1-q_\ell)}{q_\ell(1-p_\ell)}\bigg)^t + (1-q_\ell)} \bigg]^{\mmm{\ell, -i}_- - \tmmm{\ell, -i}_-}\\
  & \qquad \times \bigg[ \frac{q_\ell \bigg(\frac{q_\ell(1-p_\ell)}{p_\ell(1-q_\ell)}\bigg)^t + (1-q_\ell)}{p_\ell \bigg(\frac{q_\ell(1-p_\ell)}{p_\ell(1-q_\ell)}\bigg)^t + (1-p_\ell)}\bigg]^{\mmm{\ell, -i}_+ - \tmmm{\ell, -i}_+}
  \cdot \Indc_{E_i\cap F_i}
  ~\bigg|~ \{\zzz{\ell}_{-i}\}_{\ell=1}^L\bigg\}\\
  & = \exp\bigg\{-tx \log\bigg(\frac{1-\rho}{\rho}\bigg)\bigg\} 
  \times \prod_{\ell\in S}\bbE\bigg\{ \mathscr{T}_{1, \ell}  \times \mathscr{T}_{2, \ell} \times \mathscr{T}_{3, \ell}
  \cdot \Indc_{E_i\cap F_i}
  ~\bigg|~ \{\zzz{\ell}_{-i}\}_{\ell=1}^L\bigg\},
\end{align*}
where
\begin{align*}
  \mathscr{T}_{1, \ell} & := \bigg(\frac{1-p_\ell}{1-q_\ell}\bigg)^{t(\mmm{\ell, -i}_- - \mmm{\ell, -i}_+)} 
  \times \bigg[q_\ell\bigg(\frac{p_\ell(1-q_\ell)}{q_\ell(1-p_\ell)}\bigg)^t + (1-q_\ell)\bigg]^{\frac{\mmm{\ell, -i}_- - \mmm{\ell, -i}_+}{2}}  \\
  & \qquad 
  \times \bigg[p_\ell \bigg(\frac{q_\ell(1-p_\ell)}{p_\ell(1-q_\ell)}\bigg)^t + (1-p_\ell)\bigg]^{\frac{\mmm{\ell, -i}_+ - \mmm{\ell, -i}_-}{2}} \\
  \mathscr{T}_{2, \ell} & := \bigg[ q_\ell\bigg(\frac{p_\ell(1-q_\ell)}{q_\ell(1-p_\ell)}\bigg)^t + (1-q_\ell) \bigg]^{\frac{\mmm{\ell, -i}_- + \mmm{\ell, -i}_+}{2}} 
  \times \bigg[p_\ell\bigg(\frac{q_\ell(1-p_\ell)}{p_\ell(1-q_\ell)}\bigg)^t + (1-p_\ell)\bigg]^{\frac{\mmm{\ell, -i}_+ + \mmm{\ell, -i}_-}{2}} \\
  \mathscr{T}_{3, \ell} & := \bigg[ \frac{p_\ell \bigg(\frac{p_\ell(1-q_\ell)}{q_\ell(1-p_\ell)}\bigg)^t + (1-p_\ell)}{q_\ell \bigg(\frac{p_\ell(1-q_\ell)}{q_\ell(1-p_\ell)}\bigg)^t + (1-q_\ell)} \bigg]^{\mmm{\ell, -i}_- - \tmmm{\ell, -i}_-}
  \times \bigg[ \frac{q_\ell \bigg(\frac{q_\ell(1-p_\ell)}{p_\ell(1-q_\ell)}\bigg)^t + (1-q_\ell)}{p_\ell \bigg(\frac{q_\ell(1-p_\ell)}{p_\ell(1-q_\ell)}\bigg)^t + (1-p_\ell)}\bigg]^{\mmm{\ell, -i}_+ - \tmmm{\ell, -i}_+}.
\end{align*}
In summary, we arrive at
\begin{align}
  & \bbP(\barzzz{\star, -i}_i = -\bz^\star_i , E_i\cap F_i)  \nonumber\\
  & \leq \sum_{S \subseteq [L]} 
  \sum_{x \in\{ -|S^c|+2k: 0\leq k \leq |S^c|\}} 
  \binom{|S^c|}{\frac{|S^c| + x}{2}} 
  \exp\bigg\{ -|S^c|\log\frac{1}{\sqrt{\rho(1-\rho)}} + x(1-2t)\log \sqrt{\frac{1-\rho}{\rho}}\bigg\} \nonumber\\
  \label{eq:glob_nodewise_err_decom_3terms}
  & ~~ \times \sum_{\xi\in\{\pm 1\}^S}\bbP(\zzz{S}_i =  \bz^\star_i \xi ) 
  \cdot \bbE_{\{\zzz{\ell}_{-i}\}_{\ell=1}^L} \bigg\{
  \prod_{\ell\in S}\bbE\bigg[ \mathscr{T}_{1, \ell}\times \mathscr{T}_{2, \ell} \times \mathscr{T}_{3, \ell} \cdot \Indc_{E_i\cap F_i} ~\bigg|~ \{\zzz{\ell}_{-i}\}_{\ell=1}^L \bigg]\bigg\},
\end{align}
We now bound the three terms $\mathscr{T}_{1, \ell}, \mathscr{T}_{2, \ell}$ and $\mathscr{T}_{3, \ell}$ separately.

~\\
{\noindent \bf Bounding the first term.} We can write
\begin{align*}
  \mathscr{T}_{1, \ell} & = \bigg[ \frac{(1-p_\ell)^{2t}}{(1-q_\ell)^{2t}} 
  \times \frac{q_\ell \bigg(\frac{p_\ell(1-q_\ell)}{q_\ell(1-p_\ell)}\bigg)^t + (1-q_\ell)}{p_\ell\bigg( \frac{q_\ell(1-p_\ell)}{p_\ell(1-q_\ell)} \bigg)^t + (1-p_\ell)} \bigg]^{\frac{\mmm{\ell, -i}_+ - \mmm{\ell, -i}_-}{2}} \\
  & = \bigg[ \frac{p_\ell^t q_\ell^{1-t} + (1-p_\ell)^t (1-q_\ell)^{1-t}}{p_\ell^{1-t}q_\ell^t + (1-p_\ell)^{1-t}(1-q_\ell)^t} \bigg]^{\frac{\mmm{\ell, -i}_+ - \mmm{\ell, -i}_-}{2}} \\
  & = \exp\bigg\{\frac{\mmm{\ell, -i}_+ - \mmm{\ell, -i}_-}{2}
  \times \log \bigg(\frac{p_\ell^t q_\ell^{1-t} + (1-p_\ell)^t (1-q_\ell)^{1-t}}{p_\ell^{1-t}q_\ell^t + (1-p_\ell)^{1-t}(1-q_\ell)^t}\bigg)\bigg\}.
\end{align*}
We need the following two lemmas.
\begin{lemma}
  \label{lemma:glob_nodewise_err_decomp_1st_term_part1}
  Under the setup of Proposition \ref{prop:refine_glob_nodewise}, on the event $E_i$, for any $\ell\in S$, we have
  $$
    \bigg|(\mmm{\ell, -i}_+ - \mmm{\ell, -i}_-) - (\nnn{\ell, -i}_+ - \nnn{\ell, -i}_{-})\bigg| \leq 2n\delta_n
  $$
\end{lemma}
\begin{proof}
  Note that for $\ell\in S$, we have $\xi_\ell \bz^\star_i = \zzz{\ell}_i$. By definition \eqref{eq:ml-i+}, we have
  \begin{align*}
    \mmm{\ell, -i}_+ - \nnn{\ell, -i}_+ & = \sum_{j\neq i}\bigg(\indc{\tzzz{\star, -i}_j = \zzz{\ell}_i} - \indc{\zzz{\ell}_j = \zzz{\ell}_i}\bigg)\\
    & = \sum_{j\neq i} \bigg(\indc{\tzzz{\star, -i}_j = \zzz{\ell}_i, \tzzz{\star , -i}_j = \zzz{\ell}_j} + \indc{\tzzz{\star, -i}_j = \zzz{\ell}_i , \tzzz{\star, -i} \neq \zzz{\ell}_j} \\
    & \qquad - \indc{\zzz{\ell}_j = \zzz{\ell}_i , \tzzz{\star, -i}_j = \zzz{\ell}_j } - \indc{\zzz{\ell}_j = \zzz{\ell}_i , \tzzz{\star, -i}_j \neq \zzz{\ell}_j} \bigg)\\
    & = \sum_{j\neq i} \bigg(\indc{\tzzz{\star, -i}_j = \zzz{\ell}_i, \tzzz{\star , -i}_j = \zzz{\ell}_j}  - \indc{\zzz{\ell}_j = \zzz{\ell}_i , \tzzz{\star, -i}_j \neq \zzz{\ell}_j} \bigg)\\
    & \leq \sum_{j\neq i}\indc{\tzzz{\star, -i}_j = \zzz{\ell}_i, \tzzz{\star , -i}_j = \zzz{\ell}_j}  \\ 
    & \leq \#\{j\neq i: \tzzz{\ell, -i}_j \neq \zzz{\ell}_j\} \\
    & \leq n\delta_n
  \end{align*}
  where the last inequality holds on the event $E_i$. On the other hand, we have
  \begin{align*}
    \mmm{\ell, -i}_+ - \nnn{\ell, -i}_+ & = \sum_{j\neq i} \bigg(\indc{\tzzz{\star, -i}_j = \zzz{\ell}_i, \tzzz{\star , -i}_j = \zzz{\ell}_j}  - \indc{\zzz{\ell}_j = \zzz{\ell}_i , \tzzz{\star, -i}_j \neq \zzz{\ell}_j}\bigg)\\
    & \geq -\sum_{j\neq i}  \indc{\zzz{\ell}_j = \zzz{\ell}_i , \tzzz{\star, -i}_j \neq \zzz{\ell}_j} \\
    & \geq -\# \{j\neq i: \tzzz{\star, -i}_j \neq \zzz{\ell}_j\} \geq -n\delta_n,
  \end{align*}
  where the last inequality again holds on the event $E_i$. Thus, we arrive at
  $$
    |\mmm{\ell, -i}_+ - \nnn{\ell, -i}_+| \leq n\delta_n
  $$
  on $E_i$. A similar argument shows that $|\mmm{\ell, -i}_- - \nnn{\ell, -i}_-|\leq n\delta_n$ on $E_i$, and the proof is concluded by invoking the triangle inequality.
\end{proof}
\begin{lemma}
  \label{lemma:glob_nodewise_err_decomp_1st_term_part2}
  Under the setups of Proposition \ref{prop:refine_glob_nodewise}, for any $t\in[0, 1]$, we have
  $$
     \bigg|\log \frac{p_\ell^t q_\ell^{1-t} + (1-p_\ell)^t (1-q_\ell)^{1-t}}{p_\ell^{1-t}q_\ell^t + (1-p_\ell)^{1-t}(1-q_\ell)^t}\bigg| \leq \III{\ell}_{t},
  $$
  where $\III{\ell}_t$ is defined in   \eqref{eq:layerwise_info}.
\end{lemma}
\begin{proof}
  We first show the numerator $p_\ell^t q_\ell^{1-t} + (1-p_\ell)^t (1-q_\ell)^{1-t}\leq 1 $. To do this, we take the derivative w.r.t. $t$:
  $$
    \frac{\partial}{\partial t} \bigg(p_\ell^t q_\ell^{1-t} + (1-p_\ell)^t (1-q_\ell)^{1-t} \bigg)= q_\ell \bigg(\frac{p_\ell}{q_\ell}\bigg)^t \log \frac{p_\ell}{q_\ell} + (1-q_\ell) \bigg(\frac{1-p_\ell}{1-q_\ell}\bigg)^t \log \frac{1-p_\ell}{1-q_\ell}. 
  $$
  Note that since $p_\ell > q_\ell$, the right-hand side above is an increasing function in $t$. So the numerator is a convex function in $t$. This means that its maximum must occur at the boundary, which is at either $t = 0$ or $t = 1$, both of which gives $p_\ell^t q_\ell^{1-t} + (1-p_\ell)^t (1-q_\ell)^{1-t} = 1$. By symmetry, the denominator also satisfies $p_\ell^{1-t}q_\ell^t + (1-p_\ell)^{1-t}(1-q_\ell)^t \leq 1$. Thus, we can proceed by
  \begin{align*}
    & \bigg|\log \frac{p_\ell^t q_\ell^{1-t} + (1-p_\ell)^t (1-q_\ell)^{1-t}}{p_\ell^{1-t}q_\ell^t + (1-p_\ell)^{1-t}(1-q_\ell)^t}\bigg| \\
    & = \bigg|\log\bigg(p_\ell^t q_\ell^{1-t} + (1-p_\ell)^t (1-q_\ell)^{1-t}\bigg) - \log\bigg(p_\ell^{1-t}q_\ell^t + (1-p_\ell)^{1-t}(1-q_\ell)^t\bigg)\bigg| \\
    & \leq - \log \bigg(p_\ell^t q_\ell^{1-t} + (1-p_\ell)^t (1-q_\ell)^{1-t}\bigg) - \log \bigg(p_\ell^{1-t}q_\ell^t + (1-p_\ell)^{1-t}(1-q_\ell)^t\bigg) \\
    & = \III{\ell}_t,
  \end{align*}
  which is the desired result.
\end{proof}

Invoking Lemma \ref{lemma:glob_nodewise_err_decomp_1st_term_part1}, we know the on the event $E_i \cap F_i$, 
\begin{align*}
  & \bigg|(\mmm{\ell, -i}_+ - \mmm{\ell, -i}_-) - (1-2\rho)(n^{\star}_{\zzz{\ell}_i} - n^{\star}_{-\zzz{\ell}_i})\bigg| \\ 
  & \leq \bigg|(\mmm{\ell, -i}_+ - \mmm{\ell, -i}_-) - (\nnn{\ell, -i}_+ - \nnn{\ell, -i}_-)\bigg| + \bigg|(\nnn{\ell, -i}_+ - \nnn{\ell, -i}_-) - (1-2\rho)(n^{\star}_{\zzz{\ell}_i} - n^{\star}_{-\zzz{\ell}_i})\bigg|\\
  & \leq n\cdot ( 2\delta_n + n^{-(1-c')/2} + n^{-1}).
\end{align*}
Combining the above inequality with Lemma \ref{lemma:glob_nodewise_err_decomp_1st_term_part2}, and recalling that $\zzz{\ell}_i = \xi_\ell \bz^\star_i$ for $\ell\in[L]$, we get
\begin{align*}
  & \mathscr{T}_{1, \ell}\cdot \Indc_{E_i\cap F_i} \\
  & \leq \exp\bigg\{\frac{(1-2\rho)(n^{\star}_{\bz^\star_i} - n^{\star}_{-\bz^\star_i})}{2}\cdot \bigg(\indc{\xi_\ell = 1} - \indc{\xi_\ell = -1}\bigg)
  \times \log \bigg(\frac{p_\ell^t q_\ell^{1-t} + (1-p_\ell)^t (1-q_\ell)^{1-t}}{p_\ell^{1-t} q_\ell^t + (1-p_\ell)^{1-t}(1-q_\ell)^t}\bigg)\\
  & \qquad  \qquad + \frac{ n\cdot ( 2\delta_n + n^{-(1-c')/2} + n^{-1} )}{2} \cdot \III{\ell}_t\bigg\}\\
  & \leq \exp\bigg\{   \frac{n\III{\ell}_t}{2} \cdot \bigg((1-2\rho)(\beta  - \beta^{-1}) + n^{-(1-c')/2} + n^{-1} \bigg)\bigg\} \nonumber\\
  \label{eq:glob_nodewise_err_decomp_1st_term}
  & = \exp\{o(1)\cdot n \III{\ell}_t\}, \numberthis
\end{align*}
where the last inequality is by $\beta = 1+o(1)$.

~\\
{\noindent \bf Bounding the second term.} Since $\mmm{\ell, -i}_+ + \mmm{\ell, -i}_- = n-1$, we have
\begin{align*}
  \mathscr{T}_{2, \ell} & = \exp\bigg\{ \frac{n-1}{2}\cdot \log\bigg(p_\ell q_\ell + (1-p_\ell)(1-q_\ell) + [q_\ell(1-p_\ell)]^{1-t}[p_\ell(1-q_\ell)]^t + [q_\ell(1-p_\ell)]^t [p_\ell(1-q_\ell)]^{1-t}\bigg)\bigg\} \\
  \label{eq:glob_nodewise_err_decomp_2ed_term}
  & = \exp\{-(n-1)\III{\ell}_t/2\} \numberthis
\end{align*}

~\\
{\noindent \bf Bounding the third term.} With some algebra, one can show that
\begin{align*}
   \frac{p_\ell \bigg(\frac{p_\ell(1-q_\ell)}{q_\ell(1-p_\ell)}\bigg)^t + (1-p_\ell)}{q_\ell \bigg(\frac{p_\ell(1-q_\ell)}{q_\ell(1-p_\ell)}\bigg)^t + (1-q_\ell)}  
  & = 1 + \frac{(p_\ell - q_\ell)\bigg[\bigg(1 + \frac{p_\ell - q_\ell}{q_\ell - p_\ell q_\ell}\bigg)^t - 1\bigg]}{1-q_\ell + q_\ell\bigg(1 + \frac{p_\ell - q_\ell}{q_\ell - p_\ell q_\ell}\bigg)^t}.
\end{align*}
Since $(1+ \frac{p_\ell-q_\ell}{q_\ell-p_\ell q_\ell})^t - 1\leq \frac{p_\ell-q_\ell}{q_\ell-p_\ell q_\ell}$ when $t \in[0, 1]$, along with the assumption that $q _\ell < p_\ell \leq 1- c$, the right-hand side above can be bounded by
$$
  \frac{p_\ell \bigg(\frac{p_\ell(1-q_\ell)}{q_\ell(1-p_\ell)}\bigg)^t + (1-p_\ell)}{q_\ell \bigg(\frac{p_\ell(1-q_\ell)}{q_\ell(1-p_\ell)}\bigg)^t + (1-q_\ell)} 
  \leq 1 + \calO\bigg(\frac{(p_\ell-q_\ell)^2}{p_\ell}\bigg) \leq \exp\bigg\{\calO\bigg(\frac{(p_\ell-q_\ell)^2}{p_\ell}\bigg)\bigg\}.
$$
A similar argument shows that 
$$
  \frac{q_\ell \bigg(\frac{q_\ell(1-p_\ell)}{p_\ell(1-q_\ell)}\bigg)^t + (1-q_\ell)}{p_\ell \bigg(\frac{q_\ell(1-p_\ell)}{p_\ell(1-q_\ell)}\bigg)^t + (1-p_\ell)} \leq \exp\bigg\{\calO\bigg(\frac{(p_\ell-q_\ell)^2}{p_\ell}\bigg)\bigg\}.
$$
Now, on the event $E_i$, we have
\begin{align*}
  |\mmm{\ell, -i}_- - \tmmm{\ell, -i}_-| & = \bigg|\sum_{j\neq i}\bigg(\indc{\tzzz{\star, -i}_j = -\zzz{\ell}_i} - \indc{\tzzz{\star, -i}_j = -\zzz{\ell}_i, \tzzz{\star, -i}_j = \zzz{\ell}_j}\bigg)\bigg|\\
  & \leq \sum_{j\neq i} \indc{\tzzz{\star, -i}_j = -\zzz{\ell}_j} \\
  & \leq n\delta_n,
\end{align*}
and the same bound holds for $|\mmm{\ell, -i}_+ - \tmmm{\ell, -i}_+|$. Thus, we get
\begin{align*}
  \label{eq:glob_nodewise_err_decomp_3ed_term}
  \mathscr{T}_{3, \ell}\cdot \Indc_{E_i} \leq \exp\bigg\{\calO\bigg(\frac{n\delta_n (p_\ell-q_\ell)^2}{p_\ell}\bigg)\bigg\} = \exp\bigg\{ o(1) \cdot n\III{\ell}_{1/2} \bigg\}, \numberthis
\end{align*}
where the last inequality is by Lemma \ref{lemma:asymp_equiv_of_I_t}.

~\\
{\bf \noindent Summarizing the three terms.} Plugging   \eqref{eq:glob_nodewise_err_decomp_1st_term}, \eqref{eq:glob_nodewise_err_decomp_2ed_term} and \eqref{eq:glob_nodewise_err_decomp_3ed_term} to   \eqref{eq:glob_nodewise_err_decom_3terms}, we get
\begin{align*}
  & \bbP(\barzzz{\star, -i}_i = -\bz^\star_i , E_i\cap F_i)  \nonumber\\
  & \leq \sum_{S \subseteq [L]} 
  \sum_{x \in\{ -|S^c|+2k: 0\leq k \leq |S^c|\}} 
  \binom{|S^c|}{\frac{|S^c| + x}{2}} 
  \exp\bigg\{ -|S^c|\log\frac{1}{\sqrt{\rho(1-\rho)}} + x(1-2t_{S, x, \xi})\log \sqrt{\frac{1-\rho}{\rho}}\bigg\} \nonumber\\
  & ~~ \times \sum_{\xi\in\{\pm 1\}^S}\bbP(\zzz{S}_i =  \bz^\star_i \xi ) 
  \cdot \bbE_{\{\zzz{\ell}_{-i}\}_{\ell=1}^L} \bigg\{
  \prod_{\ell\in S}\bbE\bigg[ \exp\big\{ -n\III{\ell}_{t_{S, x, \xi}}/2 + o(1)\cdot n(\III{\ell}_{t_{S, x , \xi}} + \III{\ell}_{1/2})  \big\} ~\bigg|~ \{\zzz{\ell}_{-i}\}_{\ell=1}^L \bigg]\bigg\} \\
  & \leq \sum_{S \subseteq [L]} 
  \sum_{x \in\{ -|S^c|+2k: 0\leq k \leq |S^c|\}} 
  \binom{|S^c|}{\frac{|S^c| + x}{2}} 
  \exp\bigg\{ -|S^c|\log\frac{1}{\sqrt{\rho(1-\rho)}} + x(1-2t_{S, x, \xi})\log \sqrt{\frac{1-\rho}{\rho}}\bigg\} \nonumber\\
  \label{eq:nodewise_err_decomp_3terms_simplified}
  & ~~ \times \sum_{\xi\in\{\pm 1\}^S}\bbP(\zzz{S}_i =  \bz^\star_i \xi ) 
  \cdot \exp\bigg\{ -\frac{n}{2}\sum_{\ell\in S}\III{\ell}_{t_{S, x, \xi}}+ o(1)\cdot n\sum_{\ell\in S}\III{\ell}_{1/2}\bigg\},\numberthis
\end{align*}
where the last inequality is by $\III{\ell}_t\leq \III{\ell}_{1/2}$ for any $t\in[0, 1]$, proved in Lemma \ref{lemma:monotonicity_of_I_t}. This is a summation over all $2^L$ subsets of $[L]$, and we now carefully choose $t_{x, S, \xi}$ to make each summand as tight as possible.

~\\
{\bf \noindent Case A: either $\boldsymbol{|S^c|}$ is even, or $\boldsymbol{\log e|S^c| \geq \sqrt{J_\rho}}$.} In this case, we choose $t_{S, x, \xi} = 1/2$. For any fixed $|S^c|$ falling into this case, the corresponding summand in the right-hand side of   \eqref{eq:nodewise_err_decomp_3terms_simplified} becomes
\begin{align*}
  & \sum_{x \in\{ -|S^c|+2k: 0\leq k \leq |S^c|\}} 
  \binom{|S^c|}{\frac{|S^c| + x}{2}} 
  \exp\bigg\{ -|S^c|\log\frac{1}{\sqrt{\rho(1-\rho)}} 
  \bigg\} 
    \times \exp\bigg\{-\frac{(1+o(1))n}{2}\sum_{\ell\in S} \III{\ell}_{1/2}\bigg\}\\
  & = \exp\bigg\{-(1+o(1)) \bigg( |S^c|J_\rho + \sum_{\ell\in S}\III{\ell}_{1/2} \bigg)\bigg\}\\
  & = \exp\bigg\{-(1+o(1))\bigg(|S^c| J_\rho + \psi^\star_S(0)\bigg)\bigg\},
\end{align*}
where the last equality is by Lemma \ref{lemma:ineq_for_cvx_conjugate}. We now make the following claim.
\begin{lemma}
  \label{lemma:nodewise_err_caseA_asymp_equiv}
  Under the setup of Proposition \ref{prop:refine_glob_nodewise}, if $\log e|S^c| \geq \sqrt{J_\rho}$, then we have
  \begin{equation}
    \label{eq:nodewise_err_caseA_asymp_equiv}
    1 - \frac{1}{e^{\sqrt{J_\rho}-1}}\leq \frac{|S^c|J_\rho + \psi^\star_S(0)}{(|S^c|+1)J_\rho + \psi_S(-2J_\rho)} \leq 1.
  \end{equation}  
\end{lemma}
\begin{proof}
  By Lemma \ref{lemma:ineq_for_cvx_conjugate}, the denominator in the left-hand side of \eqref{eq:nodewise_err_caseA_asymp_equiv} lies between
  $
    -J_\rho + \sum_{\ell\in S} \III{\ell}_{1/2}
  $ 
  and  
  $
  \sum_{\ell \in S} \III{\ell}_{1/2}.
  $
  Thus, the left-hand side of \eqref{eq:nodewise_err_caseA_asymp_equiv} is between
  $$
    \bigg[\frac{|S^c| J_\rho + \sum_{\ell\in S} \III{\ell}_{1/2}}{(|S^c| + 1) J_\rho + \sum_{\ell\in S} \III{\ell}_{1/2}}~ , ~1\bigg].
  $$
  Note that the lower bound in the above display can be further lower bounded by
  $$
    1 - \frac{J_\rho}{(|S^c|+ 1)J_\rho + \sum_{\ell\in S}\III{\ell}_{1/2}} \geq 1 - \frac{1}{|S^c|} \geq 1 - \frac{1}{e^{\sqrt{J_\rho} - 1}},
  $$
  where the last inequality is by our assumption that $\log e|S^c| \geq \sqrt{J_\rho}$.
\end{proof}
The above lemma tells that under Case A, each summand (for a fixed $S$) in the right-hand side of \eqref{eq:nodewise_err_decomp_3terms_simplified} can be upper bounded by
\begin{align*}
\begin{cases}
  \exp\bigg\{-(1+o(1)) \bigg(|S^c| J_\rho + \psi^\star_S(0)\bigg)\bigg\} & \textnormal{ if } |S^c| \textnormal{ is even},\\
  \exp\bigg\{-(1+o(1)) \bigg((|S^c| + 1) J_\rho + \psi^\star_S(-2J_\rho)\bigg)\bigg\} & \textnormal{ if } |S^c| \textnormal{ is odd}.
\end{cases}
\end{align*}

~\\
{\bf \noindent Case B: $\boldsymbol{|S^c|}$ is odd and $\boldsymbol{\log e|S^c| \leq \sqrt{J_\rho}}$.} With the requirement that $t_{S, x , \xi} = t_{S, x}$ (i.e., $t$ is independent of $\xi$), each summand (for a fixed $S$) in the right-hand side of \eqref{eq:nodewise_err_decomp_3terms_simplified} becomes
\begin{align*}
  & \sum_{x \in\{ -|S^c|+2k: 0\leq k \leq |S^c|\}} 
  \binom{|S^c|}{\frac{|S^c| + x}{2}} 
  \exp\bigg\{ -|S^c|\log\frac{1}{\sqrt{\rho(1-\rho)}} + x(1-2t_{S, x})\log \sqrt{\frac{1-\rho}{\rho}}\bigg\} \\
  & \qquad \times \exp\bigg\{ - \frac{n}{2}\sum_{\ell\in S} \III{\ell}_{t_{S, x}} + o(1)\cdot n\sum_{\ell\in S} \III{\ell}_{1/2}\bigg\} \\
  & \leq 
  \sum_{x \in\{ -|S^c|+2k: 0\leq k \leq |S^c|\}} 
  \exp\bigg\{ |S^c| \log (e|S^c|) - |S^c| \log \frac{1}{\sqrt{\rho(1-\rho)}} + x(1- 2t_{S, x}) \log \sqrt{\frac{1-\rho}{\rho}}\bigg\}\\
  & \qquad \times \exp\bigg\{ - \frac{n}{2}\sum_{\ell\in S} \III{\ell}_{t_{S, x}} + o(1)\cdot n\sum_{\ell\in S} \III{\ell}_{1/2}\bigg\} \\
  & \leq 
  \sum_{x \in\{ -|S^c|+2k: 0\leq k \leq |S^c|\}} 
  \exp\bigg\{|S^c| \bigg(\sqrt{J_\rho} - \log \frac{1}{\sqrt{\rho(1-\rho)}}\bigg) + x(1- 2t_{S, x}) \log \sqrt{\frac{1-\rho}{\rho}}\bigg\}\\
  \label{eq:nodewise_err_decomp_3terms_simplified_caseB}
  & \qquad \times \exp\bigg\{ - \frac{n}{2}\sum_{\ell\in S} \III{\ell}_{t_{S, x}} + o(1)\cdot n \sum_{\ell\in S} \III{\ell}_{1/2}\bigg\},\numberthis
\end{align*}
where the first inequality is by $\binom{n}{k}\leq (en/k)^k$ for any values of $n, k$ such that $1\leq k\leq n$, and the second inequality is by our assumption that $\log e |S^c|\leq \sqrt{J_\rho}$. 

Let us require $t_{x, S}$ to be symmetric about $1/2$:
$$
  \indc{x\geq 0}\cdot t_{S, x}  - \frac{1}{2} = \frac{1}{2} - \indc{x\leq 0} \cdot t_{S, x} .
$$
Under such a requirement, using the fact that $\III{\ell}_t$ is also symmetric about $1/2$ (i.e., $\III{\ell}_{1/2-\delta} = \III{\ell}_{1/2 + \delta}$ for any $\delta\in[0, 1/2]$), the right-hand side of \eqref{eq:nodewise_err_decomp_3terms_simplified_caseB} becomes 
\begin{align*}
  & 2 \times \sum_{x\in\{-|S^c| + 2k: 0\leq k \leq (|S^c|-1)/2\}} 
  \exp\bigg\{|S^c| \bigg(\sqrt{J_\rho} - \log \frac{1}{\sqrt{\rho(1-\rho)}}\bigg) - |x|(1- 2t_{S, x}) \log \sqrt{\frac{1-\rho}{\rho}}\bigg\}\\
  & \qquad \times \exp\bigg\{ - \frac{n}{2}\sum_{\ell\in S} \III{\ell}_{t_{S, x}} + o(1)\cdot n\sum_{\ell\in S} \III{\ell}_{1/2}\bigg\} \\
  & \leq (|S^c| + 1) \exp\bigg\{|S^c| \bigg(\sqrt{J_\rho} - \log \frac{1}{\sqrt{\rho(1-\rho)}}\bigg) - (1- 2t_{S, x}) \log \sqrt{\frac{1-\rho}{\rho}} - \frac{n}{2}\sum_{\ell\in S} \III{\ell}_{t_{S, x}} + o(1)\cdot n\sum_{\ell\in S}\III{\ell}_{1/2}\bigg\},
\end{align*}
where the inequality is because the minimum value that $|x|$ can take is $1$, a consequence of $|S^c|$ being odd. Rearranging terms and using $\log e|S^c|\leq \sqrt{J_\rho}$, the right-hand side above can be further upper bounded by
\begin{align*}
  & \exp\bigg\{ |S^c|\cdot \calO(\sqrt{J_\rho}) + o(1)\cdot n\sum_{\ell\in S}\III{\ell}_{1/2} \\
  & \qquad - |S^c|\log \frac{1}{\sqrt{\rho(1-\rho)}} - \log \sqrt{\frac{1-\rho}{\rho}} - 2t_{S, x}\cdot \log\sqrt{\frac{1-\rho}{\rho}} - \frac{n}{2}\sum_{\ell\in S}\III{\ell}_{t_{S, x}}\bigg\}\\
  & \leq \exp\bigg\{ |S^c|\cdot \calO(\sqrt{J_\rho}) + o(1)\cdot n\sum_{\ell\in S}\III{\ell}_{1/2}  -(1+o(1))\bigg( (|S^c|+1)J_\rho  -  2t_{S, x}J_\rho - \psi_S(t_{S, x}) \bigg) \bigg\},
\end{align*}
where the inequality is by $\rho =o(1)$. Now the optimal choice of $t_{S, x}$ is clear. For $x < 0$, choosing
$
  t_{S, x} = \argmax_{0\leq t\leq 1} -2t J_\rho - \psi_S(t)
$
gives the following upper bound for the right-hand side of \eqref{eq:nodewise_err_decomp_3terms_simplified_caseB}:
\begin{align*}
  & \exp\bigg\{ |S^c|\cdot \calO(\sqrt{J_\rho}) + o(1)\cdot n\sum_{\ell\in S}\III{\ell}_{1/2}- (1+o(1))\bigg( (|S^c| + 1) J_\rho + \psi_S^\star(-2J_\rho)\bigg)\bigg\} \\
  \label{eq:glob_nodewise_err_caseB}
  & \leq \exp\bigg\{- (1+o(1))\bigg( (|S^c| + 1) J_\rho + \psi_S^\star(-2J_\rho)\bigg)\bigg\},\numberthis
\end{align*}
where the last inequality is by $\sqrt{J_\rho}\ll J_\rho$ and $\frac{n}{2}\cdot \sum_{\ell\in S}\III{\ell}_{1/2}\leq \psi_S^\star(-2J_\rho) + J_\rho$, the latter of which is proved in Lemma \ref{lemma:ineq_for_cvx_conjugate}.

~\\
{\bf \noindent Finishing the proof of Proposition \ref{prop:refine_glob_nodewise}.} Summarizing Case A and Case B, we have
\begin{align*}
  \bbP(\barzzz{\star, -i}_i = -\bz^\star_i, E_i\cap F_i) 
  & \leq \sum_{S\subseteq [L]: |S^c|\textnormal{ even}} \exp\bigg\{- (1+o(1)) \bigg(|S^c| J_\rho + \psi^\star_S(0)\bigg)\bigg\} \\ 
  & \qquad + \sum_{S\subseteq[L]: |S^c|\textnormal{ odd}} \exp\bigg\{ -(1+o(1))\bigg((|S^c|+1) J_\rho + \psi_S^\star(-2J_\rho)\bigg)\bigg\},
\end{align*}
where we emphasize that the $o(1)$ terms are independent of $S$. The proof is concluded by plugging the above inequality to   \eqref{eq:glob_nodewise_err_decomp}.

\subsection{Proof of Theorem \ref{thm:refine_ind}}\label{prf:thm:refine_ind}
We use the same notations as those in the proof of Theorem \ref{thm:refine_glob}. 
The proof is based on the following counterpart to Proposition \ref{prop:refine_glob_nodewise}.
\begin{proposition}
  \label{prop:refine_ind_nodewise}
  Fix $\ell\in[L]$. Under the setup of Proposition \ref{prop:refine_glob_nodewise}, there exists a sequence $\delta_n'=o(1)$ and an absolute constant $C''>0$ such that for any $i\in[n]$, we have 
    \begin{align}
    \label{eq:upper_bound_z_l}
    \bbP(\pi_i\barzzz{\ell, -i}_i \neq  \zzz{\ell}_i) 
    & \leq  C''n^{-(1+\ep_{\init})}+  \sum_{S\subseteq[L]\setminus\{\ell\}} \bigg(e^{-(1-\delta_n') \globinfo_{S\cup\{\ell\}}} + e^{-(1-\delta_n')\indinfo_{S\cup\{\ell\}}}\bigg),
  \end{align}
  where $\globinfo_S, \indinfo_S$ are defined in   \eqref{eq:globinfo} and \eqref{eq:indinfo} respectively. 
\end{proposition}
\begin{proof}
  See Appendix \ref{prf:prop:refine_ind_nodewise}.
\end{proof}

With Lemma \ref{lemma:align} and the above proposition at hand, the rest of the proof is nearly identical to the proof of Theorem \ref{thm:refine_glob}, and we omit the details.

\subsubsection{Proof of Proposition \ref{prop:refine_ind_nodewise}}\label{prf:prop:refine_ind_nodewise}
Without loss of generality we consider the first layer and we assume $\pi_i = +1$. We start by computing
\begin{align*}
  \bbP(\barzzz{1, -i}_i\neq \zzz{1}_i) = \bbP(\barzzz{1, -i}_i = -\zzz{1}_i) \leq \bbP(\barzzz{1, -i}_i = -\zzz{1}_i, E_i\cap F_i) + C' n^{-(1+\ep_{\init})},
\end{align*}
where $E_i, F_i$ are defined in   \eqref{eq:Ei} and \eqref{eq:Fi} respectively, and the last inequality is by   \eqref{eq:high_prob_Ei_Fi}. We now proceed by
\begin{align*}
  & \bbP(\barzzz{1, -i}_i = -\zzz{1}_i, E_i\cap F_i) \\
  & = \sum_{S\subseteq \{2, \hdots, L\}} \bigg(\bbP(\barzzz{\star, -i}_i = -\bz^\star_i, \barzzz{1, -i}_i = -\zzz{1}_i, \barzzz{S, -i}_i = -\zzz{S}_i, \barzzz{S^c, -i}_i = \zzz{S^c}_i, E_i\cap F_i) \\
  \label{eq:ind_nodewise_err_decomp_all_subsets}
  & \qquad + \bbP(\barzzz{\star, -i}_i = \bz^\star_i, \barzzz{1, -i}_i = -\zzz{1}_i, \barzzz{S, -i}_i = -\zzz{S}_i, \barzzz{S^c, -i}_i = \zzz{S^c}_i, E_i\cap F_i)\bigg) ,\numberthis
\end{align*}
where we denoted $\barzzz{S, -i}_i = \{\barzzz{\ell, -i}_i: \ell\in S\}$ and $\zzz{S}_i=\{\zzz{\ell}_i: \ell\in S\}$. The right-hand side above is the superposition of two terms, the first of which has already been calculated in the proof of Proposition \ref{prop:refine_glob_nodewise} (see   \eqref{eq:glob_nodewise_err_decomp_all_subsets}):
\begin{align*}
  & \sum_{S\subseteq\{2, \hdots, L\}}\bbP(\barzzz{\star, -i}_i = -\bz^\star_i, \barzzz{1, -i}_i = -\zzz{1}_i, \barzzz{S, -i}_i = -\zzz{S}_i, \barzzz{S^c, -i}_i = \zzz{S^c}_i, E_i\cap F_i)\\
  \label{eq:ind_nodewise_err_globinfo}
  & \leq \sum_{S\subseteq\{2, \hdots, L\}} e^{-(1+o(1))\globinfo_{S\cup\{1\}}},\numberthis
\end{align*}
where the $o(1)$ term is independent of $S$. For the second term in the right-hand side of \eqref{eq:ind_nodewise_err_decomp_all_subsets}, we have
\begin{align*}
  &\sum_{S\subseteq\{2, \hdots, L\}} \bbP(\barzzz{\star, -i}_i = \bz^\star_i, \barzzz{S\cup\{1\}, -i}_i = -\zzz{S\cup\{1\}}_i, \barzzz{(S\cup\{1\})^c, -i}_i = \zzz{(S\cup\{1\})^c}_i, E_i\cap F_i)\\
  & \leq \sum_{S\subseteq\{2, \hdots, L\}}
  \bbP\bigg\{ \log\bigg(\frac{1-\rho}{\rho}\bigg) \cdot \bigg( \#\{\ell\in S\cup\{1\}: \bz^\star_i = -\zzz{\ell}_i\}-\#\{\ell\in S\cup\{1\}: \bz^\star_i = \zzz{\ell}_i\} \bigg) \\
  & \qquad\qquad \qquad \qquad + \sum_{\ell\in S\cup\{1\}} \sum_{j\neq i: \tzzz{\star, -i}_j = -\zzz{\ell}_i} \bigg[\log\bigg(\frac{p_\ell(1-q_\ell)}{q_\ell(1-p_\ell)}\bigg) \AAA{\ell}_{ij} + \log \bigg(\frac{1-p_\ell}{1-q_\ell}\bigg)\bigg] \\
  &\qquad\qquad \qquad \qquad  - \sum_{\ell\in S\cup\{1\}} \sum_{j\neq i: \tzzz{\star, -i}_j = \zzz{\ell}_i} \bigg[\log\bigg(\frac{p_\ell(1-q_\ell)}{q_\ell(1-p_\ell)}\bigg) \AAA{\ell}_{ij} + \log \bigg(\frac{1-p_\ell}{1-q_\ell}\bigg)\bigg] \geq 0 \textnormal{ and }  E_i\cap F_i
  \bigg\},
\end{align*}
where the inequality is by the fact that the occurrence of the event in each summand of the left-hand side above implies
$$
\sum_{\ell\in S\cup\{1\}} \fff{\ell}_i(\bz^{\star}_i, -\zzz{\ell}_i, \tzzz{\star, -i}) + \sum_{\ell \in (S\cup\{1\})^c} \fff{\ell}_i(\bz^{\star}_i, \zzz{\ell}_i , \tzzz{\star, -i}) \geq \sum_{\ell\in [L]} \fff{\ell}_i(\bz^{\star}_i, \zzz{\ell}_i, \tzzz{\star, -i}).
$$
Since $E_i, F_i$ are both independent of $\{\zzz{\ell}_i\}_{\ell=1}^L$, we can do the following decomposition: 
\begin{align*}
  &\sum_{S\subseteq\{2, \hdots, L\}} \bbP(\barzzz{\star, -i}_i = \bz^\star_i, \barzzz{S\cup\{1\}, -i}_i = -\zzz{S\cup\{1\}}_i, \barzzz{(S\cup\{1\})^c, -i}_i = \zzz{(S\cup\{1\})^c}_i, E_i\cap F_i)\\
  & \leq \sum_{S\subseteq\{2, \hdots, L\}}
  \sum_{\xi\in\{\pm 1\}^{S\cup\{1\}}} \bbP(\zzz{S\cup\{1\}}_i = \bz^\star_i \xi)
  \times \bbP \big(\clover\big)\\
  & = \sum_{S\subseteq\{2, \hdots, L\}}
  \sum_{\xi\in\{\pm 1\}^{S\cup\{1\}}} \bbP(\zzz{S\cup\{1\}}_i = \bz^\star_i \xi) \times \bbE_{\{\zzz{\ell}_{-i}\}_{\ell=1}^L}\bigg[\bbP\bigg( \clover ~\bigg|~ \{\zzz{\ell}_{-i}\}_{\ell=1}^L \bigg)\bigg],
\end{align*}
where
\begin{align*}
  \clover &:= \bigg\{ \log\bigg(\frac{1-\rho}{\rho}\bigg) \cdot \bigg( \#\{\ell\in S\cup\{1\}: \xi_\ell = -1\}-\#\{\ell\in S\cup\{1\}: \xi_\ell = 1\} \bigg) \\
  & \qquad \qquad + \sum_{\ell\in S\cup\{1\}} \sum_{j\neq i: \tzzz{\star, -i}_j = -\xi_\ell \bz^\star_i} \bigg[\log\bigg(\frac{p_\ell(1-q_\ell)}{q_\ell(1-p_\ell)}\bigg) \AAA{\ell}_{ij} + \log \bigg(\frac{1-p_\ell}{1-q_\ell}\bigg)\bigg] \\
  &\qquad\qquad  - \sum_{\ell\in S\cup\{1\}} \sum_{j\neq i: \tzzz{\star, -i}_j = \xi_\ell\bz^\star_i} \bigg[\log\bigg(\frac{p_\ell(1-q_\ell)}{q_\ell(1-p_\ell)}\bigg) \AAA{\ell}_{ij} + \log \bigg(\frac{1-p_\ell}{1-q_\ell}\bigg)\bigg] \geq 0 \textnormal{ and }  E_i\cap F_i \bigg\}.
\end{align*}
By Markov's inequality, we have
\begin{align*}
  &\bbP\bigg( \clover ~\bigg|~ \{\zzz{\ell}_{-i}\}_{\ell=1}^L \bigg) \\
  & \leq \bbE\bigg[ \exp\bigg\{t \log \bigg(\frac{1-\rho}{\rho}\bigg)
  \cdot \bigg(\# \{\ell\in S\cup\{1\}: \xi_\ell = -1\} - \#\{\ell\in S\cup\{1\}: \xi_\ell = 1\}\bigg)\\
  & \qquad + t \sum_{\ell\in S\cup\{1\}} \sum_{j\neq i: \tzzz{\star, -i}_j = -\xi_\ell \bz^\star_i} \bigg[\log\bigg(\frac{p_\ell(1-q_\ell)}{q_\ell(1-p_\ell)}\bigg) \AAA{\ell}_{ij} + \log \bigg(\frac{1-p_\ell}{1-q_\ell}\bigg)\bigg]\\
  & \qquad - t \sum_{\ell\in S\cup\{1\}} \sum_{j\neq i: \tzzz{\star, -i}_j = \xi_\ell\bz^\star_i}\bigg[ \log\bigg(\frac{p_\ell(1-q_\ell)}{q_\ell(1-p_\ell)}\bigg) \AAA{\ell}_{ij} + \log \bigg(\frac{1-p_\ell}{1-q_\ell}\bigg)\bigg]
  \bigg\}  \cdot \Indc_{E_i\cap F_i}
  ~\bigg|~ \{\zzz{\ell}_{-i}\}_{\ell=1}^L  \bigg],
\end{align*}
where the value of $t$ will be determined later. Using the same arguments as those that give rise to   \eqref{eq:nodewise_err_decomp_3terms_simplified}, we get
\begin{align*}
  & \bbE\bigg[ \exp\bigg\{ t \sum_{\ell\in S\cup\{1\}} \sum_{j\neq i: \tzzz{\star, -i}_j = -\xi_\ell \bz^\star_i} \bigg[\log\bigg(\frac{p_\ell(1-q_\ell)}{q_\ell(1-p_\ell)}\bigg) \AAA{\ell}_{ij} + \log \bigg(\frac{1-p_\ell}{1-q_\ell}\bigg)\bigg]\\
  & \qquad - t \sum_{\ell\in S\cup\{1\}} \sum_{j\neq i: \tzzz{\star, -i}_j = \xi_\ell\bz^\star_i} \bigg[\log\bigg(\frac{p_\ell(1-q_\ell)}{q_\ell(1-p_\ell)}\bigg) \AAA{\ell}_{ij} + \log \bigg(\frac{1-p_\ell}{1-q_\ell} \bigg\}\bigg] 
  \cdot \Indc_{E_i\cap F_i}
  ~\bigg|~ \{\zzz{\ell}_{-i}\}_{\ell=1}^L\bigg] \\
  &\leq \prod_{\ell\in S\cup\{1\}} \exp\bigg\{-\frac{n}{2}\cdot \III{\ell}_t + o(1)\cdot n\III{\ell}_{1/2}\bigg\}. 
\end{align*}
Hence, we arrive at 
\begin{align*}
  &\sum_{S\subseteq\{2, \hdots, L\}} \bbP(\barzzz{\star, -i}_i = \bz^\star_i, \barzzz{S\cup\{1\}, -i}_i = -\zzz{S\cup\{1\}}_i, \barzzz{(S\cup\{1\})^c, -i}_i = \zzz{(S\cup\{1\})^c}_i, E_i\cap F_i)\\
  & \leq \sum_{S\subseteq \{2, \hdots, L\}} \sum_{\xi\in\{\pm 1\}^{S\cup\{1\}}} \bbP(\zzz{S\cup\{1\}}_i = \bz^\star_i \xi)\\
  &\qquad \times  \exp\bigg\{t_{S, \xi} \log \bigg(\frac{1-\rho}{\rho}\bigg)
  \cdot \bigg(\# \{\ell\in S\cup\{1\}: \xi_\ell = -1\} - \#\{\ell\in S\cup\{1\}: \xi_\ell = 1\}\bigg) \bigg\}\\
  \label{eq:ind_nodewise_err_decomp_indinfo}
  &\qquad \times \exp\bigg\{-\frac{n}{2}\sum_{\ell\in S\cup\{1\}} \III{\ell}_{t_{S, \xi}} + o(1)\cdot n\sum_{\ell\in S\cup\{1\}}\III{\ell}_{1/2} \bigg\}.\numberthis
\end{align*}
We discuss according to two cases.

~\\
{\noindent \bf Case A: either $\boldsymbol{|S\cup\{1\}|}$ is even, or $\boldsymbol{\log e |S\cup\{1\}| \geq \sqrt{J_\rho}}$.} In this case, we choose $t_{S, \xi} = 1/2$. Then for any fixed $S$ satisfying the assumptions made in Case A, the corresponding summand in the right-hand side of   \eqref{eq:ind_nodewise_err_decomp_indinfo} becomes
\begin{align*}
  &\bbE_{\{\zzz{\ell}_i: \ell\in S\cup\{1\}\}} \bigg[\exp\bigg\{\log\sqrt{\frac{1-\rho}{\rho}}\cdot \bigg( \# \{\ell\in S\cup\{1\}: \zzz{\ell}_i = -\bz^\star_i\} - \#\{\ell\in S\cup\{1\}: \zzz{\ell}_i = \bz^\star_i\} \bigg)\bigg\}\bigg] \\
  & \qquad \times \exp\bigg\{-\frac{(1+o(1))\cdot n}{2}\sum_{\ell\in S\cup\{1\}}\III{\ell}_{1/2} \bigg\} \\
  \label{eq:ind_nodewise_err_caseA}
  & = \exp\bigg\{ -|S\cup\{1\}|J_\rho -  \frac{(1+o(1)) \cdot n}{2}\sum_{\ell\in S\cup\{1\}}\III{\ell}_{1/2}\bigg\}\numberthis
\end{align*}
If $|S\cup\{1\}|$ is even, then the right-hand side above is 
$$
  \exp\bigg\{-(1+o(1))\bigg( |S\cup\{1\}| J_\rho - \psi^\star_{S\cup\{1\}}(0)\bigg)\bigg\} = e^{-(1+o(1))\indinfo_{S\cup\{1\}}}.
$$
If $|S\cup\{1\}|$ is odd but $\log e|S\cup\{1\}| \geq \sqrt{J_\rho}$, we have
\begin{align*}
  & (|S\cup\{1\}| + 1) J_\rho + \sup_{0\leq t\leq 1}\bigg\{-2tJ_\rho + \frac{n}{2}\sum_{\ell\in S\cup\{1\}} \III{\ell}_t\bigg\} \\
  &\leq (|S\cup\{1\}| + 1) J_\rho + \frac{n}{2} \sum_{\ell\in S\cup\{1\}} \III{\ell}_{1/2}\\
  & \leq |S\cup\{1\}| J_\rho \cdot ( 1 + e^{-\sqrt{J_\rho} +1} )+ \frac{n}{2}\sum_{\ell\in S\cup\{1\}} \III{\ell}_{1/2} \\
  & \leq (1+o(1)) \bigg(|S\cup\{1\}|J_\rho + \frac{n}{2}\sum_{\ell\in S\cup\{1\}} \III{\ell}_{1/2}\bigg),
\end{align*}
where the first inequality is by Lemma \ref{lemma:monotonicity_of_I_t}, the second inequality is by $\log e|S\cup\{1\}|\geq \sqrt{J_\rho}$ and the third inequality is by $J_\rho \to \infty$ as $n\to\infty$. So we can upper bound the right-hand side of \eqref{eq:ind_nodewise_err_caseA} by
$$
  \exp\bigg\{ (1+o(1))\bigg( (|S\cup\{1\}| + 1) J_\rho + \psi^\star_{S\cup\{1\}}(-2J_\rho)\bigg)\bigg\} = e^{-(1+o(1))\indinfo_{S\cup\{1\}}}.
$$
In summary, for any fixed $S$ satisfying the assumptions made in Case A, the corresponding summand in the right-hand side of \eqref{eq:ind_nodewise_err_decomp_indinfo} can be upper bounded by
$
  e^{-(1+o(1))\indinfo_{S\cup\{1\}}}.
$

~\\
{\noindent \bf Case B: $\boldsymbol{|S\cup\{1\}|}$ is odd, and $\boldsymbol{\log e|S\cup\{1\}| \leq \sqrt{J_\rho}}$.} In this case, we re-write each summand in the right-hand side of \eqref{eq:ind_nodewise_err_decomp_indinfo} by
\begin{align*}
  &\sum_{\xi\in\{\pm 1\}^{S\cup\{1\}}} \bbP(\zzz{S\cup\{1\}}_i = \bz^\star_i \xi)\\
  &\qquad \times  \exp\bigg\{t_{S, \xi} \log \bigg(\frac{1-\rho}{\rho}\bigg)
  \cdot \bigg(\# \{\ell\in S\cup\{1\}: \xi_\ell = -1\} - \#\{\ell\in S\cup\{1\}: \xi_\ell = 1\}\bigg) \bigg\}\\
  &\qquad \times \exp\bigg\{-\frac{n}{2}\sum_{\ell\in S\cup\{1\}} \III{\ell}_{t_{S, \xi}} + o(1)\cdot n\sum_{\ell\in S\cup\{1\}}\III{\ell}_{1/2} \bigg\}\\
  & = \sum_{x\in\{-|S\cup\{1\}| + 2k: 0\leq k\leq |S\cup\{1\}|\}} \binom{|S\cup\{1\}|}{\frac{|S\cup\{1\}| + x}{2}} (1-\rho)^{(|S\cup\{1\}| + x)/2} \rho^{(|S\cup\{1\}| - x)/2} \\
  & \qquad \times \exp\bigg\{ -2x t_{x, S}\log \sqrt{\frac{1-\rho}{\rho}}- \frac{n}{2}\sum_{\ell\in S\cup\{1\}} \III{\ell}_{t_{x, S}} + o(1)\cdot n\sum_{\ell\in S\cup\{1\}} \III{\ell}_{1/2}\bigg\}.
\end{align*}
Using similar arguments as those that give rise to   \eqref{eq:glob_nodewise_err_caseB}, we can bound the right-hand side above by
$$
  \exp\bigg\{ -(1+o(1))\bigg((|S\cup\{1\}| + 1) J_\rho + \psi^\star_{S\cup\{1\}}(-2J_\rho)\bigg)\bigg\} = e^{-(1+o(1))\indinfo_{S\cup\{1\}}}.
$$

~\\
{\bf \noindent Finishing the proof of Proposition \ref{prop:refine_ind_nodewise}}. The proof is concluded by combining Case A and B above.

\subsection{Proof of Theorem \ref{thm:opt_glob_est}}\label{prf:thm:opt_glob_est}
The desired result follows from the following two propositions.
\begin{proposition}
\label{prop:opt_glob_even}
For any $\overline{\delta}_n = o(1)$ and any $c\in(0, 1)$, there exists another $\overline{\delta}_n'=o(1)$ such that
\begin{align*}
  & \sum_{S: |S^c|\textnormal{ even}} \exp\big\{- (1-\overline{\delta}_n) \big(|S^c|J_\rho + \psi_S^\star(0)\big)\big\} \\
  & \leq  2\exp\big\{-(1-\overline{\delta}_n') \min_{S:|S^c|\textnormal{ even}}\big(|S^c|J_\rho + \psi_S^\star(0)\big) + \log L + L e^{-(1-\overline{\delta}_n) cJ_\rho}\big\}.
\end{align*}  
\end{proposition}
\begin{proof}
  See Appendix \ref{prf:prop:opt_glob_even}.
\end{proof}

\begin{proposition}
\label{prop:opt_glob_odd}
For any $\overline{\delta}_n = o(1)$ and any $c\in(0, 1)$, there exists another $\overline{\delta}_n' = o(1)$ such that
\begin{align*}
  & \sum_{S: |S^c|\textnormal{ odd}} \exp\big\{- (1-\overline{\delta}_n) \big(|S^c|J_\rho + \psi_S^\star(-2J_\rho)\big)\big\} \\
  &  \leq 2 \exp\big\{-(1-\overline{\delta}_n') \min_{S:|S^c|\textnormal{ odd}}\big(|S^c|J_\rho + \psi_S^\star(-2J_\rho)\big) + \log L + Le^{-(1-\overline{\delta}_n)cJ_\rho}\big\}.
\end{align*}  
\end{proposition}
\begin{proof}
  See Appendix \ref{prf:prop:opt_glob_odd}.
\end{proof}

\subsubsection{Proof of Proposition \ref{prop:opt_glob_even}}\label{prf:prop:opt_glob_even}
For any $t\in[0, 1]$, let us define
\begin{align}
  \label{eq:sets}
  S_{\geq , t}:= \{\ell\in[L]: n\III{\ell}_t/2 \geq J_\rho\}, \qquad S_{<, t}:= \{\ell\in[L]: n\III{\ell}_t/2 < J_\rho\}.
\end{align}
For notational simplicity we let $m = n/2$. In this proof we will assume both $S_{\geq, 1/2}$ and $S_{<, 1/2}$ are non-empty. The proof when one of them is empty is nearly identical. We begin by noting that
\begin{align*}
  &\sum_{S\subseteq[L]:  |S^c|\textnormal{ even}} \exp\bigg\{ -(1-\overline{\delta}_n)\cdot \bigg( |S^c|J_\rho + \psi_S^\star(0)\bigg) \bigg\} + \sum_{S\subseteq[L]:  |S^c|\textnormal{ odd}} \exp\bigg\{ -(1-\overline{\delta}_n)\cdot \bigg( |S^c|J_\rho + \psi_S^\star(0)\bigg) \bigg\}\\
  \label{eq:opt_glob_even_summation}
  & = \prod_{\ell\in[L]} \bigg(e^{-(1-\overline{\delta}_n) J_\rho} + e^{-(1-\overline{\delta}_n) m\III{\ell}_{1/2}}\bigg)\numberthis,
\end{align*}
and that
\begin{align*}
  &\sum_{S\subseteq[L]:  |S^c|\textnormal{ even}} \exp\bigg\{ -(1-\overline{\delta}_n)\cdot \bigg( |S^c|J_\rho + \psi_S^\star(0)\bigg) \bigg\} - \sum_{S\subseteq[L]:  |S^c|\textnormal{ odd}} \exp\bigg\{ -(1-\overline{\delta}_n)\cdot \bigg( |S^c|J_\rho + \psi_S^\star(0)\bigg) \bigg\}\\
  \label{eq:opt_glob_even_difference}
  & = \prod_{\ell\in[L]} \bigg(-e^{-(1-\overline{\delta}_n) J_\rho} + e^{-(1-\overline{\delta}_n) m\III{\ell}_{1/2}}\bigg). \numberthis
\end{align*}
We split the discussion into two cases.

~\\
{\noindent \bf Case A: $|\boldsymbol{S_{\geq, 1/2}}|$ is even.}  In this case, by   \eqref{eq:opt_glob_even_summation}, we have
\begin{align*}
  & \sum_{S: |S^c|\textnormal{ even}} \exp\big\{- (1-\overline{\delta}_n) \big(|S^c|J_\rho + \psi_S^\star(0)\big)\big\} \\
  & \leq \prod_{\ell\in[L]} \bigg(e^{-(1-\overline{\delta}_n) J_\rho} + e^{-(1-\overline{\delta}_n) m\III{\ell}_{1/2}}\bigg)\\
  & = \exp\bigg\{-(1-\overline{\delta}_n) \bigg(|S_{\geq , 1/2}|J_\rho + \sum_{\ell\in S_{<, 1/2}}m\III{\ell}_{1/2}\bigg)\bigg\}\\
  & \qquad \times \prod_{\ell\in S_{\geq, 1/2}} \bigg(1 + e^{-(1-\overline{\delta}_n)(m\III{\ell}_{1/2} - J_\rho)}\bigg)
  \cdot \prod_{\ell\in S_{<, 1/2}} \bigg(1 + e^{-(1-\overline{\delta}_n)(J_\rho-m\III{\ell}_{1/2})}\bigg)\\
  & \leq\exp\bigg\{-(1-\overline{\delta}_n) \bigg(|S_{\geq , 1/2}|J_\rho + \sum_{\ell\in S_{<, 1/2}}m\III{\ell}_{1/2}\bigg)\bigg\}\\
  & \qquad \times \exp\bigg\{\sum_{\ell\in S_{\geq, 1/2}} e^{-(1-\overline{\delta}_n)(m \III{\ell}_{1/2}-J_\rho)} + \sum_{\ell\in S_{<, 1/2}} e^{-(1-\overline{\delta}_n)(J_\rho - m\III{\ell}_{1/2})}\bigg\}\\
  & {\leq} \exp\bigg\{-(1-\delta_{1,n }) \bigg(|S_{\geq , 1/2}|J_\rho + \sum_{\ell\in S_{<, 1/2}}m\III{\ell}_{1/2}\bigg) + \sum_{\ell\in S_{<, 1/2}} e^{-(1-\overline{\delta}_n)(J_\rho - m\III{\ell}_{1/2})}\bigg\},
\end{align*}
where the last inequality is by $e^{-(1-\overline{\delta}_n)(m\III{\ell}_{1/2}-J_\rho)}\leq 1\ll J_\rho$ for any $\ell \in S_{\geq ,1/2}$, and $\delta_{1, n}=o(1)$.
Let
$$
  S_{<, 1/2}^{-} := \{\ell\in S_{<, 1/2}: m\III{\ell}_{1/2}\leq (1-c)J_\rho\}, \qquad S^+_{<, 1/2} := \{\ell\in S_{<, 1/2}: m\III{\ell}_{1/2}>  (1-c)J_\rho\}.
$$
For $\ell\in S_{<, 1/2}^-$, we have $e^{-(1-\overline{\delta}_n)(J_\rho - m\III{\ell}_{1/2})}\leq e^{-(1-\overline{\delta}_n)cJ_\rho}$ and for $\ell\in S_{<,1/2}^+$, we have $e^{-(1-\overline{\delta}_n)(J_\rho - m\III{\ell}_{1/2})}\leq 1$. 
Thus we have
\begin{align*}
  & \sum_{S: |S^c|\textnormal{ even}} \exp\big\{- (1-\overline{\delta}_n) \big(|S^c|J_\rho + \psi_S^\star(0)\big)\big\} \\
  & \overset{(*)}{\leq} \exp\bigg\{-(1-\delta_{1,n})\bigg(|S_{\geq, 1/2}| J_\rho + \sum_{\ell\in S_{<, 1/2}^+}m\III{\ell}_{1/2} \big(1- \frac{1}{m\III{\ell}_{1/2}} \big) + \sum_{\ell\in S_{<,1/2}^-} m\III{\ell}_{1/2}\bigg) + L e^{-(1-\overline{\delta}_n)cJ_\rho}\bigg\}\\
  & \overset{(**)}{\leq} \exp\bigg\{-(1-\delta_{1,n})\bigg(|S_{\geq, 1/2}| J_\rho + \sum_{\ell\in S_{<, 1/2}^+}m\III{\ell}_{1/2} \big(1- \frac{1}{(1-c)J_\rho} \big) + \sum_{\ell\in S_{<,1/2}^-} m\III{\ell}_{1/2}\bigg) + L e^{-(1-\overline{\delta}_n)cJ_\rho}\bigg\}\\
  \label{eq:opt_glob_even_caseA}
  & \overset{(***)}{\leq} \exp\bigg\{-(1-\delta_{2,n }) \min_{S: |S^c|\textnormal{ even}}\bigg(|S^c|J_\rho + \sum_{\ell\in S}m\III{\ell}_{1/2}\bigg) + Le^{-(1-\overline{\delta}_n) cJ_\rho} \bigg\}\numberthis,
\end{align*}
where $(*)$ is by $e^{-(1-\overline{\delta}_n)(J_\rho - m\III{\ell}_{1/2})}\leq 1$ for $\ell\in S_{<, 1/2}^+$ and $e^{-(1-\overline{\delta}_n)(J_\rho - m\III{\ell}_{1/2})}\leq e^{-(1-\overline{\delta}_n)cJ_\rho}$ for $\ell\in S_{<, 1/2}^-$, $(**)$ is by the definition of $S_{<, 1/2}^+$, $(***)$ is by our assumption that $|S_{\geq, 1/2}|$ is even, and $\delta_{2, n} = o(1)$.

~\\
{\noindent \bf Case B: $|\boldsymbol{S_{\geq ,1/2}}|$ is odd.} Invoking   \eqref{eq:opt_glob_even_summation} and \eqref{eq:opt_glob_even_difference}, we have
\begin{align*}
  & \sum_{S: |S^c|\textnormal{ even}} \exp\big\{- (1-\overline{\delta}_n) \big(|S^c|J_\rho + \psi_S^\star(0)\big)\big\} \\
  & = \frac{1}{2} \exp\bigg\{-(1-\overline{\delta}_n)\bigg(|S_{\geq, 1/2}|J_\rho + \sum_{\ell\in S_{<, 1/2}}m\III{\ell}_{1/2}\bigg)\bigg\} \\
  & \qquad \times \bigg[ 
  \prod_{\ell\in S_{\geq ,1/2}} \bigg(1+ e^{-(1-\overline{\delta}_n)(m\III{\ell}_{1/2}-J_\rho)}\bigg) 
  \prod_{\ell\in S_{<, 1/2}} \bigg(1+ e^{-(1-\overline{\delta}_n)(J_\rho - m\III{\ell}_{1/2})}\bigg)\\
  & \qquad \qquad - 
  \prod_{\ell\in S_{\geq ,1/2}} \bigg(1- e^{-(1-\overline{\delta}_n)(m\III{\ell}_{1/2}-J_\rho)}\bigg) 
  \prod_{\ell\in S_{<, 1/2}} \bigg(1- e^{-(1-\overline{\delta}_n)(J_\rho - m\III{\ell}_{1/2})}\bigg)
  \bigg]\\
  & = \frac{1}{2} \exp\bigg\{-(1-\overline{\delta}_n)\bigg(|S_{\geq, 1/2}|J_\rho + \sum_{\ell\in S_{<, 1/2}}m\III{\ell}_{1/2}\bigg)\bigg\}\cdot \prod_{\ell\in S_{\geq, 1/2}}\bigg(1+ e^{-(1-\overline{\delta}_n)(m\III{\ell}_{1/2}-J_\rho)}\bigg)  \\
  & \qquad \times 
  \bigg[ \prod_{\ell\in S_{<, 1/2}}\bigg(1+ e^{-(1-\overline{\delta}_n)(J_\rho - m\III{\ell}_{1/2})}\bigg) - \prod_{\ell\in S_{<, 1/2}}\bigg(1- e^{-(1-\overline{\delta}_n)(J_\rho - m\III{\ell}_{1/2})}\bigg) \bigg]\\
  & \qquad + \frac{1}{2} \exp\bigg\{-(1-\overline{\delta}_n)\bigg(|S_{\geq, 1/2}|J_\rho + \sum_{\ell\in S_{<, 1/2}}m\III{\ell}_{1/2}\bigg)\bigg\}
  \cdot \prod_{\ell\in S_{<, 1/2}}\bigg(1- e^{-(1-\overline{\delta}_n)(J_\rho - m\III{\ell}_{1/2})}\bigg)  \\
  & \qquad \qquad \times 
  \bigg[ \prod_{\ell\in S_{\geq, 1/2}}\bigg(1 + e^{-(1-\overline{\delta}_n)(m\III{\ell}_{1/2}- J_\rho )}\bigg) - \prod_{\ell\in S_{\geq, 1/2}}\bigg(1- e^{-(1-\overline{\delta}_n)(m\III{\ell}_{1/2}-J_\rho )}\bigg) \bigg].
\end{align*}

We need the following estimate.
\begin{lemma}
  \label{lemma: diff_of_products}
  Let $\{x_i: i\in[n]\}, \{y_i:i\in[n]\}$ be two collections of complex numbers. Assume $x_i$'s are outside of the unit disk and $y_i$'s are inside the unit disk, then
  $$
    \bigg|\prod_{i\in[n]} x_i - \prod_{i\in[n]} y_i\bigg| \leq \bigg(\sum_{i\in[n]}|x_i - y_i|\bigg) \cdot  \prod_{i\in[n]} |x_i|.
  $$
\end{lemma}
\begin{proof}
  We first prove the following algebraic identity:
  $$
    \prod_{i\in[n]} x_i - \prod_{i\in[n]} y_i = \sum_i (x_i - y_i) \cdot \bigg(\prod_{j=1}^{i-1} x_j \bigg) \cdot \bigg(\prod_{j=i+1}^n y_i\bigg)
  $$
  with the convention that $\prod_{i = j}^k x_i = 1$ if $k<j$, and then the desired result follows from triangle inequality. We induct on $n$. The case of $n = 1$ is trivial. Assume the identity holds for $n = k$. Now for $n = k+1$, we have
  \begin{align*}
    \prod_{i=1}^{k+1} x_i - \prod_{i=1}^{k+1} y_i & =  x_{k+1} \prod_{i=1}^k x_i - y_{k+1} \prod_{i=1}^k y_i\\
    & = x_{k+1} \prod_{i=1}^k x_i - y_{k+1} \prod_{i=1}^k x_i + y_{k+1} \prod_{i=1}^k x_i - y_{k+1} \prod_{i=1}^k y_i\\
    & = (x_{k+1}-y_{k+1}) \prod_{i=1}^k x_i + y_{k+1} \cdot \sum_{i=1}^k(x_i - y_i)\cdot \prod_{j=1}^{i-1} x_j \cdot \prod_{j=i+1}^k y_i \\
    & = (x_{k+1}-y_{k+1}) \prod_{i=1}^k x_i +  \sum_{i=1}^k(x_i - y_i)\cdot \prod_{j=1}^{i-1} x_j \cdot \prod_{j=i+1}^{k+1} y_i\\
    & = \sum_{i=1}^{k+1}(x_i - y_i)\cdot \prod_{j=1}^{i-1} x_j \cdot \prod_{j=i+1}^{k+1} y_i,
  \end{align*}
  which finishes the proof.
\end{proof}

By Lemma \ref{lemma: diff_of_products}, we have
\begin{align*}
  & \sum_{S: |S^c|\textnormal{ even}} \exp\big\{- (1-\overline{\delta}_n) \big(|S^c|J_\rho + \psi_S^\star(0)\big)\big\} \\
  & \leq \exp\bigg\{-(1-\overline{\delta}_n)\bigg(|S_{\geq, 1/2}|J_\rho + \sum_{\ell\in S_{<, 1/2}}m\III{\ell}_{1/2}\bigg)\bigg\}
  \cdot \prod_{\ell\in S_{\geq, 1/2}} \bigg(1+ e^{-(1-\overline{\delta}_n)(m\III{\ell}_{1/2}-J_\rho)}\bigg)\\
  & \qquad \times \bigg(\sum_{\ell\in S_{<, 1/2}} e^{-(1-\overline{\delta}_n)(J_\rho - m\III{\ell}_{1/2})}\bigg) \cdot \prod_{\ell\in S_{<, 1/2}} \bigg(1+ e^{-(1-\overline{\delta}_n)(J_\rho - m\III{\ell}_{1/2})}\bigg)\\
  & \qquad + 
  \exp\bigg\{-(1-\overline{\delta}_n)\bigg(|S_{\geq, 1/2}|J_\rho + \sum_{\ell\in S_{<, 1/2}}m\III{\ell}_{1/2}\bigg)\bigg\}\\
  & \qquad \qquad \times \bigg(\sum_{\ell\in S_{\geq, 1/2}} e^{-(1-\overline{\delta}_n)(m \III{\ell}_{1/2}-J_\rho)} \bigg) 
  \cdot \prod_{\ell\in S_{\geq ,1/2}} \bigg(1+ e^{-(1-\overline{\delta}_n)(J_\rho - m \III{\ell}_{1/2})}\bigg)\\
  & \leq \exp\bigg\{-(1-\delta_{2, n})\bigg(|S_{\geq, 1/2}|J_\rho + \sum_{\ell\in S_{<, 1/2}}m\III{\ell}_{1/2}\bigg)\bigg\}\\
  & \qquad \times \bigg(\sum_{\ell\in S_{<, 1/2}} e^{-(1-\overline{\delta}_n)(J_\rho - m\III{\ell}_{1/2})}\bigg) \cdot \prod_{\ell\in S_{<, 1/2}} \bigg(1+ e^{-(1-\overline{\delta}_n)(J_\rho - m\III{\ell}_{1/2})}\bigg)\\
  & \qquad + 
  \exp\bigg\{-(1-\delta_{2, n})\bigg(|S_{\geq, 1/2}|J_\rho + \sum_{\ell\in S_{<, 1/2}}m\III{\ell}_{1/2}\bigg)\bigg\}
  \cdot \bigg(\sum_{\ell\in S_{\geq, 1/2}} e^{-(1-\overline{\delta}_n)(m \III{\ell}_{1/2}-J_\rho)} \bigg), 
\end{align*}
where $\delta_{2,n} = o(1)$ and the last inequality is by $e^{-(1-\overline{\delta}_n)(m\III{\ell}_{1/2}-J_\rho)}\leq 1 \ll J_\rho$ for any $\ell\in S_{\geq, 1/2}$. 
Defining
$$
  \ell_{<, 1/2} := \argmax_{\ell\in S_{<, 1/2}} \frac{m\III{\ell}_{1/2}}{J_\rho}, \qquad \ell_{\geq , 1/2} := \argmin_{\ell\in S_{\geq, 1/2}} \frac{m\III{\ell}_{1/2}}{J_\rho},
$$
we have
\begin{align*}
  & \sum_{S: |S^c|\textnormal{ even}} \exp\big\{- (1-\overline{\delta}_n) \big(|S^c|J_\rho + \psi_S^\star(0)\big)\big\} \\
  & \leq \exp\bigg\{-(1-{\delta_{2, n}})\bigg(|S_{\geq, 1/2}|J_\rho + (J_\rho - m\III{\ell_{<, 1/2}}_{1/2}) + \sum_{\ell\in S_{<, 1/2}}m\III{\ell}_{1/2}\bigg)\bigg\} \\
  & \qquad \times |S_{<, 1/2}| \exp\bigg\{\sum_{\ell \in S_{<, 1/2}} e^{-(1-\overline{\delta}_n)(J_\rho - m\III{\ell}_{1/2})}\bigg\}\\
  & \qquad + \exp\bigg\{-(1-{\delta_{2,n}})\bigg(|S_{\geq, 1/2}|J_\rho + (m\III{\ell_{\geq, 1/2}}_{1/2} - J_\rho ) + \sum_{\ell\in S_{<, 1/2}}m\III{\ell}_{1/2}\bigg)\bigg\} \times |S_{\geq, 1/2}|,
\end{align*}
Using similar argument as those in Case A, we can proceed by
\begin{align*}
  & \sum_{S: |S^c|\textnormal{ even}} \exp\big\{- (1-\overline{\delta}_n) \big(|S^c|J_\rho + \psi_S^\star(0)\big)\big\} \\
  & \leq \exp\bigg\{-(1-{\delta_{3, n}})\bigg(|S_{\geq, 1/2}|J_\rho + (J_\rho - m\III{\ell_{<, 1/2}}_{1/2}) + \sum_{\ell\in S_{<, 1/2}}m\III{\ell}_{1/2}\bigg) + \log L + Le^{-(1-\overline{\delta}_n) c J_\rho}\bigg\} \\
  & \qquad + \exp\bigg\{-(1-{\delta_{2,n}})\bigg(|S_{\geq, 1/2}|J_\rho + (m\III{\ell_{\geq, 1/2}}_{1/2} - J_\rho ) + \sum_{\ell\in S_{<, 1/2}}m\III{\ell}_{1/2}\bigg) + \log L\bigg\}
\end{align*}
for some $\delta_{3,n}=o(1)$. 
Since $|S_{\geq, 1/2}\cup\{\ell_{<, 1/2}\}|$ and $|S_{\geq, 1/2}\setminus \{\ell_{\geq, 1/2}\}|$ are both even, we conclude that 
\begin{align*}
  & \sum_{S: |S^c|\textnormal{ even}} \exp\big\{- (1-\overline{\delta}_n) \big(|S^c|J_\rho + \psi_S^\star(0)\big)\big\}\\
  \label{eq:opt_glob_even_caseB}
  & \leq 2\exp\bigg\{-(1-\delta_{4,n }) \min_{S: |S^c|\textnormal{ even}}\bigg(|S^c|J_\rho + \sum_{\ell\in S}m\III{\ell}_{1/2}\bigg) + \log L + Le^{-(1-\overline{\delta}_n)cJ_\rho}\bigg\} \numberthis 
\end{align*}
for some $\delta_{4, n}=o(1)$

The proof is concluded by combining   \eqref{eq:opt_glob_even_caseA} and \eqref{eq:opt_glob_even_caseB}.

\subsubsection{Proof of Proposition \ref{prop:opt_glob_odd}}\label{prf:prop:opt_glob_odd}
Recall the definition of $S_{\geq, t}$ and $S_{<, t}$ in \eqref{eq:sets}. Let $m = n/2$. For any $S\subseteq[L]$, define
$$
  t_S:= \argmax_{t\in[0, 1]} -2t J_\rho + \sum_{\ell \in S}m\III{\ell}_{t}.
$$
Fix an arbitrary constant $t\in [0, 1]$. 
Similar to the proof of Proposition \ref{prop:opt_glob_even}, we will assume $S_{\geq, t}$ and $S_{< , t}$ are both non-empty, and the case of one of them being empty is treated similarly.
We have
\begin{align}
  & \sum_{S\subseteq[L]:  |S^c|\textnormal{ odd}} \exp\bigg\{ -(1-\overline{\delta}_n)\cdot \bigg( (|S^c|+1)J_\rho + \psi_S^\star(-2J_\rho)\bigg) \bigg\}\nonumber\\
  & =\sum_{S\subseteq[L]:  |S^c|\textnormal{ odd}} \exp\bigg\{ -(1-\overline{\delta}_n)\cdot \bigg( |S^c|J_\rho + (1-2t_S) J_\rho + \sum_{\ell \in S} m\III{\ell}_{t_S}\bigg) \bigg\}\nonumber\\
  \label{eq:opt_glob_even_general_t}
  & \leq \sum_{S\subseteq[L]:  |S^c|\textnormal{ odd}} \exp\bigg\{ -(1-\overline{\delta}_n)\cdot \bigg( |S^c|J_\rho + (1-2t) J_\rho + \sum_{\ell \in S} m\III{\ell}_{t}\bigg) \bigg\}.
\end{align}
Similar to   \eqref{eq:opt_glob_even_summation} and \eqref{eq:opt_glob_even_difference}, we have
\begin{align*}
  & \sum_{S\subseteq[L]:  |S^c|\textnormal{ odd}} \exp\bigg\{ -(1-\overline{\delta}_n)\cdot \bigg( |S^c|J_\rho + (1-2t) J_\rho + \sum_{\ell \in S} m\III{\ell}_{t}\bigg) \bigg\} \\
  & \qquad +\sum_{S\subseteq[L]:  |S^c|\textnormal{ even}} \exp\bigg\{ -(1-\overline{\delta}_n)\cdot \bigg( |S^c|J_\rho + (1-2t) J_\rho + \sum_{\ell \in S} m\III{\ell}_{t}\bigg) \bigg\}  \\
  \label{eq:opt_glob_odd_summation}
  & = e^{-(1-\overline{\delta}_n) (1-2t)J_\rho} \prod_{\ell\in[L]} \bigg(e^{-(1-\overline{\delta}_n)J_\rho} + e^{-(1-\overline{\delta}_n) m\III{\ell}_{t}}\bigg),\numberthis
\end{align*}
and
\begin{align*}
  & -\sum_{S\subseteq[L]:  |S^c|\textnormal{ odd}} \exp\bigg\{ -(1-\overline{\delta}_n)\cdot \bigg( |S^c|J_\rho + (1-2t) J_\rho + \sum_{\ell \in S} m\III{\ell}_{t}\bigg) \bigg\} \\
  & \qquad +\sum_{S\subseteq[L]:  |S^c|\textnormal{ even}} \exp\bigg\{ -(1-\overline{\delta}_n)\cdot \bigg( |S^c|J_\rho + (1-2t) J_\rho + \sum_{\ell \in S} m\III{\ell}_{t}\bigg) \bigg\}  \\
  \label{eq:opt_glob_odd_difference}
  & = e^{-(1-\overline{\delta}_n) (1-2t)J_\rho} \prod_{\ell\in[L]} \bigg(-e^{-(1-\overline{\delta}_n)J_\rho} + e^{-(1-\overline{\delta}_n) m\III{\ell}_{t}}\bigg).\numberthis
\end{align*}
With the above two equations at hand, using similar arguments as those in the proof of Proposition \ref{prop:opt_glob_even}, we arrive at
\begin{align*}
  & \sum_{S\subseteq[L]:  |S^c|\textnormal{ odd}} \exp\bigg\{ -(1-\overline{\delta}_n)\cdot \bigg( (|S^c|+1)J_\rho + \psi_S^\star(-2J_\rho)\bigg) \bigg\}\\
  \label{eq:opt_glob_odd_summary_two_cases}
  &  {\leq} 2\exp\bigg\{-(1-\delta_{1,n})\min_{S: |S^c|\textnormal{ odd}}\bigg(|S^c|J_\rho + (1-2t)J_\rho + \sum_{\ell\in S} m\III{\ell}_t\bigg)+ \log L + Le^{-(1-\overline{\delta}_n) cJ_\rho}\bigg\},\numberthis
\end{align*}
where $\delta_{1,n}=o(1)$.

We claim that
\begin{equation}
  \label{eq:minimax_equality_z_star}
  \sup_{t\in[0, 1]}\min_{\substack{S\subseteq[L]\\ |S^c|\textnormal{ odd}}} \big(|S^c|J_\rho + (1-2t)J_\rho + \sum_{\ell\in S} m\III{\ell}_{t} \big)= \min_{\substack{S\subseteq[L]\\ |S^c|\textnormal{ odd}}} \big(\sup_{t\in[0, 1]} |S^c|J_\rho + (1-2t)J_\rho + \sum_{\ell\in S} m\III{\ell}_{t}\big).
\end{equation}  
If this claim holds, then plugging the optimal $t$ to   \eqref{eq:opt_glob_odd_summary_two_cases} gives the desired result.

To prove   \eqref{eq:minimax_equality_z_star}, we need the following theorem.
\begin{theorem}[Sion's minimax theorem]
  \label{thm:minimax}
  Let $X$ be a compact convex subset of a topological vector space and $Y$ be a convex subset of a topological vector space. If $f$ is a real-valued function on $X\times Y$ such that
  \begin{enumerate}
    \item for each $x\in X$, $f(x, \cdot)$ is upper semi-continuous (usc) and quasi-concave on $Y$,
    \item for each $y\in Y$, $f(\cdot, y)$ is lower semi-continuous (lsc) and quasi-convex on $X$,
  \end{enumerate}
  then 
  $$
    \min_{x\in X} \sup_{y\in Y} f(x, y) = \sup_{y\in Y} \min_{x\in X} f(x, y).
  $$
\end{theorem}
To use the above theorem, we let
$$
  \calV := \bigg\{v\in\{0, 1\}^L: \#\{\ell\in[L]: v_\ell = 0\} \textnormal{ is odd}\bigg\}.
$$
Then the right-hand side of \eqref{eq:minimax_equality_z_star} is
\begin{align*}
   \sup_{t\in[0, 1]} \min_{v\in\calV}\ f(v, t),
\end{align*}
where
\begin{align*}
  f(v, t) & = \langle J_\rho \mathbf{1}_L, \mathbf{1}_L - v \rangle + (1-2t)J_\rho + \langle v, m I_{*, t} \rangle\\
  & = LJ_\rho + (1-2t)J_\rho + \langle v, mI_{*, t} - J_\rho \mathbf{1}_L\rangle,
\end{align*}
where we let $I_{*, t}$ be the $L$-dimensional vector whose $\ell$-th entry is $\III{\ell}_{t}$. It suffices to show
$$
  \min_{v\in\calV}\sup_{t\in[0, 1]}\ f(v, t) =  \sup_{t\in[0, 1]} \min_{v\in\calV}\ f(v, t).
$$
We are to invoke a version of minimax theorem, but an immediate difficulty is that $\calV$ is non-convex. Fortunately we have the following lemma.
\begin{lemma}
  \label{lemma: hidden_convexity_of_the_vertex_set}
  For any $t\in[0, 1]$, we have
  $$
    \min_{v\in \calV} f(v, t) = \min_{v\in \textnormal{conv}(\calV)} f(v, t),
  $$
  where $\textnormal{conv}(\calV)$ is the convex hull of $\calV$. 
\end{lemma}
\begin{proof}
  This follows from the fact that the convex hull of $\calV$ is a polytope. 
\end{proof}

By Lemma \ref{lemma: hidden_convexity_of_the_vertex_set}, we conclude that 
$$
   \sup_{t\in[0, 1]} \min_{v\in\calV}\ f(v, t) =  \sup_{t\in[0, 1]} \min_{v\in\textnormal{conv}(\calV)}\ f(v, t)
$$
and
$$
   \min_{v\in\calV}\sup_{t\in[0, 1]} \ f(v, t) =   \min_{v\in\textnormal{conv}(\calV)}\sup_{t\in[0, 1]}\ f(v, t).
$$
Hence, it suffices to show 
$$
  \min_{v\in\textnormal{conv}(\calV)}\sup_{t\in[0, 1]}\ f(v, t) =   \sup_{t\in[0, 1]} \min_{v\in\textnormal{conv}(\calV)}\ f(v, t).
$$
We make a few observations:
\begin{enumerate}
  \item $\textnormal{conv}(\calV)$ is a compact convex subset of the topological vector space $\bbR^L$ (with the Euclidean topology);
  \item $[0, 1]$ is a convex subset of the topological vector space $\bbR$ (again with the Euclidean topology);
  \item For each $v\in \textnormal{conv}(\calV)$, the function $f(v, \cdot)$ is continuous and concave (and hence quasi-concave) in $t$;
  \item For each $t\in[0, 1]$, the function $f(\cdot, t)$ is linear (and hence quasi-convex) in $v$. 
\end{enumerate}
Thus, invoking Theorem \ref{thm:minimax}, we obtain the desired equality.

\subsection{Proof of Theorem \ref{thm:opt_ind_est}}\label{prf:thm:opt_ind_est}
This theorem is a consequence of the following two propositions.
\begin{proposition}
\label{prop:opt_ind_given_glob}
For any $\overline{\delta}_n = o(1)$, there exists another $\overline{\delta}_n'=o(1)$ such that for any $\ell\in[L]$, we have
\begin{align*}
  & \sum_{\substack{S\subseteq L\setminus\{\ell\} }} e^{-(1-\overline{\delta}_n)\indinfo_{S\cup\{\ell\}}}\leq 2 \exp\big\{-(1-\overline{\delta}_n)\indinfo_{\{\ell\}} + \log L + Le^{-(1-\overline{\delta}_n)J_\rho}\big\}. 
\end{align*}  
\end{proposition}
\begin{proof}
  See Appendix \ref{prf:prop:opt_ind_given_glob}.
\end{proof}

\begin{proposition}
\label{prop:opt_ind_est_glob}
For any $\overline{\delta}_n = o(1)$ and any $c\in(0, 1)$, there exists another $\overline{\delta}_n'=o(1)$ such that for any $\ell\in[L]$, we have
\begin{align*}
  &\sum_{\substack{S\subseteq [L]\setminus\{\ell\}\\ |(S\cup\{\ell\})^c|\textnormal{ even}}} \exp\big\{-(1-\overline{\delta}_n) \big(|(S\cup\{\ell\})^c| J_\rho + \psi_{S\cup\{\ell\}}^\star(0)\big)\big\}\\
  & \leq 2\exp\big\{ -(1-\overline{\delta}_n') \min_{\substack{S\subseteq[L]\setminus\{\ell\} \\ |(S\cup\{\ell\})^c| \textnormal{ even}}} \big( |(S\cup\{\ell\})^c| J_\rho + \psi_{S\cup\{\ell\}}^\star(0) \big) + \log L + Le^{-(1-\overline{\delta}_n)c J_\rho}\big\},
\end{align*}
and
\begin{align*}
    & \sum_{\substack{S\subseteq[L]\setminus\{\ell\} \\ |(S\cup\{\ell\})^c| \textnormal{ odd}}} \exp\big\{- (1-\overline{\delta}_n) \big(|(S\cup\{\ell\})^c|J_\rho + \psi_{S\cup\{\ell\}}^\star(-2J_\rho)\big)\big\} \\
  &  \leq 2 \exp\big\{-(1-\overline{\delta}_n') \min_{\substack{S\subseteq[L]\setminus\{\ell\}\\ |(S\cup\{\ell\})^c| \textnormal{ odd}}}\big(|(S\cup\{\ell\})^c|J_\rho + \psi_{S\cup\{\ell\}}^\star(-2J_\rho)\big) + \log L + Le^{-(1-\overline{\delta}_n)cJ_\rho}\big\}.
\end{align*}
\end{proposition}
\begin{proof}
  The proof is nearly identical to that of Propositions \ref{prop:opt_glob_even} and \ref{prop:opt_glob_odd}, so we omit the details.
\end{proof}

\subsubsection{Proof of Proposition \ref{prop:opt_ind_given_glob}}\label{prf:prop:opt_ind_given_glob}
Without loss of generality we consider the first layer, i.e., the $\ell$ in the statement of this proposition is $1$.

~\\
{\bf The even terms.} We first consider the terms whose $|S\cup\{1\}|$'s are even.
Note that
\begingroup
\allowdisplaybreaks
\begin{align*}
  & \sum_{\substack{S\subseteq\{2, \hdots, L\}\\ |S\cup\{1\}|\textnormal{ even}}} \exp\bigg\{ -(1-\overline{\delta}_n)\bigg(|S\cup\{1\}|J_\rho + \sum_{\ell \in S\cup\{1\}}m\III{\ell}_{1/2}\bigg) \bigg\} \\
  & \qquad + \sum_{\substack{S\subseteq\{2, \hdots, L\}\\ |S\cup\{1\}|\textnormal{ odd}}} \exp\bigg\{ -(1-\overline{\delta}_n)\bigg(|S\cup\{1\}|J_\rho + \sum_{\ell \in S\cup\{1\}}m\III{\ell}_{1/2}\bigg) \bigg\} \\
  & = e^{-(1-\overline{\delta}_n)(J_\rho + m\III{1}_{1/2})} \times \prod_{\ell \neq 1} \bigg(1 + e^{-(1-\overline{\delta}_n)(J_\rho + m\III{\ell}_{1/2})}\bigg).
\end{align*}
\endgroup
On the other hand, we have
\begingroup
\allowdisplaybreaks
\begin{align*}
  & \sum_{\substack{S\subseteq\{2, \hdots, L\}\\ |S\cup\{1\}|\textnormal{ even}}} \exp\bigg\{ -(1-\overline{\delta}_n)\bigg(|S\cup\{1\}|J_\rho + \sum_{\ell \in S\cup\{1\}}m\III{\ell}_{1/2}\bigg) \bigg\} \\
  & \qquad - \sum_{\substack{S\subseteq\{2, \hdots, L\}\\ |S\cup\{1\}|\textnormal{ odd}}} \exp\bigg\{ -(1-\overline{\delta}_n)\bigg(|S\cup\{1\}|J_\rho + \sum_{\ell \in S\cup\{1\}}m\III{\ell}_{1/2}\bigg) \bigg\} \\
  & = -e^{-(1-\overline{\delta}_n)(J_\rho + m\III{1}_{1/2})} \prod_{\ell\neq 1} \bigg(1 - e^{-(1-\overline{\delta}_n)(J_\rho + m\III{\ell}_{1/2})}\bigg).
\end{align*}
\endgroup
Hence, we have
\begin{align*}
  & \sum_{\substack{S\subseteq\{2, \hdots, L\}\\ |S\cup\{1\}|\textnormal{ even}}} \exp\bigg\{ -(1-\overline{\delta}_n)\bigg(|S\cup\{1\}|J_\rho + \sum_{\ell \in S\cup\{1\}}m\III{\ell}_{1/2}\bigg) \bigg\} \\
  & = \frac{1}{2} e^{-(1-\overline{\delta}_n)(J_\rho + m\III{1}_{1/2})} \times \bigg[ \prod_{\ell\neq 1} \bigg(1 + e^{-(1-\overline{\delta}_n)(J_\rho + m\III{\ell}_{1/2})}\bigg) - \prod_{\ell\neq 1} \bigg(1 - e^{-(1-\overline{\delta}_n)(J_\rho + m\III{\ell}_{1/2})}\bigg)\bigg]\\
  & \leq e^{-(1-\overline{\delta}_n)(J_\rho + m\III{1}_{1/2})} \cdot \bigg(\sum_{i\neq 1} e^{-(1-\overline{\delta}_n)(J_\rho + m\III{\ell}_{1/2})}\bigg) \cdot \prod_{\ell\neq 1} \bigg(1+ e^{-(1-\overline{\delta}_n)(J_\rho + m\III{\ell}_{1/2})}\bigg)\\
  & \leq e^{-(1-\overline{\delta}_n)(2J_\rho + m\III{1}_{1/2})}\cdot L \cdot \bigg(1+e^{-(1-\overline{\delta}_n)J_\rho}\bigg)^{L}\\ 
  \label{eq:opt_ind_given_glob_even}
  & \leq e^{-(1-\overline{\delta}_n)\indinfo_{\{1\}}} \cdot L \exp\big\{ L e^{-(1-\overline{\delta}_n)J_\rho}\big\},\numberthis
\end{align*}
where the third line is by Lemma \ref{lemma: diff_of_products} and the last inequality is by Lemma \ref{lemma:ineq_for_cvx_conjugate}. 

~\\
{\bf\noindent The odd terms.} We now consider the terms whose $|S\cup\{1\}|$'s are odd. For any fixed $t\in[0, 1]$, we have
\begin{align*}
  &\sum_{\substack{S\subseteq\{2, \hdots, L\} \\ |S\cup\{1\}| \textnormal{ odd}}}\exp\bigg\{- (1-\overline{\delta}_n) \bigg(\bigl(|(S\cup\{1\})| + 1\bigr) J_\rho + \psi_{S\cup\{1\}}^\star(-2J_\rho)\bigg)  \bigg\} \\
  & \leq \sum_{\substack{S\subseteq\{2, \hdots, L\} \\ |S\cup\{1\}| \textnormal{ odd}}}\exp\bigg\{- (1-\overline{\delta}_n) \bigg(\bigl(|(S\cup\{1\})| + 1\bigr) J_\rho + (1-2t)J_\rho + \sum_{\ell \in S\cup\{1\}}m\III{\ell}_{t}\bigg)  \bigg\} \\
  & = \exp\bigg\{-(1-\overline{\delta}_n)\bigg(J_\rho + (1-2t)J_\rho + m\III{1}_{t}\bigg)\bigg\} \times \sum_{\substack{S\subseteq\{2, \hdots, L\} \\ |S\cup\{1\}| \textnormal{ odd}}} \exp\bigg\{-(1-\overline{\delta}_n)\bigg(|S|J_\rho + \sum_{\ell\in S} m\III{\ell}_{t}\bigg)\bigg\}.
\end{align*}
Note that
\begin{align*}
  & \sum_{\substack{S\subseteq\{2, \hdots, L\} \\ |S\cup\{1\}| \textnormal{ odd}}} \exp\bigg\{-(1-\overline{\delta}_n)\bigg(|S|J_\rho + \sum_{\ell\in S} m\III{\ell}_{t}\bigg)\bigg\} + \sum_{\substack{S\subseteq\{2, \hdots, L\} \\ |S\cup\{1\}| \textnormal{ even}}} \exp\bigg\{-(1-\overline{\delta}_n)\bigg(|S|J_\rho + \sum_{\ell\in S} m\III{\ell}_{t}\bigg)\bigg\} \\
  & = \prod_{\ell\neq 1} \bigg(1+ e^{-(1-\overline{\delta}_n)(J_\rho + m\III{\ell}_{t})}\bigg).
\end{align*}
Meanwhile, we have
\begin{align*}
  & \sum_{\substack{S\subseteq\{2, \hdots, L\} \\ |S\cup\{1\}| \textnormal{ odd}}} \exp\bigg\{-(1-\overline{\delta}_n)\bigg(|S|J_\rho + \sum_{\ell\in S} m\III{\ell}_{t}\bigg)\bigg\} - \sum_{\substack{S\subseteq\{2, \hdots, L\} \\ |S\cup\{1\}| \textnormal{ even}}} \exp\bigg\{-(1-\overline{\delta}_n)\bigg(|S|J_\rho + \sum_{\ell\in S} m\III{\ell}_{t}\bigg)\bigg\} \\
  & = \prod_{\ell\neq 1} \bigg(1- e^{-(1-\overline{\delta}_n)(J_\rho + m\III{\ell}_{t})}\bigg).
\end{align*}
Hence, the odd terms are bounded above by
\begin{align*}
  \exp\bigg\{-(1-\overline{\delta}_n)\bigg(J_\rho + (1-2t)J_\rho + m\III{1}_t\bigg)\bigg\} \times \frac{1}{2} \bigg[ \prod_{\ell\neq 1} \bigg(1+e^{-(1-\overline{\delta}_n)(J_\rho + m\III{\ell}_{t})}\bigg) + \prod_{\ell\neq 1} \bigg(1-e^{-(1-\overline{\delta}_n)(J_\rho + m\III{\ell}_{t})} \bigg)\bigg].
\end{align*}
Choosing $t\in[0, 1]$ such that it maximizes $-2tJ_\rho + m\III{1}_t$, and using 
\begin{align*}
  & \frac{1}{2} \bigg[ \prod_{\ell\neq 1} \bigg(1+e^{-(1-\overline{\delta}_n)(J_\rho + m\III{\ell}_{t})}\bigg) + \prod_{\ell\neq 1} \bigg(1-e^{-(1-\overline{\delta}_n)(J_\rho + m\III{\ell}_{t})} \bigg)\bigg] \leq \bigg(1 + e^{-(1-\overline{\delta}_n) J_\rho}\bigg)^{L-1},
\end{align*}
we get
\begin{align*}
  &\sum_{\substack{S\subseteq\{2, \hdots, L\} \\ |S\cup\{1\}| \textnormal{ odd}}}\exp\bigg\{- (1-\overline{\delta}_n) \bigg(\bigl(|(S\cup\{1\})| + 1\bigr) J_\rho + \psi_{S\cup\{1\}}^\star(-2J_\rho)\bigg)  \bigg\} \\
  & \leq \exp\bigg\{-(1-\overline{\delta}_n)\bigg(2J_\rho + \psi_{\{1\}}^\star(-2J_\rho)\bigg) + L e^{-(1-\overline{\delta}_n)J_\rho}\bigg\} \\
  \label{eq:opt_ind_given_glob_odd}
  & =\exp\bigg\{-(1-\overline{\delta}_n)\indinfo_{\{1\}} + L e^{-(1-\overline{\delta}_n)J_\rho}\bigg\}\numberthis.
\end{align*}  

The proof is concluded by combining  \eqref{eq:opt_ind_given_glob_even} and \eqref{eq:opt_ind_given_glob_odd}.

\subsection{Proof of Theorem \ref{thm:refine_glob_const_rho}}\label{prf:thm:refine_glob_const_rho}
We adopt the notations in the proof of Theorem \ref{thm:refine_glob}.
We first present a useful proposition that is similar to Proposition \ref{prop:refine_glob_nodewise}.
\begin{proposition}
  \label{prop:refine_glob_nodewise_const_rho}
  Assume $1\lesssim \rho <1/2$, $q _\ell < p_\ell \leq (Cq_\ell) \land (1-c), \forall \ell\in[L]$, $\beta= 1+o(1)$, and $\log L\ll n^{c'}$ for some constants $C > 1$ and $c, c' \in (0, 1)$. In addition, assume there exists a sequence $\delta_n=o(1)$ and constants $\ep_{\init}>0, C'>0$ such that $\forall i\in[n], \exists \pi_i\in\{\pm 1\}$ which makes the following holds:
  \begin{equation*}
    \bbP\bigg(d_\ham(\pi_i \tzzz{\star, -i}_{-i} ,\zzz{\ell}_{-i})\leq n\delta_n ~ \forall \ell\in[n]\bigg) \geq 1 - C' n^{-(1+\ep_{\init})},
  \end{equation*}
  Then there exists another sequence $\delta_n'=o(1)$ and an absolute constant $C''>0$ such that for any $i\in[n]$, we have 
    \begin{align*}
    \bbP(\pi_i\barzzz{\star, -i}_i \neq  \bz^\star_i) 
    & \leq  C''n^{-(1+\ep_{\init})}+ \sum_{S\subseteq[L]}e^{-(1-\delta_n') [|S^c|J_\rho + \psi_S^\star(0)]}
  \end{align*}
\end{proposition}
\begin{proof}[Proof of Proposition \ref{prop:refine_glob_nodewise_const_rho}]
Following the proof of Proposition \ref{prop:refine_glob_nodewise}, one can readily check that \eqref{eq:nodewise_err_decomp_3terms_simplified} still holds under the current setting. That is, we have
\begin{align*}
  &\bbP(\pi_i\barzzz{\star, -i}_i \neq  \bz^\star_i) \\
  & \leq C'' n^{-(1+\ep_\init)} + 
  \sum_{S \subseteq [L]} 
  \sum_{x \in\{ -|S^c|+2k: 0\leq k \leq |S^c|\}} \binom{|S^c|}{\frac{|S^c| + x}{2}}\\
  & \qquad  
  \times \exp\bigg\{ -|S^c|\log\frac{1}{\sqrt{\rho(1-\rho)}} + x(1-2t_{S, x, \xi})\log \sqrt{\frac{1-\rho}{\rho}}\bigg\} \nonumber\\
  & \qquad\times \sum_{\xi\in\{\pm 1\}^S}\bbP(\zzz{S}_i =  \bz^\star_i \xi ) 
  \cdot \exp\bigg\{ -\frac{n}{2}\sum_{\ell\in S}\III{\ell}_{t_{S, x, \xi}}+ o(1)\cdot n\sum_{\ell\in S}\III{\ell}_{1/2}\bigg\}.
\end{align*}
Choosing $t_{S, x, \xi} = 1/2$, the right-hand side above becomes 
\begin{align*}
  & C'' n^{-(1+\ep_\init)} + \sum_{S\subseteq[L]} \sum_{x \in\{ -|S^c|+2k: 0\leq k \leq |S^c|\}} \binom{|S^c|}{\frac{|S^c| + x}{2}}
  \exp\bigg\{ -|S^c|\log\frac{1}{\sqrt{\rho(1-\rho)}} \bigg\} \nonumber\\
  & \qquad\times \sum_{\xi\in\{\pm 1\}^S}\bbP(\zzz{S}_i =  \bz^\star_i \xi ) 
  \cdot \exp\bigg\{ -\frac{(1+o(1))n}{2}\sum_{\ell\in S}\III{\ell}_{1/2}\bigg\}\\
  & = C'' n^{-(1+\ep_\init)} + \sum_{S\subseteq[L]} \exp\bigg\{ -|S^c|\log\frac{1}{\sqrt{\rho(1-\rho)}} - \frac{(1+o(1))n}{2}\sum_{\ell\in S}\III{\ell}_{1/2} \bigg\}\\
  & \qquad \times \sum_{x \in\{ -|S^c|+2k: 0\leq k \leq |S^c|\}} \binom{|S^c|}{\frac{|S^c| + x}{2}}\\
  & = C'' n^{-(1+\ep_\init)} + \sum_{S\subseteq[L]} \exp\bigg\{ -|S^c|\log\frac{1}{\sqrt{\rho(1-\rho)}} - \frac{(1+o(1))n}{2}\sum_{\ell\in S}\III{\ell}_{1/2} \bigg\} \cdot 2^{|S^c|}\\
  & = C'' n^{-(1+\ep_\init)} + \sum_{S\subseteq[L]} \exp\bigg\{ -|S^c|J_\rho- \frac{(1+o(1))n}{2}\sum_{\ell\in S}\III{\ell}_{1/2} \bigg\}. 
\end{align*}
The proof is concluded by recognizing $\sum_{S}\III{\ell}_{1/2} = \psi_S^\star(0)$.
\end{proof}
With the above proposition, the rest of the proof is the same as the proof of Theorem \ref{thm:refine_glob}, and we omit the details.

\subsection{Proof of Theorem \ref{thm:refine_ind_const_rho}}\label{prf:thm:refine_ind_const_rho}
  The proof is nearly the same as that of Theorem \ref{thm:refine_ind}, except that we now use we use the following proposition instead of Proposition \ref{prop:refine_ind_nodewise}.
  \begin{proposition}
  \label{prop:refine_ind_nodewise_const_rho}
  Fix $\ell\in[L]$. Under the setup of Proposition \ref{prop:refine_glob_nodewise_const_rho}, there exists a sequence $\delta_n'=o(1)$ and an absolute constant $C''>0$ such that for any $i\in[n]$, we have 
    \begin{align*}
    \bbP(\pi_i\barzzz{\ell, -i}_i \neq  \zzz{\ell}_i) 
    & \leq  C''n^{-(1+\ep_{\init})} \\
    & \qquad + \sum_{S\subseteq[L]\setminus\{\ell\}} e^{-(1-\delta_n')\psi_{S\cup\{\ell\}}^\star(0)}\cdot  \Big( e^{-(1-\delta_n')|(S\cup \{\ell\})^c|J_\rho} + e^{-(1-\delta_n') |S\cup\{\ell\}|J_\rho} \Big).
  \end{align*}
\end{proposition}
\begin{proof}
Without loss of generality we consider the first layer and we assume $\pi_i = +1$. We start by computing
\begin{align*}
  \bbP(\barzzz{1, -i}_i\neq \zzz{1}_i) = \bbP(\barzzz{1, -i}_i = -\zzz{1}_i) \leq \bbP(\barzzz{1, -i}_i = -\zzz{1}_i, E_i\cap F_i) + C'' n^{-(1+\ep_{\init})},
\end{align*}
where $E_i, F_i$ are defined in \eqref{eq:Ei} and \eqref{eq:Fi} respectively, and the last inequality is by \eqref{eq:high_prob_Ei_Fi}. We now proceed by
\begin{align*}
  & \bbP(\barzzz{1, -i}_i = -\zzz{1}_i, E_i\cap F_i) \\
  & = \sum_{S\subseteq \{2, \hdots, L\}} \bigg(\bbP(\barzzz{\star, -i}_i = -\bz^\star_i, \barzzz{1, -i}_i = -\zzz{1}_i, \barzzz{S, -i}_i = -\zzz{S}_i, \barzzz{S^c, -i}_i = \zzz{S^c}_i, E_i\cap F_i) \\
  & \qquad + \bbP(\barzzz{\star, -i}_i = \bz^\star_i, \barzzz{1, -i}_i = -\zzz{1}_i, \barzzz{S, -i}_i = -\zzz{S}_i, \barzzz{S^c, -i}_i = \zzz{S^c}_i, E_i\cap F_i)\bigg) ,
\end{align*}
The right-hand side above is the superposition of two terms, the first of which has already been calculated in the proof of Proposition \ref{prop:refine_glob_nodewise_const_rho}:
\begin{align*}
  & \sum_{S\subseteq\{2, \hdots, L\}}\bbP(\barzzz{\star, -i}_i = -\bz^\star_i, \barzzz{1, -i}_i = -\zzz{1}_i, \barzzz{S, -i}_i = -\zzz{S}_i, \barzzz{S^c, -i}_i = \zzz{S^c}_i, E_i\cap F_i)\\
  & \leq \sum_{S\subseteq\{2, \hdots, L\}} e^{-(1+o(1))[|(S\cup\{1\})^c|J_\rho + \psi_{S\cup\{1\}}^\star(0)]}.
\end{align*}
For the second-term, we use \eqref{eq:ind_nodewise_err_decomp_indinfo} and \eqref{eq:ind_nodewise_err_caseA} (which still hold under the current setting) to conclude that
\begin{align*}
  &\sum_{S\subseteq\{2, \hdots, L\}} \bbP(\barzzz{\star, -i}_i = \bz^\star_i, \barzzz{S\cup\{1\}, -i}_i = -\zzz{S\cup\{1\}}_i, \barzzz{(S\cup\{1\})^c, -i}_i = \zzz{(S\cup\{1\})^c}_i, E_i\cap F_i)\\
  & \leq \sum_{S\subseteq[2, \hdots, L]} \exp\bigg\{ -|S\cup\{1\}|J_\rho -  \frac{(1+o(1)) \cdot n}{2}\sum_{\ell\in S\cup\{1\}}\III{\ell}_{1/2}\bigg\}\\
  & = \sum_{S\subseteq[2, \hdots, L]} e^{-(1+o(1))[|S\cup\{1\}|J_\rho + \psi^\star_{S\cup\{1\}}(0)]}.
\end{align*}
The proof is concluded by summarizing the above two displays.
\end{proof}

\subsection{Proof of Theorem \ref{thm:refine_glob_and_ind_cluster_imbalance}} \label{prf:thm:refine_glob_and_ind_cluster_imbalance}
The proof is almost the same as the proofs of Theorem \ref{thm:refine_glob} and \ref{thm:refine_ind}, except that instead of using \eqref{eq:glob_nodewise_err_decomp_1st_term}, we use
\begin{align*}
  & \mathscr{T}_{1, \ell}\cdot \Indc_{E_i\cap F_i} \\
  & \leq \exp\bigg\{\frac{(1-2\rho)(n^{\star}_{\bz^\star_i} - n^{\star}_{-\bz^\star_i})}{2}\cdot \bigg(\indc{\xi_\ell = 1} - \indc{\xi_\ell = -1}\bigg)
  \times \log \bigg(\frac{p_\ell^t q_\ell^{1-t} + (1-p_\ell)^t (1-q_\ell)^{1-t}}{p_\ell^{1-t} q_\ell^t + (1-p_\ell)^{1-t}(1-q_\ell)^t}\bigg)\\
  & \qquad  \qquad + \frac{ n\cdot ( 2\delta_n + n^{-(1-c')/2} + n^{-1} )}{2} \cdot \III{\ell}_t\bigg\}\\
  & \leq \exp\bigg\{o(1) \cdot n \III{\ell}_t/2 + (1+o(1)) (\beta - \beta^{-1}) \frac{n}{2} \cdot \bigg|\log \frac{p_\ell^t q_\ell^{1-t} + (1-p_\ell)^t (1-q_\ell)^{1-t}}{p_\ell^{1-t} q_\ell^t + (1-p_\ell)^{1-t}(1-q_\ell)^t}\bigg|\bigg\}.
\end{align*}
So we omit the details.

\subsection{Proof of Theorem \ref{thm:refine_glob_with_est_param}}\label{prf:thm:refine_glob_with_est_param}
We adopt the notations in the proof of Theorem \ref{thm:refine_glob}. 
The following result is analogous to Proposition \ref{prop:refine_glob_nodewise}.

\begin{proposition}
  \label{prop:refine_glob_nodewise_with_est_param}
  Under the setup of Theorem \ref{thm:refine_glob_with_est_param}, if there exists a sequence $\delta_n=o(1)$ and constants $\ep_{\init}>0, C'>0$ such that $\forall i\in[n], \exists \pi_i\in\{\pm 1\}$ which makes the following holds:
  \begin{equation}
    \label{eq:ind_consistent_init_with_est_param}
    \bbP\bigg(d_\ham(\pi_i \tzzz{\star, -i}_{-i} ,\zzz{\ell}_{-i})\leq n\delta_n ~ \forall \ell\in[n]\bigg) \geq 1 - C' n^{-(1+\ep_{\init})},
  \end{equation}
  Then there exists another sequence $\delta_n'=o(1)$ and an absolute constant $C''>0$ such that for any $i\in[n]$, we have 
    \begin{align}
    \label{eq:upper_bound_z_star_with_est_param}
    \bbP(\pi_i\barzzz{\star, -i}_i \neq  \bz^\star_i) 
    & \leq  C''n^{-(1+\ep_{\init})}+ \sum_{S\subseteq[L]}e^{-(1-\delta_n')[|S^c|J_\rho^\dagger + \psi_S^\star(0)]}
  \end{align}
\end{proposition}
Given Proposition \ref{prop:refine_glob_nodewise_with_est_param}, the rest of the proof for Theorem \ref{thm:refine_glob_with_est_param} is exactly the same as the proof of Theorem \ref{thm:refine_glob}, so we omit the details.

The proof of Proposition \ref{prop:refine_glob_nodewise_with_est_param} largely follows the proof of Proposition \ref{prop:refine_glob_nodewise}, with some modifications to take into the randomness in the estimated parameters and the misspecified $\rho^\dagger \neq \rho$. 
Recall the events $E_i$ and $F_i$ defined in \eqref{eq:Ei} and \eqref{eq:Fi}, respectively. We have proved in Section \ref{prf:prop:refine_glob_nodewise} that $E_i \cap F_i$ happens with probability at least $1 - \calO(n^{-(1+ \ep_{\mathtt{init}})})$ (see \eqref{eq:high_prob_Ei_Fi}). We introduce another high probability event in the following lemma.

\begin{lemma}
  Let the input to Stage \RN{1} of Algorithm \ref{alg: provable_refinement} be an instance generated by an $\textnormal{IMLSBM} \in \calP_n(\rho, \{p_\ell\}_1^L, \{q_\ell\}_1^L, \beta)$ satisfying $\rho = o(1), q_\ell < p_\ell \leq (C q_\ell) \land (1-c), \forall \ell\in[L], \beta = 1 + o(1)$ and $L \lesssim n^{c'}$, where $C>1, c\in(0, 1), c' \geq 0$ are absolute constants. Let Assumption \ref{assump:consistent_init} hold. Then there exists a sequence $\delta_{\mathtt{est}, n} = o(1)$ and a constant $\ep_{\mathtt{est}} > 0$ such that 
  $$
    \max_{i\in[n]} \bbP(G_i) \lesssim n^{-(1+ \ep_{\mathtt{est}})},
  $$
  where
  \begin{equation}
    \label{eq:Gi}
    G_i = \{|\hat p^{(-i)}_\ell - p_\ell| \lor |\hat q^{(-i)}_\ell - q_\ell| \leq \delta_{\mathtt{est},n} (p_\ell - q_\ell), \forall \ell \in [L]\}.
  \end{equation}  
\end{lemma}
\begin{proof}
 Under the current assumptions on $\rho$ and $L$, by Corollary \ref{cor:specc_ind_est}, we have
 $
  \calL(\tzzz{\star, -i}, \zzz{\ell}_{-i}) = o(1)
 $
 uniformly over $\ell \in[L]$
 with probability at least $1 - \calO(n^{-r})$ where $r$ can be arbitrarily large. Now, by the same arguments as in the proof of Lemma 1 in \cite{gao2017achieving}, we have
 $$
  \bbP\bigg(|\hat p^{(-i)}_\ell - p_\ell| \lor |\hat q^{(-i)}_\ell - q_\ell| \leq \delta_{\mathtt{est},n} (p_\ell - q_\ell)\bigg) \lesssim n^{-r}.
 $$
 The proof is concluded by invoking a union bound over $L\in[L]$.
\end{proof}

Now, we have $\bbP(E_i\cap F_i \cap G_i) \geq 1- C'' n^{-(1+\ep_{\init})}$. Thus, we have
\begin{equation*}
  \bbP(\barzzz{\star, -i}_i\neq \bz^\star_i) \leq \bbP(\barzzz{\star, -i}_i = -\bz^\star_i, E_i\cap F_i\cap G_i) + C''n^{-(1+\ep_{\init})}. 
\end{equation*}
Following the proof of Proposition \ref{prop:refine_glob_nodewise}, we arrive the following inequality, which is a counterpart to \eqref{eq:glob_nodewise_err_decom_3terms}:
\begin{align*}
  & \bbP(\barzzz{\star, -i}_i = -\bz^\star_i , E_i\cap F_i \cap G_i)  \nonumber\\
  & \leq \sum_{S \subseteq [L]} 
  \sum_{x \in\{ -|S^c|+2k: 0\leq k \leq |S^c|\}} 
  \binom{|S^c|}{\frac{|S^c| + x}{2}} 
  \exp\bigg\{ -|S^c|\log\frac{1}{\sqrt{\rho(1-\rho)}} + x\log \sqrt{\frac{1-\rho}{\rho}} - 2tx \log \sqrt{\frac{1-\rho^\dagger}{\rho^\dagger}}\bigg\} \nonumber\\
  & ~~ \times \sum_{\xi\in\{\pm 1\}^S}\bbP(\zzz{S}_i =  \bz^\star_i \xi ) 
  \cdot \bbE_{\{\zzz{\ell}_{-i}\}_{\ell=1}^L} \bigg\{
  \prod_{\ell\in S}\bbE\bigg[ \hat{\mathscr{T}}_{1, \ell}\times \hat{\mathscr{T}}_{2, \ell} \times \hat{\mathscr{T}}_{3, \ell} \cdot \Indc_{E_i\cap F_i\cap G_i} ~\bigg|~ \{\zzz{\ell}_{-i}\}_{\ell=1}^L \bigg]\bigg\},
\end{align*}
where
\begin{align*}
  \hat{\mathscr{T}}_{1, \ell} & := \bigg(\frac{\hat p^{(-i)}_\ell \Big(\frac{\hat q^{(-i)}_\ell (1-\hat p^{(-i)}_\ell)}{\hat p^{(-i)}_\ell (1-\hat q^{(-i)}_\ell)}\Big)^{1/2} + 1- \hat p^{(-i)}_\ell}{\hat q^{(-i)}_\ell \Big(\frac{\hat p^{(-i)}_\ell (1- \hat q^{(-i)}_\ell)}{\hat q^{(-i)}_\ell (1-\hat p^{(-i)}_\ell)}\Big)^{1/2} + 1-\hat q^{(-i)}_\ell}\bigg)^{t(\mmm{\ell, -i}_- - \mmm{\ell, -i}_+)} \\
  & \qquad \times \bigg[q_\ell\bigg(\frac{\hat p^{(-i)}_\ell(1-\hat q^{(-i)}_\ell)}{\hat q^{(-i)}_\ell(1-\hat p^{(-i)}_\ell)}\bigg)^t + (1-q_\ell)\bigg]^{\frac{\mmm{\ell, -i}_- - \mmm{\ell, -i}_+}{2}}  \\
  & \qquad 
  \times \bigg[p_\ell \bigg(\frac{\hat q^{(-i)}_\ell(1-\hat p^{(-i)}_\ell)}{\hat p^{(-i)}_\ell(1-\hat q^{(-i)}_\ell)}\bigg)^t + (1-p_\ell)\bigg]^{\frac{\mmm{\ell, -i}_+ - \mmm{\ell, -i}_-}{2}}
   \\
  \hat{\mathscr{T}}_{2, \ell} & := \bigg[ q_\ell\bigg(\frac{\hat p^{(-i)}_\ell(1-\hat q^{(-i)}_\ell)}{\hat q^{(-i)}_\ell(1-\hat p^{(-i)}_\ell)}\bigg)^t + (1-q_\ell) \bigg]^{\frac{\mmm{\ell, -i}_- + \mmm{\ell, -i}_+}{2}} 
  \times \bigg[p_\ell\bigg(\frac{\hat q^{(-i)}_\ell(1-\hat p^{(-i)}_\ell)}{\hat p^{(-i)}_\ell(1-\hat q^{(-i)}_\ell)}\bigg)^t + (1-p_\ell)\bigg]^{\frac{\mmm{\ell, -i}_+ + \mmm{\ell, -i}_-}{2}} \\
  \hat{\mathscr{T}}_{3, \ell} & := \bigg[ \frac{p_\ell \bigg(\frac{\hat p^{(-i)}_\ell(1-\hat q^{(-i)}_\ell)}{\hat q^{(-i)}_\ell(1-\hat p^{(-i)}_\ell)}\bigg)^t + (1-p_\ell)}{q_\ell \bigg(\frac{\hat p^{(-i)}_\ell(1-\hat q^{(-i)}_\ell)}{\hat q^{(-i)}_\ell(1-\hat p^{(-i)}_\ell)}\bigg)^t + (1-q_\ell)} \bigg]^{\mmm{\ell, -i}_- - \tmmm{\ell, -i}_-}
  \times \bigg[ \frac{q_\ell \bigg(\frac{\hat q^{(-i)}_\ell(1-\hat p^{(-i)}_\ell)}{\hat p^{(-i)}_\ell(1-\hat q^{(-i)}_\ell)}\bigg)^t + (1-q_\ell)}{p_\ell \bigg(\frac{\hat q^{(-i)}_\ell(1-\hat p^{(-i)}_\ell)}{\hat p^{(-i)}_\ell(1-\hat q^{(-i)}_\ell)}\bigg)^t + (1-p_\ell)}\bigg]^{\mmm{\ell, -i}_+ - \tmmm{\ell, -i}_+}.
\end{align*}

We now prove the counterparts of \eqref{eq:glob_nodewise_err_decomp_1st_term}, \eqref{eq:glob_nodewise_err_decomp_2ed_term}, and \eqref{eq:glob_nodewise_err_decomp_3ed_term} in order.

~\\
{\noindent \bf Bounding the first term.} Taking $t=1/2$, the first term can be expressed as
\begin{align*}
  \hat{\mathscr{T}}_{1, \ell} & = \bigg[\frac{\hat p^{(-i)}_\ell \Big(\frac{\hat q^{(-i)}_\ell (1-\hat p^{(-i)}_\ell)}{\hat p^{(-i)}_\ell (1-\hat q^{(-i)}_\ell)}\Big)^{1/2} + 1- \hat p^{(-i)}_\ell}{p_\ell \Big(\frac{\hat q^{(-i)}_\ell (1-\hat p^{(-i)}_\ell)}{\hat p^{(-i)}_\ell (1-\hat q^{(-i)}_\ell)}\Big)^{1/2} + 1- p_\ell}\bigg]^{\frac{\mmm{\ell, -i}_- - \mmm{\ell, -i}_+}{2}} \\
  & \qquad \times 
  \bigg[\frac{q_\ell \Big(\frac{\hat p^{(-i)}_\ell (1- \hat q^{(-i)}_\ell)}{\hat q^{(-i)}_\ell (1-\hat p^{(-i)}_\ell)}\Big)^{1/2} + 1-q_\ell}{\hat q^{(-i)}_\ell \Big(\frac{\hat p^{(-i)}_\ell (1- \hat q^{(-i)}_\ell)}{\hat q^{(-i)}_\ell (1-\hat p^{(-i)}_\ell)}\Big)^{1/2} + 1-\hat q^{(-i)}_\ell}\bigg]^{\frac{\mmm{\ell, -i}_- - \mmm{\ell, -i}_+}{2}}\\
  & \leq
  \exp\bigg\{\frac{|\mmm{\ell, -i}_- - \mmm{\ell, -i}_+|}{2}\cdot \bigg|\log  \bigg[\frac{\hat p^{(-i)}_\ell \Big(\frac{\hat q^{(-i)}_\ell (1-\hat p^{(-i)}_\ell)}{\hat p^{(-i)}_\ell (1-\hat q^{(-i)}_\ell)}\Big)^{1/2} + 1- \hat p^{(-i)}_\ell}{p_\ell \Big(\frac{\hat q^{(-i)}_\ell (1-\hat p^{(-i)}_\ell)}{\hat p^{(-i)}_\ell (1-\hat q^{(-i)}_\ell)}\Big)^{1/2} + 1- p_\ell}\bigg]\bigg| \\
  & \qquad\qquad + \frac{|\mmm{\ell, -i}_- - \mmm{\ell, -i}_+|}{2}\cdot \bigg|\log \bigg[\frac{q_\ell \Big(\frac{\hat p^{(-i)}_\ell (1- \hat q^{(-i)}_\ell)}{\hat q^{(-i)}_\ell (1-\hat p^{(-i)}_\ell)}\Big)^{1/2} + 1-q_\ell}{\hat q^{(-i)}_\ell \Big(\frac{\hat p^{(-i)}_\ell (1- \hat q^{(-i)}_\ell)}{\hat q^{(-i)}_\ell (1-\hat p^{(-i)}_\ell)}\Big)^{1/2} + 1-\hat q^{(-i)}_\ell}\bigg] \bigg|\bigg\}.
\end{align*}
We have
\begin{align*}
  \frac{\hat p^{(-i)}_\ell \Big(\frac{\hat q^{(-i)}_\ell (1-\hat p^{(-i)}_\ell)}{\hat p^{(-i)}_\ell (1-\hat q^{(-i)}_\ell)}\Big)^{1/2} + 1- \hat p^{(-i)}_\ell}{p_\ell \Big(\frac{\hat q^{(-i)}_\ell (1-\hat p^{(-i)}_\ell)}{\hat p^{(-i)}_\ell (1-\hat q^{(-i)}_\ell)}\Big)^{1/2} + 1- p_\ell}
  & = 1 + \frac{(\hat p^{(-i)}_\ell - p_\ell) \cdot \bigg( \Big( \frac{\hat q^{(-i)}_\ell (1-\hat p^{(-i)}_\ell )}{\hat p^{(-i)}_\ell (1- \hat q^{(-i)}_\ell)} \Big)^{1/2} - 1\bigg)}{p_\ell \Big(\frac{\hat q^{(-i)}_\ell (1-\hat p^{(-i)}_\ell)}{\hat p^{(-i)}_\ell (1-\hat q^{(-i)}_\ell)}\Big)^{1/2} + 1- p_\ell} \\
  & \overset{(*)}{\leq} 1 + o(p_\ell - q_\ell) \cdot \bigg| \Big( \frac{\hat q^{(-i)}_\ell (1-\hat p^{(-i)}_\ell )}{\hat p^{(-i)}_\ell (1- \hat q^{(-i)}_\ell)} \Big)^{1/2} - 1\bigg|\\
  & \leq 1 + o(p_\ell - q_\ell) \cdot \frac{\Big|\frac{\hat q^{(-i)}_\ell (1-\hat p^{(-i)}_\ell )}{\hat p^{(-i)}_\ell (1- \hat q^{(-i)}_\ell)}  - 1\Big|}{\Big(\frac{\hat q^{(-i)}_\ell (1-\hat p^{(-i)}_\ell )}{\hat p^{(-i)}_\ell (1- \hat q^{(-i)}_\ell)} \Big)^{1/2} + 1}\\
  & \leq 1 + o(p_\ell - q_\ell) \cdot \frac{|\hat p_\ell^{(-i)} - \hat q_\ell^{(-i)}|}{\hat p_\ell^{(-i)} (1-\hat q_\ell^{(-i)})}\\
  & \overset{(**)}{\leq} 1 + o(p_\ell-q_\ell) \cdot \frac{(p_\ell - q_\ell)}{p_\ell}\\
  & \leq \exp\bigg\{o\bigg(\frac{(p_\ell - q_\ell)^2}{p_\ell}\bigg)\bigg\},
\end{align*}
where $(*)$ and $(**)$ hold under $G_i$. Hence, under $E_i\cap F_i \cap G_i$, we have
$$
  \exp\bigg\{\frac{|\mmm{\ell, -i}_- - \mmm{\ell, -i}_+|}{2} \cdot \bigg|\log  \bigg[\frac{\hat p^{(-i)}_\ell \Big(\frac{\hat q^{(-i)}_\ell (1-\hat p^{(-i)}_\ell)}{\hat p^{(-i)}_\ell (1-\hat q^{(-i)}_\ell)}\Big)^{1/2} + 1- \hat p^{(-i)}_\ell}{p_\ell \Big(\frac{\hat q^{(-i)}_\ell (1-\hat p^{(-i)}_\ell)}{\hat p^{(-i)}_\ell (1-\hat q^{(-i)}_\ell)}\Big)^{1/2} + 1- p_\ell}\bigg]\bigg| \bigg\}\leq \exp\{o(1) \cdot n \III{\ell}_{1/2}\}.
$$
A similar argument shows that under the same event, 
$$
\exp\bigg\{\frac{|\mmm{\ell, -i}_- - \mmm{\ell, -i}_+|}{2}\cdot \bigg|\log \bigg[\frac{q_\ell \Big(\frac{\hat p^{(-i)}_\ell (1- \hat q^{(-i)}_\ell)}{\hat q^{(-i)}_\ell (1-\hat p^{(-i)}_\ell)}\Big)^{1/2} + 1-q_\ell}{\hat q^{(-i)}_\ell \Big(\frac{\hat p^{(-i)}_\ell (1- \hat q^{(-i)}_\ell)}{\hat q^{(-i)}_\ell (1-\hat p^{(-i)}_\ell)}\Big)^{1/2} + 1-\hat q^{(-i)}_\ell}\bigg] \bigg|\bigg\} \leq \exp\{o(1) \cdot n \III{\ell}_{1/2}\}.
$$
Thus, we get
\begin{equation}
  \label{eq:glob_nodewise_err_decomp_1st_term_est_param}
  \hat{\mathscr{T}}_{1, \ell}\cdot \Indc_{E_i\cap F_i \cap G_i} \leq \exp\{o(1) \cdot n \III{\ell}_{1/2}\}
\end{equation}  
under the choice of $t=1/2$.

~\\
{\noindent \bf Bounding the second term.} 
Since $\mmm{\ell, -i}_+ + \mmm{\ell, -i}_- = n-1$, we have
\begin{align*}
  \hat{\mathscr{T}}_{2, \ell} & = 
  \exp\bigg\{\frac{n-1}{2} \log \bigg[ p_\ell q_\ell + (1-p_\ell)(1-q_\ell) \\
  & \qquad + p_\ell(1-q_\ell)\bigg(\frac{\hat q_\ell^{(-i)} (1 - \hat p^{(-i)}_\ell)}{\hat p^{(-i)}_\ell (1 - \hat q^{(-i)}_\ell) }\bigg)^{t} + q_\ell (1-p_\ell) \bigg(\frac{\hat p^{(-i)}_\ell (1 - \hat q^{(-i)}_\ell)}{\hat q^{(-i)}_\ell (1 - \hat p^{(-i)}_\ell)}\bigg)^t \bigg]\bigg\} \\
  & \overset{(*)}{=} 
  \exp\bigg\{\frac{n-1}{2} \log \bigg[ p_\ell q_\ell + (1-p_\ell)(1-q_\ell) \\
  & \qquad + (1+o(1))\cdot \bigg(
    p_\ell(1-q_\ell)   \bigg(\frac{q_\ell (1-p_\ell)}{p_\ell (1-q_\ell)}\bigg)^t
  \bigg)
  + q_\ell(1-p_\ell)\bigg(\frac{p_\ell (1-q_\ell)}{q_\ell(1-p_\ell)}\bigg)^t \bigg]\bigg\} \\
  & = \exp\bigg\{-\frac{n-1}{2} \cdot (1+o(1)) \III{\ell}_t\bigg\} \\
  \label{eq:glob_nodewise_err_decomp_2ed_term_est_param}
  & = \exp\bigg\{-(1+o(1)) n \III{\ell}_t / 2\bigg\} \numberthis
\end{align*}
where $(*)$ holds under $G_i$.

~\\
{\noindent \bf Bounding the third term.} 
We can write
\begin{align*}
  \frac{p_\ell \bigg(\frac{\hat p^{(-i)}_\ell(1-\hat q^{(-i)}_\ell)}{\hat q^{(-i)}_\ell(1-\hat p^{(-i)}_\ell)}\bigg)^t + (1-p_\ell)}{q_\ell \bigg(\frac{\hat p^{(-i)}_\ell(1-\hat q^{(-i)}_\ell)}{\hat q^{(-i)}_\ell(1-\hat p^{(-i)}_\ell)}\bigg)^t + (1-q_\ell)} & = 
  1 + \frac{(p_\ell - q_\ell) \bigg[ \bigg(1 + \frac{\hat p_\ell^{(-i)} - \hat q_\ell^{(-i)}}{\hat q_\ell^{(-i)} (1- \hat p_\ell^{(-i)})}\bigg)^t - 1\bigg]}{1 - q_\ell + q_\ell \bigg(1 + \frac{\hat p_\ell^{(-i)} - \hat q_\ell^{(-i)}}{\hat q_\ell^{(-i)}(1-\hat p_\ell^{(-i)})}\bigg)^t}.
\end{align*}
Under $G_i$, we have
  $$
    \hat p^{(-i)}_\ell - \hat q^{(-i)}_\ell = (\hat p^{(-i)}_\ell - p_\ell) - (\hat q^{(-i)}_\ell - q_\ell) + (p_\ell - q_\ell) \geq (1-2\delta_{\mathtt{est}, n})(p_\ell - q_\ell) > 0.
  $$
Thus, we can further upper bound 
\begin{align*}
  \frac{p_\ell \bigg(\frac{\hat p^{(-i)}_\ell(1-\hat q^{(-i)}_\ell)}{\hat q^{(-i)}_\ell(1-\hat p^{(-i)}_\ell)}\bigg)^t + (1-p_\ell)}{q_\ell \bigg(\frac{\hat p^{(-i)}_\ell(1-\hat q^{(-i)}_\ell)}{\hat q^{(-i)}_\ell(1-\hat p^{(-i)}_\ell)}\bigg)^t + (1-q_\ell)}
  & \leq  1 + \frac{(p_\ell - q_\ell) \cdot \frac{\hat p^{(-i)}_\ell - \hat q^{(-i)}_\ell}{\hat q^{(-i)}_\ell (1 - \hat p^{(-i)}_\ell) }}{1 - q_\ell} \\
  & \leq 1 + \calO\bigg(\frac{(p_\ell-q_\ell)^2}{p_\ell}\bigg) \\
  & \leq \exp\bigg\{\calO\bigg(\frac{(p_\ell - q_\ell)^2}{p_\ell}\bigg)\bigg\},
\end{align*}  
where the second inequality above again holds under $G_i$. A similar argument shows that under $G_i$, we have
$$
  \frac{q_\ell \bigg(\frac{\hat q^{(-i)}_\ell(1-\hat p^{(-i)}_\ell)}{\hat p^{(-i)}_\ell(1-\hat q^{(-i)}_\ell)}\bigg)^t + (1-q_\ell)}{p_\ell \bigg(\frac{\hat q^{(-i)}_\ell(1-\hat p^{(-i)}_\ell)}{\hat p^{(-i)}_\ell(1-\hat q^{(-i)}_\ell)}\bigg)^t + (1-p_\ell)} \leq \exp\bigg\{\calO\bigg(\frac{(p_\ell - q_\ell)^2}{p_\ell}\bigg)\bigg\}.
$$
Thus, we have
\begin{align*}
  \label{eq:glob_nodewise_err_decomp_3ed_term_est_param}
  \hat{\mathscr{T}}_{3, \ell}\cdot \Indc_{E_i} \leq \exp\bigg\{\calO\bigg(\frac{o(n)(p_\ell-q_\ell)^2}{p_\ell}\bigg)\bigg\} = \exp\bigg\{ o(1) \cdot n\III{\ell}_{1/2} \bigg\}, \numberthis
\end{align*}

~\\
{\noindent \bf Summarizing the three terms.} Summarizing \eqref{eq:glob_nodewise_err_decomp_1st_term_est_param}, \eqref{eq:glob_nodewise_err_decomp_2ed_term_est_param}, and \eqref{eq:glob_nodewise_err_decomp_3ed_term_est_param}, we get
\begin{align*}
  & \bbP(\barzzz{\star, -i}_i = -\bz^\star_i , E_i\cap F_i \cap G_i)\\
  & \leq \sum_{S \subseteq [L]} 
  \sum_{x \in\{ -|S^c|+2k: 0\leq k \leq |S^c|\}} 
  \binom{|S^c|}{\frac{|S^c| + x}{2}} 
  \exp\bigg\{ -|S^c|\log\frac{1}{\sqrt{\rho(1-\rho)}} + x\log \sqrt{\frac{1-\rho}{\rho}} - 2tx \log \sqrt{\frac{1-\rho^\dagger}{\rho^\dagger}}\bigg\} \\
  & \qquad \times \exp\{-(1+o(1)) \psi^\star_S(0)\}\\
  & \leq \sum_{S \subseteq [L]} 
  \sum_{x \in\{ -|S^c|+2k: 0\leq k \leq |S^c|\}} 
  \binom{|S^c|}{\frac{|S^c| + x}{2}} 
  \exp\bigg\{ -|S^c|\log\frac{1}{\sqrt{\rho(1-\rho)}} + \frac{1}{2} \bigg|\log {\frac{(1-\rho)/(1-\rho^\dagger)}{\rho / \rho^\dagger}}\bigg|\bigg\} \\
  & \qquad \times \exp\{-(1+o(1)) \psi^\star_S(0)\}\\
  & \leq \sum_{S\subseteq[L]} \exp\bigg\{- J_\rho + \frac{1}{2}\bigg|\log {\frac{(1-\rho)/(1-\rho^\dagger)}{\rho / \rho^\dagger}}\bigg|  - (1+o(1))\psi_S^\star(0)\bigg\} \\
  & = \sum_{S\subseteq [L]} \exp\{-(1+o(1))(J^\dagger_\rho + \psi_S^\star(0))\},
\end{align*}
where the second inequality holds under the choice of $t=1/2$.

\subsection{Proof of Theorem \ref{thm:refine_ind_with_est_param}}\label{prf:thm:refine_ind_with_est_param}
The proof is nearly the same as that of Theorem \ref{thm:refine_ind}, except that we now use we use the following proposition instead of Proposition \ref{prop:refine_ind_nodewise}.
\begin{proposition}
\label{prop:refine_ind_nodewise_const_rho_est_params}
Fix $\ell\in[L]$. Under the setup of Theorem \ref{thm:refine_ind_with_est_param}, there exists a sequence $\delta_n'=o(1)$ and an absolute constant $C''>0$ such that for any $i\in[n]$, we have 
\begin{align*}
    \bbP(\pi_i\barzzz{\ell, -i}_i \neq  \zzz{\ell}_i) 
    & \leq  C''n^{-(1+\ep_{\init})} \\
    & \qquad + \sum_{S\subseteq[L]\setminus\{\ell\}} e^{-(1-\delta_n')\psi_{S\cup\{\ell\}}^\star(0)}\cdot  \Big( e^{-(1-\delta_n')|(S\cup \{\ell\})^c|J_\rho^\dagger} + e^{-(1-\delta_n') |S\cup\{\ell\}|J_\rho^\dagger} \Big).
\end{align*}
\end{proposition}
\begin{proof}
  Without loss of generality we consider the first layer. Following the proof of Proposition \ref{prop:refine_ind_nodewise}, we get
  \begin{align*}
    & \bbP(\pi_i\barzzz{\ell, -i}_i \neq  \zzz{\ell}_i) \\
    & \leq  C''n^{-(1+\ep_{\init})} \\
    & \qquad + 
    \sum_{S\subseteq \{2, \hdots, L\}} \bigg(\bbP(\barzzz{\star, -i}_i = -\bz^\star_i, \barzzz{1, -i}_i = -\zzz{1}_i, \barzzz{S, -i}_i = -\zzz{S}_i, \barzzz{S^c, -i}_i = \zzz{S^c}_i, E_i\cap F_i \cap G_i) \\
    \label{eq:ind_nodewise_err_decomp_all_subsets_with_est_param}
    & \qquad + \bbP(\barzzz{\star, -i}_i = \bz^\star_i, \barzzz{1, -i}_i = -\zzz{1}_i, \barzzz{S, -i}_i = -\zzz{S}_i, \barzzz{S^c, -i}_i = \zzz{S^c}_i, E_i\cap F_i \cap G_i)\bigg),\numberthis
  \end{align*}
  where $E_i, F_i, G_i$ are defined in \eqref{eq:Ei}, \eqref{eq:Fi}, and \eqref{eq:Gi}, respectively.
  According to the proof of Proposition \ref{prop:refine_glob_nodewise_with_est_param}, the second term in the right-hand side above can be upper bounded by
  $$
    \sum_{S \subseteq\{2, \hdots, L\}} \exp\bigg\{-(1+o(1)) \bigg(|(S \cup \{\ell\})^c|J_\rho^\dagger + \psi^\star_{S\cup \{\ell\}}(0)\bigg)\bigg\}.
  $$
  By the arguments that led to \eqref{eq:ind_nodewise_err_decomp_indinfo}, \eqref{eq:glob_nodewise_err_decomp_1st_term_est_param}, \eqref{eq:glob_nodewise_err_decomp_2ed_term_est_param}, and \eqref{eq:glob_nodewise_err_decomp_3ed_term_est_param}, the third term in the right-hand side of \eqref{eq:ind_nodewise_err_decomp_all_subsets_with_est_param} can be upper bounded by
  \begin{align*}
    &\sum_{S\subseteq\{2, \hdots, L\}}
    \bbE_{\{\zzz{\ell}_i: \ell\in S\cup\{1\}\}} \bigg[
    \exp\bigg\{\log\sqrt{\frac{1-\rho}{\rho}}\cdot \bigg( \# \{\ell\in S\cup\{1\}: \zzz{\ell}_i = -\bz^\star_i\} - \#\{\ell\in S\cup\{1\}: \zzz{\ell}_i = \bz^\star_i\} \bigg)\bigg\}\bigg] \\
    & \qquad \times \exp\bigg\{-\frac{(1+o(1))\cdot n}{2}\sum_{\ell\in S\cup\{1\}}\III{\ell}_{1/2} \bigg\}\\
    & = \sum_{S\subseteq\{2, \hdots, L\}}
    \prod_{\ell \in S\cup\{\ell\}}\bbE\bigg[ \exp\bigg\{\log\sqrt{\frac{1-\rho}{\rho}}\cdot \bigg( \Indc\{\zzz{\ell}_i = -\bz^\star_i\} - \Indc\{\zzz{\ell}_i = \bz^\star_i\} \bigg)\bigg\}\bigg]
    \cdot \exp\{-(1+o(1))\psi^\star_{S\cup\{\ell\}}(0)\}\\
    & = \sum_{S\subseteq\{2, \hdots, L\}}
    \prod_{\ell\in S\cup\{\ell\}} \sqrt{\rho(1-\rho)} \cdot \sqrt{\frac{(1 - \rho^\dagger)/(1-\rho)}{\rho^\dagger /\rho}} + \sqrt{\rho(1-\rho)} \cdot \sqrt{\frac{\rho^\dagger/\rho}{(1-\rho^\dagger)/(1-\rho)}}\\
    & \qquad \times \exp\{-(1+o(1))\psi^\star_{S\cup\{\ell\}}(0)\}\\
    & \leq \sum_{S\subseteq\{2, \hdots, L\}}
    \prod_{\ell\in S\cup\{\ell\}} 2\sqrt{\rho(1-\rho)} \cdot \bigg(\sqrt{\frac{(1 - \rho^\dagger)/(1-\rho)}{\rho^\dagger /\rho}} \lor  \sqrt{\frac{\rho^\dagger/\rho}{(1-\rho^\dagger)/(1-\rho)}}\bigg)
    \cdot \exp\{-(1+o(1))\psi^\star_{S\cup\{\ell\}}(0)\}\\
    & = \sum_{S\subseteq\{2, \hdots, L\}} \exp\bigg\{-(1+o(1))\bigg( |S\cup\{\ell\}| J_\rho^\dagger +  \psi_{S\cup\{\ell\}}^\star(0)\bigg)\bigg\}.
  \end{align*}
  The proof is thus concluded.
\end{proof}

\subsection{Proof of Theorem \ref{thm:specc_large_L}} \label{prf:thm:specc_large_L}
In view of Lemma \ref{lem: structure lem}, it suffices to give an upper bound on
$
  \|\bar A - \bbE[\bar A]\|_2^2.
$
Since $n = \calO(1)$, we bound $\|\bar A - \bbE \bar A\|^2_2$ by $\|\bar A - \bbE \bar A\|_F^2$, the later of which can be further controlled by bounding individual terms.
By triangle inequality, we have 
\begin{equation}
  \label{eq:norm-triangle-fixed-n}
\|\bar A - \bbE \bar A\|_F \leq 
\|\bar A - \bbE^{(1:L)} \bar A\|_F 
+ \|\bbE^{(1:L)}\bar A - \bbE \bar A\|_F.
\end{equation}
For any fixed $i<j$, we have
\begin{equation}
  \label{eq:T1-rep}
\bar{A}_{ij} - \bbE^{(1:L)}\bar{A}_{ij} = \sum_{\ell\in [L]} \omega_\ell(A_{ij}^{(\ell)} - \bbE^{(\ell)} A_{ij}^{(\ell)}).
\end{equation}
Then Bennett's inequality implies that for any $t > 0$, 
\begin{align}
  \label{eq:T1-bd1}
\bbP\left( \bar{A}_{ij} - \bbE^{(1:L)}\bar{A}_{ij} > t \right)
\leq \exp\left( -\frac{\sigma^2}{a^2}h\left(\frac{at}{\sigma^2}\right) \right)
\end{align}
where $h(u) = (1+u)\log(1+u) - u$,
\begin{equation*}
a = \max_{\ell} \omega_\ell , \quad\mbox{and}\quad
\sigma^2 = \sum_{\ell\in  [L]} \omega_\ell^2\,\mathrm{Var}^{(1:L)}( A_{ij}^{(\ell)})
\leq \sum_{\ell \in [L]} \omega_\ell^2\,p_{\ell}.
\end{equation*}
Note that when $u\in [0, \frac{1}{2}]$, $\frac{2}{5}u^2\leq  h(u) \leq \frac{1}{2}u^2$, and so we could further bound the right-hand side of 
\eqref{eq:T1-bd1} as
\begin{align*}
\exp\left( -\frac{\sigma^2}{a^2}h\left(\frac{at}{\sigma^2}\right) \right)
\leq 
\exp\left( -\frac{2}{5}\frac{\sigma^2}{a^2}\frac{a^2t^2}{\sigma^4} \right)
= 
\exp\left( -\frac{2t^2}{5\sigma^2} \right)
\end{align*}
as long as $at/\sigma^2 \to 0$.
Let $T_L$ be a sequence whose exact form will be specified later, and define
\begin{equation}
\label{eq:t1}
t_1 = \sqrt{5 T_L \sum_{\ell \in [L]} \omega_\ell^2\,p_{\ell}}.
\end{equation}
Then we obtain 
\begin{equation*}
\bbP\left( \bar{A}_{ij} - \bbE^{(1:L)}\bar{A}_{ij} > t_1 \right)
\leq e^{-T_L}.
\end{equation*}
As long as $T_L\to \infty$ and 
\begin{equation*}
\frac{\max_\ell \omega_\ell }{\sqrt{ \sum_\ell \omega_\ell^2 q_\ell }} \,\sqrt{T_L} \to 0,
\end{equation*}
we have $at_1/\sigma^2 \to 0$.
Repeating the foregoing argument for $-(\bar{A}_{ij} - \bbE^{(1:L)}\bar{A}_{ij})$ and applying union bound, 
we have with probability at least $1 - 2n^2e^{-T_L}$,
\begin{equation}
  \label{eq:T1-bd-final}
\|\bar A - \bbE^{(1:L)} \bar A\|_F  \leq  nt_1.
\end{equation}
Now let us turn to the second term of \eqref{eq:norm-triangle-fixed-n}.
In view of the proof of Lemma \ref{lemma:cluster_flips}, we need to control the deviation of 
\begin{equation}
\label{eq:T2-rep}
\sum_{\ell} \omega_\ell(p_\ell - q_\ell)[B_\ell - 2\rho(1-\rho)]
\end{equation}
from zero on both sides, where $B_\ell\stackrel{iid}{\sim}\mathrm{Bern}(2\rho(1-\rho))$. 
Applying Bennett's inequality again, we have for any fixed $i<j$ and all $t > 0$,
\begin{equation}
  \label{eq:T2-bd1}
\bbP\left( \bbE^{(1:L)}\bar{A}_{ij} - \bbE\bar{A}_{ij} > t \right)
\leq
\exp\left( -\frac{\sigma^2}{a^2}h\left( \frac{at}{\sigma^2} \right) \right),
\end{equation}
where 
\begin{equation*}
a = \max_\ell \omega_\ell(p_\ell -q_\ell)
\quad\mbox{and}\quad
\sigma^2 = \sum_\ell \omega_\ell^2(p_\ell-q_\ell)^2 2\rho(1-\rho)[1- 2\rho(1-\rho)]
\leq 2\rho\sum_\ell \omega_\ell^2(p_\ell-q_\ell)^2.
\end{equation*}
Under the condition that as $L\to\infty$, $T_L\to\infty$ and
\begin{equation*}
\frac{\max_\ell \omega_\ell(p_\ell - q_\ell) }{\sqrt{\rho\sum_{\ell} \omega_\ell^2 (p_\ell-q_\ell)^2}}\sqrt{T_\ell} \to 0,
\end{equation*}
if we pick
\begin{equation}
  \label{eq:t2}
t_2 = \sqrt{10 T_L \rho\sum_\ell \omega_\ell^2(p_\ell-q_\ell)^2},
\end{equation}
we obtain 
\begin{equation*}
\bbP\left( \bbE^{(1:L)}\bar{A}_{ij} - \bbE\bar{A}_{ij} > t_2 \right)
\leq e^{-T_L}.
\end{equation*}
Repeating the argument for $-(\bbE^{(1:L)}\bar{A}_{ij} - \bbE\bar{A}_{ij})$ and applying union bound, 
we have with probability at least $1 - 2n^2 e^{-T_L}$,
\begin{equation}
  \label{eq:T2-bd-final}
\|\bbE^{(1:L)} \bar A - \bbE \bar A \|_F  \leq nt_2.
\end{equation}
Combining Lemma \ref{lem: structure lem} and Equations \eqref{eq:norm-triangle-fixed-n}, \eqref{eq:T1-bd-final} and \eqref{eq:T2-bd-final},
we obtain that with probability at least $1 - 4n^2 e^{-T_L}$,
\begin{equation*}
\calL(\btz^\star, \bz^\star) 
\leq 
\frac{ C n^2 T_L}{n^2 (1-2\rho)^4 (\bar p-\bar q)^2} 
\left( \sum_\ell \omega_\ell^2 p_l + \rho\sum_\ell \omega_\ell^2(p_\ell-q_\ell)^2 \right) .
\end{equation*}
The righthand side tends to zero as $L\to\infty$ if the following two conditions are satisfied:
\begin{align*}
\frac{T_L \sum_\ell \omega_\ell^2 p_\ell^2}{(1-2\rho)^4 (\bar p - \bar q)^2} \to 0,
\qquad 
\frac{\rho\, T_L \sum_\ell \omega_\ell^2 (p_\ell - q_\ell)^2}{(1-2\rho)^4 (\bar p - \bar q)^2} \to 0.
\end{align*}
In summary, we have shown that if there exists $T_L \to \infty$ such that 
\begin{align}
  \label{eq:assump_on_L_large_L_asymptotics}
  \frac{\max_\ell \omega_\ell }{\sqrt{ \sum_\ell \omega_\ell^2 q_\ell }} \,\sqrt{T_L} 
  \lor \frac{\max_\ell \omega_\ell(p_\ell - q_\ell) }{\sqrt{\rho\sum_{\ell} \omega_\ell^2 (p_\ell-q_\ell)^2}}\sqrt{T_\ell}
  \lor \frac{T_L \sum_\ell \omega_\ell^2 p_\ell^2}{(1-2\rho)^4 (\bar p - \bar q)^2}
  \lor \frac{\rho\, T_L \sum_\ell \omega_\ell^2 (p_\ell - q_\ell)^2}{(1-2\rho)^4 (\bar p - \bar q)^2} \to \infty,
\end{align} 
then with probability $1 - \calO(e^{-T_L}) = 1-o(1)$, we have
$$
  \calL(\tilde \bz^\star, \bz^\star) = o(1) < 1/n
$$
for $L$ large enough. The above inequality implies $\calL(\tilde \bz^\star, \bz^\star) = 0$ by the definition of $\calL$.
We finish the proof by noting that the choice of $T_L$ that makes \eqref{eq:assump_on_L_large_L_asymptotics} holds is possible when \eqref{eq:assump_specc_large_L} holds.

\section{Concentration and Regularization of Multilayer Networks}\label{append:concentration}

In this section, we overload the notation and let $\{\AAA{\ell}\}_1^L$ be adjacency matrices of a ``multilayer'' inhomogeneous Erd\"os-R\'enyi graph, where each $\AAA{\ell}$ is independently generated by 
\begin{equation*}
  \AAA{\ell}_{ij} = \AAA{\ell}_{ji} \sim \Bern{(\ppp{\ell}_{ij})}, \qquad \forall i, j \in[n].
\end{equation*}
As usual, we let $\weightvec = (\weight_1, \hdots, \weight_L)$ be a weight vector and define 
\begin{equation*}
  \bar A = \sum_{\ell\in[L]}\weight_\ell \AAA{\ell}.
\end{equation*}
For $I, J\subseteq[n]$, we let $\barA_{I\times J}\in\bbR^{|I|\times |J|}$ be the submatrix of $\bar A$ with rows indexed by $I$ and columns indexed by $J$. For a generic subset $\calE\subseteq[n]\times[n]$, not necessarily of the form $\calE = I\times J$ for some $I, J \subseteq[n]$, we let $\barA_{\calE}$ denote the submatrix of $\barA$ whose dimension is $|I_\calE|\times |J_\calE|$, where $I_\calE = \{i\in[n]: (i,j)\in\calE \textnormal{ for some } j \in[n]\}, J_\calE = \{j\in[n]: (i,j)\in\calE \textnormal{ for some } i\in[n]\}$, and whose entries are given by
$$
  (\bar A_{\calE})_{i, j} = 
  \begin{cases}
    \bar A_{i, j} & \textnormal{ if } (i, j) \in \calE,\\
    0 & \textnormal{ otherwise.}
  \end{cases}
$$
The performance of spectral clustering is highly contingent upon the concentration behavior of $\barA$ around its expectation. 
Let us define
\begin{equation*}
  d_\ell = \max_{i, j} n\ppp{\ell}_{ij}.
\end{equation*}
Note that $d_\ell$ is an upper bound of the expected degree of the $\ell$-the layer (which is defined as $\max_{i\in[n]}\sum_{j\in[n]}n\ppp{\ell}_{ij}$).
Ideally, we would want the concentration of $\bar A$ happens at an $\calO(L^{-c})$ rate for $c > 0$, because otherwise there is no point in pooling $\AAA{\ell}$'s together. Such a rate, intuitively, would require that the weight we put on each layer is ``relatively balanced'', and this is exactly the intention of the assumption below.
\begin{assump}[Balanced weights]  
\label{assump:wts_for_concentration}
Assume $\weight_\ell > 0,\forall \ell\in[L]$ and $\sum_{\ell \in[L]}\weight_\ell=1$. 
Moreover, assume there exist two absolute constants $c_0 > 0, c_1 \geq 1$ such that the following two inequalities hold:
\begin{align}
      \label{eq: assump_on_weights_general_bernoulli_matrix}
      & \| \weightvec \|_\infty \sum_{\ell\in[L]} \omega_\ell d_\ell \leq c_0 \sum_{\ell\in[L]} \omega_\ell^2 d_\ell , \\
      \label{eq: another_assump_on_weights_general_bernoulli_matrix}
      & \|\weightvec \|_\infty \cdot  \sup_{i\in[n], J\subseteq[n]} \frac{\sum_{\ell\in[L]}\sum_{j\in J} \ppp{\ell}_{ij}}{\sum_{\ell\in[L]}\sum_{j\in J} \omega_\ell \ppp{\ell}_{ij}} \leq c_1.
\end{align}      
\end{assump}


We are now ready to present the main theorem of this section.
\begin{theorem}[Concentration of regularized adjacency matrices]
      \label{thm:concentration_reg_mat}
      Let Assumption \ref{assump:wts_for_concentration} hold with $c_0>0, c_1 \geq 1$ and fix two constants $c> 0, r \geq 1$. 
      Let $I\subseteq[n]$ be any subset of nodes with size at most $cn\|\weightvec\|_\infty/\sum_{\ell\in[L]}\omega_\ell d_\ell$. 
      For $\calE = (I\times[n])\cup([n]\times I)$, we {down-weight} (i.e., shrink the elements toward zero) the submatrix $A_\calE$ in an arbitrary way so that the resulting matrix $\tau(A)$ satisfies 
      \begin{equation*}
        0\leq [\tau(\barA)]_\calE\leq \barA_\calE  
      \end{equation*}
      entry-wise.
      Then with probability at least $1-3n^{-r}$, we have
      \begin{equation*}
            \| \tau(\barA) - \bbE \barA\| \leq C \cdot \bigg(\sqrt{\sum_{\ell\in[L]} \omega_\ell^2 d_\ell } + \sqrt{\| \omega\|_\infty d_\tau}\bigg),
    \end{equation*}            
      where 
      \begin{equation*}
            d_\tau =\max_{i\in[n]} \sum_{j\in[n]} [\tau(\barA)]_{ij}
    \end{equation*}            
    is the maximum degree of the regularized (i.e., down-weighted) matrix, and $C = C(c_0, c_1, c, r)$ is an absolute constant.
\end{theorem}

In the proof of Theorem \ref{thm:specc}, we will extensively use the following two corollaries of Theorem \ref{thm:concentration_reg_mat}.
\begin{corollary}[Concentration of trimmed adjacency matrices]
    \label{cor:concentration_trimmed_mat}
    Let Assumption \ref{assump:wts_for_concentration} hold with $c_0 > 0, c_1\geq 1$ and fix two constants $\gamma > e^{c_1}, r \geq 1$. 
    Define $I$ to be 
    $$
      I:=\{i\in[n]:\, \sum_{j\in[n]} \barA_{ij} > \gamma \sum_{\ell\in[L]} \omega_\ell d_\ell \}.
    $$
    We trim the entries of $\barA$ in $\calE=(I\times[n])\cup([n]\times I)$, so that the resulting matrix $\tau(\barA)$ is zero on $\calE$. Then with probability at least $1 - 3n^{-r} - c_2^{-n}$, we have
    \begin{equation*}
            \| \tau(\barA) - \bbE \barA\| \leq C \cdot \sqrt{\sum_{\ell\in[L]} \omega_\ell^2 d_\ell },
  \end{equation*}            
      where $c_2 = c_2(c_1, \gamma), C = C(c_0, c_1, \gamma, r)$ are two absolute constants. 
\end{corollary}

\begin{corollary}[Concentration of adjacency matrices without regularization]
      \label{cor:concentration_w/o_reg}
      Let Assumption \ref{assump:wts_for_concentration} hold with $c_0 > 0, c_1\geq 1$ and fix two constants $r \geq 1, c_2 > 0$. Then with probability at least $1 - 3n^{-r} - n^{-c_2}$, we have
      \begin{equation*}
            \| \barA - \bbE \barA\| \leq C \cdot  \bigg(\sqrt{\sum_{\ell\in[L]} \omega_\ell^2 d_\ell } + \|\weightvec \|_\infty\sqrt{\log n}\bigg),
    \end{equation*}            
    where $C = C(c_0, c_1, c_2, r)$ is an absolute constant.
\end{corollary}

Our proofs of the above results are based on a generalization of the \emph{graph decomposition} approach taken by \cite{le2017concentration}, where they proved the above results for $L=1$. The fact that we are dealing with a weighted average of multiple adjacency matrices calls for nontrivial modifications of the original arguments in \cite{le2017concentration}. Compared to the approach of applying matrix Bernstein's inequality (e.g., as done in \cite{paul2020spectral}), our approach, albeit being substantially more technically involved, \emph{gains a poly-log factor} in the final upper bound. Compared to \cite{bhattacharyya2018spectral}, where they adopted the combinatorial approach originally introduced by \cite{feige2005spectral}, our proof is largely probabilistic and is able to deal with \emph{non-uniform weights}. 

The rest of this section is devoted to proving the above results. Before we go into details, let us note that we can without loss of generality assume $\AAA{\ell}_{ij}\sim \Bern(\ppp{\ell}_{ij})$ independently for all $i,j\in[n]$ (i.e., $\AAA{\ell}$ is not necessarily symmetric). Indeed, such a relaxation will give the same upper bound up to a factor of 2 because we can bound the upper and lower triangular parts of the symmetric $\barA $ separately and invoke triangle inequality. Thus, in the rest of this section, we will assume $\AAA{\ell}$'s have independent entries.

\subsection{Step \texorpdfstring{\RN{1}}{I}: Concentration on a Big Block}

We first introduce a technical tool called \emph{Grothendieck-Pietsch factorization}, which allows us to ``upgrade'' an $\ell_\infty$-to-$\ell_2$ norm bound to an $\ell_2$-to-$\ell_2$ norm bound.
\begin{lemma}[Grothendieck-Pietsch factorization, Theorem 3.2 of \cite{le2017concentration}]
\label{lemma: grothendieck_pietsch_factorization}
Let $B\in \bbR^{k\times m}$ and $\delta > 0$. Then there exists $J\subseteq [m]$ with $|J|\geq (1-\delta)m$ such that the following holds:
$$
      \|B_{[k]\times J}\| \leq \frac{2 \|B \|_{\infty\to 2}}{\sqrt{\delta m}}.
$$
\end{lemma}

With the above lemma at hand, the strategy now is to first establish a concentration result in $\ell_\infty\rightarrow\ell_2$ norm and then to upgrade it to the operator norm using Lemma~\ref{lemma: grothendieck_pietsch_factorization}.

\begin{lemma}[Concentration in $\infty$-to-$2$ norm]
\label{lemma: concentration_in_infty_to_two_norm}
      Assume   \eqref{eq: assump_on_weights_general_bernoulli_matrix} holds with $c_0 >0$.
      For any $r \geq 1$, the following holds with probability at least $1-n^{-r}$: 
      uniformly for any $m\in[n]$ and any block $I\times J\subseteq [n]\times[n]$ with $|I|=|J|=m$, 
      if we let $I'$ be the indices of rows of $\barA_{I\times J}$ whose $\ell_1$-norm is bounded above by $\alpha \sum_{\ell\in[L]}\omega_\ell d_\ell$, where $\alpha$ is any number satisfying $\alpha\geq m/n$, 
      then we have
      \begin{equation}
        \label{eq:concentration_inf_to_two_norm}
            \| (\barA - \bbE \barA)_{I'\times J} \|_{\infty \to 2} \leq C\sqrt{\alpha \cdot (\sum_{\ell\in[L]} \omega_\ell^2 d_\ell) \cdot mr \log (\frac{en}{m})},
  \end{equation}            
      where $C= C(c_0)$ is an absolute constant.
\end{lemma}
\begin{proof}[Proof of Lemma \ref{lemma: concentration_in_infty_to_two_norm}]
Let us fix any $m\in[n]$, $\alpha \geq m/n$,
and take any block $I\times J\subseteq [n]\times[n]$ with $|I| = |J| = m$. 
By definition, we have
\begin{align*}
      \|(\barA - \bbE \barA)_{I'\times J}\|_{\infty \to 2}^2 &= \sup_{\|x\|_\infty \leq 1} \|Ax\|^2_2.
\end{align*}
Since the right-hand side is the supremum of a convex function over a convex set, the supremum is attained at the boundary. Hence we have
\begin{align*}
      \|(\barA - \bbE \barA)_{I'\times J}\|_{\infty \to 2}^2 = \max_{x\in\{\pm 1\}^m} \sum_{i\in I^\prime} \bigg(\sum_{j\in J} (\barA_{ij} - \bbE \barA_{ij})x_j\bigg)^2 
      := \max_{x\in\{\pm 1\}^m} \sum_{i\in I} (X_i \xi_i)^2,
\end{align*}
where we let
\begin{align*}
      X_i & := \sum_{j\in J} (\barA_{ij}-\bbE \barA_{ij})x_j = \sum_{\ell\in[L]} \sum_{j\in J} \weight_\ell(\AAA{\ell}_{ij}-\bbE \AAA{\ell}_{ij})x_j, \\
      \xi_i &:= \indc{i\in I'} = \Indicator\bigg\{\sum_{j\in J} \barA_{ij}\leq \alpha \sum_{\ell\in[L]} \omega_\ell d_\ell \bigg\} = \Indicator \bigg\{\sum_{\ell\in[L]} \sum_{j\in J} \omega_\ell \AAA{\ell}_{ij}\leq \alpha \sum_{\ell\in[L]}\omega_\ell d_\ell \bigg\}.
\end{align*}    
Note that $X_i$ has mean zero and its variance satisfies
\[
\Var{(X_i)}=\sum_{\ell\in[L]}\sum_{j\in J}\omega_\ell^2\ppp{\ell}_{ij}(1-\ppp{\ell}_{ij})\leq\sum_{\ell\in[L]}\sum_{j\in J}\omega_\ell^2\ppp{\ell}_{ij}\leq
\frac{m}{n}\sum_{\ell\in[L]}\omega_\ell^2d_\ell
\]
since $d_\ell=\max_{i,j} np_{ij}^{(\ell)}$.
Meanwhile, we have $|\omega_\ell(\AAA{\ell}_{ij}-\bbE \AAA{\ell}_{ij})x_j|\leq\|\weightvec\|_\infty$.
Invoking Bernstein's inequality, we get
\begin{align*}
      \bbP(|X_i\xi_i| \geq tm)\leq \bbP(|X_i|\geq tm) & \leq 2\exp\bigg\{\frac{-m^2t^2/2}{ (m/n)\cdot \sum_{\ell\in[L]} \omega_\ell^2 d_\ell + \frac{1}{3}mt\|\weightvec\|_\infty} \bigg\}.\numberthis\label{equ: bernstein X_i}
\end{align*}
Note that
\[
      |X_i\xi_i|  \leq \sum_{\ell\in[L]} \sum_{j\in J}  \omega_\ell (\AAA{\ell}_{ij}\xi_i + \bbE \AAA{\ell}_{ij})  \leq \sum_{\ell\in [L]} \sum_{j\in J} \omega_\ell \AAA{\ell}_{ij}\xi_i + \frac{m}{n} \sum_{\ell\in[L]} \omega_{\ell} d_\ell\leq (\alpha+\frac{m}{n})\sum_{\ell\in[L]}\omega_\ell d_\ell,
\]
where the last inequality is by the definition of $\xi_i$. 
Since $m/n\leq\alpha$, the above display translates to $|X_i\xi_i|\leq 2\alpha\sum_{\ell\in[L]}\omega_\ell d_\ell$. This means that if $tm > 2\alpha\sum_{\ell\in[L]}\weight_t d_\ell$, then the probability in the left-hand side of \eqref{equ: bernstein X_i} is zero. On the other hand, if we we assume $tm\leq 2\alpha\sum_{\ell\in[L]}\omega_\ell d_\ell$, then we can further bound the right-hand side of \eqref{equ: bernstein X_i} by
\begin{align*}
\bbP(|X_i\xi_i|\geq tm)
&\leq 2\exp\bigg\{\frac{-m^2t^2/2}{ \alpha\cdot \sum_{\ell\in[L]} \omega_\ell^2 d_\ell + \frac{2\alpha}{3}  \|\weightvec\|_\infty \sum_{\ell\in[L]}\omega_\ell d_\ell} \bigg\}\\
&\leq 2\exp\bigg\{\frac{-m^2t^2/2}{ (1+\frac{2c_0}{3}) \cdot \alpha\sum_{\ell\in[L]} \omega_\ell^2 d_\ell} \bigg\},
\end{align*}
where the last inequality is due to   \eqref{eq: assump_on_weights_general_bernoulli_matrix}. Combining the two cases, we conclude that the above display holds for all choices of $tm > 0$. 
This means that $X_i\xi_i$ has sub-Gaussian norm $\lesssim \sqrt{\alpha\sum_{\ell\in[L]}\omega^2_\ell d_\ell}$ (see, e.g., Lemma 5.5 of \cite{vershynin2010introduction}), and hence $(X_i\xi_i)^2$ has sub-exponential norm $\lesssim\alpha \sum_{\ell\in[L]} \omega_\ell^2 d_\ell$ (see, e.g., Lemma 5.14 of \cite{vershynin2010introduction}). Invoking Corollary 5.17 of \cite{vershynin2010introduction}, we have
$$
      \bbP\bigg(\sum_{i\in I}(X_i \xi_i)^2 > \ep m {\alpha} \sum_{\ell\in[L]} \omega_\ell^2 d_\ell \bigg) \leq 2\exp\bigg\{ -c (\ep^2 \land \ep) m \bigg\}
$$
for some constant $c$ only depending on $c_0$. Choosing $\ep = (10/c)r \log(en/m)$ for some constant $r\geq 1$, we deduce that with probability at least $1-(en/m)^{-5rm}$, we have
$$
      \sum_{i\in I} (X_i\xi_i)^2 \leq (10/c) r \log (en/m) \cdot m\alpha \sum_{\ell\in[L]} \omega_\ell^2 d_\ell.
$$
Taking a union bound over all possible configurations of $m\in[n], x\in\{\pm 1\}^m$, and $I, J$ with $|I|=|J|=m$, the conclusion of the lemma holds with probability at least
\begin{equation}
  \label{eq:concentration_linf_l2_norm_prob}
      1 - \sum_{m=1}^n 2^m \binom{n}{m}^2 \bigg(\frac{en}{m}\bigg)^{-5rm}\geq 1- \sum_{m=1}^n 2^m \bigg(\frac{en}{m}\bigg)^{-(5r-2)m}\geq 1- \sum_{m=1}^n \bigg(\frac{en}{m}\bigg)^{-(5r-3)m},
\end{equation}      
where the first inequality is by $\binom{n}{m}\leq (en/m)^m$ and the second inequality is by $2\leq en/m$.
We claim that
$
  (en/m)^{-(5r-3)m} \geq  \big(en/(m+1)\big)^{-(5r-3)(m+1)}
$
for any $m\in[n]$.
Indeed, with some algebra, this claim is equivalent to
\[
    (m+1)\log(m+1)-m\log n\leq \log(en)
\]
which holds for any $1\leq m\leq n$. Now, the right-hand side of \eqref{eq:concentration_linf_l2_norm_prob} can be further lower bounded by 
$$
  1 - n \cdot (en)^{-(5r-3)} \geq 1-n^{-r},
$$
where we have used $r\geq 1$, and this is exactly the desired result.
\end{proof}

The above lemma, along with Lemma \ref{lemma: grothendieck_pietsch_factorization} (with $\delta =1/4$), gives the following result.
\begin{lemma}[Concentration in spectral norm]
    \label{lemma: concentration_in_spectral_norm}
    Assume   \eqref{eq: assump_on_weights_general_bernoulli_matrix} holds with $c_0 >0$. 
    Then for any $r \geq 1$, the following holds with probability at least $1-n^{-r}$: 
    uniformly for any $m\in[n]$ and any block $I\times J \subseteq[n]\times [n]$ with $|I|=|J|=m$, if we let $I'$ be the indices of rows of $A_{I\times J}$ whose $\ell_1$-norm is bounded above by $\alpha \sum_{\ell\in[L]}\omega_\ell d_\ell$, where $\alpha$ an arbitrary (but fixed) number satisfying $\alpha \geq m/n$, then there exists a subset $J'\subseteq J$ with $|J'|\geq 3m/4$ such that
    \begin{equation}
      \label{eq:concentration_in_spectral_norm}
            \|(\barA - \bbE \barA)_{I'\times J'}\| \leq C\sqrt{\alpha \cdot (\sum_{\ell\in[L]} \omega_\ell^2 d_\ell) \cdot r \log (\frac{en}{m})},
  \end{equation}            
      where $C=C(c_0)$ is an absolute constant.
\end{lemma}

\subsection{Step \texorpdfstring{\RN{2}}{II}: Restricted \texorpdfstring{$\ell_1$}{l1} Norm}
The following lemma shows that {most of the rows} of $\barA$ have $\ell_1$ norm bounded from above by a constant multiple of $r\alpha\sum_{\ell\in[L]}\omega_{\ell\in[L]} d_\ell$.
\begin{lemma}[Degree of subgraphs]
\label{lemma: deg_of_subgraphs}
    Assume   \eqref{eq: another_assump_on_weights_general_bernoulli_matrix} holds $c_1 \geq 1$. 
    Then for any $r\geq 1$, the following holds with probability at least $1-n^{-r}$: 
    uniformly for any $m\in[n]$ and any block $I\times J\subseteq [n]\times[n]$ with $|I|=|J|=m$, all but $m\|\weightvec\|_\infty/(\alpha \sum_{\ell\in[L]}\omega_\ell d_\ell)$ rows of $\barA_{I\times J}$ have $\ell_1$-norm bounded above by $Cr\alpha \sum_{\ell\in[L]}\omega_\ell d_\ell$, where $\alpha$ is an arbitrary (but fixed) number satisfying $\alpha \geq\sqrt{m/n}$ and $C = C(c_1)$ is an absolute constant.
\end{lemma}

The proof of this lemma relies on the following concentration inequality for the weighted average of Bernoulli random variables, which is a generalization of the classical concentration inequality for the sum of independent Bernoulli random variables proved in \cite{hoeffding1963probability}.
\begin{lemma}[Concentration inequality for weighted Bernoulli sum]
      \label{lemma: concentration_of_weighted_bernoulli_sum}
      Let $\{X_i\}_{1\leq i\leq n}$ be independent random variables, each distributed as $X_i\sim \textnormal{\Bern}(p_i)$.  Let $\weightvec = \{\omega_i\}_{1\leq i\leq n}$ be a weight vector such that $\weight_i>0 ,\forall i\in[n]$. Assume there exists a constant $c_1\geq 1$ such that
      $$
            \|\weightvec \|_\infty \sum_{i\in[n]} p_i \leq c_1 \sum_{i\in[n]} \omega_i p_i.
      $$
      Then for any $t\geq\sum_{i\in[n]}\omega_i p_i$, we have
      $$
            \bbP(\sum_{i\in[n]}\omega_i X_i  \geq t) \leq \bigg(\frac{C \sum_{i\in[n]} \omega_i p_i}{t}\bigg)^{t/\|\weightvec\|_\infty},
      $$
      where $C = e^{c_1}$. 
\end{lemma}
\begin{proof}[Proof of Lemma \ref{lemma: concentration_of_weighted_bernoulli_sum}]
      For any $\lambda > 0$ we have
      \begin{align*}
            \bbP(\sum_{i\in[n]}\omega_i X_i  \geq t) & \leq \exp\{-\lambda t\}       \prod_{i\in[n]} \bbE \exp\{\lambda \omega_i X_i\} \\
            & \leq \exp\{-\lambda t\} \bigg(\frac{\sum_{i\in[n]}\bbE\exp\{\lambda\omega_iX_i\}}{n}\bigg)^n, 
      \end{align*}
      where the second inequality is due to the inequality of arithmetic and geometric means. Since $\exp\{\lambda \omega_i x\}$ is convex in $x$, its graph for $x\in[0, 1]$ is dominated by the line segment connecting the two points $(0, 1)$ and $(1, e^{\lambda\omega_i})$ in $\bbR^2$. Hence we have
      $
            e^{\lambda \omega_i X_i} \leq (e^{\lambda \omega_i } -1)  X_i + 1.
      $
      Taking expectation on both sides, we get
      $
            \bbE e^{\lambda \omega_i X_i} \leq (e^{\lambda\omega_i} -1) p_i + 1.
      $
      This gives
      \begin{align*}
            \bbP(\sum_{i\in[n]} \omega_i X_i \geq t) \leq \exp\{-\lambda t\} \bigg(\frac{\sum_{i\in[n]}\big( p_i e^{\lambda \omega_i}+ (1- p_i)\big)}{n}\bigg)^n.
      \end{align*}
      Taking $\lambda = \|\weightvec\|_\infty^{-1}\log ({t}/{\sum_{i\in[n]}\omega_ip_i})$, the right-hand side above is equal to
      \begin{align*}
            & \bigg(\frac{\sum_{i\in[n]}\omega_i p_i}{t}\bigg)^{t/\|\weightvec\|_\infty} \bigg(1 + \frac{\sum_{i\in[n]}p_i (t/\sum_{i\in[n]} \omega_i p_i)^{w_i/\|\weightvec\|_\infty} }{n} - \frac{\sum_{i\in[n]} p_i}{n}\bigg)^n \\
            & \overset{(1)}{\leq} \bigg(\frac{\sum_{i\in[n]}\omega_i p_i}{t}\bigg)^{t/\|\weightvec\|_\infty} \exp\bigg\{ \sum_{i\in[n]}p_i (t/\sum_{i\in[n]} \omega_i p_i)^{w_i/\|\weightvec\|_\infty}\bigg\} \\
            & \overset{(2)}\leq \bigg(\frac{ \sum_{i\in[n]} \omega_i p_i}{t}\bigg)^{t/\|\weightvec\|_{\infty}} \exp\bigg\{ \sum_{i\in[n]}p_i (t/\sum_{i\in[n]} \omega_i p_i)\bigg\} \\
            & = \bigg(\frac{ \sum_{i\in[n]} \omega_i p_i}{t}\bigg)^{t/\|\weightvec\|_{\infty}} \exp\bigg\{t \cdot \frac{\sum_{i\in[n]} p_i}{\sum_{i\in[n]}\omega_i p_i}\bigg\} \\
            & \overset{(3)}\leq \bigg(\frac{ \sum_{i\in[n]} \omega_i p_i}{t}\bigg)^{t/\|\weightvec\|_{\infty}} (e^{c_1})^{t/\|\weightvec\|_\infty} \\
            & = \bigg(\frac{ e^{c_1}\sum_{i\in[n]} \omega_i p_i}{t}\bigg)^{t/\|\weightvec\|_{\infty}}
      \end{align*}      
      where the (1) is by $1+x \leq e^x$ for any $x\in\bbR$, (2) is by our assumption that $t \geq \sum_{i\in[n]}\omega_i p_i$, and (3) is by our assumption on the weight vector.
\end{proof}

We now present the proof of Lemma \ref{lemma: deg_of_subgraphs}.
\begin{proof}[Proof of Lemma \ref{lemma: deg_of_subgraphs}]
Let the $\ell_1$-norm of the $i$-th row of $\barA_{I\times J}$ be
$$
      D_i = \sum_{j\in J} \barA_{ij} =  \sum_{\ell \in [L]}\sum_{j\in J} \omega_\ell \AAA{\ell}_{ij}.
$$
We have
$$
      \bbE D_i =\sum_{\ell\in[L]} \sum_{j\in J}  \omega_\ell \ppp{\ell}_{ij} \leq \frac{m}{n} \sum_{\ell\in[L]} \omega_\ell d_\ell \leq \alpha \sum_{\ell\in[L]} \omega_\ell d_\ell.
$$
Using   \eqref{eq: another_assump_on_weights_general_bernoulli_matrix}, for any $J\subseteq[n]$, we have
\[
\|\weightvec\|_\infty \sum_{\ell\in[L]}\sum_{j\in J}\ppp{\ell}_{ij}\leq c_1\sum_{\ell\in[L]}\sum_{j\in J}\omega_\ell\ppp{\ell}_{ij}.
\]
Thus we can invoke Lemma \ref{lemma: concentration_of_weighted_bernoulli_sum} to conclude that for a large enough $C'$,
\begin{align*}
      \bbP(D_i > C'r\alpha \sum_{\ell\in[L]}\omega_\ell d_\ell)  
      & \leq \bigg(\frac{C (m/n)\sum_{\ell\in[L]}\omega_\ell d_\ell}{ C'r\alpha \sum_{\ell\in[L]}\omega_\ell d_\ell }\bigg)^{C' r\alpha \sum_{\ell\in[L]}\omega_\ell d_\ell /\|\omega \|_\infty} \\
      & \leq \bigg(\frac{C'\alpha n}{Cm }\bigg)^{-C'r\alpha \sum_{\ell\in[L]}\omega_\ell d_\ell /\|\omega \|_\infty} \\
      & =: \mu.
\end{align*}
Let $S$ be the number of rows $i \in I$ such that $D_i >C'r\alpha \sum_{\ell\in[L]}\omega_\ell d_\ell$. Then $S$ is a sum of $m = |I|$ independent Bernoulli random variables, each having head probability at most $\mu$. So invoking Lemma \ref{lemma: concentration_of_weighted_bernoulli_sum} again (with $c_1 = 1$, $\omega_i=1$ for all $i$), we have
$$
      \bbP\bigg( S > \frac{m\|\weightvec\|_\infty}{\alpha \sum_{\ell\in[L]} \omega_\ell d_\ell} \bigg) \leq  \bigg(\frac{em\mu}{m\|\weightvec\|_\infty/(\alpha \sum_{\ell\in[L]}\omega_\ell d_\ell)}\bigg)^{m\|\weightvec\|_\infty/(\alpha \sum_{\ell\in[L]}\omega_\ell d_\ell)}.
$$
We claim that the right-hand side above is at most
$
      \mu^{m\|\weightvec\|_\infty/(2\alpha\sum_{\ell\in[L]}\omega_\ell d_\ell)}
$
for $C'$ large enough.
Indeed, this claim is equivalent to
$$
      \frac{e\alpha\sum_{\ell\in[L]}\omega_\ell d_\ell}{\|\weightvec\|_\infty} \leq \mu^{-1/2} = \bigg(\frac{C'\alpha n}{Cm}\bigg)^{ \frac{C'r\alpha \sum_{\ell\in[L]}\omega_\ell d_\ell}{2\|\weightvec\|_\infty}}.
$$
Since $\alpha n/m \geq 1$, it is true if 
$$
      e \cdot {\alpha\sum_{\ell\in[L]}\omega_\ell d_\ell}/{\|\weightvec\|_\infty} \leq \bigg((C'/C)^{C'r/2}\bigg)^{{\alpha\sum_{\ell\in[L]}\omega_\ell d_\ell}/{\|\weightvec\|_\infty}}.
$$
For a given constant $C$ (which only depends on $c_1$), we can choose $C'$ large enough such that the above inequality holds. Hence, we have
$$
      \bbP\bigg(S > m\| \omega\|_\infty/(\alpha \sum_{\ell\in[L]}\omega_\ell d_\ell)\bigg) \leq \mu^{m\|\weightvec\|_\infty/(2\alpha\sum_{\ell\in[L]}\omega_\ell d_\ell)} = \bigg(\frac{C'\alpha n}{C m}\bigg)^{-C'rm/2}\leq \bigg(\frac{(C')^2n}{C^2m}\bigg)^{-C'rm/4},
$$
where the last inequality is due to $\alpha^2n/m\geq1$.
Taking a union bound over all possible $m\in[n]$ and $I, J$ with $|I|=|J|=m$, we know that $S > m\| \omega\|_\infty/(\alpha\sum_{\ell\in[L]}\omega_\ell d_\ell)$ with probability at least 
$$
      1 - \sum_{m=1}^n \binom{n}{m}^2 \bigg(\frac{(C')^2n}{C^2m}\bigg)^{-C'rm/4} 
      \geq 1 - \sum_{m=1}^n \bigg(\frac{en}{m}\bigg)^{2m} \bigg(\frac{(C')^2n}{C^2m}\bigg)^{-C'rm/4}
      \geq 1 - \sum_{m=1}^n \bigg(\frac{({C'})^2n}{{C^2}m}\bigg)^{-(\frac{C'r}{4}-2)m},
$$
where the last inequality holds by choosing a large enough $C'$. Similar to the proof of Lemma \ref{lemma: concentration_in_infty_to_two_norm}, one readily checks that among the summands in the right-hand side above, the one with $m=1$ is the dominating term, and thus the right-hand side above can be further lower bounded by
\begin{equation}
  \label{eq:deg_of_subgraphs_prob}
  1 - n\cdot \bigg(\frac{({C'})^2 n}{{C}^2}\bigg)^{-(\frac{C'r}{4} - 2)} \geq 1-n^{-r}
\end{equation}  
for $C'$ large enough, and this concludes the proof of Lemma \ref{lemma: deg_of_subgraphs}.
\end{proof}


The following lemma shows that if a block has a small number of rows, then {most of its columns} has small $\ell_1$-norm.
\begin{lemma}[More on degrees of subgraphs]
\label{lemma: more_on_deg_of_subgraphs}
   Assume   \eqref{eq: another_assump_on_weights_general_bernoulli_matrix} holds for $c_1 \geq 1$. 
   Then for any $r \geq 1$, the following holds with probability at least $1-n^{-r}$: 
   uniformly for any $m\in[n]$, any 
   $$
   k\leq \bigg(m \|\weightvec\|_\infty/(\alpha \sum_{\ell\in[L]}\omega_\ell d_\ell)\bigg) \land m
   $$ 
   where $\alpha$ is an arbitrary (but fixed) number satisfying $\alpha \geq \sqrt{m/n}$,
   and any block $I \times J \subseteq [n]\times[n]$ with $|I|=k, |J|=m$, all but $m/4$ columns of $\barA_{I\times J}$ have $\ell_1$-norm bounded above by $Cr\|\weightvec\|_\infty$, where $C = C(c_1)$ is an absolute constant.
\end{lemma}
\begin{proof}[Proof of Lemma \ref{lemma: more_on_deg_of_subgraphs}]
We define the $\ell_1$ norm of the $j$-th column of the matrix $\barA_{I\times J}$ as 
$$
      D_j = \sum_{i\in I} \barA_{ij} =\sum_{\ell \in[L]} \sum_{i\in I} \omega_\ell \AAA{\ell}_{ij}.
$$
Now, $D_j$ is a weighted sum of Bernoulli random variables with
$$
      \bbE D_j = \sum_{\ell\in[L]}\sum_{i\in I} \omega_\ell \ppp{\ell}_{ij} \leq \frac{k}{n} \sum_{\ell\in[L]}\omega_\ell d_\ell \leq \frac{m\|\weightvec\|_\infty}{\alpha n},
$$
where the last inequality is due to $k\leq m\|\weightvec\|_\infty/(\alpha \sum_{\ell\in[L]}\omega_\ell d_\ell)$.
By Lemma \ref{lemma: concentration_of_weighted_bernoulli_sum}, for large enough $C'$ we have
$$
      \bbP(D_j > C'r\|\weightvec\|_\infty) \leq \bigg(\frac{C m\|\weightvec\|_\infty/(\alpha n)}{C'r \|\weightvec\|_\infty}\bigg)^{C'r} = \bigg(\frac{C' r \alpha n}{Cm}\bigg)^{-C'r} =: \mu.
$$
Let $S$ be the number of columns $j\in J$ with $D_j > C'r\|\weightvec\|_\infty$. Then $S$ is a sum of independent Bernoulli random variables, each having success probability at most $\mu$. Applying Lemma \ref{lemma: concentration_of_weighted_bernoulli_sum} (with $c_1 = 1$) gives
$
      \bbP(S > m/4) \leq (4e\mu)^{m/4}.
$
We claim that the above probability is at most $\mu^{m/6}$. This claim is equivalent to
$
      4e < \mu^{-1/3} = \big(\frac{C' r\alpha n }{Cm}\big)^{C'r/3}.
$
Since $\alpha n/m\geq 1$ and $r \geq 1$, it suffices to require
$
      4e < \big(\frac{C' }{C}\big)^{C'r/3},
$
which is true for $C'$ large enough. Hence
$$
      \bbP(S > m/4) \leq \mu^{m/6} = \bigg(\frac{C'r\alpha n}{Cm}\bigg)^{-C'rm/6} \leq \bigg(\frac{C'\alpha n}{Cm}\bigg)^{-C'rm/6} \leq \bigg(\frac{(C')^2n }{C^2 m}\bigg)^{-C'rm/12},
$$
where the last inequality is by $\alpha \geq \sqrt{m/n}$. 
We now take a union bound over all $m, k$ and $I, J$ with $|I| = k, |J| = m$. Note that it suffices to consider the largest possible $k$, which is at most $m$. So $S > m/4$ happens with probability at least 
$$
      1-\sum_{m=1}^n \binom{n}{m}^2  \bigg(\frac{(C')^2n }{C^2 m}\bigg)^{-C'rm/12}\leq 1-n^{-r},
$$
where the last inequality holds by choosing a large $C'$ and using similar arguments as those which lead to \eqref{eq:deg_of_subgraphs_prob}. 
\end{proof}


\subsection{Step \texorpdfstring{\RN{3}}{III}: Graph Decomposition}

The main idea in \cite{le2017concentration} is to seek for a partition of the set $[n]\times [n]$ into three blocks $\calN$, $\calR$ and $\calC$, where $\calN$ is a big block with good concentration behavior, and the rows of $\calR$ and the columns of $\calC$ have small $\ell_1$ norm. This \emph{graph decomposition} is implemented below.
\begin{proposition}[Graph decomposition]
      \label{prop:graph_decomp}
      Let Assumption \ref{assump:wts_for_concentration} hold with constants $c_0 > 0, c_1 \geq 1$. For any $r\geq 1$, with probability at least $1-3n^{-r}$, we can decompose $[n]\times [n]$ into three classes $\calN, \calR, \calC$ so that the following holds:
      \begin{itemize}
            \item The matrix $\barA$ concentrates well on $\calN$ in the sense that 
            $$
                  \|(\barA-\bbE \barA)_\calN\|\leq C_1 r^{3/2}\sqrt{\sum_{\ell\in[L]} \omega_\ell^2 d_\ell},
            $$
            where $C_1 = C_1(c_0, c_1)$ is an absolute constant; 
            \item Each row of $\barA_\calR$ and each column of $\barA_\calC$ has $\ell_1$-norm bounded above by $C_2r\|\weightvec\|_\infty$, where $C_2 = C_2(c_0, c_1)$ is another absolute constant; 
            \item Moreover, $\calR$ intersects at most $n\|\weightvec\|_\infty/{\sum_{\ell\in[L]} \omega_\ell d_\ell}$ columns and $\calC$ intersects at most $n\|\weightvec\|_\infty/{\sum_{\ell\in[L]} \omega_\ell d_\ell}$ rows of $[n]\times[n]$.
      \end{itemize}     
\end{proposition}
The proof of the above result is based on iterative applications of the following lemma.
\begin{lemma}[Decomposition of one block]
      \label{lemma: decomp_of_one_block}
      Let Assumption \ref{assump:wts_for_concentration} hold with constants $c_0 > 0, c_1 \geq 1$.
      Then for $r \geq 1$, the following holds with probability at least $1-3n^{-r}$: 
      uniformly for any $m\in[n]$, any block $I\times J \subseteq [n]\times[n]$ with $|I|=|J| = m$ and an arbitrary (but fixed) number $\alpha$ satisfying $\alpha \geq \sqrt{m/n}$, there exists a sub-block $I_1\times J_1 \subseteq I \times J$ with $|I_1|, |J_1|\leq m/2$, such that the remaining part of the block, namely $(I\times J)\setminus (I_1 \times J_1)$, can be decomposed into three parts, $\calR\subseteq (I\setminus I_1)\times J$, $\calC\subseteq I \times (J\setminus J_1)$, and $\cal N$, so that the following holds:
      \begin{itemize}
            \item The matrix $\barA$ concentrates well on $\calN$ in the sense that 
            $$
            \|(\barA-\bbE \barA)_\calN\|\leq C_1 r^{3/2}\sqrt{\alpha (\sum_{\ell\in[L]} \omega_\ell^2 d_\ell) \log (\frac{en}{m})},
            $$
            where $C_1 = C_1(c_0, c_1)$ is an absolute constant;
            \item Each row of $\barA_\calR$ and each column of $\barA_\calC$ has $\ell_1$-norm bounded above by $C_2r\|\weightvec\|_\infty$, where $C_2 = C_2(c_0, c_1)$ is another absolute constant;
            \item Moreover, $\calR$ intersects at most $m\|\weightvec\|_\infty/(\alpha \sum_{\ell\in[L]}\omega_\ell d_\ell)$ columns and $\calC$ intersects at most $m\|\weightvec\|_\infty/(\alpha \sum_{\ell\in[L]}\omega_\ell d_\ell)$ rows of $I\times J$.
      \end{itemize}     
\end{lemma}
\begin{proof}
      The proof is an adaptation of arguments in the proof of Lemma 3.7 of \cite{le2017concentration}. We fix a realization of $\barA$ such that Lemmas \ref{lemma: concentration_in_spectral_norm}, \ref{lemma: deg_of_subgraphs}, and \ref{lemma: more_on_deg_of_subgraphs} hold. Note that this event happens with probability at least $1- 3n^{-r}$. 
      \begin{figure}[t]
      \centering
      \includegraphics[width=\textwidth]{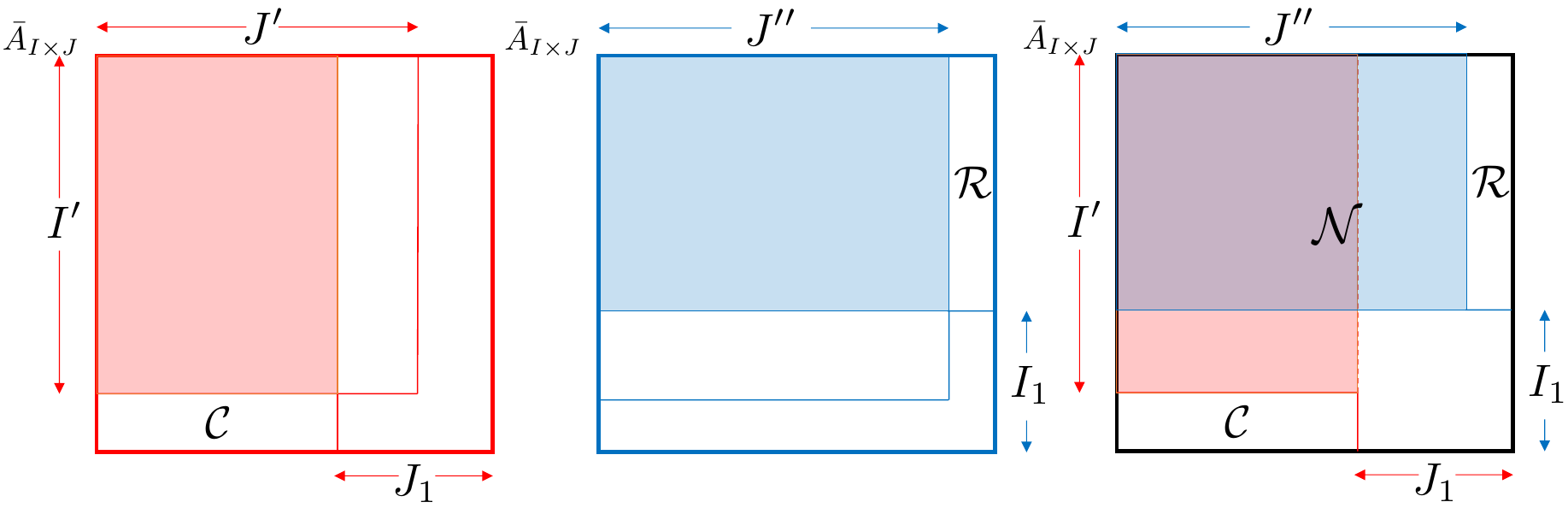}
      \caption{\small A pictorial illustration of one step of graph decomposition in Lemma \ref{lemma: decomp_of_one_block}. In the leftmost figure, we construct from $\bar A_{I\times J}$ the red shaded block which concentrates well, and a block $\calC$ which satisfies the requirements in Lemma \ref{lemma: decomp_of_one_block}. In the middle figure, we apply the same construction on $\barA_{I\times J}^\top$ to get the blue shaded block which concentrates well, and a block $\calR$ which satisfies the requirements in Lemma \ref{lemma: decomp_of_one_block}. In the rightmost figure, we combine the two constructions to obtain the desired decomposition in Lemma \ref{lemma: decomp_of_one_block}. The union of the red and blue shaded block is precisely $\calN$.}
    \label{fig:graph_decomp}
    \end{figure}

      We first construct the ``bad columns'' $J_1$. Fix some $\alpha \geq \sqrt{m/n}$. 
      By Lemma \ref{lemma: deg_of_subgraphs}, all but $m\|\weightvec\|_\infty/(\alpha \sum_{\ell\in[L]}\omega_\ell d_\ell)$ rows of $\barA_{I\times J}$ have $\ell_1$ norm bounded above by $C r\alpha \sum_{\ell\in[L]}\omega_\ell d_\ell$, where $C$ only depends on $c_1$. 
      Let $I'\subseteq I$ be the indices of those rows whose $\ell_1$ norm are bounded above by $Cr\alpha \sum_{\ell\in[L]} \omega_\ell d_\ell$ and $|I\setminus I'|\leq m\|\weightvec\|_\infty/(\alpha \sum_{\ell\in[L]}\omega_\ell d_\ell)$. 
      By Lemma \ref{lemma: concentration_in_spectral_norm} (with $\alpha$ replaced by $Cr\alpha$), we know that there exists a subset $J'\subseteq J$ with $|J'|\geq 3m/4$ such that
      $$
            \|(\barA - \bbE \barA)_{I'\times J'} \| \leq C' r \sqrt{\alpha (\sum_{\ell\in[L]}\omega_\ell^2 d_\ell) \log(\frac{en}{m})},
      $$
      where $C'$ depends on both $c_0$ and $c_1$. 
      For rows in $I\setminus I'$, whose cardinality is bounded above by both $m\|\weightvec\|_\infty/(\alpha \sum_{\ell\in[L]}\omega_\ell d_\ell)$ and $m$, we use Lemma \ref{lemma: more_on_deg_of_subgraphs} to deduce that, all but $m/4$ columns of $\barA_{(I\setminus I')\times J}$ have $\ell_1$ norm bounded above by $C''r\|\weightvec\|_\infty$, where $C''$ only depends on $c_1$. Let $J_1$ be the union of columns in $J\setminus J'$ and the columns of $A_{(I\setminus I')\times J}$ whose $\ell_1$ norm is larger than $C'' r \|\weightvec\|_\infty$. Note that $|J_1|\leq m/4 + m/4 = m/2$ by construction. In summary, we have found row indices $I'$ and column indices $J_1$, such that: 
      \begin{itemize}
            \item The block $(\barA - \bbE \barA)_{I' \times (J\setminus J_1)}$ satisfies the concentration inequality in the last display; 
            \item The block $\calC := (I\setminus I')\times (J\setminus J_1)$ satisfies the property specified in the lemma, i.e., columns of $\barA_\calC$ has $\ell_1$ norm bounded above by $C''r\|\weightvec\|_\infty$, and it intersects at most $m\|\weightvec\|_\infty/(\alpha \sum_{\ell\in[L]}\omega_\ell d_\ell)$ rows of $I\times J$;
            \item The size of $J_1$ is at most $m/2$.
      \end{itemize}
      See the leftmost part of Figure \ref{fig:graph_decomp} for a pictorial illustration.

      Now, we apply the same arguments to $A^\top$, which allows us to find row indices $I_1$ and column indices $J''$, such that
      \begin{itemize}
            \item The block $(\barA - \bbE \barA)_{(I\setminus I_1) \times J''}$ satisfies the concentration inequality in the last display; 
            \item The block $\calR := (I\setminus I_1)\times (J\setminus J'')$ satisfies the property specified in the lemma, i.e., rows of $A_\calR$ has $\ell_1$ norm bounded above by $C'''r\|\weightvec\|_\infty$ for some constant $C'''$ only depending on $c_0, c_1$, and it intersects at most $m\|\weightvec\|_\infty/(\alpha \sum_{\ell\in[L]}\omega_\ell d_\ell)$ columns of $I\times J$;
            \item The size of $I_1$ is at most $m/2$.
      \end{itemize}
      See the middle part of Figure \ref{fig:graph_decomp} for a pictorial illustration.

      To this end, we let $\calN = (I'\times (J\setminus J_1)) \cup ((I\setminus I_1)\times J'')$. See the rightmost part of Figure \ref{fig:graph_decomp} for a pictorial illustration. It is clear that $(A-\bbE A)_\calN$ satisfies the concentration inequality required by the lemma, which completes the proof.
\end{proof}

\begin{proof}[Proof of Proposition \ref{prop:graph_decomp}]
Proposition \ref{prop:graph_decomp} follows by iteratively applying Lemma \ref{lemma: decomp_of_one_block} for $\calO(\log n)$ times. Such arguments are nearly identical to arguments used in the proof of Theorem 2.6 in \cite{le2017concentration}, so we omit the details. 
\end{proof}

\subsection{Proof of Theorem \ref{thm:concentration_reg_mat}}
First, let us observe that bounded row and column $\ell_1$ norms leads to bounded operator norm.
\begin{lemma}[Lemma 2.7 of \cite{le2017concentration}]
      \label{lemma: relation_between_spectral_norm_and_l1_norm}
      Consider a matrix $B$ in which each row has $\ell_1$-norm at most $a$, and each column has $\ell_1$-norm at most $b$. Then $\|B\|\leq \sqrt{ab}$.
\end{lemma}

Following the route taken by Theorem 2.1 in \cite{le2017concentration}, we start with the following decomposition:
$$
      \tau(\barA) -\bbE \barA = (\tau(\barA) - \bbE \barA)_\calN + (\tau(\barA) - \bbE \barA)_\calR + (\tau(\barA) - \bbE \barA)_\calC.
$$
We take a realization of $\barA$ such that the conclusions in Proposition \ref{prop:graph_decomp} hold, which happens with probability at least $1-3n^{-r}$. 

For the well-behaved part, we write
$$
      (\tau(\barA) - \bbE \barA)_\calN = (\barA - \bbE \barA)_\calN - (\barA - \tau(\barA))_\calN.
$$
By Proposition \ref{prop:graph_decomp}, we have
$$
      (\barA - \bbE \barA)_\calN \leq Cr^{3/2}\sqrt{\sum_{\ell\in[L]}\omega_\ell^2 d_\ell}.
$$
On the other hand, since $\tau$ only takes effect on the elements in $\calE$, we have $\barA_{\calE^c} = \tau(\barA)_{\calE^c}$, and hence
$$
      \|(\barA - \tau(\barA))_\calN\|=\|(\barA-\tau(\barA))_{\calN\cap \calE}\| \leq \|\barA_{\calN \cap \calE}\| \leq \|(\barA - \bbE \barA)_{\calN \cap\calE}\| + \|\bbE \barA_{\calN \cap \calE}\|,
$$
where the first inequality is due to $0\leq \barA - \tau(\barA) \leq \barA$ entry-wise (recall that we only do down-weighting in the regularization process). Since $\calE$ is a disjoint union two of rectangular blocks, we have
$$
      \| (\barA - \bbE \barA)_{\calN \cap \calE}\| \leq 2\|(\barA - \bbE \barA)_\calN\| \leq Cr^{3/2}\sqrt{\sum_{\ell\in[L]}\omega_\ell^2 d_\ell}.
$$
Moreover, since the spectral norm of a non-negative matrix can only reduce by restricting onto any subset of $[n]\times[n]$, we get
$$
      \|\bbE \barA_{\calN\cap\calE}\|\leq \|\bbE \barA_\calE\| \leq \|\bbE \barA_{I \times [n]}\| + \|\bbE \barA_{I^c\times I}\|, 
$$
where we recall that $I$ contains the the vertices we choose to regularize. By construction, the $\ell_1$-norm of all rows of $\bbE \barA_{I\times[n]}$ is bounded above by $\sum_{\ell\in[L]}\omega_\ell d_\ell$.  
Meanwhile, by definition, $|I|\leq cn\|\weightvec\|_\infty/(\sum_{\ell\in[L]}\omega_\ell d_\ell)$, and thus the columns of $\bbE \barA_{I\times[n]}$ have $\ell_1$ norm bounded above by $c\|\omega \|_\infty$. By Lemma \ref{lemma: relation_between_spectral_norm_and_l1_norm}, we have 
$$
            \|\bbE \barA_{I\times[n]}\| \leq \sqrt{c \|\omega \|_\infty \sum_{\ell\in[L]}\omega_\ell d_\ell} \leq \sqrt{c c_0 \sum_{\ell\in[L]}\omega_\ell^2 d_\ell},
$$
where the last inequality is by our assumption on the weights \eqref{eq: assump_on_weights_general_bernoulli_matrix}. 
A similar bound holds for $\bbE \barA_{I^c\times I}$, which gives
$$
      \|\bbE \barA_{\calN \cap \calE}\|\lesssim \sqrt{c_0 \sum_{\ell\in[L]}\omega_\ell^2 d_\ell}.
$$
Hence we arrive at
$$
      (\tau(\barA) - \bbE \barA)_\calN \leq C_1 r^{3 / 2} \sqrt{\sum_{\ell\in[L]} \omega_\ell^2 d_\ell},
$$
where $C_1$ only depending on $c_0, c_1$. 

Now we deal with the block $\calR$. We have
$$
      \|(\tau(\barA) - \bbE \barA)_\calR\| \leq \|\tau(\barA)_\calR\| + \|\bbE \barA_\calR\|.
$$
Recall that $0\leq \tau(\barA)_\calR \leq \barA_\calR$ entry-wise because we only do down-weighting. By Proposition \ref{prop:graph_decomp}, each row of $\barA_\calR$, and thus of $\tau(\barA)_\calR$, has $\ell_1$-norm at most $Cr\|\weightvec\|_\infty$. Moreover, by our regularization process, each column of $\tau(\barA)$, and thus of $\tau(\barA)_\calR$, has $\ell_1$-norm at most $d_\tau$. Then, Lemma \ref{lemma: relation_between_spectral_norm_and_l1_norm} gives
$$
      \|\tau(\barA)_\calR \|       \leq \sqrt{C r \|\weightvec\|_\infty d_\tau}. 
$$
For $\bbE \barA_\calR$, by Proposition \ref{prop:graph_decomp}, each row of it has at most $n\|\weightvec\|_\infty/{\sum_{\ell\in[L]} \omega_\ell d_\ell}$ non-zero entries, and all entries are bounded above by $\sum_{\ell\in[L]}\omega_\ell d_\ell/n$. Hence each row of $\bbE \barA_\calR$ has $\ell_1$-norm at most $\|\weightvec\|_\infty$. Meanwhile, each column of $\bbE \barA_\calR$ has $\ell_1$-norm at most $\sum_{\ell\in[L]}\omega_\ell d_\ell$. An application of Lemma \ref{lemma: relation_between_spectral_norm_and_l1_norm} and   \eqref{eq: assump_on_weights_general_bernoulli_matrix} gives 
$$
      \|\bbE \barA_\calR\| \leq \sqrt{c_0 \sum_{\ell\in[L]}\omega_\ell^2 d_\ell},
$$
and hence
$$
      \|(\tau(\barA)-\bbE \barA )_\calR\| \leq \sqrt{C r \|\omega \|_\infty d_\tau} + \sqrt{c_0 \sum_{\ell\in[L]}\omega_\ell^2 d_\ell}.
$$
      A nearly identical argument gives
      $$
            \|(\tau(\barA)-\bbE \barA )_\calC\| \leq  \sqrt{C r \| \omega\|_\infty d_\tau} + \sqrt{c_0 \sum_{\ell\in[L]}\omega_\ell^2 d_\ell}.
      $$

Finally, we combine the bounds above to conclude
$$
      \|\tau(\barA) - \bbE \barA\| \leq C_1 r^{3/2}\sqrt{\sum_{\ell\in[L]} \omega_\ell^2 d_\ell} + 2 \sqrt{C_2 r\|\weightvec\|_\infty d_\tau}\\ \leq C r^{3/2}\left( \sqrt{\sum_{\ell\in[L]} \weight_\ell^2 d_\ell}+\sqrt{\|\weightvec\|_\infty} d_\tau \right),
$$
where $C$ only depends on $c_0,c_1, c$.

\subsection{Proof of Corollary \ref{cor:concentration_trimmed_mat}}
Since $\gamma > e^{c_1} \geq e$, by Lemma \ref{lemma: concentration_of_weighted_bernoulli_sum}, we have
$$
      \bbP(\sum_{j\in[n]} \barA_{ij} > \gamma \sum_{\ell\in[L]} \omega_\ell d_\ell) 
      \leq \bigg(\frac{e^{c_1} \sum_{\ell\in[L]}\sum_{j\in[n]}\omega_\ell p_{ij}^{(\ell)}}{\gamma\sum_{\ell\in[L]}\omega_\ell d_\ell}\bigg)^{\gamma \sum_{\ell\in[L]}\omega_\ell d_\ell / \| \omega\|_\infty} \leq (e^{c_1}/\gamma)^{\gamma \sum_{\ell\in[L]}\omega_\ell d_\ell / \| \omega\|_\infty} =: \mu.
$$
Invoking Lemma \ref{lemma: concentration_of_weighted_bernoulli_sum} again, for large enough $c$ we have
$$
      \bbP(|I| \geq cn\|\weightvec\|_\infty/\sum_{\ell\in[L]} \omega_\ell d_\ell) \leq \bigg(\frac{en\mu}{cn\| \omega\|_\infty /\sum_{\ell\in[L]}\omega_\ell d_\ell}\bigg)^{cn\| \omega\|_\infty /\sum_{\ell\in[L]}\omega_\ell d_\ell}.
$$
We claim that, we can specify $c$ based only on $c_1$ and $\gamma$, so that the above probability is less than or equal to $\mu^{cn\|\weightvec\|_\infty/2\sum_{\ell\in[L]}\omega_\ell d_\ell}$. Indeed, this claim is equivalent to
$$
      \frac{e \sum_{\ell\in[L]}\omega_\ell d_\ell}{c\|\weightvec\|_\infty} \leq \mu^{-1/2} = \bigg(\frac{\gamma}{e^{c_1}}\bigg)^{\gamma \sum_{\ell\in[L]}\omega_\ell d_\ell / 2\| \omega\|_\infty}.
$$
Since we've chosen $\gamma$ such that $\gamma/e^{c_1}> 1$, it suffices to choose $c$ such that
$$
      \frac{ex}{c} \leq \bigg((\gamma/e^{c_1})^{\gamma/2}\bigg)^x , \qquad \forall x > 0.
$$
This can be done by choosing $c$ only based on $c_1$ and $\gamma$. Hence, with such a choice of $c$, we have
$$
      \bbP(|I| \geq cn\|\weightvec\|_\infty/\sum_{\ell\in[L]} \omega_\ell d_\ell ) \leq \mu^{cn\|\weightvec\|_\infty/2\sum_{\ell\in[L]}\omega_\ell d_\ell} = (\gamma/e^{c_1})^{-\gamma c n /2}.
$$
Thus, an application of Theorem \ref{thm:concentration_reg_mat} and   \eqref{eq: assump_on_weights_general_bernoulli_matrix} gives the desired result.

\subsection{Proof of Corollary \ref{cor:concentration_w/o_reg}}
Recall that in the proof of Corollary \ref{cor:concentration_trimmed_mat}, we have established:
$$
      \bbP(\sum_{j\in[n]} \barA_{ij} > \gamma \sum_{\ell\in[L]} \omega_\ell d_\ell)  \leq (e^{c_1}/\gamma)^{\gamma \sum_{\ell\in[L]}\omega_\ell d_\ell / \| \omega\|_\infty},
$$
where $\gamma$ is some constant satisfying $\gamma > e^{c_1}$. If we are in the not-too-sparse regime, i.e., $\sum_{\ell\in[L]}\omega_\ell d_\ell \geq c_2 \|\omega \|_\infty\log n$ for some constant $c_2 > 0$, then the above probability can be bounded by
$$
      n^{-c_2 \gamma \log (\gamma/e^{c_1})}.
$$
Applying a union bound over $[n]$, we conclude that in this regime, every row of $A$ has $\ell_1$ norm bounded above by $\gamma \sum_{\ell\in[L]}\omega_\ell d_\ell$ with high probability. So without any regularization, we obtain the following guarantee:
$$
      \|A - \bbE A \| \leq C r^{3/2} \sqrt{\sum_{\ell\in[L]} \omega_\ell^2 d_\ell}
$$
with probability at least $1-3n^{-r} - n^{-c_3}$, where $c_3$ only depends on $c_1, c_2$.

On the other hand, if we are in the very sparse regime, i.e., $\sum_{\ell\in[L]}\omega_\ell d_\ell \leq c_2 \|\omega \|_\infty\log n$, by Lemma \ref{lemma: concentration_of_weighted_bernoulli_sum} we have
\begin{align*}
      \bbP(\sum_{j\in[n]} \barA_{ij} > \gamma \| \omega\|_\infty \log n) & \leq \bigg(\frac{e^{c_1} \sum_{j}\sum_{\ell\in[L]} \omega_\ell p_{ij}^{(\ell)}}{\gamma \| \omega\|_\infty \log n}\bigg)^{\gamma \| \omega\|_\infty \log (n) / \| \omega \|_\infty} \\
      & \leq (c_2 e^{c_1}/\gamma)^{\gamma \log n} \\
      & = n^{-\gamma \log (\gamma / c_2e^{c_1})}.
\end{align*}
By choosing $\gamma$ and taking a union bound over $[n]$, we conclude that in this regime, every row of $A$ has $\ell_1$ norm bounded above by $\gamma \|\omega \|_\infty\log n$ with high probability. Hence, invoking Theorem \ref{thm:concentration_reg_mat}, without regularization, we obtain the following guarantee:
$$
      \|A - \bbE A \| \leq Cr^{3/2} \sqrt{\sum_{\ell\in[L]}\omega_\ell^2 d_\ell} + C_2 \| \omega\|_\infty \sqrt{\log n},
$$
with probability at least $1 - 3n^{-r} - n^{-c_3}$, where $c_3$ and $C_2$ only depend on $c_1$. The desired result follows by combining the two regimes.

\section{Properties of Key Information-Theoretic Quantities} \label{append:properties_of_info}

In this section, we state and prove some useful properties of $\psi_S$ and $\psi_S^\star$, the two key information-theoretic quantities in the minimax rate. 

\begin{lemma}
\label{lemma:monotonicity_of_I_t}
For any $\ell\in[L]$, the quantity $\III{\ell}_t$ is increasing in $t$ for $t\in[0, 1/2]$ and decreasing in $t$ for $t\in[1/2, 1]$. 
\end{lemma}
\begin{proof}
  Note that $-\III{\ell}_{t}$ is the cumulant generating function of the following random variable (see   \eqref{eq:mgf_opt_test_z_star} and \eqref{eq:cgf_opt_test_z_star}):
  $
    \Bern(p_\ell) \cdot \log \frac{q_\ell(1-p_\ell)}{p_\ell(1-q_\ell)} + \Bern(q_\ell) \cdot \log \frac{p_\ell(1-q_\ell)}{q_\ell(1-p_\ell)}.
  $
  Since the cumulant generating function, if it exists, is always convex, we know that $\III{\ell}_t$ is a concave function in $t$. Thus it suffices to show $\III{\ell}_{t}$ attains its maximum at $t = 1/2$. We can expand $\III{\ell}_t$ as
  $$
    \III{\ell}_{t} = -\log\bigg(p_\ell q_\ell + (1-p_\ell)(1-q_\ell) + [p_\ell(1-q_\ell)]^{1-t} [(1-p_\ell)q_\ell]^t + [p_\ell(1-q_\ell)]^{t}[(1-p_\ell)q_\ell]^{-1-t} \bigg).
  $$
  Using the fact that $a+ b\geq 2\sqrt{ab}$ for any $a, b\geq 0$, with equality only if $a = b$, one finds that the maximum of $\III{\ell}_t$ is attained at $t = 1/2$, and the proof is concluded.
\end{proof}

\begin{lemma}
\label{lemma:asymp_equiv_of_I_t}
Assume there exist constants $C > 1, c\in(0, 1)$ such that $q_\ell < p_\ell \leq (C q_\ell) \land (1-c)$ for any $\ell\in[L]$. Then for any $\ell\in[L]$, we have
\begin{equation}
\label{eq:asymp_equiv_of_I_1/2}
  \III{\ell}_{1/2}\asymp \frac{(p_\ell-q_\ell)^2}{p_\ell}.
\end{equation}
If in addition, $p_\ell = o(1)$ for any $\ell \in [L]$, then 
$$
  \III{\ell}_{1/2} = (1+o(1)) (\sqrt{p_\ell} - \sqrt{q_\ell})^2.
$$
\end{lemma}
\begin{proof}
  This is a direct consequence of Lemma B.1 in \cite{zhang2016minimax}.
\end{proof}

\begin{lemma}[Formula for $\III{\ell}_t$ under simplified setups]
  \label{lemma:formula_I_l_simplified}
  For any $\ell$, assume $p_\ell = a_\ell\eta_\ell, b_\ell = b_\ell\eta_\ell$ where $b_\ell < a_\ell < \infty$ are two constants and $\eta_\ell$ is a positive sequence such that
  \begin{equation*}
    \liminf_{n\to\infty} \inf_{\ell\in[L]} a_\ell-b_\ell > 0, \ \ \ \limsup_{n\to\infty} \sup_{\ell\in[L]} a_\ell <\infty, \ \ \ \lim_{n\to\infty}\sup_{\ell\in[L]}\eta_\ell = 0.
  \end{equation*}
  Then, we have
  \begin{equation*}
    \lim_{n\to\infty}\sup_{\substack{\ell\in[L]\\ t\in(0, 1/2]}} \bigg|\frac{(a_\ell^t - b_\ell^t)(a_\ell^{1-t}- b_\ell^{1-t})\eta_\ell}{\III{\ell}_t} - 1\bigg| = 0.
  \end{equation*}
  Moreover, since $\III{\ell}_0 = 0$, so the formula $(a_\ell^t - b_\ell^t)(a_\ell^{1-t}- b_\ell^{1-t})\eta_\ell$ is also accurate at $t = 0$. 
\end{lemma}

The proof of the above lemma will based on the following fact.
\begin{lemma}[Newton's generalized binomial theorem]
  \label{lemma:binom_thm}
  Let $x, t$ be two arbitrary complex numbers. We have
  $$
    (1-x)^t = \sum_{k=0}^\infty (-x)^k \cdot\frac{t (t-1) \cdots (t-k+1)}{k!}.
  $$
\end{lemma}

We now provide a proof of Lemma \ref{lemma:formula_I_l_simplified}.
\begin{proof}[Proof of Lemma \ref{lemma:formula_I_l_simplified}]
For fixed $\ell\in[L]$ and $t\in(0, 1/2]$, we have
\begingroup
\allowdisplaybreaks
\begin{align*}
    \III{\ell}_t & = - \log \bigg(1 - p_\ell -q_\ell + 2p_\ell q_\ell + p_\ell^{1-t}q_\ell^{t}(1-p_\ell)^{t} (1-q_\ell)^{1-t} + p_\ell^t q_\ell^{1-t}(1-p_\ell)^{1-t}(1-q_\ell)^t\bigg) \\
    & = - \log \bigg( 1 - (p_\ell^t-q_\ell^t)(p_\ell^{1-t} - q_\ell^{1-t}) -p_\ell^t q_\ell^{1-t} - p_\ell^{1-t}q_\ell^t + 2p_\ell q_\ell \\
    & \qquad \qquad +   p_\ell^{1-t}q_\ell^{t}(1-p_\ell)^{t} (1-q_\ell)^{1-t} + p_\ell^t q_\ell^{1-t}(1-p_\ell)^{1-t}(1-q_\ell)^t  \bigg)\\
    & = -\log \bigg(1 - (p_\ell^t-q_\ell^t)(p_\ell^{1-t}- q_\ell^{1-t}) - p_\ell^t q_\ell^{1-t} \cdot \bigl(1 - (1-p_\ell)^{1-t}(1-q_\ell)^t - p_\ell^{1-t}q_\ell^t\bigr) \\
    & \qquad \qquad - p_\ell^{1-t}q_\ell^t \cdot \bigl( 1 - (1-p_\ell)^t(1-q_\ell)^{1-t} - p_\ell^t q_\ell^{1-t}\bigr)\bigg).
\end{align*}
\endgroup
We then proceed by
\begingroup
\allowdisplaybreaks
\begin{align*}
  & 1 - (1-p_\ell)^{1-t}(1-q_\ell)^t - p_\ell^{1-t} q_\ell^t \\
  & = (1-p_\ell) \bigg[1 - \bigg(\frac{1-q_\ell}{1-p_\ell}\bigg)^t\bigg] + p_\ell \bigg[1 - \bigg(\frac{q_\ell}{p_\ell}\bigg)^t\bigg] \\
  & = (1-p_\ell) \cdot \frac{-\frac{p_\ell-q_\ell}{1-p_\ell} + \bigg(\frac{1-q_\ell}{1-p_\ell}\bigg)^{1-t} - \bigg(\frac{1-q_\ell}{1-p_\ell}\bigg)^t}{1 + \bigg(\frac{1-q_\ell}{1-p_\ell}\bigg)^{1-t}} + p_\ell \cdot \frac{\frac{p_\ell-q_\ell}{p_\ell} + \bigg(\frac{q_\ell}{p_\ell}\bigg)^{1-t} - \bigg(\frac{q_\ell}{p_\ell}\bigg)^t}{1 + \bigg(\frac{q_\ell}{p_\ell}\bigg)^{1-t}}\\
  & = \frac{(1-p_\ell)-(1-q_\ell) + (1-p_\ell)^t (1-q_\ell)^{1-t} - (1-p_\ell)^{1-t}(1-q_\ell)^t}{1 + \bigg(\frac{1-q_\ell}{1-p_\ell}\bigg)^{1-t}} + \frac{p_\ell - q_\ell + p_\ell^{t}q_\ell^{1-t} - p_\ell^{1-t}q_\ell^t}{1 + \bigg(\frac{q_\ell}{p_\ell}\bigg)^{1-t}}\\
  & = \bigg((1-p_\ell)^t - (1-q_\ell)^t\bigg) \times \frac{(1-p_\ell)^{1-t} + (1-q_\ell)^{1-t}}{1 + \bigg(\frac{1-q_\ell}{1-p_\ell}\bigg)^{1-t}} + (p_\ell^t - q_\ell^t) \times \frac{p_\ell^{1-t} + q_\ell^{1-t}}{1 + \bigg(\frac{q_\ell}{p_\ell}\bigg)^{1-t}}.
\end{align*}
\endgroup
Thus, we have
\begin{align*}
  & |1 - (1-p_\ell)^{1-t}(1-q_\ell)^t - p_\ell^{1-t} q_\ell^t |
   = \Theta\bigg((1-q_\ell)^t - (1-p_\ell)^t\bigg) + \Theta\bigg((p_\ell^t - q_\ell^t) (p_\ell^{1-t} + q_\ell^{1-t})\bigg),
\end{align*}
where we use $x_n = \Theta(y_n)$ to denote $x_n \asymp y_n$. A similar calculation gives 
\begin{align*}
  |1 - (1-p_\ell)^t(1-q_\ell)^{1-t} - p_\ell^t q_\ell^{1-t}\bigr)| = \Theta\bigg((1-q_\ell)^{t} - (1-p_\ell)^{t}\bigg) + \Theta\bigg((p_\ell^t - q_\ell^t) (q_\ell^{1-t} + p_\ell^{1-t})\bigg).
\end{align*}
Note that under the current assumptions, $p_\ell^t q_\ell^{1-t}\asymp \eta_\ell$ and $p_\ell^{1-t} q_\ell^{t}\asymp\eta_\ell$, uniformly over $\ell$. Thus, the desired result is implied by
$$
  \sup_{t\in(0, 1/2]}\frac{[(1-b_\ell\eta_\ell)^t - (1-a_\ell\eta_\ell)^t] + (a_\ell^t-b_\ell^t)(a_\ell^{1-t} + b_\ell^{1-t}) \eta_\ell}{ (a_\ell^t-b_\ell^t)(a_\ell^{1-t}-b_\ell^{1-t})} \to 0.
$$
So it suffices to show
\begin{enumerate}
  \item[(A).] $(1-q_\ell)^t - (1-p_\ell)^t \ll (a_\ell^t - b_\ell^t)(a_\ell^{1-t}-b_\ell^{1-t})$ uniformly over $t\in(0, 1/2]$;
  \item[(B).] $(a_\ell^t-b_\ell^t)(a_\ell^{1-t}+b_\ell^{1-t})\eta_\ell\ll (a_\ell^t - b_\ell^t)(a_\ell^{1-t}-b_\ell^{1-t})$ uniformly over $t\in(0, 1/2]$.
\end{enumerate}

We first show (A). By Lemma \ref{lemma:binom_thm}, we have
\begin{align*}
  (1-q_\ell)^t - (1-p_\ell)^t & = \sum_{k=0}^\infty (-1)^k (q_\ell^k - p_\ell^k) \cdot \frac{t(t-1)\cdots (t-k+1)}{k!}.
\end{align*}
Since $t\in(0, 1/2]$, we have
$$
  \bigg|\frac{t(t-1)\cdots (t-k+1)}{k!}\bigg| = t\cdot \prod_{j=2}^k \bigg|\frac{t-j+1}{j}\bigg| \leq  1.
$$
This means that
$$
  (1-q_\ell)^t - (1-p_\ell)^t \leq t \sum_{k=0}^\infty (p_\ell^k - q_\ell^k)  = t \cdot \bigg(\frac{1}{1-p_\ell} - \frac{1}{1-q_\ell}\bigg) \asymp t(p_\ell-q_\ell) = t(a_\ell-b_\ell)\eta_\ell.
$$
So (A) is implied by 
$$
  \sup_{t\in(0, 1/2]} \frac{t(a_\ell-b_\ell) \eta_\ell }{(a_\ell^t - b_\ell^t)(a_\ell^{1-t}-b_\ell^{1-t})} \to 0.
$$
Under the current assumptions, for large $n$, we can find an absolute constant $c$ such that $a_\ell -b_\ell > c >0$. Thus, we have
\begin{align*}
  a_\ell^t - b_\ell^t  &=  e^{t\log a_\ell} -e^{t\log b_\ell} \\
  & = \sum_{k=0}^\infty \frac{t^k [(\log a_\ell)^k-(\log b_\ell)^k] }{k!} \\
  & \geq t(\log a_\ell - \log b_\ell)\\
  & \geq c' t,
\end{align*}
and $a_\ell^{1-t}-b_\ell^{1-t} \geq c'' $ for some $c', c'' >0$. Thus, (A) is implied by
$$
  \sup_{t\in(0, 1/2]} \frac{t(a_\ell-b_\ell) \eta_\ell}{c'c'' t} \to 0,
$$
which trivially holds by $\sup_{\ell\in[L]}\eta_\ell \to 0$. 

Finally we show (B). This is equivalent to
$$
  \sup_{t\in(0, 1/2]} \frac{(a_\ell^{1-t} + b_\ell^{1-t})\eta_\ell}{a_\ell^{1-t}-b_\ell^{1-t}} \to 0,
$$
which holds because $\sup_{\ell\in[L]}\eta_\ell \to 0$ and $a_\ell^{1-t} - b_\ell^{1-t}> c'' > 0$ for large $n$.
\end{proof}

\begin{lemma}
  \label{eq:convexity_of_cum_gen_func}
  For any $S\subseteq[L]$, the function $\psi_S(t)$ is convex in $t$ for $t\geq 0$. Moreover, we have
  \begin{align*}
    \frac{d \psi_S(t)}{dt} 
    & = \frac{n}{2}\sum_{\ell\in S} \bigg(\log \frac{p_\ell(1-q_\ell)}{q_\ell(1-p_\ell)}\bigg) \\
    & ~~\times \frac{[p_\ell(1-q_\ell)]^{t} [q_\ell(1-p_\ell)]^{1-t}  - [p_\ell(1-q_\ell)]^{1-t} [q_\ell(1-p_\ell)]^t }{p_\ell q_\ell + (1-p_\ell)(1-q_\ell) + [p_\ell(1-q_\ell)]^{1-t} [q_\ell(1-p_\ell)]^t + [p_\ell(1-q_\ell)]^{t} [q_\ell(1-p_\ell)]^{1-t}},
  \end{align*}  
  which is also equal to $\bbE[\tilde\mu_t]$, where $\tilde\mu_t$ is the exponentially tilted law defined in \eqref{eq:tilted_law}.
\end{lemma}
\begin{proof}
  Let $\mu$ be defined in \eqref{eq:original_law}. 
  The convexity of $\psi_S(t)$ follows from the fact that it is the cumulant generating function of $\mu$:
  $
    \psi_S(t) = \log \int e^{t x} d\mu(x).
  $
  Now, by construction, we have
  $
    \frac{d\tilde\mu_t}{d\mu}(x) \propto e^{tx}.
  $
  Hence, we have
  $$
    \frac{d\psi_S(t)}{dt} = \frac{\int x e^{t x}  d\mu(x)}{\int e^{tx} d\mu(x)} = \int x d\tilde\mu_t(x).
  $$
  The exact form of $\tilde\mu_t$ is calculated in   \eqref{eq:tilted_rvs}, and the exact formula of the right-hand side above follows from direct computations, so we omit the details. 
\end{proof}

\begin{lemma}
\label{lemma:ineq_for_cvx_conjugate}
For any $S\subseteq[L]$, the function $\psi_S^\star$ satisfies $\psi_S^\star(0) = \frac{n}{2} \sum_{\ell\in S} \III{\ell}_{1/2}$
and
\begin{align}
  \label{eq:ineq_for_cvx_conjugate}
  & 0 \lor \bigg(-J_\rho + \frac{n}{2}\sum_{\ell\in S}\III{\ell}_{1/2}\bigg)\leq \psi_S^\star(-2J_\rho) \leq \frac{n}{2}\sum_{\ell\in S}\III{\ell}_{1/2}.
\end{align}
\end{lemma}
\begin{proof}
  By definition, we have
  $$
    \psi_S^\star(0) = \sup_{t\in[0, 1]} -\psi_S(t) = \frac{n}{2} \sup_{t\in[0, 1]}  \sum_{\ell\in S}\III{\ell}_{t} = \frac{n}{2} \sum_{\ell\in S}\III{\ell}_{1/2},
  $$
  where the last equality follows from Lemma \ref{lemma:monotonicity_of_I_t}. On the other hand, we have
  $$
    \psi_S^\star(-2J_\rho) = \sup_{t\in[0, 1]} -2tJ_\rho + \frac{n}{2} \sum_{\ell\in S}\III{\ell}_t
  $$
  Choosing $t = 1/2$ and $t = 0$ in the right-hand side above respectively gives the two lower bounds in   \eqref{eq:ineq_for_cvx_conjugate}. 
  Finally, we can upper bound $\psi^\star_S(-2J_\rho)$ by
  $$
    \sup_{t\in[0, 1]} \frac{n}{2} \sum_{\ell\in S} \III{\ell}_t = \frac{n}{2} \sum_{\ell\in S} \III{\ell}_{1/2},
  $$
  and the proof is concluded.
\end{proof}

\begin{lemma}
  \label{lemma:optimal_t_for_cvx_conjugate}
  The optimal $t$ that gives rise to $\psi_S^\star(-2J_\rho) = \sup_{t\in[0, 1]}  -2t J_\rho -\psi_S(t)$ satisfies $t\in[0, 1/2]$. 
\end{lemma}
\begin{proof}
  Note that $\psi_{S}(t)$ is symmetric over $t=1/2$: $\psi_S(1/2-\delta) = \psi_S(1/2+\delta)$ for any $\delta \in[0, 1/2]$. Thus, for any $t_1 \in [1/2, 1]$, its reflected point $t_1\in[0, 1/2]$ w.r.t. the $t=1/2$ axis always satisfies
  $$
    -2t_2J_\rho - \psi_S(t_2) \geq -2t_1 J_\rho - \psi_S(t_1),
  $$
  from which the desired result follows.
\end{proof}

\section{More Details on Experiments}\label{append:more_exp}

\subsection{Spectral Clustering and Choice of Weights}\label{subappend:weight_choice}
Recall that our Algorithms \ref{alg: generic_refinement} and \ref{alg: provable_refinement} both require an initialization scheme, which by default is set to spectral clustering on the trimmed weighted adjacency matrix $\bar A = \sum_{\ell\in[L]}{\omega_\ell}\AAA{\ell}$ (i.e., Algorithm \ref{alg: specc}). In this experiment, we set the trimming threshold $\gamma = 5$, and we explore three choices of weights: (1) $\omega_\ell \propto 1$ (uniform weight), (2) $\omega_\ell \propto 1/p_\ell$ (scale by variance), and (3) $\omega_\ell \propto 1/\sqrt{p_\ell}$ (scale by standard deviation), where $p_\ell$'s are either known or estimated from the data using the method of moment (see Appendix \ref{subappend:est_deg} for the detailed estimation procedure). We consider the setup in Section \ref{sec:exp}, and we set $n =1000, L = 100, \rho = 0.1$ and vary $c$. 

\begin{figure}[t]
        \centering
        \includegraphics[width=0.6\textwidth]{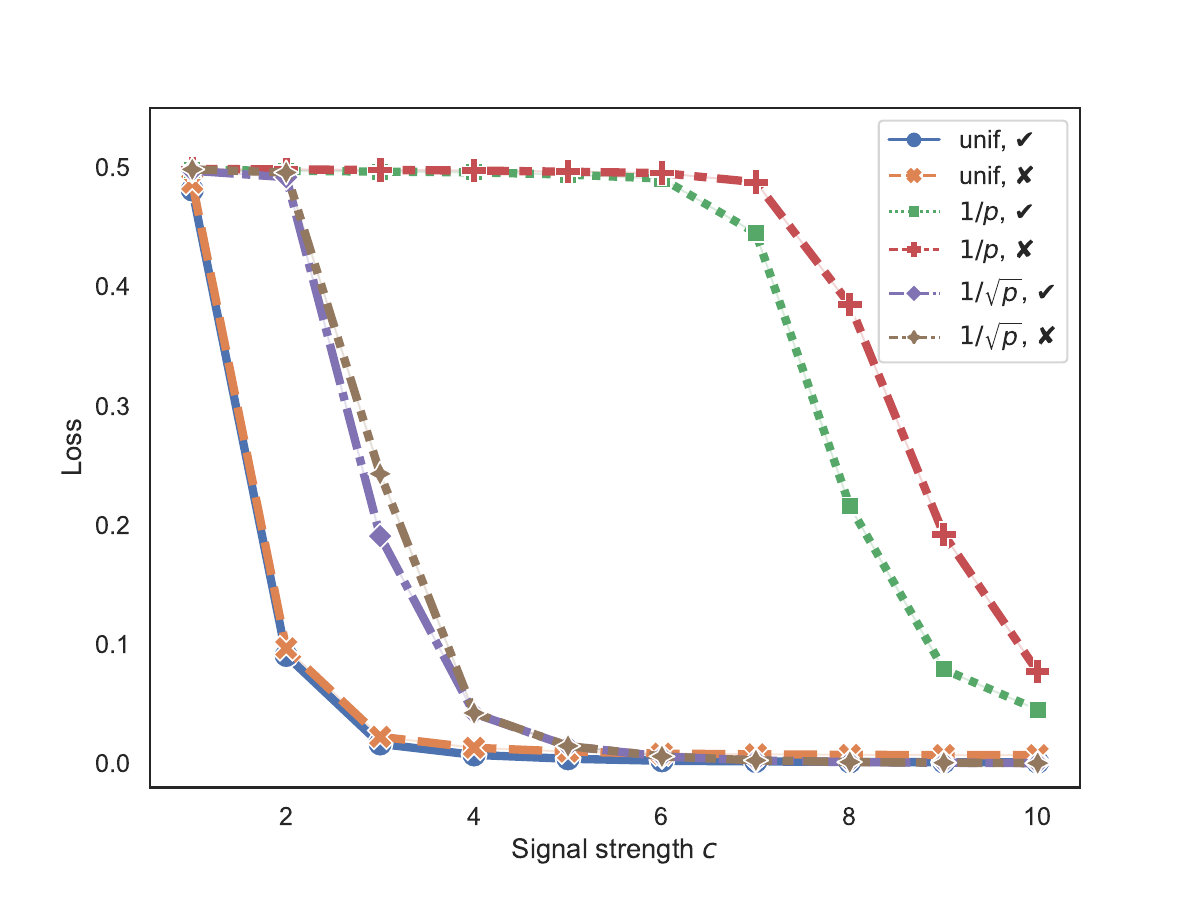}
        \caption{\small Average misclustering proportion against the signal strength $c$ for spectral clustering under different choices of weights. A check mark means the ground truth $p_\ell$'s are used, and a cross mark means $p_\ell$'s are estimated from the data. 
        }
        \label{fig:compare_wts}
\end{figure}

Figure \ref{fig:compare_wts} shows the misclustering proportion (for estimating $\bz^\star$) over $500$ simulations. We see that overall, uniform weight performs the best, regardless of whether $p_\ell$'s are known or not. Scaling by $\sqrt{p_\ell}$ performs slightly better than uniform weight when the signal strength $c$ is large, but is significantly worse when $c$ is small. Scale by $p$ is the worst among the three. In the rest of the experiments, the initialization scheme is always set to spectral clustering with uniform weight.

\subsection{Co-Regularized Spectral Clustering}\label{subappend:co_reg}
{Co-regularized spectral clustering is a popular algorithm for clustering in multilayer networks originally proposed by \cite{kumar2011co}, and it was later shown to be consistent in the $\rho=0$ regime by \cite{paul2020spectral}.}
It solves the following optimization problem:
\[
[\hat U^{(1)},\ldots,\hat U^{(L)},\hat U^\star]=\argmax_{\substack{{U^{(\ell)}}\tran U^{(\ell)}=I,\forall\ell\in[L] \\ {U^\star}\tran U^\star=I }}\sum_{\ell=1}^L\tr{{U^{(\ell)}}\tran \AAA{\ell}U^{(\ell)} } + \gamma_\ell \tr{{U^{\star}}\tran U^{(\ell)}{U^{(\ell)}}\tran U^{\star} }.
\]
In our implementation, the regularization parameter is set to be $\gamma_\ell=\|\AAA{\ell}\|_2$ as suggested by \cite{paul2020spectral}, and we solve the problem by alternating between optimizing $U^{(\ell)}$ and $U^\star$ via eigen-decomposition. The maximum number of iterations is 20. We then apply $k$-means on $\hat U^{\star}$ and $\hat U^{(\ell)}$ to get the global assignment $\bhz^{\star}$ and individual assignments $\{\bhz^{(\ell)}\}$. We emphasize that \cite{paul2020spectral} only proposed to use $\hat U^\star$ to get the global assignment and proved the consistency, and did not propose to use $\hat U^{(\ell)}$ to get the individual assignments.

\subsection{Varying the Number of Layers}
In this subsection, we do a simulation to explore the effect of the number of layers $L$. The setup is similar to the simulation in Section \ref{subsec:effect_of_snr}: we set $n=1000, \rho = 0.1, c = 3$, and we vary $L$ from $20$ to $100$ while maintaining the proportions of weak, intermediate, and strong layers.
We run Algorithm \ref{alg: generic_refinement} over $100$ instances of the model (assuming $\{p_\ell\}, \{q_\ell\}, \rho$ are known) and run co-regularized spectral clustering over $50$ instances of the model and record their misclustering proportions for both \globest~and \indest. 

\begin{figure}[t]
\centering
    \begin{subfigure}{.45\textwidth}
      \centering
      \includegraphics[width=1.1\linewidth]{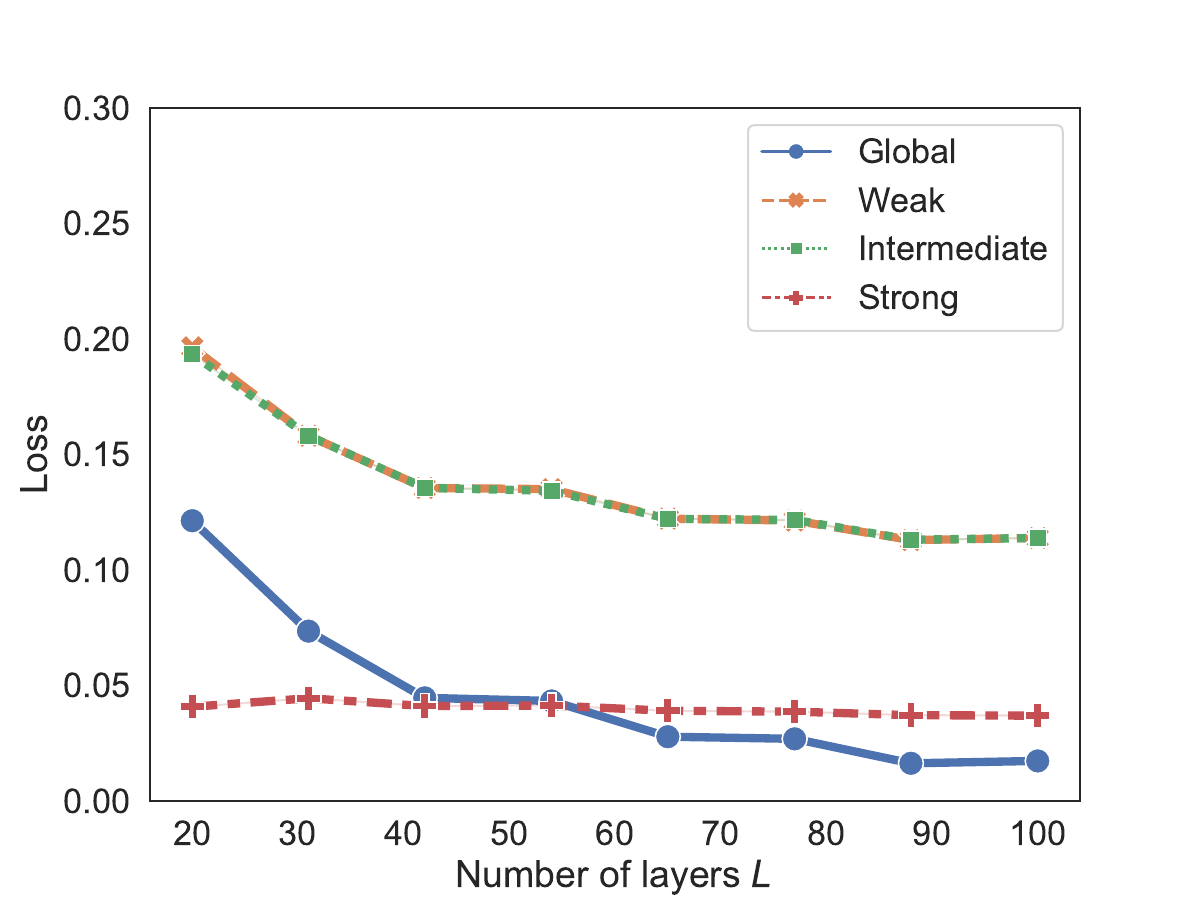}
      \caption{Our method: Algorithm \ref{alg: specc}+\ref{alg: generic_refinement}}
      \label{fig:vary_rho}
    \end{subfigure}
    \begin{subfigure}{.45\textwidth}
      \centering
      \includegraphics[width=1.1\linewidth]{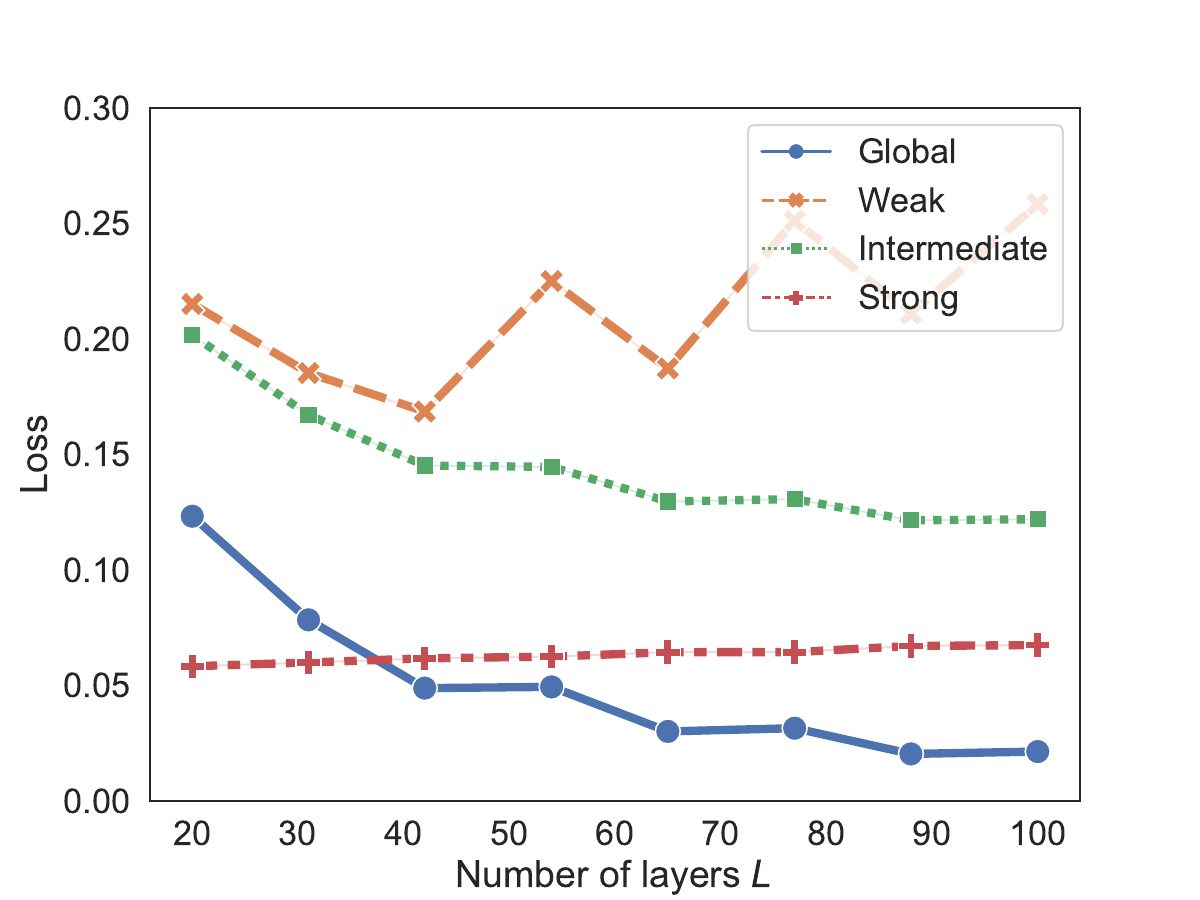}
      \caption{Co-regularized spectral clustering}
      \label{fig:vary_c}
    \end{subfigure}
        \vskip -0.33\baselineskip
        \caption{\small Performance of our method (left panel) and co-regularized spectral clustering (right panel) \cite{kumar2011co,paul2020spectral} as the number of layers vary.
        }
    \label{fig:effect_of_L}
\end{figure}

The results are shown in Figure \ref{fig:effect_of_L}. From Figure \ref{fig:effect_of_L}(a), we see that for our algorithm, the misclustering proportions for  global, weak and intermediate layers all go down monotonically as $L$ increases. This is because the additional layers introduces more information. 
In contrast, the misclustering proportion for strong layers stays constant, because the benefits from the other layers have saturated. The trends are largely the same for the co-regularized spectral clustering algorithm (Figure \ref{fig:effect_of_L}(b)), except that the misclustering proportions are higher and the weak layers have oscillating and diverging misclustering proportion.

\subsection{Estimating \texorpdfstring{$p_\ell$}{pl}'s}\label{subappend:est_deg}

The inputs to Algorithm \ref{alg: specc} include $\{p_\ell\}$. To estimate them, we begin by recalling that $E[\AAA{\ell}_{ij}] =\tilde p_\ell =  p_\ell - 2(p_\ell-q_\ell)(\rho-\rho^2) \approx p_\ell$ if $\bz^\star_i = \bz^\star_j$ and $E[\AAA{\ell}_{ij}] =\tilde q_\ell= q_\ell + 2(p_\ell-q_\ell)(\rho-\rho^2) \approx q_\ell$ otherwise.
Assuming $\tilde p_\ell = C\tilde q_\ell$, we then have 
$$
\sum_{i<j} \mathbb{E}[A^{\ell}_{ij}] = \tilde{p}_\ell\times \frac{n^2 - n - 2(1-C^{-1})n^\star_+n^\star_-}{2}.
$$
If $n^\star_+ = C'n$, then we have 
$$
\sum_{i<j} \mathbb{E}[\AAA{\ell}_{ij}] = \tilde{p}_\ell\times \frac{[1-2(1-C^{-1})C'(1-C')]n^2 - n}{2}.
$$
Specifically, if $C' \approx 1/2$ (i.e., the two clusters are approximately symmetric), we have
$$
\sum_{i<j} \mathbb{E}[\AAA{\ell}_{ij}] \approx \tilde{p}_\ell\times \frac{(0.5+0.5C^{-1})n^2 - n}{2} \geq \tilde{p}_\ell\times \frac{0.5n^2 - n}{2}.
$$
So a conservative estimator for $p_\ell$ is given by
$$
\hat{p}_\ell = \frac{2\sum_{i<j} \mathbb{E}[\AAA{\ell}_{ij}]}{0.5n^2 -n}.
$$

\end{appendices}

\end{document}